\documentclass[11pt, a4paper, twoside]{article}%
\pdfoutput=1
\usepackage{lspaper}
\usepackage{genyoungtabtikz}
\usepackage{mfabacus,genyoungtabtikz}
\usepackage{tikzsymbols}
\usepackage{booktabs}
\usepackage{multirow, multicol}
\usepackage{tabularx}
\usepackage{placeins}
\usepackage[normalem]{ulem}
\usepackage{amssymb,amsmath}
\usepackage{pdflscape,afterpage}
\usetikzlibrary{matrix,arrows,decorations.pathmorphing,decorations.pathreplacing}
\usetikzlibrary{positioning,intersections,shapes.geometric,calc,math}
\usetikzlibrary{decorations.markings}
\RequirePackage{pxfonts}
\usepackage[toc,page]{appendix}
\usepackage{scalefnt,caption}
\usepackage[all, knot]{xy}
\usepackage{mathtools}
\setitemize{itemsep=0cm}
\usepackage{makecell}
\usepackage{mathtools}

\renewcommand\ge\geqslant
\renewcommand\le\leqslant
\renewcommand\succeq\succcurlyeq
\renewcommand\preceq\preccurlyeq
\usepackage{stackengine,graphicx}
\stackMath

\newcommand{\succprec}{%
  \mathrel{\stackon[-2pt]{\prec}{\succ}}%
}

\newcommand{\nsuccprec}{%
  \mathrel{\raisebox{-0.25ex}{%
    \ooalign{%
      $\succprec$\cr
      \hidewidth\raisebox{.4ex}{\rotatebox{-20}{\scalebox{1.1}{/}}}\hidewidth\cr
    }%
  }}%
}
\newcommand{\skipover}[1]{}

\newcommand{\Z}{\mathbb{Z}}

\newcommand{\defect}{\mathrm{def}}

\newcommand{\Ext}{\operatorname{Ext}}
\newcommand{\Rad}{\operatorname{Rad}}
\newcommand{\Soc}{\operatorname{Soc}}

\newcommand{\downoperation}{\operatorname{down}}
\newcommand{\upoperation}{\operatorname{up}}
\newcommand{\roof}{\operatorname{roof}}
\newcommand{\base}{\operatorname{base}}
\newcommand{\core}{\operatorname{core}}

\newcommand{\St}{\mathcal{S}}
\newcommand{\R}{\mathcal{R}}

\newcommand{\AB}{\mathscr{A}^2}
\newcommand{\Bc}{\mathscr{B}^{\circ}}
\newcommand{\Bl}{\mathscr{B}}
\newcommand{\Cl}{\mathscr{C}}

\newcommand{\Tl}{\mathscr{T}}
\newcommand{\cB}{\hat{B}}
\newcommand{\cC}{\hat{C}}
\newcommand{\cD}{\hat{D}}
\newcommand{\hk}{{\rm h}}
\newcommand{\nn}{{\rm n}}
\newcommand{\pp}{{\rm p}}
\newcommand{\TT}{{\rm T}}
\newcommand{\tb}{{\rm t}}

\newcommand{\Pn}{\mathscr{P}}
\newcommand{\Kl}{\mathscr{K}^{\bs}_n}

\DeclareMathOperator{\Shh}{Sh}

\newcommand{\ScA}{\rightarrow_{\text{Sc}^\ast}}
\newcommand{\Sw}{\equiv_{\text{Sw}}}
\newcommand{\Sh}{\equiv_{\text{Sh}}}
\newcommand{\Sc}{\equiv_{\text{Sc}}}
\newcommand{\SU}{\equiv_{\text{S}}}
\newcommand{\nB}[1][\cB]{\|#1\|}
\DeclareMathOperator{\Res}{Res}

\renewcommand\phi\varphi
\newcommand{\balph}{\boldsymbol{\alpha}}
\newcommand{\bbeta}{\boldsymbol{\beta}}
\newcommand{\bgam}{\boldsymbol{\gamma}}
\newcommand{\bdelta}{\boldsymbol{\delta}}
\newcommand{\beps}{\boldsymbol{\varepsilon}}
\newcommand{\bzeta}{\boldsymbol{\zeta}}

\newcommand{\bla}{\boldsymbol{\lambda}}
\newcommand{\bmu}{\boldsymbol{\mu}}
\newcommand{\bnu}{\boldsymbol{\nu}}
\newcommand{\bsig}{\boldsymbol{\sigma}}
\newcommand{\btau}{\boldsymbol{\tau}}
\newcommand{\ba}{\boldsymbol{a}}
\newcommand{\bs}{\boldsymbol{s}}
\newcommand{\vn}{\varnothing}
\newcommand{\ra}{\raisebox{0.1cm}{
\begin{tikzpicture}
\draw [<-] (0,0) to [out =60, in = 120] (.4,0);
\end{tikzpicture}}}

\renewcommand{\theenumi}{\roman{enumi}}

\usepackage{ytableau}

\def\mod{\mathsf{mod}\hspace{.05in}}
\def\proj{\mathsf{proj}\hspace{.05in}}

\def\Kb{\mathsf{K}^{\rm b}}

\def\add{\operatorname{\mathsf{add}}}

\newcommand{\tautilt}{\mbox{\sf $\tau$-tilt}\hspace{.05in}}
\newcommand{\stautilt}{\mbox{\sf s$\tau$-tilt}\hspace{.05in}}
\usepackage{listings}
\definecolor{codegreen}{rgb}{0,0.6,0}
\definecolor{codegray}{rgb}{0.5,0.5,0.5}
\definecolor{codepurple}{rgb}{0.58,0,0.82}
\definecolor{backcolour}{rgb}{0.95,0.95,0.92}
\lstdefinestyle{mystyle}{
  backgroundcolor=\color{backcolour}, commentstyle=\color{codegreen},
  keywordstyle=\color{magenta},
  numberstyle=\tiny\color{codegray},
  stringstyle=\color{codepurple},
  basicstyle=\ttfamily\footnotesize,
  breakatwhitespace=false,         
  breaklines=true,                 
  captionpos=b,                    
  keepspaces=true,                 
  numbers=left,                    
  numbersep=5pt,                  
  showspaces=false,                
  showstringspaces=false,
  showtabs=false,                  
  tabsize=2
}
\begin{document}

\title{Schurian-finiteness of blocks of type $\ttb$ Hecke algebras}

\author{%
\begin{tabular}{c c}
\begin{tabular}{c}
Susumu Ariki \\
\texttt{ariki@ist.osaka-u.ac.jp}
\end{tabular}
&
\begin{tabular}{c}
Sin\'ead Lyle \\
\normalsize University of East Anglia \\
\normalsize Norwich NR4 7TJ, UK \\
\texttt{s.lyle@uea.ac.uk}
\end{tabular}
\\[3em]
\begin{tabular}{c}
Liron Speyer \\
\normalsize Okinawa Institute of Science and Technology \\
\normalsize Okinawa, Japan 904-0495 \\
\texttt{liron.speyer@oist.jp}
\end{tabular}
&
\begin{tabular}{c}
Qi Wang \\
\normalsize Dalian University of Technology \\
\normalsize Dalian, China 116024 \\
\texttt{wang2025@dlut.edu.cn}
\end{tabular}
\end{tabular}
}

\renewcommand\auth{Susumu Ariki, Sin\'ead Lyle, Liron Speyer, \& Qi Wang}

\runninghead{Schurian-finiteness of blocks of type $\ttb$ Hecke algebras}


\toptitle

\begin{abstract}
Schurian-finiteness, also known as $\tau$-tilting finiteness, is equivalent to the finiteness of various representation theoretic objects such as wide subcategories. The first three authors classified Schurian-finite blocks of type $\tta$ Hecke algebras in~\cite{als23}. Here we study the Schurian-finiteness of blocks of type $\ttb$ Hecke algebras, and determine the Schurian-finiteness of all blocks if the Hecke algebra is `non-integral', and for almost all blocks in the integral case.
The only remaining cases are a small number of blocks in defect $3$ when $e=3$, and a family of blocks in defects $3$ and $4$ for $e\geq4$.
The classification is mostly achieved by methods using decomposition numbers, with many degenerate cases requiring direct study using standard methods from the representation theory of quivers.
\end{abstract}

\tableofcontents

\section{Introduction}

Let $\bbf$ denote an algebraically closed field.
For an $\bbf$-algebra $A$, an $A$-module $M$ is called Schurian, or a brick, if $\End_A(M) \cong \bbf$.
The algebra $A$ is called Schurian-finite (or brick-finite) if it has finitely many isoclasses of Schurian modules, and Schurian-infinite (or brick-infinite) otherwise.
Since Schurian modules are indecomposable, determining the Schurian-finiteness of an algebra may be seen as a refinement of determining whether it has finite representation type -- if $A$ has finite representation type, it is certainly Schurian-finite, but the converse need not hold.
In fact, whether $A$ is tame or wild, it may still be Schurian-finite or Schurian-infinite.
By \cite[Theorem~4.2]{DIJ}, Schurian-finiteness coincides with $\tau$-tilting finiteness.

Experts studying $\tau$-tilting theory usually start with an algebra given by its bound quiver presentation.
For those algebras where such a presentation is unavailable, results are quite sparse.
One such family of algebras without a known presentation by quiver and relations is the family of type $\tta$ Hecke algebras, which includes the group algebras of symmetric groups.
A major difficulty here is that, although a standard method for reducing the rank of a block algebra is to use induction and restriction functors, Schurian-finiteness does not behave well under those functors.

In \cite{als23} the first three authors determined the Schurian-finiteness of blocks of type $\tta$ Hecke algebras under the assumption that the quantum characteristic $e$ is not equal to $2$.
There, it was shown that, in fact, a block is Schurian-finite if and only if it has finite representation type, which is known by older work of Erdmann and Nakano~\cite{enreptype} to hold if and only if the block has \emph{defect} $0$ or $1$.
While some edge cases there required some ad hoc methods, our main tool was the study of graded decomposition matrices, coming from the Hecke algebras' incarnation as level 1 cyclotomic Khovanov--Lauda--Rouquier algebras in affine type $\tta$, thanks to the famous Brundan--Kleshchev isomorphism~\cite{bkisom}.
We showed that finding a certain submatrix of the graded decomposition matrix of a block suffices to show that the block is Schurian-infinite, and were able to find the necessary submatrices in most cases, completely avoiding the need to explicitly compute bound quiver presentations. 
One may use graded dimensions for that purpose, but the advantage of using decomposition numbers is that we may use many reduction techniques, including row-removal and diagonal cut results in order to find decomposition numbers.

The third author later used similar methods to extend the Schurian-finiteness result in \cite{lsstrictlywild}, showing that in almost all cases, these blocks of infinite representation type, previously known to be wild, are not only Schurian-infinite, but are in fact \emph{strictly wild}, an even stronger condition.

In this paper, we study the type $\ttb$ Hecke algebras $\hhh=\hhh(q,Q_1,Q_2)$, i.e., deformations of the type $\ttb$ Weyl groups, again assuming that we are not in quantum characteristic $e=2$ and that $q\ne 1$.
(See \cref{subsec:hecke} for the presentation of $\hhh$.)
Here, we note that $q$ and $Q_1, Q_2$ are non-zero elements in our ground field $\bbf$.
We say that $\hhh$ is integral if $Q_i = q^{a_i}$ for $a_i\in \Z$, $i=1,2$, and non-integral otherwise.
By \cite{bkisom}, the integral algebras are level 2 cyclotomic Khovanov--Lauda--Rouquier algebras in affine type $\tta$.
Our main tool to show that most blocks are Schurian-infinite is via graded decomposition matrices again -- see \cref{prop:matrixtrick}.
The situation is more delicate for these type $\ttb$ Hecke algebras, however.
Just as in type $\tta$, there is a block invariant known as the defect, which roughly indicates the complexity of the block, and a block $B$ of a type $\ttb$ Hecke algebra has finite representation type if and only if it has defect $0$ or $1$, which can be seen by interpreting \cite[Theorem 7.5]{arikirep} in terms of defect. 
One major new feature in this type $\ttb$ work is that we find some blocks of infinite representation type that are Schurian-finite, which was not the case in type $\tta$.

Those Schurian-finite wild blocks are worth studying in more detail, although we do not pursue that direction in this paper, because there has been considerable interest in the study of various kinds of subcategories of module categories~\footnote{Our previous paper \cite{als23} was the first instance of such a study for Hecke algebras.} such as torsion classes and wide subcategories, to name just two,  from the viewpoint of representation theory of algebras. 
The number of wide subcategories is finite for those blocks and they can be analysed in detail.

An obstacle we must overcome is that explicit computation of graded decomposition numbers is more difficult in type $\ttb$, and there are fewer existing results from which we can draw.
The graded decomposition numbers are now known to be characteristic-free for blocks of defect 2 and 3, by \cite{fay06,fayput24}, though only those in defect 2 have been computed explicitly.
Despite this difficulty, we are able to completely determine the Schurian finiteness of blocks of non-integral Hecke algebras, and we almost completely determine the Schurian-finiteness of blocks of integral Hecke algebras for $e\geq 3$, with some exceptions in defects $3$ and $4$ that remain to be resolved. Our main results for non-integral Hecke algebras and for integral Hecke algebras for $e=3$ are the following.

\begin{thm}\label{thm:mainnonintegral}
    Let $B$ be a block of a non-integral Hecke algebra $\hhh$.
    Then in any characteristic, $B$ is Schurian-infinite if and only if it has defect at least $2$.
    In particular, $B$ is Schurian-finite if and only if it has finite representation type.
\end{thm}

The above theorem appears as \cref{thm:mainnonintegralrepeated} in the body of the paper.
If $\hhh$ is integral, we may assume, up to isomorphism, that $Q_1 = q^0$ and $Q_2 = q^s$ for some $s\in \{0,1,\dots, e-1\}$.
When $e=3$, we may further assume that $s=0$ or $1$ -- see \cref{L:01}. The next theorem appears as \cref{thm:e3def2b,thm:e3def3s1,thm:2ormorehookscasesNot2,thm:2ormorehooksbadcases}.

\begin{thm}\label{thm:maine3}
    Suppose $e=3$, and that $\hhh$ is integral.
    If $B$ is a block of $\hhh$ of defect at least 4, then $B$ is Schurian-infinite in any characteristic.
    
    If $B$ is a block of defect $2$ and $s=0$, then $B$ is Schurian-finite.
    If $s=1$, then there are six Scopes classes.
    Two of these are Schurian-infinite, and the remaining four are Schurian-finite.
    
    If $B$ is a block of defect $3$ and $s=0$, then there are ten Scopes classes of blocks.
    Two of them are Schurian-infinite, and the other eight remain unsolved.     
    If $B$ is a block of defect $3$ and $s=1$, then $B$ is Schurian--infinite.
\end{thm}

For $e\geq 4$ and $\hhh$ integral, we expect the following result.

\begin{conj}\label{conj:e>3}
    Suppose $e\geq 4$, $\hhh$ is integral, and that $B$ is a block of $\hhh$ of defect at least 2.
    Then $B$ is Schurian-infinite in any characteristic.
    In particular, $B$ is Schurian-finite if and only if it has finite representation type.
\end{conj}

We are able to prove the above conjecture outside of some classes of blocks in defects $3$ and $4$.
Our main theorem for $e\geq4$ is as follows, and is proved as \cref{T:CoreBlocksInf1,T:nlarge,thm:2ormorehooks,T:Defect2,} in the main body of the paper.

\begin{thm}\label{thm:main}
    Suppose $e\geq 4$, $\hhh$ is integral, and that $B$ is a block of $\hhh$ of defect at least 2.
    Unless the defect of $B$ is 3, with core block $C$ of defect $1$, or $s = 0$ and the defect of $B$ is $4$ with core block $C$ of defect $2$, $B$ is Schurian-infinite in any characteristic.
\end{thm}

The authors are working to resolve the missing cases in defects 3 and 4, although we note that defect 4 is made more difficult since the decomposition numbers are no longer characteristic-free in this case.

As we mentioned above, we find examples, e.g.~\eqref{SF-alg-2}, \eqref{SF-alg-3}, \eqref{SF-alg-4} of blocks of defect $2$ when $e=3$, which have wild representation type, but are Schurian-finite.
We also find the first example we know of where two blocks are Morita equivalent and the equivalence is not obtained by a composition of Scopes equivalences and twists by the sign automorphism -- see \cref{rem:newMorita}.

We are also able to determine analogous results for blocks of level 2 cyclotomic $q$-Schur algebras.
We have the following conjecture for the classification there.

\begin{conj}\label{conj:mainSchuralg}
Let $e\geq 3$ and let $B$ be a block of the cyclotomic $q$-Schur algebra $\mathcal{S}_n$.
Then $B$ is Schurian-infinite if and only if $\defect(B) \geq 2$.
\end{conj}

Towards this conjecture, we are currently able to prove the following result, which appears as \cref{thm:Schuralg} at the end of the paper.
We note that the exceptions not covered by the theorem are as in \cref{thm:main}.

\begin{thm}\label{thm:mainSchuralg}
Let $e\geq 3$ and let $B$ be a block of the cyclotomic $q$-Schur algebra $\mathcal{S}_n$.
If $e=3$, then $B$ is Schurian-infinite if and only if $\defect(B) \geq 2$.

Suppose $e\geq 4$, and that the corresponding block of $\hhh$ is not a defect $3$ block with core block $C$ of defect $1$, and is not a defect $4$ block with core block $C$ of defect $2$ with $s=0$.
Then $B$ is Schurian-infinite in any characteristic.
\end{thm}

\begin{rem}
    There are related work on Schurian finiteness of the (entire) Schur algebras $S(n,r)$, and of blocks of the Schur algebras $S(n,r)$. See Wang~\cite{qiwang22} and Aoki--Wang~\cite{aokiwang21}.
\end{rem}

In order to find the necessary decomposition numbers to prove most of the cases in \cref{thm:maine3,thm:main}, we use the non-recursive characterisation of Kleshchev bipartitions in \cite{akt08} and row removal theorems from \cite{bs16} in order to reduce many cases to the type $\tta$ setting.
Many remaining cases, particularly those for defect 2 blocks when $e=3$, are resolved using standard methods from $\tau$-tilting theory. 
Indeed, almost all blocks are Schurian-infinite, and we do not know any other method than $\tau$-tilting theory to conclude that a block is Schurian-finite. Because of computational difficulties, we have succeeded only partially in those remaining cases. 

The paper is organised as follows.
In \cref{sec:hecke,sec:bipartitions and abacus}, we introduce the type $\ttb$ Hecke algebras and their blocks, defect of a block, special kinds of blocks known as core blocks, much of the necessary combinatorics of bipartitions, abacus displays, combinatorial blocks, combinatorial core blocks, and several combinatorial equivalence relations that induce Morita equivalences on blocks of our Hecke algebras.
We also quickly prove our main result, \cref{thm:mainnonintegral}, for blocks of non-integral type $\ttb$ Hecke algebras in \cref{subsec:nonintegral}.
Armed with this setup, we introduce Specht modules in \cref{sec:spechts} to relate blocks and combinatorial blocks, as well as the graded decomposition numbers, and explain how we use those to prove that blocks are Schurian-infinite, and explain the reduction techniques from \cite{bs16}.
In \cref{sec:Demazurecrystals}, we use \cite{akt08} to prove an important result in \cref{T:StillKlesh}, that allows us to generate Kleshchev bipartitions starting from one with an $e$-core in the second component and adding $e$-hooks to this second component only.

In \cref{sec:SIblocks}, we will prove most of the Schurian-infinite cases of \cref{thm:main}.
In particular, we start in \cref{subsec:SIcoreblocks}, showing that core blocks of defect at least $3$ are Schurian-infinite, as are those of defect $2$ if $s\neq0$; this can be found in \cref{T:CoreBlocksInf1}.
Next, we show in \cref{subsec:SIblockswithlargecoreblock} that any blocks built from those core blocks by adding $e$-hooks are also Schurian-infinite, which appears in \cref{T:nlarge}.
In \cref{subsec:ManyRemHooks}, we prove \cref{thm:2ormorehooks}, which gives us that outside of the case $(e,p)=(3,2)$, any block that contains some bipartitions with at least two removable $e$-hooks is Schurian-infinite in any characteristic.
Finally, in \cref{S:SIC2}, we prove \cref{T:Defect2}, which shows that defect 2 non-core blocks are Schurian-infinite when $s\neq0$.

In \cref{sec:def2}, we handle the remaining cases of \cref{thm:main}.
Namely, the remaining defect $2$ core blocks for $s=0$ are resolved in \cref{subsec:core-def-2-equal}, and the defect $2$ non-core blocks with $s=0$ are resolved in \cref{subsec:def2s1=s2}.

In \cref{sec:e3}, we handle the $e=3$ cases of \cref{thm:maine3}.
We first consider the blocks of defect $2$ and defect $3$. The results for defect $2$ are summarised in \cref{thm:e3def2b} while the results for defect $3$ are summarised in \cref{thm:e3def3s0,thm:e3def3s1}.
We resolve the blocks of larger defect in \cref{sec:e=3defect4+} -- in most cases we are able to reduce to type $\tta$ and solve these cases easily, using \cref{T:StillKlesh} and some `diagonal cuts' from \cite{bs16}.
For RoCK blocks of defect $4,5,6$, or $7$ when $p=2$, more care is needed and we handle these in the rest of the section.
We end the paper with \cref{sec:Schuralg}, in which we prove \cref{thm:mainSchuralg}.


\begin{ack}
The first author is partially supported by JSPS Kakenhi grant number \linebreak 21K03163. He thanks Kavli Institute for the Physics and Mathematics of the Universe for the current visiting position, and Okinawa Institute of Science and Technology for his two one month long visits by TSVP (Theoretical Sciences Visiting Program) in the academic year 2024. 
The third author is partially supported by JSPS Kakenhi grant number 23K03043.
The fourth author is partially supported by National Natural Science Foundation of China grant number 12401048 and Fundamental Research Funds for the Central Universities grant number DUT25RC(3)132.
We thank Peng Shan for telling us about \cite{Mak14}, which we needed for the proof of \cref{prop:matrixtrick}, and thank Matt Fayers for making his LaTeX style file for abacus displays publicly available.
We also thank Ben Webster for sharing his Sage code which computes Scopes classes of blocks, which we initially used for \cref{sec:e=3defect4+}. 
We thank Tsinghua University and Yu Qiu for their hospitality during a visit in April 2025, where much progress on the project took place.
\end{ack}

\section{The Hecke algebra of type $\ttb$}\label{sec:hecke}

\subsection{The Hecke algebra}\label{subsec:hecke}

Throughout, we let $\bbf$ denote an algebraically closed field of characteristic $p\geq 0$, and all our modules are left modules unless otherwise stated.
Take $q,Q_1,Q_2 \in \bbf^\times$ and suppose $n \geq 0$.   
The Hecke algebra $\hhh=\hhh(q,Q_1,Q_2)$ of type $\ttb$ is the unital associative $\bbf$-algebra with generators $T_0,T_1,\dots,T_{n-1}$ and relations 
\begin{gather*}
(T_0-Q_1)(T_0-Q_2)=0, \quad (T_i-q)(T_i+1)=0\;\;(1\le i\le n-1),\\
T_0T_1T_0T_1=T_1T_0T_1T_0, \quad T_iT_{i+1}T_i=T_{i+1}T_iT_{i+1}\;\;(1\le i\le n-2),\\
T_iT_j=T_jT_i\;\; (0\le i\le j-2\le n-3).
\end{gather*}

The algebra deforms the group algebra of the Weyl group of type $\ttb_n$. If $q\ne 1$, it has graded algebra structure through the Brundan--Kleshchev isomorphism theorem \cite{bkisom}. 
When $q=1$, we instead must work with the \emph{degenerate} cyclotomic Hecke algebra of type $\ttb$ for introducing the graded algebra structure, as the above presentation gives an algebra with very different behaviour in this case.
For example, it has more simple modules (c.f.~\cite[Theorem~3.7]{m98}).
As we use the graded algebra structure in an essential way, we assume \emph{$q\ne 1$} throughout the paper.\footnote{Computations in this paper which use only the defining relations of cyclotomic quiver Hecke algebras are valid even when $q=1$ and may be viewed as 
computations for degenerate cyclotomic Hecke algebras.}

\begin{rem}
   There exists a unified presentation for these algebras given by Hu and Mathas~\cite[Definition~2.2]{hmseminorm}, and details can also be found in \cite[\S1.1]{m14surv}. 
\end{rem}

The quantum characteristic of $\hhh$ is the smallest positive integer $e$ such that 
\[
1 + q + q^2 + \dots + q^{e-1} = 0, 
\]
if such an $e$ exists, and we set $e = \infty$ otherwise.
In fact, the results of this paper only hold for $e \geq 3$, so we will later include this further assumption.
Throughout the paper, we will often identify $\Z/e\Z$ and $\{0,1,\dots,e-1\}$ if $e$ is finite. 
We note that if $e=\infty$, every block of $\hhh$ is a core block, and our main results for these cases are resolved in \cref{subsec:SIcoreblocks,subsec:core-def-2-equal}.

We will call $\hhh$ \emph{integral} if $Q_1 = q^s Q_2$ for some $s \in \Z$, and \emph{non-integral} if $Q_1 \neq q^s Q_2$ for all $s \in \Z$.
In~\cite[Theorem~4.17]{dj92}, Dipper and James showed that if $\hhh$ is non-integral, then it is Morita equivalent to a sum of tensor products of type $\tta$ Hecke algebras; we note that the Schurian-infinite type $\tta$ Hecke algebras for $e \ne 2$ were classified by the first three authors~\cite{als23}, and we may analyse the Schurian-finiteness of the tensor product.
We resolve this quickly in \cref{subsec:nonintegral}.
Hence for most of the paper, we may assume that $Q_1=q^a Q_2$ for some $a \in \Z$; then we can rescale so that there exist $a_1, a_2\in \Z$ such that $Q_k=q^{a_k}$ for $k=1,2$.

\begin{rem}
We note that if $q=1$, the definition of integral (for the degenerate Hecke algebra) is a little less clean; however, an analogous Morita equivalence theorem can be derived by following the arguments of \cite{dm02} using the aforementioned unified presentation of Hu--Mathas, as mentioned in \cite[\S1.2]{m14surv}.
\end{rem}

\begin{defn}
    A Hecke algebra with bicharge is a pair $(\hhh, \ba)$, where $\ba=(a_1,a_2)\in \Z^2$ and $\hhh=\hhh(q,q^{a_1},q^{a_2})$.
\end{defn}

Note that the bicharge $\ba=(a_1,a_2)$ is an ordered pair. 
For $k=1,2$, we often let $s_k \in \Z/e\Z$ denote the equivalence class of $a_k$ modulo $e$.
Under the identification of $\Z/e\Z$ and $\{0,1,\dots,e-1\}$ above, we may assume $0 \leq s_1,s_2 <e$.

\begin{rem}
With further rescaling, we could assume that $s_1=0$. We do not make this assumption, firstly because we are motivated by the study of the Ariki--Koike algebras~\cite{ak94}, and secondly because in some later sections it will be convenient to also allow $1 \le s_1 <e$. 
\end{rem}

\subsection{Blocks}\label{subsec:blocks}

A block (of $\hhh$) is an indecomposable two-sided ideal which is a direct summand of the $(\hhh,\hhh)$-bimodule $\hhh$. 
Lyle and Mathas \cite{lm07} proved that 
blocks of $\hhh$ are labelled by the weight poset of an integrable highest weight module over $\mathfrak{g}(\tta^{(1)}_{e-1})$. Let $P$ be the weight lattice for the Kac--Moody Lie algebra $\mathfrak{g}(\tta^{(1)}_{e-1})$, and let $\La_s$, for $s\in \Z/e\Z$, be the fundamental weights. We set $\La=\La_{s_1}+\La_{s_2}$ here. 
We denote by $V(\La)$ the integrable highest weight module with highest weight $\La$, and by $V(\La)_\mu$, for $\mu\in P$, weight spaces. The weight poset of $V(\La)$ is 
\[
P(\La)=\{ \mu \in P \mid V(\La)_\mu\neq0 \}. 
\]
Let $Q_+=\Z_{\ge0}\alpha_0+\dots+\Z_{\ge0}\alpha_{e-1}$ be the positive cone of the root lattice. Then
\[
P(\La) \subset \{ \Lambda-\beta \mid \beta \in Q_+\}.
\]
Let $\delta$ be the null root. 
The set of maximal weights is defined to be 
\[
\max(\La)=\{ \mu\in P(\La) \mid \mu+\delta\not\in P(\La)\}.
\]

\begin{defn} \label{D:Core1}
   We call blocks labelled by $\max(\La)$ core blocks.
   Let $B$ be the block labelled by $\mu\in P(\La)$.
   Then, there exists a non-negative integer $k\in\Z$ such that $\mu+k\delta\in \max(\La)$. Let $C$ be the block labelled by $\mu+k\delta$. We call $C$ the core block of $B$.
\end{defn}

The Brundan--Kleshchev isomorphism theorem~\cite[Main Theorem]{bkisom} tells us that each block is isomorphic to a cyclotomic quiver Hecke algebra, also known as a cyclotomic Khovanov--Lauda--Rouquier algebra, which we denote by $R^\La(\beta)$ if the block is labelled by $\La-\beta\in P(\La)$.
We do not give the defining relations of $R^\La(\beta)$ here, but we will need the explicit defining relations for $e=3$.
See \cref{Cyclotomic quiver Hecke algebras}. 
Further, the isomorphism theorem shows that blocks are graded algebras, and we may define graded decomposition numbers. We will utilise the graded algebra structure and graded decomposition numbers in essential ways. 

\begin{rem}
    We can show that $R^\La(\beta)=0$ when $\La-\beta\not\in P(\La)$. Thus, $R^\La(\beta)$, for $\La-\beta\in P(\La)$, is a core block if and only if $R^\La(\beta-\delta)=0$. Note that the label $(\La, \beta)$ defines the block directly, and the definition of core block does not require any additional structure such as cellular algebra structure. 
\end{rem}

\begin{defn}
The defect of the block labelled by $\La-\beta\in P(\La)$
is defined to be the defect of $\beta$: 
\[
\defect_\La(\beta)=(\La \mid \beta)-\frac{1}{2}(\beta \mid \beta).
\]
We denote the defect of a block $B$ by $\defect(B)$.
\end{defn}

Then $\defect(B)=0$ if and only if the block $B$ is simple.
In general, the higher the defect, the more complicated the block tends to be.
From the perspective of quiver Hecke algebras, it is clear that a Hecke algebra with quantum characteristic $e=\infty$ is just isomorphic to one with finite $e$, so long as $e$ is chosen to be large enough, since the presentation of $R^\La(\beta)$ is identical in these two instances.
One standard idea for studying representations of Hecke algebras is to use a combinatorial block, which is a set of bipartitions together with a bicharge.
For this, we need an additional structure on blocks. Namely, we
will need the notion of block for Hecke algebras with bicharge, which is defined as follows. 
The link from blocks of the Hecke algebra with bicharge to combinatorial blocks will be explained in \cref{L:CoresMatch}.

\begin{defn}
    Let $(\hhh,\ba)$ be a Hecke algebra with bicharge. A block of $(\hhh,\ba)$ is a pair $(B,\ba)$ where $B$ is a block of $\hhh$ in the above sense. 
\end{defn}

\subsection{Non-integral Hecke algebras}\label{subsec:nonintegral}

Suppose $e\ge3$, and that $\hhh$ is a non-integral type $\ttb$ Hecke algebra, that is, $Q_1 \ne q^\Z Q_2$.
Then $\hhh$ is Morita equivalent to the direct sum of tensor products of Hecke algebras of type $\tta$, so that any block is Morita equivalent to the tensor product of one block $A$ from the Hecke algebra of the symmetric group $S_m$ and another block $B$ from the Hecke algebra of the symmetric group $S_{n-m}$. 

There is a notion of weight, or defect, of a block of a type $\tta$ Hecke algebra, and in this non-integral case the defect of a type $\ttb$ block is naturally the sum of the defects of the corresponding type $\tta$ blocks.
We are now ready to prove \cref{thm:mainnonintegral}, which we restate below for the reader's convenience.

\begin{thm}\label{thm:mainnonintegralrepeated}
    Let $B$ be a block of a non-integral Hecke algebra $\hhh$.
    Then in any characteristic, $B$ is Schurian-infinite if and only if it has defect at least $2$.
    In particular, $B$ is Schurian-finite if and only if it has finite representation type.
\end{thm}

\begin{proof}
First, observe that $A\otimes B$ as above is representation-finite if and only if the sum of the defect of $A$ and $B$ is less than or equal to $1$.
Thus, the blocks in those cases are Schurian-finite.

If the defect of $A$ is greater than or equal to $2$, we choose a primitive idempotent $e\in B$ and consider the local algebra $eBe$. 
Then $A\otimes eBe$ is Schurian-infinite by \cite{als23}.
It follows that $A\otimes B$ is Schurian-infinite. 

Finally, we consider the case both $A$ and $B$ are defect $1$ blocks. Since those blocks are derived equivalent to the principal block of the Hecke algebra of the symmetric group $S_e$, the number of simple $A$-modules and the number of simple $B$-modules are both $e-1\ge 2$. 
Moreover, it is known that defect $1$ blocks are Morita equivalent to Brauer line algebras. Thus, we may assume $A=B$. 
We choose an idempotent $e\in A$ 
such that the Gabriel quiver of $eAe$ is the following quiver.
\begin{center}
$\vcenter{\xymatrix@C=1cm{\circ \ar@<0.5ex>[r]&\circ \ar@<0.5ex>[l]}}$
\end{center}
The tensor product of the path algebra of affine Lie type $\tta^{(1)}_1$ 
modulo the radical square with itself is a factor algebra of $eAe\otimes eAe$, and the path algebra of affine Lie type $\tta^{(1)}_3$ with zigzag orientation 
is a factor algebra of that tensor product of the radical square algebras, so that $eAe\otimes eAe$ is Schurian-infinite.
\end{proof}

\section{Bipartitions}\label{sec:bipartitions and abacus}
In this section, we introduce the combinatorics that we will use in the paper. This section is purely combinatorial; the items introduced will be shown to have algebraic significance later. Throughout~\cref{sec:bipartitions and abacus}, we fix $e \ge 2$ and we continue to identify the sets $\Z/e\Z$ and $\{0,1,\dots,e-1\}$. 

\subsection{Bipartitions and abacus combinatorics}\label{subsec:bipartition}
A partition is a weakly decreasing sequence $\la=(\la_1,\la_2,\dots)$ of non-negative integers such that $\la_i=0$ for all sufficiently large values of $i$. If $\sum_{i \geq 1} \la_i=n$, then we say that $\la$ is a partition of $n$ and write $|\la|=n$. Let $\Pn_n$ denote the set of partitions of $n$ and let $\Pn$ denote the set of all partitions. A bipartition $\bla=(\la^{(1)},\la^{(2)})$ is then an ordered pair of partitions. We define $|\bla|=|\la^{(1)}|+|\la^{(2)}|$, and if $|\bla|=n$, then we say that $\bla$ is a bipartition of $n$. Let $\Pn^2_n$ denote the set of bipartitions of $n$ and let $\Pn^2$ denote the set of all bipartitions. We define a partial order $\dom$ on $\Pn^2_n$ by saying that
\[
\bla \dom \bmu \iff \sum_{u=1}^s |\la^{(u)}| + \sum_{i=1}^t \la^{(s+1)}_i \geq \sum_{u=1}^s |\mu^{(u)}| +  \sum_{i=1}^t \mu^{(s+1)}_i \text{ for all } s \in \{0,1\} \text { and } t \geq 0.
\]
Suppose $\bla \in \Pn^2$. We define the Young diagram of $\bla$ to be the set
\[
[\bla]=\{(r,c,k) \in \Z_{>0} \times \Z_{>0}\times \{1,2\} \mid c \leq \la^{(k)}_{r}\}.
\]
We call the elements of $\Z_{>0} \times \Z_{>0}\times \{1,2\}$ nodes and say that a node $\mathfrak{n}\notin [\bla]$ is an addable node of $[\bla]$ if $[\bla] \cup \{\mathfrak{n}\}$ is the Young diagram of a bipartition.  
The rim of $\bla$ is the set of nodes $(r,c,k)\in [\bla]$ such that $(r+1,c+1,k) \notin [\bla]$.
For $h \geq 1$, an $h$-rim hook is a connected set of $h$ nodes contained in the rim of $\bla$ such that when they are removed from $[\bla]$ the remaining nodes give the Young diagram of a bipartition.  

Now let $\ba=(a_1,a_2)\in \Z^2$. The $e$-residue diagram of $\bla$ with respect to $\ba$ is the diagram obtained by replacing each node $\mathfrak{n}=(r,c,k) \in [\bla]$ with $\res_{\ba}(\mathfrak{n})=c-r+a_k \pmod e$. If $\res_{\ba}(\mathfrak{n})=i$, we call $\mathfrak{n}$ an $i$-node. 
Define $\Res_{\ba}(\bla)$ to be the multiset $\{\res_{\ba}(\mathfrak{n}) \mid \mathfrak{n} \in [\bla]\}$ and define an equivalence relation $\sim_{\ba}$ on $\Pn^2$ by saying that 
\[\bla \sim_{\ba} \bmu \quad\text{ if and only if }\quad \Res_{\ba}(\bla)=\Res_{\ba}(\bmu).\]
Note that $\bla \sim_{\ba} \bmu$ implies that $|\bla|=|\bmu|$. 

\begin{eg}
Let $e=3$ and $\ba=(0,2)$. Take $\bla=((6,5,4,1),(3,1))$ and $\bmu=((6,3),(5,4,2))$. Then the $e$-residue diagrams of $\bla$ and $\bmu$ are given by 
\[\bla: \; \young(012012,20120,1201,0)
\quad \young(201,1)\;, \qquad \qquad \bmu: \; \young(012012,201)\quad \young(20120,1201,01)\;, \] 
and $\Res_{\ba}(\bla)=\Res_{\ba}(\bmu)=\{0,0,0,0,0,0,0,1,1,1,1,1,1,1,2,2,2,2,2,2\}$ so that $\bla \sim_{\ba} \bmu$.  
\end{eg}

If $\la \in \Pn$, define its conjugate to be the partition $\la'=(\la'_1,\la'_2,\dots)$ where $\la'_i=\#\{j \geq 1 \mid \la_j \geq i\}$, so that we obtain $[\la']$ from $[\la]$ by reflecting along the main diagonal. 
If $\bla=(\la^{(1)},\la^{(2)}) \in \Pn^2$, define its conjugate to be the bipartition $\bla'=({\la^{(2)}}',{\la^{(1)}}')$. 

Suppose that $\la\in \Pn$ and $a \in \Z$. Define the $\beta$-set 
\[\beta_a(\la)=\{\la_i-i+a \mid i \ge 1\}.\]
We define the abacus display of $\la$ with respect to $a$ as follows. Draw an abacus with $e$ runners labelled $0,1,\dots,e-1$ from left to right extending infinitely far in both directions. Label the positions on the abacus with the elements of $\Z$, going from left to right across the runners and then top to bottom down the runners so that runner $i$ contains the positions congruent to $i$ modulo $e$. Finally put a bead in each of the positions corresponding to an element of $\beta_a(\la)$. Our convention when drawing the abacus configurations is that all rows above the portion of the diagram given contain a bead on each runner, and there are no beads below the portion of the diagram given. It will often be helpful to indicate position $0$ on the abacus, and we do so by drawing a line that separates the negative and non-negative positions.

\begin{eg} Let $\la=(6,4^2,3,2,1)$ and $a=14$. Then $\beta_a(\la)=\{19,16,15,13,11,9,7,6,\dots\}$.
Take $e=5$.
Below we label the positions on the abacus with $5$ runners, and then give the abacus configuration of $\la$ with respect to $a$. 

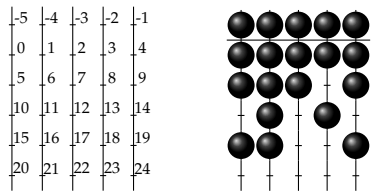
\begin{center}
\begin{tikzpicture}[xscale=0.4,yscale=0.4]
\foreach \x in {0,...,4}{
\draw(\x,-1)--(\x,5.1);
\foreach \y in {0,1,...,5}{
\draw(\x-0.1,\y-.5)--(\x+0.1,\y-.5);
\pgfmathsetmacro{\z}{int(\x+5*(5-\y)-5)};
\node at(\x+.3,\y-.25){\tiny{\z}};}}
\end{tikzpicture}
\qquad
\raisebox{1.06cm}{\abacusline(5,0,bbbbb,bbbbb,bbbnb,nbnbn,bbnnb,nnnnn)}
\end{center}
\end{eg}

Now suppose that $\bla=(\la^{(1)},\la^{(2)}) \in \Pn^2$ and $\ba=(a_1,a_2) \in\Z^2$. Define the abacus configuration of $\bla$ with respect to $\ba$ to be the pair of abacus configurations consisting of the abacus configuration of $\la^{(k)}$ with respect to $a_k$ for $k=1,2$. 
Let $\AB$ denote the set of pairs of abacus configurations. Then we have a bijection
\[\Pn^2 \times \Z^2 \longleftrightarrow \AB.\]
Henceforth, we identify the two sets, so that we may write $(\bla,\ba) \in \AB$. 
We define an equivalence relation $\sim$ on $\AB$ by saying that
\[
(\bla,\ba) \sim (\bmu,\ba') \quad \text{ if }\quad  \ba=\ba' \text{ and } \bla \sim_{\ba} \bmu.
\]
Let $\Bl$ denote the set of $\sim$-equivalence classes of $\AB$.

\begin{defn}
We refer to the elements of $\Bl$ as combinatorial blocks. If $\cB\in \Bl$, then each element of $\cB$ is equal to $(\bla,\ba)$ for some fixed $\ba \in \Z^2$; we say that $\ba$ is the bicharge of $\cB$. In fact, if $\cB$ is a combinatorial block and its bicharge $\ba$ is understood, we will often abuse notation by writing $\bla \in \cB$ rather than $(\bla,\ba) \in \cB$. If $\bla,\bmu \in \cB$ then $|\bla|=|\bmu|$ and we write
$\nB=|\bla|$. 
\end{defn}

\subsection{The notion of core blocks and defect for $\cB$}
Let $(\bla,\ba) \in \AB$.
It is well-known that removing an $e$-rim hook from $[\bla]$ corresponds to moving a bead up one position in the abacus configuration for $\bla$.
We now introduce the notion of core blocks for combinatorial blocks.
In \cref{L:CoresMatch}, we will see that \cref{D:Core1,D:Core2} are compatible.

\begin{defn} \label{D:Core2}
We say that $\bla \in \Pn^2$ is a bicore if it has no removable $e$-rim hooks, or equivalently, if there is no bead in the abacus configuration of $\bla$ with an empty space above it.
We say that $\cC \in \Bl$ is a (combinatorial) core block if every $(\bla,\ba) \in \cC$ is such that $\bla$ is a bicore. 
\end{defn}

We note that if $\cB \in \Bl$ then it is possible to have $(\bla,\ba),(\bmu,\ba) \in \cB$ with $\bla$ a bicore but $\bmu$ not. 

\begin{defn} \label{D:Core3}
For $\bla \in \Pn^2$, define $\bar{\bla}$ to be the bipartition whose Young diagram is obtained by repeatedly removing $e$-rim hooks from $[\bla]$ until we obtain a bicore and set $\hk(\bla)$ to be the number of hooks removed in this process; thus $|\bla|=|\bar{\bla}|+\hk(\bla)e$. For $\cB \in \Bl$, define
\[
\hk(\cB) = \max\{\hk(\bla) \mid (\bla,\ba) \in \cB\}.
\]
\end{defn}

\begin{lemc}{fay07core}{Theorem~3.1}
Suppose that $\cB\in \Bl$. Then 
\[\cC=\{(\bar{\bla},\ba) \mid (\bla,\ba) \in \cB \text{ and } \hk(\bla)=\hk(\cB)\}\]
is a core block. We refer to $\cC$ as the core block of $\cB$.
\end{lemc}

Note that if $\cB \in \Bl$ with core block $\cC$ then
\[\nB= \nB[\cC]+\hk(\cB)e.\]

\begin{defn}
Suppose $(\bla,\ba) \in \AB$ with $\bla$ a bicore. Let $\beta_{a_k}(\la^{(k)})$ denote the $\beta$-set of $\la^{(k)}$ with respect to $a_k$, for $k=1,2$. 
For $i \in \{0,1,\dots,e-1\}$ and $k \in \{1,2\}$, set
\[
b^{\ba}_{ik}(\bla) = \frac{\max\{b \in \beta_{a_k}(\la^{(k)}) \mid b \equiv i \pmod e\} - i}{e}.
\]
Thus $b^{\ba}_{ik}(\bla)$ is the lowest row of the abacus configuration for $\la^{(k)}$ with respect to $a_k$ that contains a bead on runner $i$.
\end{defn}

The next definition is due to Fayers~\cite[\S2.1]{fay06wts}.  

\begin{defn}\label{def:defect}
Take a combinatorial block $\cB \in \Bl$ with bicharge $\ba$. If $(\bla,\ba),(\bmu,\ba) \in \cB$, then $\Res_{\ba}(\bla)=\Res_{\ba}(\bmu)$. For $a\in\bbz$, let $c_a$ be the number of times that $a \pmod e$ occurs in $\Res_{\ba}(\bla)$, for $(\bla,\ba) \in \cB$. We define the defect of $\cB$ to be
\[
\defect(\cB) = c_{a_1} + c_{a_2} - \frac{1}{2} \sum_{i=0}^{e-1} (c_i-c_{i+1})^2.
\]
\end{defn}

Recall that $s_k$ is the equivalence class of $a_k$ modulo $e$ for $k=1,2$ and blocks are labelled by weights in $P(\La)$. If we set $\beta=\sum_{i=0}^{e-1} c_i\alpha_i$ and $\La =\La_{s_1}+\La_{s_2}$ then 
\[
\defect(\cB)=\defect_\La(\beta)=\defect(R^\La(\beta)).
\]
In \cref{sec:spechts} below, we use Specht modules to relate blocks, which are algebras, and combinatorial blocks, which are sets of bipartitions. Then, the block $B=R^\La(\beta)$ corresponds to $\cB$ and this assignment preserves defect.

\begin{lemc}{fay06wts}{Corollary~3.4}
Suppose $\cB \in \Bl$ and that $\cC$ is the core block of $\cB$. Then
\[\defect(\cB) = \defect(\cC) + 2 \hk(\cB).\] 
\end{lemc}

\subsection{Equivalences} \label{SS:Equivs}
In this section, we define four equivalence relations on $\Bl$. These are
\begin{itemize}
\item[$\Sw$:] The swap equivalence,
\item[$\Sh$:] The shift equivalence,
\item[$\Sc$:] The Scopes equivalence,
\item[$\SU$:] The equivalence relation generated by $\Sh$ and $\Sc$.
\end{itemize}

The first two relations are easy to describe. 

\begin{defn}
Suppose that $\cB,\cB' \in \Bl$. Then $\cB \Sw \cB'$ if and only if $\cB=\cB'$ or $\cB'=\{(\bla^{\text{Sw}},\ba^{\text{Sw}}) \mid (\bla,\ba) \in \cB\}$, where $(\la^{(1)},\la^{(2)})^{\text{Sw}}=(\la^{(2)},\la^{(1)})$ and $(a_1,a_2)^{\text{Sw}}=(a_2,a_1)$.
\end{defn}

\begin{defn}
Suppose that $\cB,\cB' \in \Bl$ and that their respective bicharges are $\ba=(a_1,a_2)$ and $\ba'=(a'_1,a'_2)$. Then $\cB \Sh \cB'$ if and only if $a_1-a_2 \equiv a'_1-a'_2 \pmod e$ and $\cB'=\{(\bla,\ba') \mid (\bla,\ba) \in \cB\}$. 
\end{defn}

We now move on to Scopes equivalence. 

\begin{defn} \label{D:Scopes}
Let $\la \in \Pn$, $a \in \Z$ and suppose $0 \le i <e$. Let $\beta=\beta_{a}(\la)$ and define 
\[
\begin{aligned}
\beta^i=\{b & \not \equiv i-1,i \pmod e \mid b \in \beta\} \\
&\cup \{b \equiv i \pmod e \mid b-1 \in \beta, b \notin \beta\}\\
&\cup 
\{b \equiv i-1 \pmod e \mid b \notin \beta, b+1 \in \beta\}.
\end{aligned}
\]

Define $\Phi_i^a(\la)$ to be the partition with $\beta$-set $\beta^i$ (with respect to $a$).
Now suppose $(\bla,\ba) \in \AB$ where $\bla=(\la^{(1)},\la^{(2)})$ and $\ba=(a_1,a_2)$. Define 
\[
\Phi_i(\bla,\ba)=((\Phi_i^{a_1}(\la^{(1)}),\Phi_i^{a_2}(\la^{(2)})),\ba).
\]
\end{defn}

The next result follows from~\cite[Proposition~4.6]{fay06wts}.

\begin{lem}
Suppose $(\bla,\ba),(\bmu,\ba') \in \AB$. Then 
\[
(\bla,\ba) \sim (\bmu,\ba') \iff \Phi_i(\bla,\ba) \sim \Phi_i(\bmu,\ba')
\]
that is, if $\cB \in \Bl$ then $\Phi_i(\cB) =\{\Phi_i(\bla,\ba) \mid (\bla,\ba) \in \cB\}\in \Bl$. 
\end{lem}

\begin{defn} \label{D:ScopesMove}
If $\cB,\cB' \in \Bl$ and $0 \leq i < e$, then we write 
$\cB \rightarrow_{\text{Sc}_i} \cB'$ if the following conditions are satisfied. 
\begin{enumerate}
\item If $(\bla,\ba) \in \cB$ then $[\bla]$ has no addable $i$-nodes; and 
\item $\cB'=\Phi_i(\cB)$.
\end{enumerate}
We write $\cB \rightarrow_{\text{Sc}} \cB'$ if there exists $0 \leq i <e$ such that $\cB \rightarrow_{\text{Sc}_i} \cB'$. 
Let $\Sc$ be the equivalence relation generated by $\rightarrow_{\text{Sc}}$.
\end{defn}

The easiest way to think about $\Phi_i(\bla,\ba)$ is to work with the abacus. Consider the abacus configuration for $(\bla,\ba)$. If $1 \leq i <e-1$, then we obtain the abacus configuration for $\Phi_i(\bla,\ba)$ by swapping runners $i$ and $i-1$; the condition that $[\bla]$ has no addable $i$-nodes means that every bead on runner $i-1$ of the abacus configuration has a bead directly to its right. We obtain the abacus configuration for $\Phi_0(\bla,\ba)$ by swapping runners $0$ and $e-1$ and then moving all beads on the new runner $0$ down one position and all beads on the new runner $e-1$ up one position. The condition that $[\bla]$ has no addable $0$-nodes means that, for every bead on runner $e-1$ of the abacus configuration, there is a bead in the row one lower on runner $0$.  

In fact, it has been shown by Dell'Arciprete that when considering the first condition for $\cB \rightarrow_{\text{Sc}_i} \cB'$ above, it is sufficient to consider only a subset of $\cB$.

\begin{lemc}{d'ascopes}{Theorem~5.2} \label{L:DARed}
Let $\cB \in \Bl$ with bicharge $\ba$.
Let $\cB_\hk=\{(\bmu,\ba) \in B \mid \hk(\bmu)=\hk(B)\}$. If there exists $(\bla,\ba)\in \cB$ such that $[\bla]$ has an addable $i$-node, then there exists $(\bmu,\ba) \in \cB_\hk$ such that $[\bmu]$ has an addable $i$-node.  
\end{lemc}

Recall that the core block of $\cB$ is equal to $\{(\bar{\bmu},\ba)\mid \hk(\bmu)=\hk(B)\}$, so that when we test condition (i) above, it is sufficient to consider only bipartitions which are obtained by adding hooks to bipartitions in the core block. 

\begin{cor} \label{C:DARed}
Let $\cB \in \Bl$ with bicharge $\ba$ and core block $\cC$. 
\begin{itemize}
\item Suppose $0 < i \leq e-1$ and let $\cB'=\Phi_i(\cB)$. Then $\cB \rightarrow_{\text{Sc}_i} \cB'$ if and only if there do not exist $(\bmu,\ba) \in \cC$ and $k \in \{1,2\}$ with
\[b^{\ba}_{ik}(\bmu) - b^{\ba}_{i-1k}(\bmu) < \hk(\cB).\] 
\item Suppose $i=0$ and let $\cB'=\Phi_i(\cB)$. Then $\cB \rightarrow_{\text{Sc}_i} \cB'$ if and only if there do not exist $(\bmu,\ba) \in \cC$ and $k \in \{1,2\}$ with
\[b^{\ba}_{0k}(\bmu) - b^{\ba}_{e-1k}(\bmu) < \hk(\cB)+1.\] 
\end{itemize}
\end{cor}

\begin{defn}
We define the equivalence relation $\SU$ on $\Bl$ to be the relation generated by $\Sh$ and $\Sc$. 
\end{defn}

\begin{prop} \label{P:defectpreserve}
Suppose that $\cB,\cB' \in \Bl$ and that $\cB \SU \cB'$ or $\cB \Sw \cB'$. Then $\defect(\cB)=\defect(\cB')$.  
\end{prop}

\begin{proof}
It is obvious that the relations $\Sw$ and $\Sh$ preserve the defect and it follows from~\cite[Proposition~4.6]{fay06wts} that $\Sc$ also does.   
\end{proof}

We return to the relations $\Sc$ and $\SU$ in~\cref{S:ScopesSec}, where we find a set of $\Sc$-representatives and $\SU$-representatives in $\Bl$. The utility of the relations is shown in~\cref{ScopesEquivBlocks,T:Equiv1}.

\section{Specht modules associated with bicharge}\label{sec:spechts}
Throughout this section, we assume that $(\hhh, \ba)$ is a Hecke algebra with bicharge defined over a field of characteristic $p \ge 0$ and that $\hhh$ has quantum characteristic $e \ge 2$. Let $\bs=(s_1,s_2)$, where $s_k$ is the equivalence class of $a_k$ modulo $e$ for $k=1,2$. 

\subsection{Specht modules}
For $(\hhh, \ba)$, we may define a cellular algebra structure on $\hhh$ in the sense of Graham and Lehrer~\cite{GL98}. The cellular algebra structure depends only on $\bs$. 
For each bipartition $\bla \in \Pn^2_n$, one may construct a cell module $\rspe{\bla}$, known as a Specht module.
If $\hhh$ is semisimple, the set $\{\rspe\bla \mid \bla \in \Pn^2_n\}$ is a complete set of pairwise non-isomorphic irreducible $\hhh$-modules.
Otherwise, there exists a subset $\Kl\subseteq \Pn^2_n$ such that for each $\bmu \in \Kl$, the Specht module $\rspe{\bmu}$ has a simple head, denoted by $\rD{\bmu}$; furthermore the set $\{\rD{\bmu} \mid \bmu \in \Kl\}$ is a complete set of pairwise non-isomorphic irreducible $\hhh$-modules. We call $\Kl$ the set of Kleshchev bipartitions of $n$. Note that the set $\Kl$ depends only on $e$ and $\bs$. 

Recall that if $A=B_1\oplus\dots\oplus B_s$ is the block decomposition and $M$ is an indecomposable $A$-module, then 
there is a unique $i$ such that $B_iM\neq0$ and $B_jM=0$, for $j\neq i$. This does not hold for decomposable modules in general, but if it holds for a decomposable $A$-module $M$, we say that $M$ belongs to $B_i$. 

Specht modules may be decomposable.  
For type $\tta$ Hecke algebras, Specht modules are indecomposable if $e\geq3$, but it is no longer true for higher level cyclotomic Hecke algebras. 
Replacing the quadratic relation $(T_0-Q_1)(T_0-Q_2)=0$ with $(T_0-Q_1) \dots (T_0-Q_m)=0$, we obtain the cyclotomic Hecke algebra associated with the complex reflection group $G(m,1,n)$. 
For cyclotomic Hecke algebras with $Q_i=q^{a_i}$, all Specht modules are indecomposable if $e\neq 2$ and $a_i\not\equiv a_j$ $\pmod e$, for all $i \neq j$, by \cite[Corollary 3.12 and Theorem 3.13]{fs16}. It can also be deduced from \cite[Theorem 6.6]{rouq08}.
If there is a repeat in the parameters, then the third author, Muth, and Sutton found many decomposable Specht modules in \cite{ss20,mss22}. 
Nevertheless, every Specht module belongs to a unique block because they are cell modules and \cite[Theorem 3.7(iii)]{GL98} implies that if 
$\rD{\bla}$ and $\rD{\bmu}$ appear in $\rspe{\bnu}$, for some $\bnu$, then the entry $c_{\bla\bmu}$ of the Cartan matrix is non-zero, implying that $\rD{\bla}$ and $\rD{\bmu}$ belong to the same block.

\begin{thmc}{lm07}{Theorem~2.11}\label{block=combblock}
Suppose that $\bla,\bmu \in \Pn^2_n$. Then $\rspe\bla$ and $\rspe\bmu$ lie in the same block of $\hhh$ if and only if $\Res_{\bs}(\bla)=\Res_{\bs}(\bmu)$.  
\end{thmc} 

\medskip
Define an equivalence relation $\equiv_e$ on $\Z^2$ by saying that $(a_1,a_2) \equiv_e (a'_1,a'_2)$ if $a_k\equiv a'_k \pmod e$ for $k=1,2$. The next result follows from \cref{block=combblock}. 

\begin{lem}\label{L:Matches}
Let $(\hhh, \ba)$ be a Hecke algebra with bicharge. If 
$B$ is a block of $\hhh$, then  
\[\{(\bla,\ba) \mid \text{$\rspe{\bla}$ belongs to $B$} \} \in \Bl.\]
\end{lem}

Using~\cref{block=combblock} together with Fayers's equivalent characterisations of a core block of $\hhh$, that is, conditions 4 and 5 in \cite[Theorem 3.1]{fay07core}\footnote{See \cite[Remark 3.3.(2)]{bfs13} for the fact that the hub of the label $\lambda\in P(\Lambda)$ of a block is $(\langle\alpha_i^\vee, \lambda\rangle)_{i\in \Z/e\Z}$. It is clear from this that if two weights have the same hub, then they coincide modulo $\Z\delta$, where $\delta$ is the null root. }, we have the next lemma. 

\begin{lem} \label{L:CoresMatch}
    We fix a bicharge $\ba$ such that $Q_1=q^{a_1}$ and $Q_2=q^{a_2}$.
    Then the map from the set of blocks of $\oplus_{n\geq0} \hhh(q,Q_1,Q_2)$ to the set of combinatorial blocks with bicharge $\ba$ given by 
    \[
    B \longmapsto \cB=\{ (\bla,\ba) \mid \text{$\rspe{\bla}$ belongs to $B$} \}
    \]
    is bijective.
    Moreover, core blocks in the sense of~\cref{D:Core1} correspond to core blocks in the sense of~\cref{D:Core2}.
\end{lem}

Given a Hecke algebra with bicharge, 
two Specht modules $\rspe\bla$ and $\rspe\bmu$ lie in the same block of $\hhh$ if and only if $(\bla,\bs) \sim (\bmu,\bs)$, the equivalence defined in~\cref{subsec:bipartition}. From now on, if there is no confusion which Hecke algebra with bicharge we are considering,
we identify a block of $\hhh$ with the set of bipartitions indexing the Specht modules that lie in it.
We will often equate such a block of $\hhh$ with $R^\La(\beta)$, where $\La = \La_{s_1} + \La_{s_2}$ and $\beta = \sum_{i =0}^{e-1} c_i \alpha_i$ for any $\bla$ in the block, thanks to the Brundan--Kleshchev isomorphism. Here, $c_i$ denotes the number of times that $i \pmod e$ occurs in $\Res_{\ba}(\bla)$, as in \cref{def:defect}.

\subsection{Decomposition matrices} \label{S:Blocks}

\begin{defn}
For $\bla \in \Pn^2_n$ and $\bmu \in \Kl$, define $d_{\bla\bmu}^{p,\bs}=[\rspe\bla:\rD{\bmu}]$ to be the multiplicity of the simple module $\rD{\bmu}$ as a composition factor of the Specht module $\rspe{\bla}$. If $p$ and $\bs$ are understood or unnecessary, we may omit one or both of them. We define the decomposition matrix of $\hhh$ to be the matrix whose rows are indexed by the elements of $\Pn^2_n$ and whose columns are indexed by the elements of $\Kl$, with the $(\bla,\bmu)$-entry equal to $d_{\bla\bmu}$. 
\end{defn}

The Scopes equivalence relation on combinatorial blocks translates to a Morita equivalence on blocks that also preserves decomposition matrices. 

\begin{thm}[{\cite[Corollary 4.3, Theorem 6.3]{jost_typeBHeckeAlg}}]\label{ScopesEquivBlocks}
    Suppose that $(B,\ba)$ and $(B',\ba')$ are blocks of two Hecke algebras with bicharge and let 
     $\cB,\cB' \in \Bl$ be the corresponding combinatorial blocks. If $\cB \Sc \cB'$ then $(B,\ba)$ and $(B',\ba')$ have the same decomposition matrices via the maps $\Phi_i$ introduced in~\cref{D:Scopes}. Moreover, $B$ and $B'$ are Morita equivalent. 
\end{thm}

In fact, it is possible to define the graded decomposition matrix of $\hhh$. Khovanov and Lauda~\cite{kl09}, and, independently, Rouquier~\cite{rouq}, introduced quiver Hecke algebras, or KLR algebras, in order to categorify the negative halves of quantum groups.
Khovanov and Lauda also introduced cyclotomic quotients in their paper, which we refer to as cyclotomic quiver Hecke algebras, that categorify highest weight irreducible modules over quantum groups~\cite{kk12}.
Subsequently, Brundan and Kleshchev showed in \cite{bkisom} that $\hhh$ is isomorphic to a level 2 cyclotomic quiver Hecke algebra of type $\tta^{(1)}_{e-1}$. As we mentioned in \cref{subsec:blocks}, blocks are $R^\Lambda(\beta)$, for $\beta\in Q_+$.

We will not recall the presentation of the cyclotomic quiver Hecke algebras until \cref{sec:e3}, as we will not be directly working with the definition except for a few cases that we consider in the last two sections. 
For now, it will suffice to note that this framework allows us to study the graded representation theory of $\hhh$, which is further developed in \cite{bk09,bkw11}.
Specifically, it allows us to define graded decomposition numbers since the Specht modules admit a grading. 
For a graded $\hhh$-module $D$, let $D\langle d \rangle$ denote the graded shift (by $d$) of the module $D$ -- in other words $D\langle d \rangle_r = D_{r-d}$.
Then the corresponding graded decomposition number is the Laurent polynomial
\[
d^{p,\bs}_{\bla\bmu}(v)=[\rspe\bla : \rD\bmu]_v = \sum_{d\in\bbz} [\rspe\bla : \rD\bmu \langle d \rangle] v^d \in \bbn[v,v^{-1}],
\]
and we define the graded decomposition matrix in the obvious way. Again, we may omit $p$ or $\bs$ if they are understood or if the result we are quoting holds for all possible $p$ and $\bs$.  

The next result follows, for instance, from the graded cellular basis of Hu--Mathas~\cite[Lemma~2.13, Theorem~6.11 and Section~6.4]{hm10}.

\begin{lem} \label{L:DomOrder}
Keeping the notation above, we have that $d_{\bmu\bmu}(v) = 1$ and $d_{\bla\bmu}(v) \neq 0$ only if $\bla \dom \bmu$. 
\end{lem}

Graded analogues of~\cref{ScopesEquivBlocks} which apply to more general cyclotomic Hecke algebras can be found in~\cite{d'ascopes,websterScopes}.

\begin{thm}\label{T:Equiv1}
Suppose that $(\mathscr{H}_m, \ba)$ and $(\mathscr{H}_{m'}, \ba')$ are Hecke algebras with bicharge and that $B$ and $B'$ are blocks of $\mathscr{H}_m$ and $\mathscr{H}_{m'}$ respectively. 
If the corresponding combinatorial blocks have the property that $\cB \SU \cB'$ or $\cB \Sw \cB'$ then 
$B'$ is Schurian-infinite if and only if $B$ is Schurian-infinite.  
\end{thm}

\begin{proof}
By definition, $\cB'$ may be obtained from $\cB$ by a series of swaps, shifts, and Scopes equivalences.
The swap does not change the algebra: it only changes the cellular algebra structure of the algebra. The shift is induced by an automorphism of the Dynkin diagram for the type $\tta^{(1)}_{e-1}$. Then \cite[Proposition 3.2]{arikirep} shows that 
if $\cB \Sh \cB'$ then $B$ and $B'$ are isomorphic. On the other hand, if $\cB \Sc \cB'$, then $B$ and $B'$ are Morita equivalent by \cref{ScopesEquivBlocks}.
The result follows immediately.
\end{proof}

\subsection{Using decomposition numbers to prove a block is Schurian-infinite}\label{subsec:Schur functor}

For the convenience of the reader, we refer to \cite{als23} for the next basic results. 

\begin{propc}{als23}{Proposition~2.11}\label{reduction}
Let $A$ be a Schurian-finite algebra.
\begin{enumerate}[label=(\roman*)]
\item
If $B$ is a factor algebra of $A$, then $B$ is Schurian-finite.
\item
If $B=eAe$, for an idempotent $e\in A$, then $B$ is Schurian-finite.
\end{enumerate}
\end{propc}

\begin{corc}{als23}{Corollary~2.12}\label{cor:Schurianinfinitequiver}
If the Gabriel quiver of a finite-dimensional algebra $A$ over $\bbf$ contains the quiver of an affine Dynkin diagram with zigzag orientation (i.e.~such that every vertex is a sink or a source) as a subquiver, then $A$ is Schurian-infinite.
\end{corc}

Fix $\hhh$ as in the last section and suppose that $\bla \in \Pn^2_n$ and $\bmu\in\Kl$. Note that $d_{\bla\bmu}(v)=0$ unless $(\bla,\ba)\sim (\bmu,\ba)$. The decomposition matrix of $\hhh$ can therefore be decomposed into a product of block decomposition matrices. 

\medskip
For higher levels, we have the cyclotomic $q$-Schur algebra $\mathcal{S}_n$ associated with the cyclotomic Hecke algebra $\mathcal{H}_n$ \cite{m04surv} and the Schur functor $F: \mathcal{S}_n\text{-mod} \to \mathcal{H}_n\text{-mod}$. 
As is explained in \cite[\S2.9]{Mak14}, the cyclotomic $q$-Schur algebra is quasi-hereditary, and we have standard modules (also called Weyl modules) $\Delta(\bla)$ \cite[Definition 4.12]{m04surv}, which have simple heads $L(\bla)$.
Following Mathas's book \cite{Mathas}, 
we denote $\rspe\bla=F(\Delta(\bla))$ and $\rD\bla=F(L(\bla))$, respectively.
The modules $\rspe\bla$ are `row Specht modules' in the terminology of \cite{kmr}.
$\rD\bla$ is non-zero precisely when $\bla$ is a Kleshchev bipartition. 

Suppose that the characteristic of the base field is zero.
Using a graded lift of a block of the cyclotomic $q$-Schur algebra introduced by Rouquier, Shan, Varagnolo and Vasserot, which we denote by $S^\La(\beta)$ \footnote{See \cite[Theorem 3.13]{Mak14} and the first paragraph of \cite[\S3.5]{Mak14}.}, Maksimau defines an analogue of the cyclotomic quiver Hecke algebra as an idempotent truncation of $S^\La(\beta)$, which we denote by $M^\La(\beta)$, and he proves that $M^\La(\beta)$ is isomorphic to the graded basic algebra of $R^\La(\beta)$ as graded algebras \cite[\S3.9]{Mak14}.
Hence, we may utilise graded representation theory of $M^\Lambda(\beta)$ to study Schurian-finiteness of blocks of $\mathcal{H}_n$. 

The algebra $S^\La(\beta)$ admits a Koszul grading \cite[Theorem 3.12]{Mak14}. 
Using this grading, we have natural graded lifts of standard modules and simple modules, which we also denote by $\Delta(\bla)$ and $L(\bla)$ by abuse of notation.
The Schur functor which sends graded $S^\La(\beta)$-modules to graded $M^\Lambda(\beta)$-modules is a graded lift of the Schur functor $F$.
Then, we have graded $M^\Lambda(\beta)$-modules $\rspe\bla$ and $\rD\bla$ as the image of 
$\Delta(\bla)$ and $L(\bla)$ respectively, and 
we define graded decomposition numbers $d^M_{\bla\bmu}(v)$, where 
$\bla$ and $\bmu$ are multipartitions such that $\bmu$ is Kleshchev. 

Since $M^\Lambda(\beta)$ is $\Z_{\ge0}$-graded and the degree zero part is semisimple, we have the following lemma, which is similar to \cite[Lemma 2.23]{Mak14}. 
Note that unlike $S^\La(\beta)$, the algebra $M^\La(\beta)$ is not generated by degree one elements, and is therefore \emph{not} Koszul.

\begin{lem}\label{lem:Maksimau}
    Suppose that an $M^\Lambda(\beta)$-module $\rspe\bla$ has simple head $\rD{\bla}$ and the coefficient of $v$ in $d^M_{\bla\bmu}(v)$ 
    is non-zero. Then, $[\Rad \rspe\bla/\Rad^2 \rspe\bla: \rD{\bmu}]\ne0$. 
\end{lem}

\begin{rem}
Under the equivalence of (ungraded) module categories 
\[
EH:R^\Lambda(\beta)\text{-mod}\simeq M^\Lambda(\beta)\text{-mod},
\]
Specht modules and simple modules correspond \cite[page 116]{Mak14}. 
Therefore, if the assumption of \cref{lem:Maksimau} is satisfied, then we have $\Ext^1(\rD\bla, \rD\bmu) = \Ext^1(\rD\bmu,\rD\bla)\neq 0$, for simple $R^\Lambda(\beta)$-modules.
Furthermore, \cite[Lemma 3.22]{Mak14} proves that the graded decomposition numbers of $d^M_{\bla\bmu}(v)$ coincide with the decomposition numbers $d^0_{\bla\bmu}(v)$ of $R^\Lambda(\beta)$ given by Brundan and Kleshchev \cite{bk09}, which are exactly those graded decomposition numbers introduced in \cref{S:Blocks}.%
\footnote{Both are given by explicit formulas in terms of parabolic Kazhdan--Lusztig polynomials and they coincide. The appearance of parabolic Kazhdan--Lusztig polynomias  goes back to Uglov's work on canonical bases in Fock spaces.} 
Hence, we may check the assumption in \cref{lem:Maksimau} by using 
the decomposition matrices for $R^\Lambda(\beta)$. 
\end{rem}

In this paper, however, we also use `column Specht modules' $\spe\bla$ (in the terminology of \cite{kmr}), their simple heads $\D\bmu$ (for $\bmu$ a conjugate-Kleshchev multipartition), and graded decomposition numbers $[\spe\bla:\D\bmu]_v$. 
Recall the graded algebra isomorphism 
\[
R^{\La}(\beta) \cong R^{\sigma(\La)}(\sigma(\beta)) 
\]
from \cite[\S3.3]{kmr}, 
where $\sigma: \La_i\mapsto \Lambda_{-i}$ and $\sigma: \alpha_i\mapsto \alpha_{-i}$, for $i\in \bbz/e\bbz$.
By twisting an $R^{\sigma(\Lambda)}(\sigma(\beta))$-module $M$, we obtain an $R^\Lambda(\beta)$-module, 
which we denote by $M\otimes{\rm sgn}$.
Then, 
\[
\rspe\bla\cong \spe{\bla'}\otimes{\rm sgn}, \quad 
\spe\bla\cong \rspe{\bla'}\otimes{\rm sgn},
\]
by \cite[Theorem 8.5]{kmr}\footnote{The graded dual of $\rspe\bla$ is isomorphic to $\spe\bla\langle -{\rm def}(\beta)\rangle$ and 
the graded dual of $\spe\bla$ is isomorphic to 
$\rspe\lambda\langle -{\rm def}(\beta)\rangle$ \cite[Theorem 7.25]{kmr}.}, where $\bla'=({\lambda^{(2)}}', {\lambda^{(1)}}')$ is the conjugate bipartition of $\bla=(\lambda^{(1)}, \lambda^{(2)})$ \cite[\S2.2]{kmr}. 
Hence, $d_{\bla\bmu}(v)=[\rspe\bla:\rD\bmu]_v$, for Kleshchev bipartitions $\bmu$, and 
$[\spe{\bla'}: D_{\bmu'}]_v$, for conjugate-Kleshchev bipartitions $\bmu'$, coincide. Indeed, we have
\[
\rD\bmu:= \mathrm{hd} \rspe\bmu \cong \mathrm{hd} (\spe{\bmu'} \otimes \mathrm{sgn}) = \D{\bmu'} \otimes \mathrm{sgn} \cong \D{m(\bmu')},
\]
where $m$ is the Mullineux map, and that $\D{m(\bmu')} \otimes \mathrm{sgn} \cong \D{\bmu'}$.
Then applying this, we have
\begin{equation}\label{eq:decompnumbersandduals}
d_{\bla\bmu}(v)=[\rspe\bla : \rD\bmu]_v = [\rspe\bla : \D{m(\bmu')}] = [\spe{\bla'}: D_{\bmu'}]_v.
\end{equation}

Since $R^{\sigma(\Lambda)}(\sigma(\beta))$ is Schurian-infinite if and only if $R^\Lambda(\beta)$ is Schurian-infinite, 
the following proposition using column Specht modules instead of row Specht modules also holds.  

\begin{prop}\label{prop:matrixtrick}
Suppose that $e\geq2$ and $\bbf$ has characteristic $p\geq 0$.
If a submatrix of the graded block decomposition matrix in characteristic $0$ is one of the following matrices, and $d_{\bla\bmu}^{p}(1) = d_{\bla\bmu}^{0}(1)$ holds, for all Kleshchev bipartitions $\bla, \bmu$ that label rows of the submatrix, then the block in which those partitions belong is Schurian-infinite.

\begin{multicols}{2}
\begin{equation*}\label{targetmatrix}\tag{\(\dag\)}
\begin{pmatrix}
1\\
v & 1\\
0 & v & 1\\
v & v^2 & v & 1
\end{pmatrix}
\end{equation*}
\begin{equation*}\label{targetmatrixalt}\tag{\(\ddag\)}
\begin{pmatrix}
1\\
v & 1\\
v & 0 & 1\\
v^2 & v & v & 1
\end{pmatrix}
\end{equation*}
\begin{equation*}\label{targetmatrixaltsquare}\tag{\(\clubsuit\)}
\begin{pmatrix}
1\\
0 & 1\\
v & v & 1\\
v & v & 0 & 1
\end{pmatrix}
\end{equation*}
\begin{equation*}\label{targetmatrixstar}\tag{\(\spadesuit\)}
\begin{pmatrix}
1\\
0 & 1\\
v & v & 1\\
0 & v^2 & v & 1\\
v^2 & 0 & v & 0 &1
\end{pmatrix}
\end{equation*}
\end{multicols}

\end{prop}

\begin{proof}
Let $B$ and $B_\bbf$ be the block in characteristic $0$ or $p$ in which the four or the five partitions that label the rows and columns of the submatrix belong, and denote the Kleshchev bipartitions by $\bla^{(i)}$, for $1\leq i\leq 4$ or $1\leq i\leq 5$.
Then, interpreting the submatrix as the submatrix of graded decomposition numbers of $M^\Lambda(\beta)$, we have an $A^{(1)}_3$ quiver (square) or a $D^{(1)}_4$ quiver (4-pointed star) with zigzag orientation as a subquiver of the Gabriel quiver of $B$.
By the exact same argument used in the proof of \cite[Lemma~2.2]{als23}, 
$d_{\bla\bmu}^{p}(v) = d_{\bla\bmu}^{0}(v)$.
We abuse notation and use the same $d_{\bla\bmu}^{p}(v)$ and $d_{\bla\bmu}^{0}(v)$ for the corresponding graded decomposition numbers of $M^\La(\beta)$.

Let $\rspe\bla$ and $\rD\bmu$ be Specht $M^\La(\beta)$-modules and simple $M^\La(\beta)$-modules, respectively.
As those are images of standard modules and simple modules over $S^\La(\beta)$ under the Schur functor, $\rspe\bla$ has a composition series whose layers are compatible with the dominance order on bipartitions. 
Let $\rD\bmu$ and $\rD\bmu_\bbf$ be the simple $B$-module and the simple $B_\bbf$-module labelled by $\bmu$, respectively.
We denote by $\rspe\bla_\bbf$ the Specht $B_\bbf$-module labelled by $\bla$.
Now we consider the modular reduction of $\rspe\la/\Rad^2(\rspe\la)$ as a factor module of $\rspe\la_\bbf$.
If $d_{\bla\bmu}^{0}(v) = v^2$, then $[\rspe\bla/\Rad^2(\rspe\bla):\rD\bmu] = 0$.
Then, $d_{\bla\bmu}^{p}(1) = d_{\bla\bmu}^{0}(1)\in \{0,1\}$ implies that $\rD\bmu_\bbf$ does not appear in the modular reduction of $\rspe\bla/\Rad^2(\rspe\bla)$ as a composition factor.

The same argument shows that $\rD\bla_\bbf$ appears with multiplicity $1$ as the head of the modular reduction of $\rspe\bla/\Rad^2(\rspe\bla)$, and $\rD\bmu_\bbf$ appears with multiplicity $1$ as one of the composition factors of the modular reduction if $d^{0}_{\bla\bmu}(v) = v$, 
and all the other composition factors of the modular reduction are $\rD\nu_\bbf$ where $\nu$ is not among the four or five partitions $\bla^{(i)}$.
Therefore, we may obtain an indecomposable $B_\bbf$-module that has the unique head $\rD\bla_\bbf$ and submodules $\rD\bmu_\bbf$ for $\bmu$ with $d^{0}_{\bla\bmu}(v)=v$, and all the other composition factors of the module are $\rD\nu_\bbf$ for bipartitions $\nu$ that are not among those we have labelled by $\bla^{(i)}$.

Let $P_\bbf^{\bmu}$ be the projective cover of $\rD{\bmu}_\bbf$ and let $t$ be the sum of the idempotents in the basic algebra $B_\bbf$ that are projectors to $P_\bbf^{\bla^{(i)}}$, summing over all $i$.
Then $t$ kills every simple module $\rD\bnu_\bbf$ that is not labelled by some $\bla^{(i)}$.
Since $\Rad(t{B_\bbf}t)=t\Rad(B_\bbf)t$, it follows that we have 
$\Ext^1(t\rD\bla_\bbf, t\rD\bmu_\bbf)\ne0$, for simple $t{B_\bbf}t$-modules $t\rD\bla_\bbf$ and $t\rD\bmu_\bbf$ with $d^{0}_{\bla\bmu}(v) = v$.
This implies that the Gabriel quiver of $t{B_\bbf}t$ contains either $A^{(1)}_3$ or $D^{(1)}_4$ quiver with zigzag orientation, so that the Gabriel quiver of $B_\bbf$ contains one of them as a subquiver.
We apply \cref{cor:Schurianinfinitequiver} to conclude that $B_\bbf$ is Schurian-infinite.
\end{proof}

\subsection{Diagonal cuts} \label{S:DiagonalCuts}
\cref{prop:matrixtrick} gives us a way to prove that certain blocks of a Hecke algebra $\hhh$ are Schurian-infinite. However, computing (graded) decomposition numbers for Hecke algebras is a notoriously hard problem. In this section, we show equality of certain decomposition numbers. This will later enable us to use decomposition numbers which have been computed elsewhere in order to find submatrices that satisfy the conditions of~\cref{prop:matrixtrick}.

Graded decomposition numbers depend only on the parameters $n,e,p$ and $\bs$; any pair of Hecke algebras that share those parameters will have the same decomposition matrix. Given that $e$ will be fixed, we may therefore talk about decomposition numbers $d^{p,\bs}_{\bla\bmu}(v)$ on the understanding that they are decomposition numbers for such a Hecke algebra without needing to specify the algebra in question.  

The following are special cases of \cite[Main Theorem]{bs16}, which applies more generally to various quasi-hereditary covers of our Hecke algebras, coming with different associated partial orders on the set of multipartitions.
In the `well-separated' case, with the usual dominance order, the (level 2) diagrammatic Cherednik algebras of \cite{bs16} are quasi-hereditary covers of $\calh_n$, whose standard modules $\Delta(\bla)$ map to the (row-) Specht modules $\rspe\bla$ under the Schur functor, and whose simple modules $L(\bmu)$ map to our simple modules $\rD\bmu$ if $\bmu$ is Kleshchev, and $0$ otherwise.

The representation theory of the Hecke algebras of type $\tta$ is analogous to that of type $\ttb$. For a type $\tta$ Hecke algebra $\mathcal{H}_n$ of quantum characteristic $e$, we have a set of cell modules $\rspe{\la}$, known as Specht modules, that are indexed by partitions of $n$, with the simple modules $\rD\mu$ indexed by a subset of $\Pn_n$ known as $e$-restricted partitions.
We refer to~\cite{Mathas} for details. Again, by results of~\cite{bkw11} following the work of Khovanov and Lauda~\cite{kl09} and Rouquier~\cite{rouq}, the Specht modules admit a grading and we define $d^{p}_{\la\mu}(v) = [\rspe\la : \rD\mu]_v$ to be the graded multiplicity of $\rD\mu$ as a composition factor of $\rspe{\la}$.
These graded decomposition numbers depend only on $n$, $e$ and $p$ so that we can again talk about them without needing to specify the underlying algebra.  Note that if $\bla\in\Kl$ then $\la^{(1)}$ and $\la^{(2)}$ are both $e$-restricted. 

For the remainder of~\cref{S:DiagonalCuts}, fix $n \ge 0$ and $e \ge 2$ and let $\bs=(s_1,s_2)$ where $0 \le s_1,s_2<e$. 
Take $p = 0$ or $p$ a prime.

The next result will be used in \cref{sec:e=3defect4+}; it is a consequence of \cite[Theorem~4.30]{bs15}.

\begin{lem} \label{L:OneNodeShift}
Suppose that $\bla \in \Pn^2_n$ and $\bmu \in \Kl$ and that (the Young diagram of) $\bla$ is obtained from $\bmu$ by moving one node.
Then
\[
d^p_{\bla\bmu}(v)=d^0_{\bla\bmu}(v),
\]
that is, the graded decomposition number is independent of the characteristic of the field. 
\end{lem}

The next results are special cases of~\cite[Main Theorem]{bs16}.

\begin{thm} \label{T:BowmanSpeyer1}
Suppose that $\bla = (\la^{(1)},\la^{(2)}) \in \Pn^2_n$ and $\bmu = (\mu^{(1)},\mu^{(2)}) \in \Kl$ and that $|\la^{(2)}| = |\mu^{(2)}|$.
Then 
\[
d^{p,\bs}_{\bla\bmu}(v) = d^{p}_{\la^{(1)}\mu^{(1)}}(v) \times d^{p}_{\la^{(2)}\mu^{(2)}}(v).
\]
In particular, if $\la^{(2)}=\mu^{(2)}$, then 
\[
d^{p,\bs}_{\bla\bmu}(v) = d^{p}_{\la^{(1)}\mu^{(1)}}(v).
\]
\end{thm}

Now suppose that $\la \in \Pn$ has exactly $k$ non-zero parts. For $l \geq 0$, define $\la^{l^+}$ to be the partition whose Young diagram is obtained by adding a vertical rim hook of length $le$ to the bottom of $[\la]$; that is, $\la^{l^+}=(\la_1,\la_2,\dots,\la_k,1,\dots,1)$, where there are $le$ entries equal to $1$ appended to $\la$. 

\begin{thm} \label{T:BowmanSpeyer2}
Suppose that $\bla = (\la^{(1)},\la^{(2)}) \in \Pn^2_n$ and $\bmu = (\mu^{(1)},\mu^{(2)}) \in \Kl$ and that $\la^{(2)}$ and $\mu^{(2)}$ have the same number of non-zero parts.
Take $l \ge 0$ and define $\bla^{l^+}=(\la^{(1)},(\la^{(2)})^{l^+})$ and $\bmu^{l^+}=(\mu^{(1)},(\mu^{(2)})^{l^+})$. 
If $\bmu^{l^+}\in \mathscr{K}^{\bs}_{n+le}$ then 
\[d^{p,\bs}_{\bla\bmu}(v) = d^{p,\bs}_{\bla^{l^+}\bmu^{l^+}}(v).\]
\end{thm}

\begin{cor}\label{cor:BowmanSpeyerconjugate}
Suppose $\bla = (\la^{(1)},\la^{(2)}), \bmu = (\mu^{(1)},\mu^{(2)}) \in \Pn^2_n$ with $\bmu$ conjugate-Kleshchev and that $\la^{(1)}=\mu^{(1)}$.
Then $[\spe{\bla}: \D{\bmu}]_v = [\spe{\la^{(2)}} : \D{\mu^{(2)}}]_v$.
\end{cor}

\begin{proof}
As in \cref{eq:decompnumbersandduals}, we have $[\spe{\bla}: \D{\bmu}]_v = [\rspe{\bla'}:\rD{\bmu'}]_v$ for any conjugate-Kleshchev bipartition $\bmu$.
Then since $\la^{(1)}=\mu^{(1)}$, we have that the second components of $\bla'$ and $\bmu'$ (that is, $(\la^{(1)})'$ and $(\mu^{(1)})'$) coincide, so that we may apply \cref{T:BowmanSpeyer1} to obtain 
\[
[\spe{\bla}: \D{\bmu}]_v = d^{p,\bs}_{\bla'\bmu'}(v) = d^p_{{\la^{(2)}}'{\mu^{(2)}}'}(v) 
 = [\spe{\la^{(2)}}:\D{\mu^{(2)}}]_v.\qedhere
\]
\end{proof}

We use~\cref{prop:matrixtrick} to show that certain blocks are Schurian-infinite. In order to do this, we need to find submatrices of the graded block decomposition matrices, and we would like to apply \cref{T:BowmanSpeyer1,T:BowmanSpeyer2} to reduce to known cases.
In order to apply~\cref{T:BowmanSpeyer2}, it is necessary to know whether the bipartition $\bmu^{l^+}$ is a Kleshchev bipartition. This is discussed in the next section.

\subsection{Demazure crystals and Kleshchev bipartitions}\label{sec:Demazurecrystals}

In \cite{akt08}, the authors give a non-recursive characterisation of Kleshchev bipartitions. For this, they introduced combinatorial operations `base' and `roof' on the set of $e$-restricted partitions.
Let $P$ be the weight lattice, and $V=P\otimes_\bbz \bbr$. 
In the Littelmann path model of the crystal $B(\Lambda)$, each element is represented by a piecewise linear path $[0,1] \to V$, and the first direction vector and the last direction vector of the path govern two kinds of Demazure crystals. 
The base and the roof are the incarnation of the last direction vector and the first direction vector of the path, respectively. 

Let us consider the Misra--Miwa model for $B(\Lambda_m)$, where $0\le m\le e-1$. 
Then, for each partition $\lambda=(\lambda_1,\lambda_2,\dots)$, we associate 
its beta set $\beta_m(\lambda)=\{ \lambda_i-i+m \mid i=1,2,\dots\}$.
We define $\downoperation(\lambda)$ and $\upoperation(\lambda)$ using $\beta_m(\lambda)$ as in 
\cite[Definition 2.6]{akt08} and \cite[Definition 2.2]{akt08}. 
We apply down operations and up operations as many times as possible. 
Then, we reach cores, which we denote by $\base(\la)$ and $\roof(\la)$, respectively. 

In \cite{akt08}, we developed another model which mixes the Misra--Miwa model and the Littelmann path model. In this model of $B(\La_m)$, elements are 
\[
(\kappa_1,\dots,\kappa_r:0=a_0<a_1<\dots<a_r=1),
\]
where $a_i$ are rational numbers and  $\kappa_i$ are $e$-cores. $\kappa_1$ corresponds to the first direction vector and $\kappa_r$ corresponds to the last direction vector in the Littelmann path model.%
\footnote{It is worth mentioning that the Mullineux map transposes $\kappa_i$ simultaneously. This agrees with the case where $e=\infty$, where Specht modules are irreducible and the Mullineux map transposes partitions.}
Kreiman, Lakshmibai, Magyar and Weyman's roof lemma implies that  
$\kappa_1=\roof(\lambda)$, and it is proven in \cite[Corollary 8.5]{akt08} that $\kappa_r=\base(\lambda)$. Let $M_i(\lambda)=\max\{ x\in \beta_m(\lambda) \mid x \equiv i \pmod e\}$, for $0\le i\le e-1$. We reorder them in the descending order $M_{i_1}(\lambda)>\dots>M_{i_e}(\lambda)$ and define $\tau_m(\lambda)$ by
\[
\beta_m(\tau_m(\lambda))=\beta_m(\lambda)\sqcup\{M_{i_1}(\lambda)+e,\dots, M_{i_m}(\lambda)+e\}.
\]
Then \cite[Theorem~9.5]{akt08} in our setting may be stated as follows. 

\begin{thmc}{akt08}{Theorem~9.5}\label{akt thm}
    Let $\Lambda=\Lambda_0+\Lambda_s$, for $0\le s\le e-1$. Then a bipartition  $(\lambda^{(1)},\lambda^{(2)})$ of bicharge $(0,s)$ is a Kleshchev bipartition if and only if
    \begin{itemize}
        \item[(a)] $\lambda^{(1)}$ and $\lambda^{(2)}$ are $e$-restricted; and
        \item[(b)] $\tau_s(\base(\la^{(2)}))\supseteq \roof(\la^{(1)})$.
    \end{itemize}
\end{thmc}

\noindent
The following lemma is useful for computing $\base(\lambda)$. Recall that $B(\Lambda_m)$ is the set of coloured $e$-restricted partitions, where the node ($r,c)$ is coloured $s-r+c \pmod e$. 

\begin{lemc}{akt08}{Corollary 7.9}\label{useful_lemma_for_base}
    Let $\la = (\la_1,\la_2,\dots)\in B(\Lambda_m)$ and 
    suppose that $\mu[r]=(\la_r,\la_{r+1},\dots)$ is an $e$-core. 
    For $k=r-1,\dots,1$, let $\mu[k]=(\la_k,\base(\mu[k+1]))$, which is the partition obtained from $\base(\mu[k+1])$ by appending a row of length $\la_k$ as its first row.
    Then each $\mu[k]$ is $e$-restricted and $\base(\lambda)=\base(\mu[1])$.
\end{lemc}

\begin{lem}\label{tau_m lemma}
 Suppose that $\mu, \nu\in B(\Lambda_0)$ are $e$-cores. If $\mu\supseteq \nu$, then 
 $\tau_m(\mu)\supseteq \tau_m(\nu)$, for $0\le m\le e-1$.     
\end{lem}
\begin{proof}
Recall that the set of $e$-cores is the Weyl group orbit of the empty partition. 
We denote the Coxeter generators of the Weyl group $S_e^{\rm aff}$ 
by $\{\sigma_i \mid i\in \bbz/e\bbz\}$. The action of $\sigma_i$ on an $e$-core either adds all addable $i$-nodes, or deletes all removable $i$-nodes. 
The stabiliser of $\Lambda_0$ is 
$S_e$ which is generated by $\{\sigma_i \mid i\ne0 \}$.
We write $\mu=w\varnothing$ and $\nu=y\varnothing$, where $y$ and $w$ are distinguished coset representatives of $S_e^{\rm aff}/S_e$. Then, $\mu\supseteq \nu$ is equivalent to 
$y\le w$ in the Bruhat order by \cite[Proposition 4.4]{akt08}. 

Let $w=\sigma_{i_1}\dots \sigma_{i_\ell}$ be a reduced expression. Then, by looking at the abacus, we see that swaps of runners $i_k$ and $i_{k+1}$ always increase the size of the partition.
It is also clear from the abacus picture that $w\tau_m(\varnothing) = \tau_m(w\varnothing)$.
Namely, after adding one bead to the largest $m$ runners, this sequence is still feasible, i.e.~if the length of runner $i$ exceeds the length of runner $i+1$, it remains the case after we add beads. 
 
Now, take $y<w$.
Equivalently, we have a subsequence 
$\sigma_{j_1}\dots \sigma_{j_r}$ of $\sigma_{i_1}\dots \sigma_{i_\ell}$ which is a reduced expression of $y$.
Then, we have a sequence of swaps starting from the beta set of $\tau_m(\varnothing)=(e-m)^m$, and each step increases the size of the partition. 
We prove $\tau_m(\mu)\supseteq \tau_m(\nu)$ by induction on $\ell=\ell(w)$. 
Suppose first that $i_1\ne j_1$. Then $y\le \sigma_{i_1}w$ and the induction hypothesis implies $\tau_m(\nu)\subseteq \sigma_{i_1}\tau_m(\mu)\subseteq \tau_m(\mu)$. 
Suppose that $i_1=j_1$. Then $\sigma_{i_1}y\le \sigma_{i_1}w$ and the induction hypothesis implies $\sigma_{i_1}\tau_m(\nu)\subseteq \sigma_{i_1}\tau_m(\mu)$. Then, 
it follows that $\tau_m(\nu)\subseteq\tau_m(\mu)$. 
\end{proof}

In order to prove \cref{thm:base-corerel}, we recall the definition of the down operation, and prove a preliminary lemma.
For $\la\in B(\La_0)$, we take the beta set $\beta=\beta_0(\la)$, and define 
\[
U(\beta)=\{ x\in \beta \mid x-e\not\in \beta \}.
\]
If $U(\beta)=\varnothing$, then we have $\base(\la) = \la$. Suppose that $U(\beta)\ne\varnothing$.
We define $p'=\min U(\beta)$ and consider
\[
W(\beta)=\{ x>p'-e \mid x\in \beta, x+e\not\in \beta \}\sqcup\{p'\}. 
\]
Set $q'=\min W(\beta)$.
Then we obtain $\downoperation(\beta)$ by moving the bead $q'$ to $p'-e$.
By applying the down operation repeatedly, we eventually obtain an $e$-core, at which point the down operation acts as the identity, and this $e$-core is the base of $\la$, denoted $\base(\la)$. 
Now we describe this down operation in Young diagrammatic terms.

First we recall how to read $\beta=\beta_0(\la)$ from $\la\in B(\La_0)$, and it is best explained by an example.
Let $\la =(3,2)$.
Then, we walk the rim of $\la$ from the bottom-left vertical line  to the top-right  horizontal line as follows.
We get a sequence of $\uparrow$ and $\rightarrow$.
Then, the numbers labelled on $\uparrow$ give us $\beta$. 
Namely, $\beta=\{2,0,-3,-4,\dots\}$. 

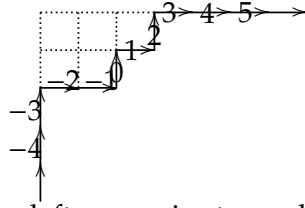
\medskip
\hspace*{2cm}
\begin{xy}
(15,17) *{2}="A",(17,20) *{3}="B",(22,20) *{4}="C", (27,20) *{5}="J",
(12,15) *{1}="D",(10,12) *{0}="E",(8,11) *{-1}="F",
(3,11) *{-2}="G", (-2,7) *{-3}="H", (-2,2) *{-4}="I",

\ar@{..} (0,20);(15,20)
\ar@{..} (0,20);(0,10)
\ar@{..} (0,15);(10,15)
\ar@{..} (5,20);(5,10)
\ar@{..} (10,20);(10,15)

\ar@{->} (0,-5);(0,0)
\ar@{->} (0,0);(0,5)
\ar@{->} (0,5);(0,10)
\ar@{->} (0,10);(5,10)
\ar@{->} (5,10);(10,10)
\ar@{->} (10,10);(10,15)
\ar@{->} (10,15);(15,15)
\ar@{->} (15,15);(15,20)
\ar@{->} (15,20);(20,20)
\ar@{->} (20,20);(25,20)
\ar@{->} (25,20);(30,20)
\ar@{->} (30,20);(35,20)
\end{xy}

Here, the top-left corner is at  coordinate $(0,0)$, 
and the tail of a $\uparrow$ on the left border matches its $y$-coordinate if the number is sufficiently small, the tail of a $\rightarrow$ on the top border matches its $x$-coordinate if the number is sufficiently large, respectively.

We denote the $e$-core of $\la$ by $\core(\la)$, and introduce the following lemma in order to prove \cref{thm:base-corerel}.

\begin{lem}\label{lem:downopvscore}
    Suppose $\mu$ and $\mu'$ are two partitions such that $\mu \supseteq \mu'$.
    Then either $\downoperation(\mu) \supseteq \mu'$, or $\mu'$ has a removable $e$-rim hook such that removing it yields a partition $\widehat{\mu'}$ satisfying $\downoperation(\mu) \supseteq \widehat{\mu'}$.
\end{lem}

\begin{proof}
    Suppose $\downoperation(\mu)$ does not contain $\mu'$.
    Define $p'$ and $q'$ as in the description of $\downoperation(\mu)$.
    Then visualising $\beta = \beta_0(\mu)$ as explained above the lemma, $p'$ labels the top-most $\uparrow$, and $p'-e$ labels some $\rightarrow$.
    Then, $q'\in[p'-e+1,p']$ is the first number in the interval which labels a $\uparrow$.
    Moving $q'$ to $p'-e$ implies that we change the subsequence
    \[
    \rightarrow\;\cdots\;\rightarrow\;\uparrow
    \]
    labelled by $p'-e,\dots, q'$ to the subsequence
    \[
    \uparrow\;\rightarrow\;\cdots\;\rightarrow
    \]
    labelled by $p'-e,\dots, q'$. In other words, we delete a horizontal strip of length $q'-p'+e$ from that row to obtain $\downoperation(\mu)$.
    Otherwise, $\mu \supseteq \mu'$ implies that the only failure in the containment $\downoperation(\mu)\not\supseteq \mu'$ occurs in the row corresponding to the subsequence labelled by $p'-e,\dots, q'$. 
    Namely, the subsequence for $\mu'$ is
    \[
    \rightarrow\;\cdots\;\rightarrow\; \uparrow \;\rightarrow\;\cdots\;\rightarrow,
    \]
    which is contained in $\rightarrow\;\cdots\;\rightarrow\;\uparrow$, but not contained in 
    $\uparrow\;\rightarrow\;\cdots\;\rightarrow$. 
    Hence, if we consider the subsequence for $\mu'$, the label $p'-e$ appears on some $\rightarrow$ and $p'$ appears on some $\uparrow$. 
    This implies that we may remove an $e$-rim hook corresponding to the interval $[p'-e,p']$ from $\mu'$, yielding the partition $\widehat{\mu'}$ in the statement of the lemma. 
\end{proof}

\begin{thm}\label{thm:base-corerel}
    For an $e$-restricted partition $\la\in B(\La_0)$, $\tau_m(\base(\la)) \supseteq \tau_m(\core(\la))$, for $0\le m\le e-1$. 
\end{thm}

\begin{proof}
    Since $\la\in B(\La_0)$, we take the beta set $\beta=\beta_0(\la)$.
    We define 
    \[
    U(\beta)=\{ x\in \beta \mid x-e\not\in \beta \}.
    \]
    It suffices to prove that $\base(\la) \supseteq \core(\la)$ by \cref{tau_m lemma}.
    We define $\mu[k]=(\la_k,\base(\mu[k+1]))$ as in \cref{useful_lemma_for_base}, and define $\mu'[k]=(\la_k,\core(\mu'[k+1]))$.
    We shall prove that $\base(\mu[k])\supseteq \core(\mu'[k])$ by downward induction on $k$.
    For the base case, we take $k$ large enough that $\mu[k] = \mu'[k] = \vn$, so that $\base(\mu[k]) = \core(\mu'[k]) = \vn$.

    Denote $\kappa=\base(\mu[k+1])\supseteq \kappa'=\core(\mu'[k+1])$.
    Then, $\mu[k]=(\la_k,\kappa)$ and $\mu'[k]=(\la_k,\kappa')$.
    Since $\kappa$ is an $e$-core, we have either $U(\beta_0(\mu[k]))=\varnothing$ or $U(\beta_0(\mu[k]))=\{p'\}$, where $p'=\la_k-1$. 
    In the former case, 
    \[
    \base(\mu[k])=\mu[k]\supseteq \mu'[k]\supseteq \core(\mu'[k]).
    \]
    Hence we consider the latter case.
    We iteratively define $F^r(\mu'[k])$ as follows:
    \[
    F^r(\mu'[k]) = \begin{cases}
        F^{r-1}(\mu'[k]) &\text{if } \downoperation^r(\mu[k]) \supseteq F^{r-1}(\mu'[k]),\\[6pt]
        \widehat{F^{r-1}(\mu'[k])} &\text{otherwise,}
    \end{cases}
    \]
    where $\widehat{F^{r-1}(\mu'[k])}$ is obtained from $F^{r-1}(\mu'[k])$ by removing the $e$-rim hook prescribed by \linebreak \cref{lem:downopvscore}.
    The by construction, and \cref{lem:downopvscore}, we necessarily have that $\downoperation^r(\mu[k]) \supseteq F^{r}(\mu'[k])$.
    For some large enough $r$, we have that
    \[
        \base(\mu[k]) = \downoperation^r(\mu[k]) \supseteq F^{r}(\mu'[k]) \supseteq \core(\mu'[k]).
    \]
    Hence, the downward induction on $k$ proceeds, and we obtain
    \[
    \base(\la) = \base(\mu[1]) \supseteq \core(\mu'[1]) = \core(\la)
    \]
    for an $e$-restricted partition $\la\in B(\La_0)$.
\end{proof}

\begin{thm} \label{T:StillKlesh}
Suppose that $(\lambda,\nu)$ is a bipartition of bicharge $(0,s)$, where $0 \le s \le e-1$, and that $\core(\nu)=\mu$. If $\nu \in B(\Lambda_s)$ and $(\lambda,\mu)$ is a Kleshchev bipartition then $(\lambda,\nu)$ is a Kleshchev bipartition.  \end{thm}

\begin{proof}
Using \cref{thm:base-corerel,akt thm} and the assumptions, and that by definition the down operation (and therefore the base operation) is the identity on $e$-cores, we obtain 
\[
\tau_s(\base(\nu))\supset \tau_s(\core(\nu))=\tau_s(\mu)=\tau_s(\base(\mu))\supset \roof(\lambda).
\]
Thus, $(\lambda,\nu)$ is a Kleshchev bipartition by \cref{akt thm} again.
\end{proof}

\subsection{Core blocks} \label{subsec:coreblocks}
We end this section by looking at (combinatorial) core blocks of $(\hhh,\ba)$. 
The representation theory of the core blocks is well-understood and their decomposition matrices are easy to compute. The material and notation in this section is mostly taken from~\cite{Lyle24coreblocks}. Recall from~\cref{L:CoresMatch} that if $(B,\ba)$ is a block of the Hecke algebra with bicharge $(\hhh,\ba)$ then $\cB$ is the corresponding combinatorial block and that $B$ is a core block if and only if $\cB$ is a core block. 

\begin{defn}\label{Definition~4.22}
Let $\Cl \subseteq \Bl$ be the set of combinatorial blocks $\cC$ such that $\cC$ is a core block and if $\cC$ has bicharge $\ba$ then there exists $(b_0,b_1,\dots,b_{e-1}) \in \Z^e$ such that if $\bla \in \cC$ then
\[
b^{\ba}_{ik}(\bla) \in \{b_i,b_i+1\}
\]
for $0 \leq i <e$ and $k=1,2$.
We then call $(b_0,b_1,\dots,b_{e-1})$ a base tuple for $\cC$.
\end{defn}

\begin{eg}
Take $e=4$. Let $\bla_1=((5,2,1),(2,1^3))$, $\bla_2=((1^3),(4,3,2,1))$ and $\bla_3=((3,2,1),(4,1^3))$ and take $\ba=(9,10)$. Then $C=\{(\bla_i,\ba) \mid i=1,2,3\}$ is a core block. Furthermore $C \in \Cl$ with base tuple $(2,3,0,2)$ (or $2,3,1,2)$). 

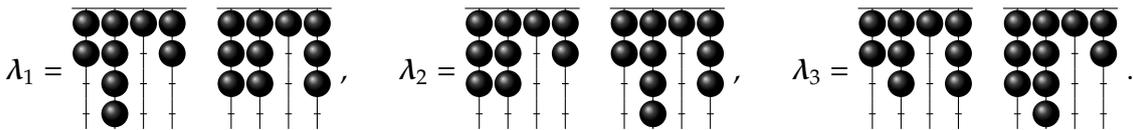
\begin{align*}
\bla_1 & =\abacusline(4,1,bbbb,bbnb,nbnn,nbnn) \quad \abacusline(4,1,bbbb,bbnb,bbnb,nnnn)\;, & 
\bla_2& =\abacusline(4,1,bbbb,bbnb,bbnn,nnnn) \quad \abacusline(4,1,bbbb,bbnb,nbnb,nbnn)\;, & 
\bla_3& =\abacusline(4,1,bbbb,bbnb,nbnb,nnnn) \quad \abacusline(4,1,bbbb,bbnb,bbnn,nbnn)\;.
\end{align*}
However we take $\ba'=(9,6)$ and $\cC'=\{(\bla_i,\ba') \mid i=1,2,3\}$ then $C'$ is still a core block but $C' \notin \Cl$. 
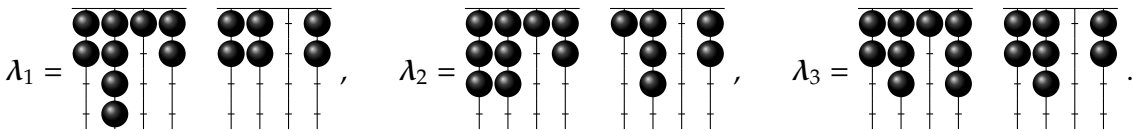
\begin{align*}
\bla_1 & =\abacusline(4,1,bbbb,bbnb,nbnn,nbnn) \quad \abacusline(4,1,bbnb,bbnb,nnnn,nnnn)\;, & 
\bla_2& =\abacusline(4,1,bbbb,bbnb,bbnn,nnnn) \quad \abacusline(4,1,bbnb,nbnb,nbnn,nnnn)\;, & 
\bla_3& =\abacusline(4,1,bbbb,bbnb,nbnb,nnnn) \quad \abacusline(4,1,bbnb,bbnb,nbnn,nnnn)\;.
\end{align*}
\end{eg}

\begin{lemc}{fay07core}{Theorem~3.1} \label{L:basetuple1}
Suppose $\cB$ is a combinatorial block with bicharge $\ba$ and that $(\bla,\ba) \in \cB$ is a bicore.  If there exist $\ba' \in \Z^2$ with $\ba' \equiv_e \ba$ and $(b_0,b_1,\dots,b_{e-1}) \in \Z^e$ such that $b^{\ba'}_{ik}(\bla) \in \{b_i,b_i+1\}$ for all $i,k$ then $\cB$ is a core block. 

\end{lemc}

\begin{lemc}{fay07core}{Theorem~3.1} \label{L:basetuple2}
Let $C$ be a core block of $\hhh$. Then we may choose 
$\ba \in \Z^2$ such that 
$(C,\ba)$ is a block of the Hecke algebra with bicharge $(\hhh,\ba)$ and the corresponding core block $\cC$ lies in $\Cl$. 
\end{lemc}

Assume for the rest of this section that $C$ is a core block of a Hecke algebra with bicharge $(\hhh,\ba)$ such that $\cC \in \Cl$ and take $(b_0,b_1,\dots,b_{e-1})$ to be the corresponding base tuple. We note that the choice of base tuple is not unique, however all the results in this section hold for any choice. For $k=1,2$, let $s_k$ be the equivalence class of $a_k$ modulo $e$ and let $\bs=(s_1,s_2)$.  

\begin{defn} \label{D:delta}
Define a total order $\prec$ on $\{0,1,\dots,e-1\}$ by setting $i \prec j$ if and only if $b_i<b_j$ or $b_i=b_j$ and $i<j$. Suppose that $j_0 \prec j_1 \prec \dots \prec j_{e-1}$.
For $0 \leq i \leq e-1$, let 
\[\delta_i(\bla) = b^{\ba}_{j_i2}(\bla)-b^{\ba}_{j_i1}(\bla).\]
Note that $\delta_i(\bla)$ corresponds to runner $j_i$ of the abacus configuration, whereas $b_{ik}(\bla)$ corresponds to runner $i$.
\end{defn}

By definition, we have $\delta_i(\bla) \in \{-1,0,1\}$ for $\bla \in \cC$.
Following~\cite{Lyle24coreblocks}, we write $-$ for $-1$ and $+$ for $1$ and we define $\Delta=\{0,-,+\}^e$. For $\bla \in \cC$, set $\delta_{\bla}=(\delta_0(\bla),\dots,\delta_{e-1}(\bla)) \in \Delta$. 

\begin{defn} \label{D:SwapCore}
Suppose that $\delta_u(\bla)=-$ and $\delta_v(\bla)=+$ for some $u,v \in \{0,1,\dots,e-1\}$. Define $s_{uv}(\bla)$ to be the bicore $\bmu$ with 
\[b^{\ba}_{ik}(\bmu) = \begin{cases} 
b^{\ba}_{ik}(\bla) + 1, & k=1 \text{ and } i=j_v \text{ or } k=2 \text{ and } i=j_u,\\
b^{\ba}_{ik}(\bla) -1, & k=1 \text{ and } i=j_u \text{ or } k=2 \text{ and } i=j_v,\\
b^{\ba}_{ik}(\bla), & \text{otherwise}.
\end{cases}
\]
Namely, we move the last bead on runner $j_u$ of the abacus for $\la^{(1)}$ to end of the runner $j_v$ of that abacus and move the last bead on runner $j_v$ of the abacus for $\la^{(2)}$ to the end of runner $j_v$ of that abacus.
\end{defn}

\begin{lemc}{fay07core}{Proposition~3.7} \label{L:deltaseq}
Let $\bla \in \cC$ and suppose that $\delta_u(\bla)=-$ and $\delta_v(\bla)=+$. Let $\bmu=s_{uv}(\bla)$. Then $\bmu \in \cC$ and $\delta_{\bmu}$ is formed from $\delta_{\bla}$ by swapping the $-$ in position $u$ with the $+$ in position $v$. Moreover, any $\bnu \in \cC$ can be formed by repeatedly performing such swaps.     
\end{lemc}

\begin{defn} \label{D:np}
Define $\nn=\nn(\cC)$ and $\pp=\pp(\cC)$ by choosing $\bla \in \cC$ and setting
\[
\nn=|\{i \in \{0,1,\dots,e-1\} \mid \delta_{i}(\bla)=-\}|,
\qquad \pp=|\{i \in \{0,1,\dots,e-1\} \mid \delta_{i}(\bla)=+\}|.
\]
\end{defn}

By \cref{L:deltaseq}, $\nn$ and $\pp$ are independent of the choice of $\bla$. 

\begin{lem} \label{L:s}
We have $a_1-a_2 \equiv \nn-\pp \pmod e$ and $\defect(C)=\min\{\nn,\pp\}$. Moreover, $|\cC|=\binom{\nn+\pp}{\pp}$ and the number of Kleshchev bipartitions 
in $\cC$ is
\[
\binom{\nn+\pp}{\defect(C)}-\binom{\nn+\pp}{\defect(C)-1}.
\]
\end{lem}
\begin{proof}
The first statement is~\cite[Lemmas~3.4]{Lyle24coreblocks} and the second statement follows from~\cite[Proposition~3.8]{fay06wts}. The third statement follows from~\cref{L:deltaseq} and the last is~\cite[Theorem~3.20(2)]{Lyle24coreblocks}.
\end{proof}

The values $\nn$ and $\pp$ are well-defined for a combinatorial core block, but not well-defined for the underlying core block. The next corollary clarifies that two combinatorial core blocks, one with $\min\{\nn,\pp\}=\nn$ and the other with $\min\{\nn,\pp\}=\pp$, correspond to a core block when $\min\{\nn,\pp\}>0$.

\begin{cor}
If $\defect(\cC)\ne 0$ then the set $\{\nn, \pp\}$ is determined by the underlying block $C$. 
\end{cor}

\begin{proof}
Recall that a block is labelled by $\Lambda-\beta\in P(\Lambda)$.
$\Res_{\ba}(\bla)$ depends on $\bs=\ba \pmod e$. 
If we write $\beta=\sum_{i=0}^{e-1} c_i\alpha_i$, the number of Specht modules in the block is the number of bipartitions $\bla$ such that the multiplicity of $i$ in $\Res_{\ba}(\bla)$ is $c_i$, for $0\le i\le e-1$, because swap of $s_1$ and $s_2$ does not affect the number of bipartitions.
It implies that the number of Specht modules does not depend on the choice of $\ba$, but the block only.

Suppose $\defect(C)=\pp$. The proof for the case $\defect(C)=\nn$ is similar. Note that $\pp=\defect(C)$ is determined by $C$. Since the number of Kleshchev bipartitions is the number of simple $C$-modules, 
   \[
   \frac{\text{the number of simple modules in $C$}}{\text{the number of Specht modules in $C$}}=1-\frac{\pp}{\nn+1}
   \]
implies that $\nn$ is also determined by $C$, since $\pp\neq 0$. 
\end{proof}

\begin{rem}
The corollary reflects the fact that, among the three equivalences introduced in \cref{SS:Equivs}, the swap equivalence is different from the other two, in the sense that it only concerns the additional structure given by bicharge and it does not affect the underlying block. The other two equivalences transform the underlying block to another block, but as was explained in the proof of \cref{T:Equiv1}, the shift equivalence does not change the isomorphism class of the algebra, and the Scopes equivalence does not change the Morita class of the algebra.
\end{rem}

We can use the sequences $\delta_{\bla}$ to describe which bipartitions in the block are Kleshchev bipartitions and to find the entries in the decomposition matrix. Take $\delta \in \Delta$. We begin by defining a sequence $\R(\delta)$ and a set $\St(\delta)$.

\begin{enumerate}
\item Take $\R(\delta)=\delta$ and $\St(\delta)=\emptyset$. 
\item If there do not exist $i,j \in \{0,1 \dots,e-1\}$ with $i < j$ and $\delta_i=-$ and $\delta_j=+$, end the process. Return $\R(\delta)$ and $\St(\delta)$. 
\item Otherwise, choose $i,j \in \{0,1,\dots,e-1\}$ with $i < j$ and $\delta_i=-$ and $\delta_j=+$ with the property that $\delta_m=0$ for all $i < m < j$. 
Add $(i,j)$ to $\St(\delta)$ and set $\R(\delta)_i = \R(\delta)_j = 0$. Go back to step (ii). 
\end{enumerate}

Set
\[
\Delta_0 = \{ \delta \in \Delta \mid - \notin \R(\delta) \text{ or } + \notin \R(\delta)\}.
\]
Suppose $\delta \in \Delta_0$.
For each $S = \{(i_1,j_1),\dots,(i_t,j_t)\} \subseteq \St(\delta)$, define $\delta^S \in \Delta$ by setting  
\[\delta^S_m = \begin{cases} +, & m=i_l \text{ for some } 1 \leq l \leq t, \\
-, & m=j_l \text{ for some } 1 \leq l \leq t, \\
\delta_m, & \text{otherwise}.
\end{cases}\]

If $\delta'=\delta^S$ for some $S \subseteq \St(\delta)$, we write $\delta' \ra \, \delta$ and set
$\ell(\delta',\delta) = |S|$.

\begin{thmc}{Lyle24coreblocks}{Proposition~3.10 \& Theorem~3.19} \label{T:Flat}
Suppose that $\bla,\bmu \in C$ with $|\bla|=n$. 
\begin{enumerate}
\item $\bmu \in \Kl$ if and only if $\delta_{\bmu} \in \Delta_0$.
\item Suppose $\bmu\in\Kl$.  
Then 
\[d_{\bla\bmu}(v) = \begin{cases} 
v^{\ell(\delta_{\bla},\delta_{\bmu})}, & \delta_{\bla} \ra \; \delta_{\bmu}, \\
0, & \text{otherwise}.
\end{cases}\]
\end{enumerate}
\end{thmc}

Note that these decomposition numbers are independent of the characteristic of the field. 

\begin{lem} \label{L:OrderCore}
Suppose that $\bla \in \cC$ with $\delta_{\bla} = (\delta_0,\delta_1,\dots,\delta_{e-1})$ and that there exist $0 \leq u<v <e$ such that $\delta_u=-$ and $\delta_v=+$. Let $\bmu=s_{uv}(\bla)$ as in~\cref{D:SwapCore}. Then $\bmu \doms \bla$. 
\end{lem}

\begin{proof}
By~\cref{D:SwapCore}, the abacus configuration for $\bmu$ is obtained from that of $\bla$ by moving a bead from the end of runner $j_u$ to the end of runner $j_v$ on the first abacus and making the reverse move on the second abacus. Since $u<v$ we have $j_u \prec j_v$, that is, we increase a $\beta$-number in the first abacus configuration and decrease one in the second; this corresponds to adding a hook to $\la^{(1)}$ and removing a hook from $\la^{(2)}$. It follows that $\bmu \doms \bla$.
\end{proof}

\begin{cor} \label{C:MinCore}
Suppose that $\bla \in \cC$ and that $\delta_{\bla}=(\delta_0,\delta_1,\dots,\delta_{e-1})$ has the property that there do not exist $0 \leq u<v<e$ such that $\delta_u=+$ and $\delta_v=-$. If $\bmu \in \cC$ then $\bmu \dom \bla$.
\end{cor}

\begin{proof}
This follows from~\cref{L:deltaseq} and~\cref{L:OrderCore}. 
\end{proof}

\subsection{Representatives of the $\Sc$ and $\SU$-equivalence classes of $\Bl$} \label{S:ScopesSec}

Recall the equivalence relations $\Sh$ and $\Sc$, and the relation $\SU$ generated by $\Sh$ and $\Sc$, which were defined in~\cref{SS:Equivs}.
In this section, we find sets of $\Sc$-representatives and $\SU$-representatives in $\Bl$. Recall that each combinatorial block $\cB$ is completely determined by its core block $\cC$ and the maximal number of removable hooks $\hk(\cB)$. 
Recall also that $\Cl$ is the set of core blocks $\cC$ whose bicharge $\ba$ admits a base tuple.

\begin{defn} \label{D:HookFixed}
Set
\[\Bc=\{\cB \in \Bl \mid \cB \text{ has core block } \cC \in \Cl\}.\]
Note that $\Bc$ is closed under the relations $\Sc$ and $\Sh$. Now for $\hk \ge 0$, set
\begin{align*}
\Bl_{\hk} &= \{\cB \in \Bl \mid \hk(\cB)=\hk\}, \\
\Bc_{\hk} &= \{\cB \in \Bc \mid \hk(\cB)=\hk\};
\end{align*}
these sets are also closed under $\Sc$ and $\Sh$. 
\end{defn}

\begin{lem} \label{L:BzOK}
Given $\cB \in \Bl$ with bicharge $\ba$, we can find $\cB' \in \Bc$ with bicharge $\ba'$ such that $\cB'=\{(\bla,\ba') \mid (\bla,\ba) \in \cB\}$ and 
$\cB \Sh \cB'$. 
\end{lem}

\begin{proof}
Following~\cref{L:basetuple2}, we can in fact find $\cB' \in \Bc$ with bicharge $\ba'$ such that $\cB'=\{(\bla,\ba') \mid (\bla,\ba) \in \cB\}$ and $\ba \equiv_e \ba'$.  
\end{proof}

Hence there is no loss in first studying combinatorial blocks $\cB \in \Bc$. In this section we see that each $\cB \in \Bl$ is Scopes-equivalent to a unique minimal combinatorial block and describe these representatives explicitly. 
Rather than working directly with the abacus, we show how to describe the blocks in $\Bc$ using elements of $\Z^e$.

\begin{defn} \label{D:I}
Suppose that $\tb=(\tb_0,\tb_1,\dots,\tb_{e-1}) \in \Z^e$. Set
\[
I (=I_{\tb})=\{i \in \{0,1,\dots,e-1\} \mid \tb_i \equiv 1 \pmod 2\}.
\]
We write our indices for $\tb$ modulo $e$; for example if $i=e-1$ then $\tb_{i+1}$ should be taken to mean $\tb_0$.
Define a total order $\prec (=\prec_{\tb})$ on $I$ by saying that $i \prec j$ if $\tb_i < \tb_j$ or if $\tb_i=\tb_j$ and $0 \le i<j \leq e-1$.
\end{defn}

\begin{defn}
Set \[\Tl=\{\TT=(\tb,\nn,\pp) \mid \tb=(\tb_0,\tb_1,\dots,\tb_{e-1}) \in \Z^e, \nn,\pp \in \Z_{\ge 0} \text{ and } \nn+\pp=|I_{\tb}|\}.\] 
\end{defn}

We now define maps $\phi^{\ast}:\Tl \rightarrow \AB$ and $\phi:\Tl \rightarrow \Bl$. 
 
\begin{defn}
Let $\TT=(\tb,\nn,\pp) \in \Tl$.  If $I=\{m_1,m_2,\dots,m_{|I|}\}$ where $m_l \prec m_{l+1}$ for $1 \leq l < |I|$ then set $I^-=\{m_1,\dots,m_{\nn}\}$ and $I^+=I \setminus I^-$. 
Now define $\phi^\ast(\TT)=(\bla,\ba) \in \AB$ to be the abacus configuration such that $\bla$ is a bicore and
\[b^{\ba}_{ik}(\bla) = \begin{cases} \frac{\tb_i}{2}, & i \notin I, \\
\frac{\tb_i+1}{2}, & k=1 \text{ and } i \in I^- \text{ or } k=2 \text{ and } i \in I^+,\\
\frac{\tb_i-1}{2}, & k=1 \text{ and } i \in I^+ \text{ or } k=2 \text{ and } i \in I^-,
\end{cases}\]
and define $\phi(\TT) \in \Bl$ to be the combinatorial block (that is, the $\sim$-equivalence class of $\AB$) containing $\phi^\ast(\TT)$. 
\end{defn}

\begin{eg}
Let $e=10$ and $\TT=((4,3,5,1,6,7,2,1,3,8),2,4)$. Then $I=\{1,2,3,5,7,8\}$ where $3 \prec 7 \prec 1 \prec 8\prec 2 \prec 5$ so that $I^-=\{3,7\}$ and $I^+=\{1,8,2,5\}$.  Then $\phi^\ast(\TT)$ is equal to
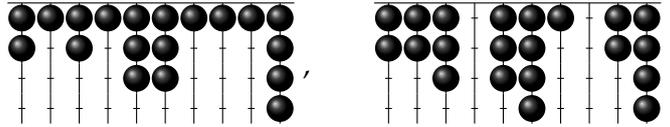
\[\abacusline(10,1,bbbbbbbbbb,bnbnbbnnnb,nnnnbbnnnb,nnnnnnnnnb)\; , \qquad
\abacusline(10,1,bbbnbbbnbb,bbbnbbnnbb,nnbnbbnnnb,nnnnnbnnnb)\;.\]
\end{eg}

Recall the set $\Cl$ of core blocks equipped with a base tuple from~\cref{Definition~4.22}.

\begin{lem} \label{L:AbMap}
Let $\TT=(\tb,\nn,\pp) \in \Tl$ and suppose $\phi^\ast(\TT) = (\bla,\ba)$ and $\phi(\TT)=\cC$. Then $\cC \in \Cl$ with $\nn(\cC)=\nn$ and $\pp(\cC)=\pp$. Moreover if $(\bmu,\ba) \in \cC$ then $\bmu \dom \bla$.  
\end{lem}

\begin{proof}
It follows from~\cref{L:basetuple1} that $\cC \in \Cl$, and it can be seen immediately from the definition that $\nn(\cC)=\nn$ and $\pp(\cC)=\pp$. The last statement follows from~\cref{C:MinCore}. 
\end{proof}

\begin{cor} \label{C:TCBij}
The map $\phi$ gives a bijection between $\Tl$ and $\Cl$. If $\cC \in \Cl$ then
$\phi^{-1}(\cC)$ is defined as follows: Using~\cref{C:MinCore} we may choose $(\bla,\ba) \in \cC$ such that $\bla$ is minimal. Define $\tb=(\tb_0,\tb_1,\dots,\tb_{e-1})$ by setting $\tb_i = b^{\ba}_{i1}(\bla)+b^{\ba}_{i2}(\bla)$. Then $\phi(\tb,\nn(\cC),\pp(\cC))=(\bla,\ba)$. 
\end{cor}

\begin{proof}
The map $\phi^\ast$ is clearly injective and from~\cref{L:AbMap} it follows that $\phi$ is also injective and that $\im(\phi) \subseteq \Cl$. 
\cref{C:MinCore} shows that the last equality holds. 
\end{proof}

Since each $\cB \in \Bc$ is determined by its core block $\cC \in \Cl$ and $\hk(\cB)$, we also have the following correspondence. 

\begin{cor}
There is a bijection between $\Bc$ and $\Tl \times \Z_{\ge 0}$ given by $\cB \mapsto (\phi^{-1}(\cC), \hk(\cB))$. 
\end{cor}

\begin{defn}
Let $\cB \in \Bc$ and suppose that $\cB$ has core block $\cC$ where $\cC=\phi(\TT)$. Set $\psi(\cB)=\TT$. 
\end{defn}

Thus $\cB \in \Bc$ is uniquely determined by $\psi(\cB)$ and $\hk(\cB)$. 
We now see how the action $\rightarrow_{\text{Sc}_i}$ on $\cB$ defined in~\cref{D:ScopesMove} relates to an action on $\psi(\cB)$. 
For the next two lemmas, we assume the following setup. Suppose that $\cB,\cB' \in \Bc$ and that their core blocks are respectively $\cC$ and $\cC'$. Suppose that $\psi(\cB)=(\tb,\nn,\pp)$ and $\psi(\cB')=(\tb',\nn',\pp')$. 

\begin{lem} \label{L:Scopest1}
Let $0 \leq i \leq e-1$. Suppose that $\min\{\nn,\pp\}=0$ or that $\tb_i \equiv 0 \pmod 2$ or that $\tb_{i+1} \equiv 0 \pmod 2$.
\begin{itemize}
\item Suppose $1 \leq i \leq e-1$. Then $\cB \rightarrow_{\text{Sc}_i}\cB'$ if and only if $\hk(\cB)=\hk(\cB')$, $\nn=\nn'$, $\pp=\pp'$, 
\[\tb'=(\tb_0,\dots,\tb_{i-2},\tb_i,\tb_{i-1},t_{i+1}\dots,\tb_{e-1})\] and $\tb_{i}-\tb_{i-1} \ge 2\hk(\cB)$. 
\item Suppose $i=0$. Then $\cB' \rightarrow_{\text{Sc}_i}\cB$ if and only if $\hk(\cB)=\hk(\cB')$, $\nn=\nn'$, $\pp=\pp'$, \[\tb=(\tb_{e-1}+2,\tb_1,\dots,\tb_{e-2},\tb_0-2)\] and $\tb_{0}-\tb_{e-1} \ge 2\hk(\cB)+2$. 
\end{itemize}
\end{lem}

\begin{proof}
Note that
\[\cB'=\Phi_i(\cB) \iff \cC'=\Phi_i(\cC) \text{ and } \hk(\cB)=\hk(\cB').\]
We consider only the case that $1 \leq i \leq e-1$; the case $i=0$ is similar.

We first show that if $\cB'=\Phi_i(\cB)$ then  $\cB\rightarrow_{\text{Sc}_i}\cB'$ if and only if $\tb_i-\tb_{i-1} \ge 2 \hk$, where $\hk=\hk(\cB)$. 
By~\cref{C:DARed}, this is equivalent to showing that there exist $(\bla,\ba) \in \cC$ and $k \in \{1,2\}$ with $b^{\ba}_{ik}(\bla)-b^{\ba}_{i-1k}(\bla)<\hk$ if and only if $\tb_i-\tb_{i-1}< 2\hk$.  
\begin{itemize}
\item Suppose that $i-1,i \notin I$. Take $(\bla,\ba) \in \cC$. By~\cref{L:deltaseq} and~\cref{C:TCBij}, $b^{\ba}_{jk}(\bla)=\frac{\tb_j}{2}$
for $j=i-1,i$ and $k=1,2$. Hence  
\[b^{\ba}_{ik}(\bla)-b^{\ba}_{i-1k}(\bla)=\frac{\tb_i}{2}-\frac{\tb_{i-1}}{2}\]
for $k=1,2$ and therefore $b^{\ba}_{ik}(\bla)-b^{\ba}_{i-1k}(\bla)<\hk$ if and only if $\tb_i-\tb_{i-1}<2\hk$.
\item Suppose $i-1 \notin I$ and $i \in I$. Take $(\bla,\ba) \in \cC$. By~\cref{L:deltaseq} and~\cref{C:TCBij},
 $b^{\ba}_{i-1k}(\bla)=\frac{\tb_{i-1}}{2}$ for $k=1,2$ and $\left\{b^{\ba}_{i1}(\bla),b^{\ba}_{i2}(\bla)\right\} = \left\{\frac{\tb_i+1}{2},\frac{\tb_{i}-1}{2}\right\}$. Hence
\[\left\{b^{\ba}_{i1}(\bla)-b^{\ba}_{i-11}(\bla), b^{\ba}_{i2}(\bla)-b^{\ba}_{i-12}(\bla)\right\}=  \left\{\frac{\tb_i+1}{2}-\frac{\tb_{i-1}}{2},\frac{\tb_i-1}{2}-\frac{\tb_{i-1}}{2}\right\}\]
and therefore $b^{\ba}_{ik}(\bla)-b^{\ba}_{i-1k}(\bla)<\hk$ for $k=1$ or $k=2$ if and only if $\tb_i-\tb_{i-1}<2\hk$.
\item Suppose $i-1 \in I$ and $i \notin I$. Take $(\bla,\ba) \in \cC$. By~\cref{L:deltaseq} and~\cref{C:TCBij},
$\left\{b^{\ba}_{i-11}(\bla),b^{\ba}_{i-12}(\bla)\right\} = \left\{\frac{\tb_{i-1}+1}{2},\frac{\tb_{i-1}-1}{2}\right\}$ and $b^{\ba}_{ik}(\bla)=\frac{\tb_{i}}{2}$ for $k=1,2$. Hence
\[\left\{b^{\ba}_{i1}(\bla)-b^{\ba}_{i-11}(\bla), b^{\ba}_{i1}(\bla)-b^{\ba}_{i-11}(\bla)\right\}=  \left\{\frac{\tb_i}{2}-\frac{\tb_{i-1}+1}{2},\frac{\tb_i}{2}-\frac{\tb_{i-1}-1}{2}\right\}\]
and therefore $b^{\ba}_{ik}(\bla)-b^{\ba}_{i-1k}(\bla)<\hk$ for $k=1$ or $k=2$ if and only if $\tb_i-\tb_{i-1}<2\hk$.
\item Suppose $i-1,i \in I$. Then $\min\{\nn,\pp\}=0$ so that $|\cC|=1$ by \cref{L:s}. 
Suppose that $\cC=\{(\bla,\ba)\}$. 
If $\nn=0$ then $b^{\ba}_{j1}(\bla)=\frac{\tb_j-1}{2}$ 
and $b^{\ba}_{j2}(\bla)=\frac{\tb_j+1}{2}$ for $j=i-1,i$; and if $\pp=0$ then $b^{\ba}_{j1}(\bla)=\frac{\tb_j+1}{2}$ 
and $b^{\ba}_{j2}(\bla)=\frac{\tb_j-1}{2}$ 
for $j=i-1,i$. In either case
\[b^{\ba}_{ik}(\bla)-b^{\ba}_{i-1k}(\bla)=\frac{\tb_i-\tb_{i-1}}{2}\]
for $k=1,2$ and therefore $b^{\ba}_{ik}(\bla)-b^{\ba}_{i-1k}(\bla)<\hk$ if and only if $\tb_i-\tb_{i-1}<2\hk$.
\end{itemize}
Hence if $\cB'=\Phi_i(\cB)$ then  $\cB\rightarrow_{\text{Sc}_i}\cB'$ if and only if $\tb_i-\tb_{i-1} \ge 2 \hk$.

Assume that $\cB\rightarrow_{\text{Sc}_i}\cB'$.  
Since $\Phi_i$ acts by swapping runners $i$ and $i-1$ on the abacus configurations of the elements of $\cB$ or $\cC$, it clearly preserves $\nn$, $\pp$ and $\hk(\cB)$. Since $\cB\rightarrow_{\text{Sc}_i}\cB'$,~\cite[Lemma~5.5 and Theorem~5.2]{d'ascopes} implies that $\Phi_i$ preserves the dominance order $\dom$, so that if $\bmu$ is minimal in $\cC$ then $\Phi_i(\bmu)$ is minimal in $\cC'$. 
Thus $\tb'$ is formed from $\tb$ by swapping the entries in positions $i-1$ and $i$. This proves one direction of the lemma. 

Conversely, suppose that $\hk(\cB)=\hk(\cB')$, $\nn=\nn'$, $\pp=\pp'$,  
$\tb'=(\tb_0,\dots,\tb_{i-2},\tb_i,\tb_{i-1},t_{i+1}\dots,\tb_{e-1})$ and $\tb_{i}-\tb_{i-1} \ge 2\hk(\cB)$. Then $\Phi_i(\cB)$ shares these properties, so $\Phi_i(\cB)=\cB'$ and therefore $\cB \rightarrow_{\text{Sc}_i} \cB'$. 
\end{proof}

\begin{lem} \label{L:Scopest2}
Let $0 \leq i \leq e-1$. Suppose that $\nn,\pp>0$ and $\tb_i \equiv  \tb_{i+1} \equiv 1 \pmod 2$. 
\begin{itemize}
\item Suppose $1 \leq i \leq e-1$. Then $\cB \rightarrow_{\text{Sc}_i}\cB'$ if and only if $\hk(\cB)=\hk(\cB')$, $\nn=\nn'$, $\pp=\pp'$, \[\tb'=(\tb_0,\dots,\tb_{i-2},\tb_i,\tb_{i-1},t_{i+1}\dots,\tb_{e-1})\] and $\tb_{i}-\tb_{i-1} \ge 2\hk(\cB)+2$. 
\item Suppose $i=0$. Then $\cB \rightarrow_{\text{Sc}_i} \cB'$ if and only if $\hk(\cB)=\hk(\cB')$, $\nn=\nn'$, $\pp=\pp'$, \[\tb=(\tb_{e-1}+2,\tb_1,\dots,\tb_{e-2},\tb_0-2)\] and $\tb_{0}-\tb_{e-1} \ge 2\hk(\cB)+4$. 
\end{itemize}
\end{lem}

\begin{proof}
Again, we prove only the case that $1 \leq i \leq e-1$. 
Following the argument in the proof of~\cref{L:Scopest1}, this is equivalent to showing that there exist $(\bla,\ba) \in \cC$ and $k \in \{1,2\}$ with $b^{\ba}_{ik}(\bla)-b^{\ba}_{i-1k}(\bla)<\hk$ if and only if $\tb_i-\tb_{i-1}<2\hk+2$. 
If $(\bla,\ba) \in \cC$ then 
\[\left\{b^{\ba}_{j1}(\bla),b^{\ba}_{j2}(\bla)\right\} = \left\{\frac{\tb_j+1}{2},\frac{\tb_j-1}{2}\right\}\]
for $j =i-1,i$. Furthermore, by~\cref{L:deltaseq}, there exists $(\bmu,\ba) \in \cC$ such that 
$b^{\ba}_{i-11}(\bmu)  = \frac{\tb_{i-1}+1}{2}$ and 
$b^{\ba}_{i1}(\bmu) = \frac{\tb_{i}-1}{2}$ 
so that \[b^{\ba}_{ik}(\bla)-b^{\ba}_{i-1k}(\bla) \ge\frac{\tb_i-\tb_{i-1}-2}{2}
\]
for all $(\bla,\ba) \in\cC$, with $(\bmu,\ba)$ and $k=1$ achieving the equality. 
Therefore there exist $(\bla,\ba) \in \cC$ and $k\in \{1,2\}$ such that $b^{\ba}_{ik}(\bla)-b^{\ba}_{i-1k}(\bla)<\hk$ if and only if $\tb_i-\tb_{i-1}<2\hk+2$.
\end{proof}

We note that if $\hk=0$ and $\tb_{i-1}=\tb_i \equiv 0 \pmod 2$ then~\cref{L:Scopest1} implies that $\cB \rightarrow_{\text{Sc}_i} \Phi_i(\cB)$. However, in this case $\rightarrow_{\text{Sc}_i}$ is just the identity map. We would like to rule out this relation, so we say that the relation $\rightarrow_{\text{Sc}_i}$ is non-trivial if $\cB \rightarrow_{\text{Sc}_i} \Phi_i(\cB)$ with $\cB \ne \cB'$. 

\begin{defn} \label{D:ScopesMin}
Let $\TT=(\tb,\nn,\pp) \in \Tl$ and suppose $\hk \ge 0$. For $0 \leq i \leq e-1$, set $d_i=2$ if $\min\{\nn,\pp\}>0$ and $\tb_{i-1},\tb_i \equiv 1 \pmod2$ and set $d_i=0$ otherwise. 
We say that $\TT$ is $\hk$-Scopes-minimal if
\begin{itemize}
\item For all $1 \leq i \leq e-1$ we have $\tb_i=\tb_{i-1}$ or $\tb_{i}-\tb_{i-1}<2\hk+d_i$; and
\item $\tb_0=\tb_{e-1}+2$ or $\tb_0-\tb_{e-1}<2\hk+2+d_0$. 
\end{itemize}
\end{defn}

\begin{defn}
If $\cB,\cB' \in \mathfrak{B}$, write $\cB \ScA \cB'$ if there exist $i_1,i_2,\dots,i_t \in \{0,1,\dots,e-1\}$ such that 
\[\cB \rightarrow_{\text{Sc}_{i_1}} \cB^{(1)} \rightarrow_{\text{Sc}_{i_2}} \cB^{(2)} \rightarrow_{\text{Sc}_{i_3}} \dots \rightarrow_{\text{Sc}_{i_t}} \cB'\]
with each map $\rightarrow_{\text{Sc}_{i_s}}$ non-trivial. 
Note that if $\cB \ScA \cB'$ and $(\bla,\ba) \in \cB$ and $(\bla',\ba) \in \cB'$ then $|\bla| \geq |\bla'|$ with equality if and only $t=0$.  
\end{defn}

\begin{prop} \label{P:ScopesMinExists}
Suppose that $\cB\in \Bc$ and $\hk(\cB)=\hk$. 
Then there exists a combinatorial block $\cB'\in \Bc$ such that $\cB \ScA \cB'$ and $\psi(\cB')$ is $\hk$-Scopes-minimal. 
\end{prop}

\begin{proof}
Suppose that $\psi(\cB)=(\tb,\nn,\pp)$. We repeatedly apply~\cref{L:Scopest1} and~\cref{L:Scopest2}. 
If there exists $1 \leq i \leq e-1$ such that $\tb_i-\tb_{i-1}\ge \max\{1,2\hk+d_i\}$ (resp.~$\tb_0-\tb_{e-1}\ge \max\{3,2\hk+2+d_0\}$) then $\cB \rightarrow_{\text{Sc}_i} \Phi_i(\cB)$ (resp.~$\cB\rightarrow_{\text{Sc}_0} \Phi_0(\cB)$). In either case, if $(\bla,\ba) \in \cB$ and $(\bla',\ba) \in \Phi_i(\cB)$ then $|\bla'|<|\bla|$, so this procedure must terminate and we end up with a combinatorial block $\cB'$ such that $\psi(\cB')$ is $\hk$-Scopes-minimal.   
\end{proof}

We now want to show that the block $\cB'$ described in~\cref{P:ScopesMinExists} is unique. 

\begin{lem} \label{L:MeetUp}
Suppose that $\cB \in \Bc$ and let $\hk = \hk(\cB)$. Suppose that there exist $0 \leq i,j \leq e-1$ with $i \ne j$ such that
$\cB \rightarrow_{\text{Sc}_i} \Phi_i(\cB)$ and $\cB \rightarrow_{\text{Sc}_j} \Phi_j(\cB)$ and such that these maps are non-trivial.
\begin{itemize}
\item If $j \not\equiv i \pm 1 \pmod e$ then
\[
\Phi_i(\cB) \rightarrow_{\text{Sc}_j} \Phi_j(\Phi_i(\cB)) \text{ and }\Phi_j(\cB) \rightarrow_{\text{Sc}_i} \Phi_i(\Phi_j(\cB))
\] and
$\Phi_j(\Phi_i(\cB))=\Phi_i(\Phi_j(\cB))$. 
\item If $j \equiv i \pm 1 \pmod e$ then 
\[
\Phi_i(\cB) \rightarrow_{\text{Sc}_{j}} \Phi_{j}(\Phi_i(\cB)) \rightarrow_{\text{Sc}_i}
\Phi_i(\Phi_{j}(\Phi_i(\cB))) \text{ and } \Phi_j(\cB) \rightarrow_{\text{Sc}_{i}} \Phi_{i}(\Phi_j(\cB)) \rightarrow_{\text{Sc}_j}
\Phi_j(\Phi_{i}(\Phi_j(\cB)))
\]
and $\Phi_j(\Phi_{i}(\Phi_j(\cB))) = \Phi_i(\Phi_{j}(\Phi_i(\cB)))$.
\end{itemize}
\end{lem}

\begin{proof}
\begin{itemize}
\item Suppose $j \not \equiv i \pm 1 \pmod e$. For simplicity we assume that $i,j \ne 0$. The operation $\Phi_i$ swaps runners $i$ and $i-1$ of the abacus configurations of the elements of $\cB$ and does not change runners $j$ and $j-1$, so that 
$\cB \rightarrow_{\text{Sc}_j} \cB' \implies\Phi_i(\cB) \rightarrow_{\text{Sc}_j} \Phi_j(\Phi_i(\cB))$; 
similarly $\Phi_j(\cB) \rightarrow_{\text{Sc}_i} \Phi_i(\Phi_j(\cB))$.
It is clear that the operations $\Phi_i$ and $\Phi_j$ commute.
\item Suppose that $j\equiv i+1 \pmod e$. (The case $j \equiv i-1 \pmod e$ is almost identical.) For simplicity, we also assume $1 \leq i \le e-1$. Let $\psi(\cB)=(\tb,\nn,\pp)$. 

First assume $\min\{\nn,\pp\}=0$ or that $\tb_{i-1} \equiv 0 \pmod 2$ or that $\tb_{i+1} \equiv 0 \pmod 2$. The operation $\Phi_i$ swaps runners $i$ and $i-1$ of the abacus configurations of the elements of $\cB$ so that $\Phi_i(\cB) \rightarrow_{\text{Sc}_{i+1}} \Phi_{i+1}(\Phi_i(\cB))$ if and only if $\tb_{i+1}-\tb_{i-1} \ge 2\hk$, by \cref{L:Scopest1}.
Since $\cB \rightarrow_{\text{Sc}_i} \Phi_i(\cB)$ and $\cB \rightarrow_{\text{Sc}_{i+1}} \Phi_{i+1}(\cB)$ we have $\tb_i-\tb_{i-1}\ge 2\hk$ and $\tb_{i+1}-\tb_{i}\ge 2\hk$, so this does indeed hold. Then 
$\Phi_{i+1}(\Phi_i(\cB)) \rightarrow_{\text{Sc}_i} \Phi_i(\Phi_{i+1}(\Phi_i(\cB)))$ if and only $\cB \rightarrow_{\text{Sc}_{i+1}} \Phi_{i+1}(\cB)$, which we know holds.  

Now assume that $\nn,\pp>0$ and that $\tb_{i-1},\tb_{i+1} \equiv 1 \pmod 2$. Then $\Phi_i(\cB) \rightarrow_{\text{Sc}_{i+1}} \Phi_{i+1}(\Phi_i(\cB))$ if and only if $\tb_{i+1}-\tb_{i-1} \ge 2\hk+2$. If $\tb_{i} \equiv 1 \pmod 2$ then $\tb_{i+1}-\tb_i \ge 2\hk +2$ and $\tb_{i}-\tb_{i-1} \ge 2 \hk+2$ so the condition holds. If $\tb_{i} \equiv 0 \pmod 2$ then $\tb_{i+1}-\tb_i \ge 2\hk$ and $\tb_{i}-\tb_{i-1} \ge 2 \hk$ by \cref{L:Scopest2}, but due to the congruences modulo $2$ we actually have $\tb_{i+1}-\tb_i,\tb_i-\tb_{i-1} \ge 2\hk+1$. Hence $\tb_{i+1}-\tb_{i-1} \ge 2\hk+2$ as required. As above, $\Phi_{i+1}(\Phi_i(\cB)) \rightarrow_{\text{Sc}_i} \Phi_i(\Phi_{i+1}(\Phi_i(\cB)))$ if and only $\cB \rightarrow_{\text{Sc}_{i+1}} \Phi_{i+1}(\cB)$.

We have shown that $\Phi_i(\cB) \rightarrow_{\text{Sc}_{i+1}} \Phi_{i+1}(\Phi_i(\cB)) \rightarrow_{\text{Sc}_i}
\Phi_i(\Phi_{i+1}(\Phi_i(\cB)))$ and a similar argument shows that $\Phi_{i+1}(\cB) \rightarrow_{\text{Sc}_{i}} \Phi_{i}(\Phi_{i+1}(\cB)) \rightarrow_{\text{Sc}_{i+1}}
\Phi_{i+1}(\Phi_{i}(\Phi_{i+1}(\cB)))$. Again, it is clear that $\Phi_{i+1}(\Phi_{i}(\Phi_{i+1}(\cB))) = \Phi_i(\Phi_{i+1}(\Phi_i(\cB)))$.\qedhere
\end{itemize}
\end{proof}

\begin{prop} \label{P:MinUnique}
Let $\mathfrak{B}$ be a $\Sc$-equivalence class of $\Bl$ such that $\cB \in \Bc$ and let $\hk = \hk(\cB)$ for $\cB \in \mathfrak{B}$.
Then there exists a unique $\cD \in \mathfrak{B}$ such that $\psi(\cD)$ is $\hk$-Scopes-minimal. 
\end{prop}

\begin{proof}
The existence of a block $\cD$ as above holds by~\cref{P:ScopesMinExists}.
Suppose the result is false, that is, $\cD$ is not unique. 
Then there exists $\cB \in \mathfrak{B}$ such that we can find distinct $\cB_1, \cB_2 \in \mathfrak{B}$ with respective core blocks $\cC_1$ and $\cC_2$ such that $\phi^{-1}(\cC_1)$ and $\phi^{-1}(\cC_2)$ are $\hk$-Scopes-minimal and $\cB \rightarrow_{\text{Sc}^\ast} \cB_1$ and $\cB \rightarrow_{\text{Sc}^\ast} \cB_2$. 
Choose $\cB$ satisfying this condition such that $\nB$ is minimal.
Then there exist $i\ne j$ such that 
$\cB \rightarrow_{\text{Sc}_i} \Phi_i(\cB)$ and $\cB \rightarrow_{\text{Sc}_j} \Phi_j(\cB)$.
By~\cref{L:MeetUp} we can find $\cB' \in \mathfrak{B}$ such that 
$\Phi_i(\cB)\rightarrow_{\text{Sc}^\ast} \cB'$ and $\Phi_j(\cB)\rightarrow_{\text{Sc}^\ast} \cB'$. 

If $\cB' \rightarrow_{\text{Sc}^\ast} \cB_1$ and $\cB' \rightarrow_{\text{Sc}^\ast} \cB_2$ then since $\nB[\cB']<\nB$ we have a contradiction to our choice of $\cB$. Otherwise 
$\cB' \rightarrow_{\text{Sc}^\ast} \cB''$ for some $\cB''$ with $\cB''\ne \cB_1$ or $\cB'' \ne\cB_2$ and $\phi^{-1}(\cB'')$ Scopes-minimal. Then $\Phi_i(\cB) \rightarrow_{\text{Sc}^\ast} \cB_1$ and $\Phi_i(\cB) \rightarrow_{\text{Sc}^\ast} \cB''$ and $\nB[\Phi_i(\cB)]<\nB$. 
Again this contradicts the choice of $\cB$. 
\end{proof}

Combining~\cref{P:ScopesMinExists} and~\cref{P:MinUnique}, we obtain the following theorem. 

\begin{thm} \label{T:ScopesReps}
Let $\hk \geq 0$. Recall that $\Bc_{\hk} = \{\cB \in \Bc \mid \hk(\cB)=\hk\}$. The set
\[\left\{\cB \in \Bc \mid \psi(\cB) \text{ is $\hk(\cB)$-Scopes-minimal} \right\}\]
is a complete set of representatives of the $\Sc$-equivalence classes
of $\Bc_{\hk}$. 
\end{thm}

\begin{cor}
Let $\mathfrak{B}$ be a $\Sc$-equivalence class in $\Bl$. Then there exists a unique combinatorial block $\cB \in \mathfrak{B}$ such that if $\cB' \in \mathfrak{B}$ then $\nB\le \nB[\cB']$. We call $\cB$ a Scopes-minimal class.
\end{cor}

\begin{proof}
Suppose that each $\cB \in \mathfrak{B}$ has bicharge $\ba=(a_1,a_2)$. 
For $\ba' \equiv_e \ba$ let $\cB^{\ba'}= \{(\bla,\ba') \mid (\bla,\ba) \in \cB\}$. Then there exists some such $\ba'$ with the further property that $\cB^{\ba'}\in \Bc$ for all $\cB \in \mathfrak{B}$. 
By~\cref{T:ScopesReps}, the set $\{\cB^{\ba'} \mid \cB \in \mathfrak{B}\}$ contains a Scopes-minimal class; hence does does $\mathfrak{B}$.  
\end{proof}

\begin{eg}
Take $e=8$. Let $\bla=((12,5,3,2,1^{13}),(11,9,8,6^2,5,4,2^2,1))$ and $\ba=(18,19)$ and let $\cB$ be the combinatorial block containing $(\bla,\ba)$. We find the Scopes-minimal block $\cB'$ with $\cB' \Sc \cB$.
We have $\bar{\bla}=((12,5,3,2,1^5),(11,9,8,6^2,5,4,2^2,1))$ so that 

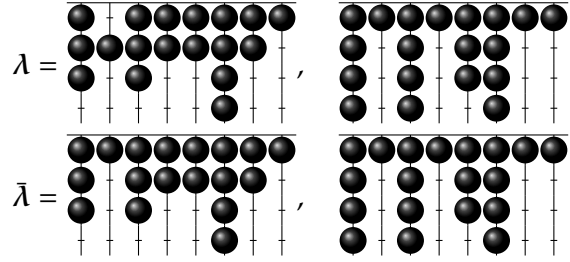
\begin{align*}
\bla & = \abacusline(8,1,bnbbbbbb,bbbbbbbn,bnbnnbnn,nnnnnbnn), \quad \abacusline(8,1,bbbbbbbb,bnbnbbnn,bnbnbbnn,bnbnnbnn) \\
\bar{\bla} & = \abacusline(8,1,bbbbbbbb,bnbbbbbn,bnbnnbnn,nnnnnbnn), \quad \abacusline(8,1,bbbbbbbb,bnbnbbnn,bnbnbbnn,bnbnnbnn)
\end{align*}

Let $\cC$ be the combinatorial block containing $(\bar{\bla},\ba)$. By~\cref{L:basetuple1} we see that $\cC$ is a core block and $\cC \in \Cl$, so $\cB \in \Bc$ with $\hk(\cB)=1$.
Furthermore if we choose base tuple $(3,0,3,1,2,3,1,0)$, then $\delta_{\bar{\bla}} = (0,0,-,-,+,+,+,0)$.
It follows that $\bar{\bla}$ is minimal in $\cC$.
Thus we have $\psi(\cB) = ((7,2,7,3,5,8,3,2),2,3)$. 

To find the Scopes-minimal block, we apply the operations $\rightarrow_{\text{Sc}_i}$ to $(7,2,7,3,5,8,3,2)$ until we reach a term $\tb'$ such that $(\tb',2,3)$ is $1$-Scopes-minimal. (This is a slight abuse of notation; we are really applying the operations to $\cC$.) 
\begin{multline*}
(7,2,7,3,5,8,3,2) \rightarrow_{\text{Sc}_2} 
(7,7,2,3,5,8,3,2) \rightarrow_{\text{Sc}_5} 
(7,7,2,3,8,5,3,2) \rightarrow_{\text{Sc}_4} 
(7,7,2,8,3,5,3,2) \rightarrow_{\text{Sc}_3} \\
(7,7,8,2,3,5,3,2) \rightarrow_{\text{Sc}_0} 
(4,7,8,2,3,5,3,5) \rightarrow_{\text{Sc}_1} 
(7,4,8,2,3,5,3,5) \rightarrow_{\text{Sc}_2} 
(7,8,4,2,3,5,3,5)
\end{multline*}
Then $\phi^\ast((7,8,4,2,3,5,3,5),2,3)=(\bmu,\ba)$ where
$\bmu=((8,1^6),(7^3,6,5,3^3,2))$:
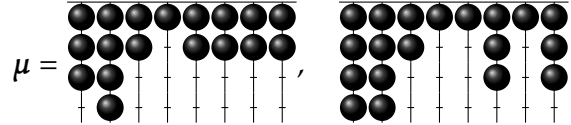
\[\bmu=\abacusline(8,1,bbbbbbbb,bbbnbbbb,bbnnnnnn,nbnnnnnn),\quad
\abacusline(8,1,bbbbbbbb,bbbnnbnb,bbnnnbnb,bbnnnnnn)\]

So the Scopes-minimal block in the $\Sc$-equivalence class containing $\cB$ is the combinatorial block $\cB'$ with $\hk(\cB')=1$ and core block containing $(\bmu,\ba)$; or equally, the combinatorial block containing $(((8,1^6),(7^3,6,5,3^3,2,1^8)),(18,19))$.
\end{eg}

Recall that the equivalence relation $\SU$ is the relation on $\Bl$ generated by $\Sc$ and $\Sh$.
Our next aim is find a set of $\SU$-representatives of $\Bl$. However, the relation $\Sh$ only changes the bicharge of a combinatorial block and so, given~\cref{T:ScopesReps}, this will be relatively straightforward. Given a $\SU$-equivalence class of $\Bl$, we just need to define a canonical choice of bicharge. Following the previous work in this section, we will choose the representative to be an element of $\Bc$.

In fact, we define two different sets of representatives. The first is probably the more natural, but the second makes it easier to describe the combinatorics needed in~\cref{S:SIC2}. 

\begin{lem} \label{L:ShiftCharge}
Suppose that $\cB \in\Bc$ has bicharge $\ba=(a_1,a_2)$. For $\ba'\equiv_e \ba$, set $\cB^{\ba'}=\{(\bla,\ba') \mid (\bla,\ba) \in \cB\}$. We assume $\ba'=(a'_1,a'_2)$.   
\begin{itemize}
\item Suppose that $\nn(\cB),\pp(\cB) >0$. Let $\ba' \equiv_e \ba$. Then $\cB^{\ba'} \in \Bc$ if and only if there exists $m \in \Z$ such that $a'_k=a_k+me$ for $k=1,2$.
\item Suppose that $\nn(\cB)=\pp(\cB)=0$. Let $\ba' \equiv_e \ba$. Then $\cB^{\ba'} \in \Bc$ if and only if there exists $m \in \Z$ such that $a'_k=a_k+me$ for $k=1,2$ or there exists $m_1 \in \Z$ such that $a'_1=a_1+(m_1+1)e$ and $a'_2=a_2+m_1e$ or there exists $m_2 \in \Z$ such that $a'_1=a_1+m_2e$ and $a'_2=a_2+(m_2+1)e$. In the second case we have $\nn(\cB^{\ba'})=e$ and $\pp(\cB^{\ba'})=0$ and in the third case we have $\nn(\cB^{\ba'})=0$ and $\pp(\cB^{\ba'})=e$. 
\item Suppose that $\nn(\cB)=0$ and $0<\pp(\cB)<e$. Let $\ba' \equiv_e \ba$. Then $\cB^{\ba'} \in \Bc$ if and only if there exists $m \in \Z$ such that $a'_k=a_k+me$ for $k=1,2$ or there exists $m_1 \in \Z$ such that $a'_1=a_1+(m_1+1)e$ and $a'_2=a_2+m_1e$. In the latter case, $\nn(\cB^{\ba'})=e-\pp(\cB)$ and $\pp(\cB^{\ba'})=0$. 
\item Suppose that $\pp(\cB)=0$ and $0 < \nn(\cB)<e$. Let $\ba'\equiv_e \ba$. Then $\cB^{\ba'} \in \Bc$ if and only if there exists $m \in \Z$ such that $a'_k=a_k+me$ for $k=1,2$ or there exists $m_1 \in \Z$ such that $a'_1=a_1+m_1e$ and $a'_2=a_2+(m_1+1)e$. In the latter case, $\nn(\cB^{\ba'})=0$ and $\pp(\cB^{\ba'})=e-\pp(\cB)$.
\end{itemize}
\end{lem}

\begin{defn}
Suppose $\cB \in \Bc$ with bicharge $\ba=(a_1,a_2)$. Set $\Shh(\ba)=(a_1+1,a_2+1)$ and $\Shh(\cB)=\{(\bla,\ba') \mid (\bla,\ba) \in \cB \text{ and } \ba'=\Shh(\ba)\}$. 
\end{defn}

\begin{defn}
Define an operation $\rightarrow_{\text{Sh}}$ on $\Z^e$ as follows. For $\tb\in \Z^e$, say that $\tb \rightarrow_{\text{Sh}} \tb'$ if $\tb'=(\tb_{e-1}+2,\tb_{0},\dots,\tb_{e-2})$. Let $\Sh$ be the equivalence relation on $\Z^e$ generated by $\rightarrow_{\text{Sh}}$. 
\end{defn}

\begin{lem} \label{L:Shiftt}
Suppose that $\cB,\cB' \in \Bc$ with $\psi(\cB)=(\tb,\nn,\pp)$ and $\psi(\cB')=(\tb',\nn',\pp')$. Then $\cB \rightarrow_{\text{Sh}} \cB'$ if and only if $\hk(\cB)=\hk(\cB')$, $\nn=\nn'$, $\pp=\pp'$ and $\tb \rightarrow_{\text{Sh}} \tb'$. 
\end{lem}

\begin{proof}
Suppose $\cC$ is the core block of $\cB$ and $\cC \rightarrow_{\text{Sh}} \cC'$. If $(\bla,\ba) \in \cC$ then $(\bla,\ba)^{\text{Sh}}=(\bla,\ba^{\text{Sh}})$ and 
\[b^{\ba^{\text{Sh}}}_{ik}(\bla) = \begin{cases} b^{\ba}_{i-1k}(\bla), & 1 \le i \leq e-1,\, k=1,2,\\  
b^{\ba}_{e-1k}(\bla)+1, & i=0, \,k=1,2.
\end{cases}\]
The claims in~\cref{L:Shiftt} then follow by definition. 
\end{proof}

\begin{defn} \label{D:TypeMin}
Let $\TT=(\tb,\nn,\pp) \in \Tl$ and suppose $\hk \ge 0$.
\begin{enumerate}
\item We say that $\TT$ is Type 1 $\hk$-minimal if
\begin{itemize}
\item $\TT$ is $\hk$-Scopes-minimal; and
\item if $\pp>0$ then $\nn>0$ and if $\pp=0$ then $\nn<e$; and 
\item $\min\{\tb_1,\tb_2,\dots,\tb_{e-1}\} \in \{1,2\}$; and
\item $\sum_{i=0}^{e-1}\tb_i \equiv \pp-\nn \pmod{2e}$.
\end{itemize}
The condition $\sum_{i=0}^{e-1}\tb_i \equiv \pp-\nn \pmod{2e}$ is equivalent to the condition that if $\phi(\tb,\nn,\pp)=(\bla,\ba)$ where $\ba =(a_1,a_2)$ then $a_1 \equiv 0 \pmod e$. 
\item We say that $\TT$ is Type 2 $\hk$-minimal if
\begin{itemize}
\item $\TT$ is $\hk$-Scopes-minimal; and 
\item if $\pp>0$ then $\nn>0$ and if $\pp=0$ then $\nn<e$; and 
\item $\tb_0 \in \{1,2\}$ and $\tb_i \geq 1$ for all $1 \leq i \leq e-1$.
\end{itemize}
\end{enumerate}
\end{defn}

For $\hk \ge 0$, recall the definitions of $\Bl_{\hk}$ and $\Bc_{\hk}$ from~\cref{D:HookFixed}. 

\begin{thm} \label{T:SUEquivReps}
Let $\hk \geq 0$.  
\begin{itemize}
\item The set
\[\{\cB \in \Bc_{\hk} \mid \psi(\cB) \text{ is Type $1$ $\hk$-minimal}\}\]
is a set of $\SU$-equivalence representatives of $\Bl_{\hk}$.
\item The set
\[\{\cB \in \Bc_{\hk} \mid \psi(\cB) \text{ is Type $2$ $\hk$-minimal}\}\]
is a set of $\SU$-equivalence representatives of $\Bl_{\hk}$.
\end{itemize}
Thus the $\SU$-equivalence classes of $\Bl_{\hk}$ are indexed by the Type $m$ $\hk$-minimal elements of $\Tl$, for $m=1$ and $m=2$. 
\end{thm}

\begin{proof}
Suppose that $\cB_1\in \Bl_{\hk}$. Then by~\cref{L:BzOK} we can find $\cB_2 \in \Bc$ such that $\cB_1 \Sh \cB_2$; and since the set $\Bl_{\hk}$ is closed under $\Sh$, we have that $\cB_2 \in \Bc_{\hk}$. By~\cref{T:ScopesReps}, we can find $\cB_3 \in \Bc_{\hk}$ such that $\psi(\cB_3)$ is $\hk$-Scopes-minimal and $\cB_3 \Sc \cB_2$. 

Thus we have $\cB_1 \SU \cB_3$ for some $\cB_3 \in \Bc_{\hk}$ with $\psi(\cB_3)$ $\hk$-Scopes-minimal. Suppose that $\cB_3$ has bicharge $\ba=(a_1,a_2)$ and that $\psi(\cB_3)=(\tb,\nn,\pp)$ and recall the results of~\cref{L:ShiftCharge}.
If $\pp>0$ and $\nn=0$, let $\cB_4=\{(\bla,(a_1+e,a_2)) \mid (\bla,\ba) \in \cB_3\}$.
Then $\cB_4\in\Bc$; suppose that $\psi(\cB_4)=(\tb',\nn',\pp')$. Then $\pp'=0$ and $\nn'=e-\pp<e$ and $\cB_3 \Sh \cB_4$.
If $\pp=0$ and $\nn=e$, let $\cB'_4=\{(\bla,(a_1,a_2+e)) \mid (\bla,\ba) \in \cB_3\}$.
Again $\psi(\cB_4) \in \Bc$; suppose that $\psi(\cB_4)=(\tb',\nn',\pp')$.
Then $\pp'=\nn'=0$ and $\cB_3 \Sh \cB_4$.
Hence $\cB_4 \SU \cB_1$ and $\psi(\cB_4)$ satisfies the first two conditions for Type $m$ $\hk$-minimality, for $m=1,2$.

Note that if $\ba'=(a_1+t,a_2+t)$ for $t \in \Z$ and $\cB_5=\{(\bla,\ba') \mid (\bla,\ba) \in \cB_4\}$  then $\psi(\cB_5)$ satisfies the first two conditions for Type $m$ $\hk$-minimality, for $m=1,2$. Using~\cref{L:Shiftt}, we may then choose $t$ such that $\cB_5 \SU \cB_1$ satisfies all conditions to be either Type $1$ $\hk$-minimal or Type 2 $\hk$-minimal.   

Now let $m\in\{1,2\}$ and suppose that we have $\cB,\cB' \in \Bc_{\hk}$ with $\cB \SU \cB'$ and $\psi(\cB)$ and $\psi(\cB')$ both Type $m$ $\hk$-minimal. By~\cref{T:ScopesReps}, $\cB \Sh \cB'$. Suppose that $\cB$ has bicharge $\ba$ and $\cB'$ has bicharge $\ba'$. 
Write $a'_1-a_1=xe+i$ where $0 \le i<e$; by definition $a'_2-a_2=ye+i$ for some $y \in \Z$. 
However, following~\cref{L:ShiftCharge}, since $\psi(\cB)$ and $\psi(\cB')$ both satisfy the second condition to be Type $m$ $\hk$-minimal, we must also have $x=y$.
It then follows from~\cref{L:Shiftt} that since $\psi(\cB)$ and $\psi(\cB')$ both satisfy the remaining conditions to be Type $m$ $\hk$-minimal, we must actually have $\ba=\ba'$ and so $\cB=\cB'$.  
\end{proof}

As suggested by~\cref{T:ScopesReps}, the only difference between the Type $1$ and Type $2$ $\hk$-minimal representatives is a potential change in the bicharge. We state this more precisely.

\begin{cor}
Suppose that $\cB \in \Bl_{\hk}$ and $\cB_1,\cB_2 \in \Bc_{\hk}$ with $\psi(\cB_m)$ Type $m$ $\hk$-minimal for $m=1,2$ and that $\cB_1 \SU \cB \SU \cB_2$. Suppose that $\cB_1$ has bicharge $\ba$ and $\cB_2$ has bicharge $\ba'$. Then there exists $s \in \Z$ such that $a'_k=a_k+s$ for $k=1,2$ and 
\[\cB_2 = \{(\bla,\ba') \mid (\bla,\ba) \in \cB_1\}.\] 
\end{cor}

One can find the Type $m$ $\hk$-minimal sequences by hand for small values of $e$ and $\hk$. In general, it is simple to write a computer program that will give the sequences.

\begin{eg}
Fix $e=3$. Suppose that we want to find all combinatorial blocks of defect $3$ up to $\SU$-equivalence. A combinatorial block $\cB$ of defect $3$ with core block $\cC$ must have $\defect(\cB)=3$ and $\defect(\cC)=1$; that is $\defect(\cC)=1$ and $\hk(\cB)=1$.

Suppose $\cC$ is a core block with $\defect(\cC)=1$. Let $\nn(\cC)=\nn$ and $\pp(\cC)=\pp$.
Then $\min\{\nn,\pp\}=1$, so we have
\[
(\nn,\pp) \in \{(1,1),(1,2),(2,1)\}.
\]
For each $(\nn,\pp)$ as above, we consider the Type $1$ $1$-minimal sequences $\TT=(\tb,\nn,\pp)$. That is, we want
\begin{itemize}
\item $\tb=(\tb_0,\tb_1,\tb_2) \in \Z^3$ with $\nn+\pp=|\{0 \leq i \leq 2 \mid \tb_i \equiv 1 \pmod 2\}$; and 
\item $\min\{\tb_0,\tb_1,\tb_2\}\in \{1,2\}$; and 
\item $\tb_0+\tb_1+\tb_2 \equiv \pp-\nn \pmod 6$; and
\item $\tb_1-\tb_0 < 2+d_1$ and $\tb_2-\tb_1 < 2+d_2$ and 
$\tb_0-\tb_2 <4+d_0$ where $d_i=2$ if $\tb_i,\tb_{i-1}\equiv 1 \pmod 2$ and $d_i=0$ otherwise. 
\end{itemize}

\underline{$(\nn,\pp)=(1,1)$:} We have $\tb$ equal to one of the following. 

\begin{align*}
&( 1, 2, 3), &&(1, 3, 2), &&(2, 1, 3), &&(2, 3, 1),&&(3, 1, 2), \\
&(3, 2, 1), &&(5, 5, 2), &&(4, 1, 1), &&(7, 2, 3),&& (5, 6, 1). 
\end{align*}

\underline{$(\nn,\pp)=(1,2)$:} We have $\tb$ equal to one of the following.

\begin{align*}
&(1, 3, 3), &&  (3, 1, 3), && (3, 3, 1), &&( 5, 1, 1), &&(5, 7, 1).   
\end{align*}

\underline{$(\nn,\pp)=(2,1)$:} We have $\tb$ equal to one of the following.

\begin{align*}
& (1, 1, 3), &&(1, 3, 1), && (3, 1, 1), && (5, 5, 1), & (7, 1, 3).
\end{align*}
Hence up to $\SU$-equivalence, there are 20 blocks of defect $3$ when $e=3$. 
\end{eg}

We will return to this example in~\cref{sec:e=3defect3}.

\section{Blocks that are Schurian-infinite: Using decomposition matrices}\label{sec:SIblocks}

Suppose $(\hhh,\ba)$ is a Hecke algebra with bicharge.
If $(B,\ba)$ is a block of $(\hhh,\ba)$ then $\cB=\{(\bla,\ba) \mid \text{$\rspe{\bla}$ belongs to $B$}\}$ will denote the corresponding combinatorial block. We will frequently abuse notation and write $\bla \in \cB$ rather than $(\bla,\ba) \in \cB$.

\begin{defn}
    We say that a set of bipartitions $S \subseteq \cB$ is an SI-subset of $B$ if each bipartition in $S$ is a Kleshchev bipartition and the submatrix of the block decomposition matrix given by the rows and columns labelled by these bipartitions is equal to one of the matrices (\ref{targetmatrix}), (\ref{targetmatrixalt}), (\ref{targetmatrixaltsquare}), (\ref{targetmatrixstar})
listed in \cref{prop:matrixtrick}.
\end{defn}

\begin{lem} \label{L:Swapsies}
Suppose that $(B,\ba)$ is a block of a Hecke algebra with bicharge $(\hhh,\ba)$ and that $\cB$ lies in $\Bc$ and has core block $\cC$. Let $\nn=\nn(\cC)$ and $\pp=\pp(\cC)$. Let \[\cB'=\{(\bla^{\text{Sw}},\ba^{\text{Sw}}) \mid (\bla,\ba) \in \cB\}\]
and suppose that $\cB'$ has core block $\cC'$. Take $B'$ to be the block of the Hecke algebra with bicharge $(\hhh,\ba^{\text{Sw}})$ corresponding to $\cB'$. Then
\begin{align*}
\nn(\cC')&=\pp(\cC), & \defect(\cB')&=\defect(\cB),\\
\pp(\cC')& =\nn(\cC),  & \defect(\cC')& =\defect(\cC),
\end{align*}
and $B$ is Schurian-infinite if and only if $B'$ is Schurian-infinite.
\end{lem}

\begin{proof}
It follows from the definition that $\nn(\cC')=\pp(\cC)$ and $\pp(\cC') =\nn(\cC)$. The formulae for the defects come from~\cref{P:defectpreserve} and the last result follows from~\cref{T:Equiv1}.
\end{proof}

\subsection{Core blocks of defect at least $2$} \label{subsec:SIcoreblocks}

Throughout this section, fix $(C,\ba)$ to be a core block of a Hecke algebra with bicharge $(\hhh,\ba)$. Following~\cref{T:Equiv1} and~\cref{L:basetuple2}, we may assume that $\cC \in \Cl$ and we take $\nn=\nn(\cC)$ and $\pp=\pp(\cC)$ to be the associated parameters as described in~\cref{D:np}. 
We also assume that $\defect(C)\geq 2$; note that by~\cref{L:s} we have $\defect(C)=\min\{\nn,\pp\}$ and we must therefore have that $e \ge 4$.
We note once again that this includes the $e=\infty$ situation, as we may instead choose a finite $e$ large enough to obtain an isomorphic algebra.

We identify a bipartition $\bla \in \cC$ with the sequence $\delta_{\bla}$ as given in~\cref{D:delta}. In the sequences $\delta_{\bla}$ below, we omit the $0$ terms.  
The next two lemmas are a consequence of \cref{T:Flat}.

\begin{lem} \label{L:PMatrix}
Suppose $\nn\geq 2$ and $\pp \ge 3$. Then the following bipartitions are an SI-subset of $C$ and the corresponding matrix is the matrix $(\ddag)$.
\[
\begin{matrix}
(-,\dots,-, -,+,-,+,+, +, \dots, + ) \\
(-,\dots,-, +,-,-,+,+, +, \dots, + ) \\
(-,\dots,-, -,+,+,-,+, +, \dots, + ) \\
(-,\dots,-, +,-,+,-,+, +, \dots, + ) \\
\end{matrix} \qquad 
\begin{pmatrix} 1 \\  v & 1 \\  v& 0 & 1 \\v^2 & v& v&1 \end{pmatrix} .
\]
\end{lem}

\begin{lem} \label{L:NMatrix}
Suppose $\nn\geq 3$ and $\pp \geq 2$. Then the following bipartitions are an SI-subset of $C$ and the corresponding matrix is the matrix $(\ddag)$.
\[
\begin{matrix}
(-,\dots,-, -,-,+,-,+, +, \dots, + ) \\
(-,\dots,-, -,+,-,-,+, +, \dots, + ) \\
(-,\dots,-, -,-,+,+,-, +, \dots, + ) \\
(-,\dots,-, -,+,-,+,-, +, \dots, + ) \\
\end{matrix} \qquad 
\begin{pmatrix} 1 \\ v & 1 \\ v& 0 & 1 \\ v^2 & v& v&1 \end{pmatrix} .
 \]
\end{lem}

\begin{thm} \label{T:CoreBlocksInf1}
Assume that $\ba=(a_1,a_2)$ and let $s_k$ be the equivalence class of $a_k$ modulo $e$ for $k=1,2$. 
\begin{itemize}
\item If $\defect(C)\geq 3$ then $C$ is Schurian-infinite.
\item If $\defect(C)=2$ and $s_1 \ne s_2$ then $C$ is Schurian-infinite. 
\end{itemize}
\end{thm}

\begin{proof}
Recall, from \cref{L:s}, that $\defect(C)=\min\{\nn,\pp\}$. 
If $\defect(C) \geq 3$, then $\nn \geq 3$ and $\pp\geq 3$. If $\defect(C)=2$ and $s_1 \ne s_2$, we have $\nn>\pp=2$ or $\pp>\nn=2$, again by \cref{L:s}. 
Hence in either case we may use \cref{L:PMatrix} or \cref{L:NMatrix} to provide an SI-subset of $C$, thus proving that $C$ is Schurian-infinite. 
\end{proof}

\subsection{Blocks whose core block has defect at least $3$ and blocks whose core block has defect $2$ where $s_1 \ne s_2$}\label{subsec:SIblockswithlargecoreblock}

For a bipartition $\bla=(\la^{(1)},\la^{(2)})$, let $\ell_2(\bla)$ denote the number of non-zero parts in $\la^{(2)}$. If $l\geq 0$, recall that $\bla^{l^+}$ is the bipartition obtained by adding a hook of shape $(1^{le})$ to the bottom of the second component of $\bla$. 

Throughout this section, we fix $C$ to be a core block of a Hecke algebra with bicharge $(\hhh,\ba)$ and let $\cC$ denote the corresponding combinatorial block; by \cref{T:Equiv1} and \cref{L:basetuple2} there is no loss in assuming that $\cC \in \Cl$.
We also fix $l \ge 0$ and let $\cB \in \Bl$ be the combinatorial block such that 
\[
\cB\supseteq \{(\bla^{l^+},\ba) \mid (\bla,\ba) \in \cC\};
\]
that is, $\cB$ is the combinatorial block with core block $\cC$ and $\hk(\cB)=l$. 
Let $B$ be the block of the Hecke algebra with bicharge $(\mathscr{H}_{n+le},\ba)$ corresponding to $\cB$. 

Set $\ba=(a_1,a_2)$ and let $s_k$ denote the equivalence class of $a_k$ modulo $e$ for $k=1,2$.

\begin{lem} \label{L:SameParts}
Suppose that $\bla,\bmu \in \cC$ are such that $\delta_0(\bla) = \delta_0(\bmu) \in \{0,-\}$.
Then $\ell_2(\bla)=\ell_2(\bmu)$.  
\end{lem}

\begin{proof}
Let $(b_0,b_1,\dots,b_{e-1})$ denote the base tuple of $\cC$. 
Fix $\bla\in \cC$ with $\delta_0(\bla) \in \{0,-\}$ and let $\bmu=(\mu^{(1)},\mu^{(2)})$ be any bipartition in $\cC$ with $\delta_0(\bmu)=\delta_0(\bla)$. Let $\beta_{a_2}(\mu^{(2)})$ denote the set of $\beta$-numbers for $\mu^{(2)}$ with respect to $a_2$. Set
\[x=\min\{ x' \in \Z \mid x' \notin \beta_{a_2}(\mu^{(2)}) \}\] 
so that 
\[\ell_2(\bmu) = \#\{x' \in \beta_{a_2}(\mu^{(2)}) \mid x<x'\}.\]
Suppose that $0 \le i <e$ is minimal in the $\prec$ ordering. 
Then, since $\delta_0(\bmu) \neq +$, we have that $x=b_i e + i$. Moreover, consider a bipartition $s_{uv}(\bmu)$ where $\delta_u(\bmu)=-$ and $\delta_v(\bmu)=+$ and $u \ne i$. The second component of this bipartition is formed from $\mu^{(2)}$ by moving a bead from the end of a runner labelled by $+$ on to the end of a runner labelled by $-$; by construction, both the bead and the empty position occur after $x$. Hence $\ell_2(s_{uv}(\bmu))=\ell_2(\bmu)$. Since we can get from $\bmu$ to $\bla$ by a sequence of these moves, we have $\ell_2(\bmu)=\ell_2(\bla)$.  
\end{proof}

\begin{lem} \label{L:BS}
Suppose that $\bla,\bmu \in\cC$ with $\bmu \in \mathscr{K}^{\bs}$ and $\ell_2(\bla)=\ell_2(\bmu)$. Let $l \geq 0$. Then $\bmu^{l^+} \in 
\mathscr{K}^{\bs}$ and
\[
d_{\bla\bmu}(v)  = d_{\bla^{l^+}\bmu^{l^+}}(v).
\]
\end{lem}

\begin{proof}
It follows from \cref{T:StillKlesh} that $\bmu^{l^+} \in \mathscr{K}^{\bs}$ and then the equality of the graded decomposition numbers follows from \cref{T:BowmanSpeyer2}.
\end{proof}

\begin{cor} \label{C:AddHook}
Suppose there exists an SI-subset $S$ of $C$ such that $\ell_2(\bla)=\ell_2(\bmu)$ for all $\bla,\bmu \in S$.
Then
\[
\{\bla^{l^+} \mid \bla \in S\}
\]
is an SI-subset of $B$.
\end{cor}

\begin{lem} \label{L:nlarge}
Suppose $\defect(C) \geq 2$ and $\nn(\cC) \geq 3$. 
Then there exists an SI-subset $S$ of $C$ such that $\ell_2(\bla)=\ell_2(\bmu)$ for all $\bla,\bmu \in S$.
\end{lem}

\begin{proof}
This follows from \cref{L:NMatrix,L:SameParts}.
\end{proof}

\begin{thm} \label{T:nlarge}
Suppose $\defect(C) \geq 3$ or $\defect(C) =2$ and $s_1 \neq s_2$.
Then $B$ is Schurian-infinite. 
\end{thm}

\begin{proof}
By \cref{L:s}, $\nn(\cC) \geq 3$ or $\pp(\cC) \geq3$.
If $\nn(\cC) \geq 3$, we may use \cref{L:NMatrix,L:SameParts} and \cref{C:AddHook} to find an SI-subset of $B$, proving that $B$ is Schurian-infinite.
Otherwise, the block $B'$ as described in \cref{L:Swapsies} satisfies the criteria above, with $\nn(\cB')\ge 3$. Hence $B'$ is Schurian-infinite and, by~\cref{L:Swapsies}, so is $B$. 
\end{proof}

\subsection{Blocks with $\defect(B)-\defect(C)\geq 4$}
\label{subsec:ManyRemHooks}

\begin{thm}\label{thm:2ormorehooks}
Let $B$ be a block of $\hhh$ with core block $C$. Suppose that $\defect(B)-\defect(C) \geq 4$ and $(e,p)\neq(3,2)$.
Then $B$ is Schurian-infinite in any characteristic. 
\end{thm}

\begin{proof}
Suppose that $\defect(B)-\defect(C)=2\hk$, where $\hk \geq 2$. 
Choose $\bmu=(\mu^{(1)},\mu^{(2)}) \in \cC$ minimal in the $\unrhd$ ordering so that $\bmu$ is a Kleshchev bipartition and $\mu^{(2)}$ is an $e$-core.
Consider the level one block $\tilde{B}$ with core $\mu^{(2)}$ and defect $\hk$.
In~\cite{als23}, the first three authors show that there is an SI-subset $S$ of $\tilde{B}$ (applying the type $\tta$ analogue of \cref{eq:decompnumbersandduals}).
Suppose that $S=\{\nu(1),\nu(2),\dots,\nu(t)\}$. Then the set
\[
S^+ = \{(\mu^{(1)},\nu(l)) \mid 1 \leq l \leq t\}
\]
is an SI-subset of $B$: each bipartition $(\mu^{(1)},\nu(l))$ is a Kleshchev bipartition by \cref{T:StillKlesh} and the submatrix of the decomposition matrix for $B$ indexed by the bipartitions in $S^+$ is equal to the submatrix of the decomposition matrix for $\tilde{B}$ indexed by the partitions in $S$ by \cref{cor:BowmanSpeyerconjugate}. 
\end{proof}

Hence, we have settled a large part of the cases we must consider. 
It remains to consider blocks for $e=3$, $p=2$, and blocks for $e\geq4$ with small defect.
More precisely, for $e\geq 4$, we must consider blocks $B$ such that $\defect(B) - \defect(C) \leq 2$ and either $\defect(C)\in\{0,1\}$ or $\defect(C)=2$ and $a_1\equiv a_2 \pmod e$.
Any block of defect 0 or 1 has finite representation type and is therefore Schurian-finite.
So it suffices to consider non-core blocks $B$ in those cases where $e\geq4$ and $\defect(C)\in\{0,1\}$.
We must consider $(e,p)=(3,2)$ separately.
This is done in \cref{sec:e3}.
Likewise, we must consider defect 2 core blocks with $a_1\equiv a_2 \pmod e$ separately, which is done in \cref{subsec:core-def-2-equal}.
In the following sections, we settle some of those exceptional cases.

\subsection{Non-core blocks of defect $2$ when $s_1 \ne s_2$} \label{S:SIC2}
Fix a block $B$ for a Hecke algebra with bicharge $(\hhh,\ba)$ such that $B$ is a block of defect $2$ which is not a core block. Let $s_k$ denote the equivalence class of $a_k$ modulo $e$ for $k=1,2$; we will assume that $s_1 \ne s_2$. 
In this section we prove that if $e \ge 4$ then $B$ is Schurian-infinite. (We consider the case $e=3$ in~\cref{sec:e=3defect2}.) This is made possible by the paper of Fayers~\cite{fay06} in which he gives a description of the decomposition numbers of such a block; these decomposition numbers are independent of the characteristic of the field. It therefore is possible to describe an SI-subset of each such block and use~\cite{fay06} to verify that the submatrix of the decomposition matrix provided by this subset has the form claimed. 

We may use \cref{P:FayersKlesh} and \cref{T:FayersDecomp} below to find an SI-subset of $B$. We first use $\SU$-equivalence as in~\cref{SS:Equivs} to reduce the number of cases that we need to consider.

\begin{defn}
Suppose that $0 < \nn<e$ and $\pp=0$. Say that $(\tb,\nn,\pp)\in \Tl$ is minimal if  
\begin{itemize}
\item For all $1 \leq i \leq e-1$ we have $\tb_i-\tb_{i-1}<2$; and 
\item $\tb_{0}-\tb_{e-1}<4$; and 
\item $\tb_0 \in \{1,2\}$ and $\tb_i \geq 1$ for all $1 \leq i \leq e-1$.
\end{itemize}
\end{defn}

\begin{lem} \label{L:AMC}
There exists a block $B'$ for a Hecke algebra with bicharge $(\hhh',\ba')$ such that $\cB' \in \Bc$, $\psi(\cB')$ is minimal and $\cB\SU\cB'$. 
Note that $B'$ is therefore also a non-core block of defect 2 with $s'_1 \ne s'_2$; moreover $B'$ is Schurian-infinite if and only if $B$ is Schurian-infinite.
\end{lem}

\begin{proof}
Note that $\hk(B)=1$ and the core block of $B$ has defect $0$. 
Furthermore, if $\min\{\nn,\pp\}=0$ and $1 \le \max\{\nn,\pp\}\leq e-1$ then the conditions for $(\tb,\nn,\pp)$ to be minimal agree with the conditions for $(\tb,\nn,\pp)$ to be Type $2$ $1$-minimal. 
The result then follows from~\cref{T:Equiv1} and~\cref{T:SUEquivReps}. 
\end{proof}

We may therefore assume that $\cB \in \Bc$ and $\psi(\cB)$ is minimal. 
Take $C$ to be the core block of $B$. Since $\defect(C)=0$, we have $|\cC|=1$ so assume $\cC=\{\bnu\}$. 
It therefore follows that 
\[
\tb_i=b^{\ba}_{i1}(\bnu)+b^{\ba}_{i2}(\bnu)
\]
for $0 \le i \leq e-1$. Set $\TT=(\tb,\nn,\pp)=\psi(\cB)$. Note that $\pp=0$ and therefore $\nn=|\{0 \leq i < e \mid \tb_i \equiv 1 \pmod 2\}|$.

In order to find an SI-subset of $B$, we recall some definitions and results of Fayers~\cite{fay06}. In the terminology of that paper, we are considering Type I blocks.
The definitions we make below are compatible with those given in~\cref{S:ScopesSec}.
In particular, the set $I$ defined in~\cref{D:I2} agrees with the set $I$ defined in~\cref{D:I}, and when we restrict the partial order $\preceq$ of~\cref{D:I2} on $\{0,1,\dots,e-1\}$ to $I$, we obtain the total order $\prec$ on $I$ from~\cref{D:I}.

\begin{defn} \label{D:I2}
We define
\begin{align*}
I & =\{0 \leq i <e \mid \tb_i \equiv 1 \pmod 2\} = \{0 \leq i <e \mid b^{\ba}_{i2}(\bnu)-b^{\ba}_{i1}(\bnu)=-1\}, \\ 
K &= \{0 \leq i <e \mid \tb_i \equiv 0 \pmod 2\}=\{0 \leq i <e \mid b^{\ba}_{i2}(\bnu)-b^{\ba}_{i1}(\bnu)=0\}.
\end{align*}
We define a partial order $\preceq$ on $\{0,1,\dots,e-1\}$ by saying that
\[
l \preceq m \iff \begin{cases}\tb_l \leq \tb_m \text{ and } l \leq m; \text{ or}\\
\tb_l+2 \leq \tb_m. \end{cases}
\]
Note that this restricts to a total order on $I$ and on $K$.
We use the symbol $\nsuccprec$ to denote incomparability under this partial order.
If $i \in I$ (resp.~$k \in K$) define $i^-$ (resp.~$k^-)$ to be the maximal element of $I$ (resp.~$K$) in the $\preceq$ order such that $i^- \prec i$ (resp.~$k^- \prec k$), if such an element exists. Define $i^+$ and $k^+$ similarly. 
\end{defn}

\begin{defn}
Suppose $0 \leq m <e$ and that $i \in I, k \in K$. Suppose $c \in \{1,2\}$.   
\begin{itemize}
\item Define $[m]^c$ to be the abacus configuration obtained from $(\bnu,\ba)$ by moving the bead on runner $m$ of component $c$ down one position.
\item Define $[i,k]$ to be the abacus configuration obtained from $(\bnu,\ba)$ by moving a bead from the lowest filled position on runner $i$ to the highest free position on runner $k$ on component $1$, and simultaneously moving a bead from the lowest filled position on runner $k$ to the highest free position on runner $i$ on component $2$.
\end{itemize}
\end{defn}

\begin{propc}{fay06}{Lemma~3.5} 
  The bipartitions in $B$ are
  \[\Big\{[m]^c \mid 0 \leq m \leq e-1, c\in\{1,2\}\Big\} \sqcup \Big\{[i,k] \mid i \in I, k \in K\Big\}.\]
\end{propc}

\begin{propc}{fay06}{Lemma~3.11} \label{P:FayersKlesh}
  We keep notation as above.
  \begin{itemize}
\item For $i \in I$, we have that $[i]^1 \in \Kl$ if and only if there is some $k \in K$ with $i \prec k$.
\item For $k \in K$, we have that $[k]^1 \in \Kl$ if and only if there are some $i \in I$, $l \in K$ with $i \prec k \prec l$.
\item For $i \in I$, we have that $[i]^2 \in \Kl$ if and only if there is some $m \in I \cup K$ with $m \npreccurlyeq i$.
\item For $k \in K$, we have that $[k]^2 \in \Kl$ if and only if there is some $m \in I \cup K$ with $k \prec m$.
  \item For $i\in I$ and $k\in K$, we have that $[i,k] \in \Kl$ if and only if either $i \npreccurlyeq k$ or there exist $j \in I$, $l \in K$ with $j \prec i$ and $k \prec l$.  
    \end{itemize}
\end{propc}

\begin{defn}
    For $\bmu \in \Kl$, we define $\bmu^{\ast}$ to be the conjugate of the Mullineux image $m(\bmu)$ of $\bmu$, as introduced in \cref{subsec:Schur functor}.
\end{defn}

\begin{thmc}{fay06}{Theorem~3.14 and Table~1} \label{T:FayersDecomp}
Suppose that $\bla, \bmu \in B$ with $\bmu \in \Kl$.
Then the graded decomposition number $d_{\bla\bmu}(v)$ is either $0$, $v$ or $v^2$, and is independent of the characteristic $p$.
For each $\bmu$, the bipartitions $\bla$ with $d_{\bla\bmu} \ne 0$ are listed in~\cref{Fayers:DecompTable}, and we have $d_{\bla\bmu}(v)=v^2$ if and only if $\bla = \bmu^{\ast}$.
In each case, conditions involving $i^-,i^+,k^-$ or $k^+$ should be ignored if these elements do not exist.
\end{thmc}

\newlength{\Fblockheight}
\setlength{\Fblockheight}{\dimexpr 4\baselineskip + 52pt\relax}
\newcommand{\Hblockheight}{27pt}

\begin{figure}
{\small
\begin{tabular}{@{} 
  >{\centering\arraybackslash}m{0.8cm}  
  >{\centering\arraybackslash}m{1.1cm}  
  >{\centering\arraybackslash}m{3.6cm}    
  >{\centering\arraybackslash}m{1.1cm}  
  >{\centering\arraybackslash}m{3.3cm}  
  >{\centering\arraybackslash}m{3.8cm}    
  @{}}
\toprule
Case & $\bmu$ & conditions & $\bmu^{\ast}$ & additional conditions & $\bla$ for which $\bmu \prec \bla \prec \bmu^{\ast}$ and $[\rspe{\bla}:\rD{\bmu}]=1$ \\
\midrule
\multirow{2}{*}{$\vcenter{\hbox{A}}$}
 & \multirow{2}{*}{$\vcenter{\hbox{$[i]^1$}}$}
 &
 \multirow{2}{*}{
  \parbox[c]{3cm}{\centering
    $i \in I$ \\
    $(\exists k \in K)(k \succcurlyeq i \npreccurlyeq k^-)$
  }
}
 & \multirow{2}{*}{\makecell{\centering $[ik]$}}
 & $(\exists i^+,\, k \succcurlyeq i^+)$ & $[i^+k], [i^+]^1$ \\
\cmidrule(lr){5-6}
 &  &  &  & $(k \not\succcurlyeq i^+)$ & $[k]^1$ \\
\midrule
\multirow{2}{*}{\makecell[c]{A$'$}} 
 & \multirow{2}{*}{\makecell[c]{$[ik]$}}
 & \multirow{2}{*}{\parbox[c]{3.6cm}{\centering
     $k \preceq i \nsucccurlyeq k^+$
   }}
 & \multirow{2}{*}{\makecell[c]{$[i]^2$}}
 & $(\exists i^-,\, k \preceq i^-)$
 & $[i^-]^2,\ [i^-k]$ \\
\cmidrule(lr){5-6}
 & & & & $(k \npreccurlyeq i^-)$ & $[k]^2$ \\
\midrule
B & \makecell[c]{$[i]^2$}
  & \parbox[c]{3.6cm}{\centering
      $i\in I,\ \exists i^+$ \\[0.5ex]
      $(\forall k\in K)(k\succeq i^+ \text{ or } i\succeq k)$
    }
  & \makecell[c]{$[i^+]^1$}
  & \makecell[c]{---}
  & \makecell[c]{$[i]^1,\ [i^+]^2$} \\
\midrule
\multirow{2}{*}{\makecell[c]{C}}
 & \multirow{2}{*}{\makecell[c]{$[i]^2$}}
 & \multirow{2}{*}{\parbox[c]{3.6cm}{\centering
     $i\in I$ \\[0.5ex]
     $(\exists k\in K)(i^+ \!\npreccurlyeq k \succeq i \succeq k^-)$
   }}
 & \multirow{2}{*}{\makecell[c]{$[k]^1$}}
 & $(\exists i^+,\, k \npreccurlyeq i^+)$
 & \makecell[c]{$[i]^1,\ [i^+k],\ [i^+]^2$} \\
\cmidrule(lr){5-6}
 & & & & $(k \preceq i^+)$ & \makecell[c]{$[i]^1,\ [k]^2$} \\
\midrule
\multirow{2}{*}{\makecell[c]{C$'$}}
 & \multirow{2}{*}{\makecell[c]{$[k]^2$}}
 & \multirow{2}{*}{\parbox[c]{3.6cm}{\centering
     $k\in K$ \\[0.5ex]
     $(\exists i\in I)(\,k^+ \succeq i \succeq k \npreccurlyeq i^-\,)$
   }}
 & \multirow{2}{*}{\makecell[c]{$[i]^1$}}
 & $(\exists i^-,\; k \nsucccurlyeq i^-)$
 & \makecell[c]{$[i]^2,\ [i^-k],\ [i^-]^1$} \\
\cmidrule(lr){5-6}
 & & & & $(k \succeq i^-)$ & \makecell[c]{$[i]^2,\ [k]^1$} \\
\midrule
\multirow{2}{*}{\makecell[c]{D}}
 & \multirow{2}{*}{\makecell[c]{$[i]^2$}}
 & \multirow{2}{*}{\parbox[c]{3.6cm}{\centering
     $i\in I$ \\[0.5ex]
     $(\exists k\in K)(\,k \nsuccprec i \succeq k^-\,)$
   }}
 & \multirow{2}{*}{\makecell[c]{$[ik]$}}
 & $(\exists i^+,\; k \npreccurlyeq i^+)$
 & \makecell[c]{$[i^+k],\ [i^+]^2$} \\
\cmidrule(lr){5-6}
 & & & & $(k \preceq i^+)$ & \makecell[c]{$[k]^2$} \\
\midrule
\multirow{2}{*}{\makecell[c]{D$'$}}
 & \multirow{2}{*}{\makecell[c]{$[ik]$}}
 & \multirow{2}{*}{\parbox[c]{3.6cm}{\centering
     $k^{+}\,\succeq\, i \,\nsuccprec\, k$
   }}
 & \multirow{2}{*}{\makecell[c]{$[i]^1$}}
 & $(\exists i^-,\; k \nsucccurlyeq i^-)$
 & \makecell[c]{$[i^-]^1,\ [i^-k]$} \\
\cmidrule(lr){5-6}
 & & & & $(k \succeq i^-)$ & \makecell[c]{$[k]^1$} \\
\midrule
\multirow{2}{*}{\makecell[c]{E}}
 & \multirow{2}{*}{\makecell[c]{$[k]^1$}}
 & \multirow{2}{*}{\parbox[c]{3.6cm}{\centering
     $k\in K,\ \exists k^{+}$ \\[0.5ex]
     $(\exists i\in I)(\,i^{+} \npreccurlyeq k \succeq i\,)$
   }}
 & \multirow{2}{*}{\makecell[c]{$[ik^{+}]$}}
 & $(\exists i^{+},\; k^{+} \succeq i^{+})$
 & \makecell[c]{$[ik],\ [i^{+}k^{+}],\ [i^{+}]^{1}$} \\
\cmidrule(lr){5-6}
 & & & & $(k^{+} \nsucccurlyeq i^{+})$ & \makecell[c]{$[ik],\ [k^{+}]^{1}$} \\
\midrule
\multirow{2}{*}{\makecell[c]{E$'$}}
 & \multirow{2}{*}{\makecell[c]{$[ik]$}}
 & \multirow{2}{*}{\parbox[c]{3.6cm}{\centering
     $\exists k^{+}$ \\[0.5ex]
     $(\,i \succeq k^{+} \npreccurlyeq i^{-}\,)$
   }}
 & \multirow{2}{*}{\makecell[c]{$[k^{+}]^{2}$}}
 & $(\exists i^{-},\; k \preceq i^{-})$
 & \makecell[c]{$[ik^{+}],\ [i^{-}]^{2},\ [i^{-}k]$} \\
\cmidrule(lr){5-6}
 & & & & $(k \npreccurlyeq i^{-})$ & \makecell[c]{$[ik^{+}],\ [k]^{2}$} \\
\midrule
\multirow{4}{0.8cm}{\centering\begin{minipage}[c][\Fblockheight]{\linewidth}%
  \centering F%
\end{minipage}}%
 & \multirow{4}{1.1cm}{\centering\begin{minipage}[c][\Fblockheight]{\linewidth}%
   \centering $[k]^2$%
 \end{minipage}}%
 & \multirow{4}{3.6cm}{\centering\begin{minipage}[c][\Fblockheight]{\linewidth}%
   \centering
   \[
   \begin{gathered}
     k\in K,\ \exists k^+ \\[0.6ex]
     (\forall i\in I)\bigl(\,k^+\not\succcurlyeq i \succeq k \\[0.6ex]
       \text{or } k^+\succeq i \nsucccurlyeq k\,\bigr)
   \end{gathered}
   \]
 \end{minipage}}%
 & \multirow{4}{1.1cm}{\centering\begin{minipage}[c][\Fblockheight]{\linewidth}%
    \centering $[k^+]^1$%
 \end{minipage}}%
 & \makecell[c]{$(\exists j\in I)\,(j \nsuccprec k \preceq j^+)$ \\ $(\exists i\in I)\,(i \nsuccprec k^+ \succeq i^-)$}
 & \makecell[c]{$[ik^+],\ [i]^2,\ [j]^1,\ [jk]$} \\
\cmidrule(lr){5-6}
 & & & & \makecell[c]{$(\nexists j\in I)\,(j \nsuccprec k)$ \\ $(\exists i\in I)\,(i \nsuccprec k^+ \succeq i^-)$}
 & \makecell[c]{$[ik^+],\ [i]^2,\ [k]^1$} \\
\cmidrule(lr){5-6}
 & & & & \makecell[c]{$(\exists j\in I)\,(j \nsuccprec k \preceq j^+)$ \\ $(\nexists i\in I)\,(i \nsuccprec k^+)$}
 & \makecell[c]{$[k^+]^2,\ [j]^1,\ [jk]$} \\
\cmidrule(lr){5-6}
 & & & & \makecell[c]{$(\nexists j\in I)\,(j \nsuccprec k)$ \\ $(\nexists i\in I)\,(i \nsuccprec k^+)$}
 & \makecell[c]{$[k^+]^2,\ [k]^1$} \\
\midrule
\multirow{2}{*}{\makecell[c]{G}}
 & \multirow{2}{*}{\makecell[c]{$[k]^2$}}
 & \multirow{2}{*}{\parbox[c]{3.6cm}{\centering
     $k\in K,\ \exists k^+$ \\[0.6ex]
     $(\exists i\in I)\,(i^+ \succeq k \nsuccprec i \nsuccprec k^+)$
   }}
 & \multirow{2}{*}{\makecell[c]{$[ik^+]$}}
 & \makecell[c]{$(\exists i^+,\,k^+ \not\preccurlyeq i^+)$}
 & \makecell[c]{$[ik],\ [i^+k^+],\ [i^+]^2$} \\
\cmidrule(lr){5-6}
 & & & & $(k^+ \preceq i^+)$ & \makecell[c]{$[ik],\ [k^+]^2$} \\
\midrule
\multirow{2}{*}{\makecell[c]{G$'$}}
 & \multirow{2}{*}{\makecell[c]{$[ik]$}}
 & \multirow{2}{*}{\parbox[c]{3.6cm}{\centering
     $\exists k^+$ \\[0.5ex]
     $k \nsuccprec i \nsuccprec k^+ \succeq i^-$
   }}
 & \multirow{2}{*}{\makecell[c]{$[k^+]^1$}}
 & $(\exists i^-,\; k \nsucccurlyeq i^-)$
 & \makecell[c]{$[ik^+],\ [i^-]^1,\ [i^-k]$} \\
\cmidrule(lr){5-6}
 & & & & $(k \succeq i^-)$ & \makecell[c]{$[ik^+],\ [k]^1$} \\
\midrule
\multirow{1}{0.8cm}{\centering\begin{minipage}[c][\Hblockheight]{\linewidth}%
  \centering H%
\end{minipage}}%
& \multirow{1}{1.1cm}{\centering\begin{minipage}[c][\Hblockheight]{\linewidth}%
  \centering $[ik]$%
\end{minipage}}%
& \parbox[c]{3.6cm}{\centering
    $\exists i^-,\,k^+$ \\[0.6ex]
    $(k \succeq i \ \text{or}\ i^- \succeq k^+)$ \\[0.6ex]
    $\text{or }(k \nsucccurlyeq i,\ i^- \nsucccurlyeq k^+)$
}%
& \multirow{1}{1.3cm}{\centering\begin{minipage}[c][\Hblockheight]{\linewidth}%
  \centering $[i^-k^+]$%
\end{minipage}}%
& \multirow{1}{1.3cm}{\centering\begin{minipage}[c][\Hblockheight]{\linewidth}%
  \centering ---%
\end{minipage}}%
& \multirow{1}{2.5cm}{\centering\begin{minipage}[c][\Hblockheight]{\linewidth}%
  \centering $[ik^+],\ [i^-k]$%
\end{minipage}} \\
\midrule
\end{tabular}}
\caption{Decomposition numbers for non-core blocks of defect $2$}
\label{Fayers:DecompTable}
\end{figure}

We are finally ready to prove the main theorem of this section. We first deal with some exceptional cases. 

\begin{lem} \label{L:Ex1}
  Let $\tb=(2,3,1,2)$. Then 
  \[
  \{[0,1],[2]^2,[0]^2,[3,1],[1]^2\}
  \]
  is an SI-subset of $B$ which gives the matrix~(\ref{targetmatrixstar}).
\end{lem}

\begin{eg}
Let us look at~\cref{L:Ex1} in more detail.
We have $\TT=((2,3,1,2),2,0)$ which means that $\phi(\TT)$ gives the following abacus configuration.
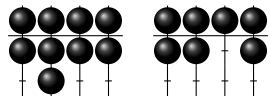
\[
\abacusline(4,0,bbbb,bbbb,nbnn) \quad \abacusline(4,0,bbbb,bbnb,nnnn)
\]
Hence $e=4$, $\bnu=((1),(1))$ and $\ba=(5,3)$ so that $\hhh=\mathscr{H}_6(q,q^1,q^3)$ where $q$ is a primitive $4$th root of unity over our field $\bbf$. If $\bla \in B$ then $\Res_{\bs}(\bla)=\{0,1,1,2,3,3\}$.  
We have $I=\{0,3\}$ and $K=\{1,2\}$. The five bipartitions in the SI-subset  as follows.
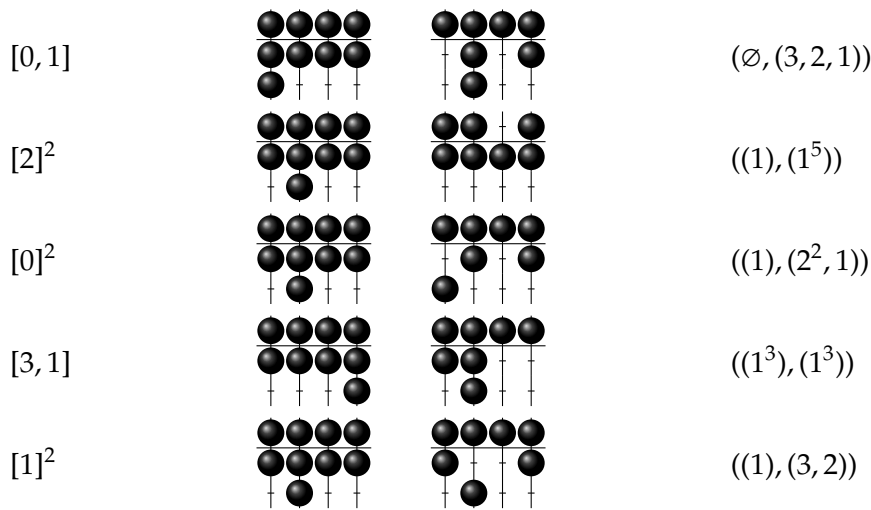
\begin{align*}
& [0,1] && \abacusline(4,0,bbbb,bbbb,bnnn) \qquad \abacusline(4,0,bbbb,nbnb,nbnn) && (\varnothing,(3,2,1)) \\
& [2]^2 && \abacusline(4,0,bbbb,bbbb,nbnn) \qquad \abacusline(4,0,bbnb,bbbb,nnnn) && ((1),(1^5)) \\
&[0]^2 && \abacusline(4,0,bbbb,bbbb,nbnn) \qquad \abacusline(4,0,bbbb,nbnb,bnnn) && ((1),(2^2,1)) \\
&[3,1] && \abacusline(4,0,bbbb,bbbb,nnnb) \qquad \abacusline(4,0,bbbb,bbnn,nbnn) && ((1^3),(1^3)) \\
&[1]^2 && \abacusline(4,0,bbbb,bbbb,nbnn) \qquad \abacusline(4,0,bbbb,bnnb,nbnn) && ((1),(3,2))
\end{align*}
This example is small enough to check by hand; or we may use~\cref{P:FayersKlesh} to verify that these are indeed Kleshchev multipartitions and~\cref{T:FayersDecomp} to see that the submatrix of the graded decomposition matrix given by these bipartitions is indeed the matrix~(\ref{targetmatrixstar}): 
\[\begin{matrix*}[r] (\varnothing,(3,2,1)) \\ ((1),(1^5)) \\((1),(2^2,1)) \\ ((1^3),(1^3)) \\ ((1),(3,2)) \end{matrix*} \quad 
\begin{pmatrix} 1 &&&& \\
0 & 1 & & &\\
v & v & 1 && \\
0 & v^2 & v & 1& \\
v^2 & 0 & v&0&1 
\end{pmatrix}
\]
\end{eg}

\begin{lem} \label{L:Ex2}
  Let $\tb=(2,3,2,1)$. Then 
  \[\{[0]^2,[3,0],[1,2],[3,2]\}\]
  is an SI-subset of $B$ which gives the matrix~(\ref{targetmatrixalt}).
\end{lem}

\begin{eg}
Let us look at~\cref{L:Ex2} in more detail. We have $\TT=((2,3,2,1),2,0)$ which means that $\phi(\TT)$ gives the following abacus configuration.
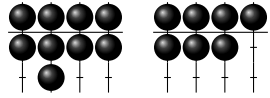
\[
\abacusline(4,0,bbbb,bbbb,nbnn) \quad \abacusline(4,0,bbbb,bbbn,nnnn)
\]
Hence $e=4$, $\bnu=((1),\varnothing)$ and $\ba=(5,3)$ so that $\hhh=\mathscr{H}_5(q,q^1,q^3)$ where $q$ is a primitive $4$th root of unity over our field $\bbf$. 
If $\bla \in B$ then $\Res_{\bs}(\bla)=\{0,1,1,2,3\}$. We have $I=\{1,3\}$ and $K=\{0,2\}$. 
The four bipartitions in the SI-subset are
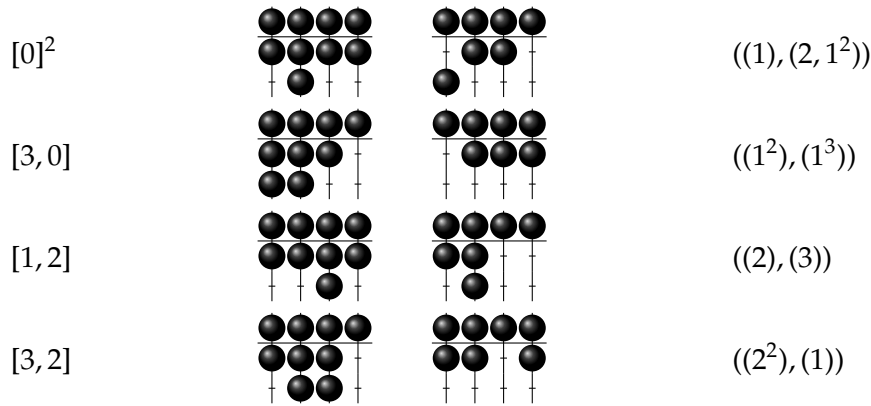
\begin{align*}
&[0]^2 && \abacusline(4,0,bbbb,bbbb,nbnn) \qquad \abacusline(4,0,bbbb,nbbn,bnnn) &&  ((1),(2,1^2)) \\
&[3,0]&&\abacusline(4,0,bbbb,bbbn,bbnn) \qquad \abacusline(4,0,bbbb,nbbb,nnnn) && ((1^2),(1^3)) \\
&[1,2]&&\abacusline(4,0,bbbb,bbbb,nnbn) \qquad \abacusline(4,0,bbbb,bbnn,nbnn) && ((2),(3)) \\
&[3,2]&&\abacusline(4,0,bbbb,bbbn,nbbn) \qquad \abacusline(4,0,bbbb,bbnb,nnnn) &&((2^2),(1)) 
\end{align*}
This example is small enough to check by hand; or we may use~\cref{P:FayersKlesh} to verify that these are indeed Kleshchev multipartitions and~\cref{T:FayersDecomp} to see that the submatrix of the graded decomposition matrix given by these bipartitions is indeed the matrix~(\ref{targetmatrixalt}): 
\[\begin{matrix*}[r] ((1),(2,1^2)) \\ ((1^2),(1^3)) \\((2),(3)) \\ ((2^2),(1)) \end{matrix*} \quad 
\begin{pmatrix} 1 &&& \\
v & 1 & & \\
v & 0 & 1 & \\
v^2 & v & v & 1 
\end{pmatrix}
\]
\end{eg}

We now look at the remaining cases for $\tb$ when $e \ge 4$. In most -- possibly all -- of the cases below, there is more than one way of choosing the SI-subset.
Other than trying to have as few different cases as possible, our choice for the subset was arbitrary. 

For $d \geq 1$, let
\[
N_d(\tb)=\#\{0 \leq m <e \mid \tb_m=d\}.
\]
If $N_d(\tb)>0$ and $\{0 \leq m \leq e-1 \mid \tb_m=d\} = \{s_1,\dots,s_r\}$ where $s_1 < \dots < s_r$ and $1 \leq u \leq r$, set $P_d^u(\tb)=s_u$, and set $P_d({\tb})=P^1_d(\tb)$. (Informally, $P_d^u(\tb)$ is just the position in which the $u^{\text{th}}$ entry equal to $d$ occurs in $\tb$.)  

\begin{thm}
Assume that $\tb \notin\{(2,3,1,2), (2,3,2,1)\}$. Assume that $e \ge 4$. 
\begin{enumerate} 
\item Suppose $\tb=(1,1,\ast,\dots,\ast)$. Then
\[\{[0]^2, [1]^2, [0]^1, [1]^1\}\]
is an SI-subset of $B$ that gives the matrix~(\ref{targetmatrixalt}).
\item Suppose $\tb \neq (1,1,\ast,\dots,\ast)$ and that $N_1(t) \geq 3$. For $i=1,2,3$, let $a_i=P^i_1(\tb)$ and 
let $x=P_2(\tb)$. Then
\[\{[a_1]^2, [a_2]^2, [a_3,x], [a_2,x]\}\]
is an SI-subset of $B$ that gives the matrix~(\ref{targetmatrix}).
\item Suppose $\tb=(1,2,\ast,\dots,\ast)$ with $N_1(\tb)=2$ and $P^2_1(\tb)=x$. Then 
\[\{[0]^2,[x]^2,[1]^2,[x,1]\}\]
is an SI-subset of $B$ that gives the matrix~(\ref{targetmatrix}).
\item Suppose that $\tb=(1,2,\ast,\dots,\ast)$ and that $N_1(\tb)=1$. 
\begin{enumerate}
\item If $\tb=(1,2,3,\dots,3)$ then 
\[
\{[3,1],[2,1],[1]^2,[2]^2\}
\]
is an SI-subset of $B$ that gives the matrix~(\ref{targetmatrix}).
\item If not, the set
\[
\{[0]^2,[1]^2,[0]^1,[1]^1\}
\]
is an SI-subset of $B$ that gives the matrix~(\ref{targetmatrixalt}).
\end{enumerate}
 \item Suppose $\tb=(2,\ast,\dots,\ast)$ and $N_1(\tb)=2$. For $i=1,2$, let $a_i=P^i_1(\tb)$. Then 
\[
\{[a_1]^2,[a_2]^2,[0]^2,[a_2 ,0]\}
\]
is an SI-subset of $B$ that gives the matrix~(\ref{targetmatrix}).
\item Suppose $\tb=(2,\ast,\dots,\ast)$ and $N_1(\tb)=1$. Let $a=P_1(\tb)$. 
  \begin{enumerate}
  \item Suppose $\tb=(2,1,\ast,\dots,\ast)$ and $N_3(\tb)=0$, i.e.~$\tb=(2,1,2,\dots,2)$.
  Then
    \[
    \{[0]^2,[2]^2,[1]^1,[2]^1\}
    \]
is an SI-subset of $B$ that gives the matrix~(\ref{targetmatrixalt}).
  \item Suppose $\tb=(2,2,\ast,\dots,\ast)$ and $N_3(\tb)=0$, i.e.~$\tb=(2,2,\dots,2,1,2,\dots,2)$.
  Then
    \[
    \{[0]^2,[1]^2,[a,0],[a,1]\}
    \]
    is an SI-subset of $B$ that gives the matrix~(\ref{targetmatrixalt}). 
  \item Suppose $\tb=(2,\ast,\dots,\ast)$ where $t_1\leq 2$ and $N_3(\tb)\geq 1$. Take $c=P_3(t)$ and $b=P^2_2(t)$. Then 
    \[
    \{[c,0],[c,b],[0]^2,[b]^2\}
    \]
    is an SI-subset of $B$ that gives the matrix~(\ref{targetmatrixalt}).
  \item Suppose $\tb=(2,3,\ast,\dots,\ast)$ and $|I|\geq 3$. If $N_3(\tb)=1$, take $d=P_5(\tb)$; else take $d=P^2_3(\tb)$. Then 
    \[
    \{[d,0],[1,0],[0]^2,[1]^2\}
    \]
    is an SI-subset of $B$ that gives the matrix~(\ref{targetmatrix}).
  \item Suppose $\tb=(2,3,\ast,\dots,\ast)$ and $|I|=2$ and $N_2(t) \geq 2$.
  Note that since $\tb \ne (2,3,1,2),(2,3,2,1)$ we have $|K| \ge 3$.
  Let $b=P^2_2(\tb)$. Then 
    \[
    \{[0]^2,[1]^2,[b]^2,[1,b]\}
    \]
    is an SI-subset of $B$ that gives the matrix~(\ref{targetmatrix}).
  \item Suppose $\tb=(2,3,\ast,\dots,\ast)$ and $|I|=2$ and $N_2(t)= 1$, i.e.~$\tb=(2,3,4,4,\dots,4,1)$.
  Then 
    \[
    \{[0]^2,[1]^2,[a]^1,[1]^1\}
    \]
    is an SI-subset of $B$ that gives the matrix~(\ref{targetmatrixalt}).
    \end{enumerate}
    \item Suppose $\tb=(2,\ast,\dots,\ast)$ and $N_1(\tb)=0$. 
\begin{enumerate}
\item Suppose $\tb=(2,2,\ast,\dots,\ast)$. Since $\nn>0$, we must have $N_3(t)\geq 1$.
Let $c=P_3(\tb)$. Then
\[
\{[c,0],[c,1],[0]^2,[1]^2\}
\] 
is an SI-subset of $B$ that gives the matrix~(\ref{targetmatrixalt}).
\item Suppose $\tb=(2,3,2,\ast,\dots,\ast)$. Then
\[
\{[0]^2,[1]^2,[2]^2,[1,2]\}
\]
is an SI-subset of $B$ that gives the matrix~(\ref{targetmatrix}).
\item Suppose $\tb=(2,3,\ast,\dots,\ast)$ where $\tb_2 \geq 3$ and $|I|\geq 2$. If $N_3(\tb)=1$, take $d=P_5(\tb)$; else take $d=P_3^2(\tb)$. Then 
   \[
   \{[d,0],[1,0],[0]^2,[1]^2\}
   \]
is an SI-subset of $B$ that gives the matrix~(\ref{targetmatrix}). 

\item Suppose $\tb=(2,3,4,\ast,\dots,\ast)$ and $|I|=1$.
and $N_2(\tb) \geq 2$. Take $b=P^2_2(\tb)$. Then
    \[
    \{[0]^2,[1]^2,[b]^2,[1,b]\}
    \]
is an SI-subset of $B$ that gives the matrix~(\ref{targetmatrix}).

\item Suppose $\tb=(2,3,4,\ast,\dots,\ast)$ and $|I|=1$ and $N_2(\tb) =1$, i.e.~$\tb=(2,3,4,4,\dots,4)$. Then
    \[
    \{[1]^2,[2]^2,[1]^1,[2]^1\}
    \]
is an SI-subset of $B$ that gives the matrix~(\ref{targetmatrixalt}).    
    \end{enumerate}
\end{enumerate}
\end{thm}

\begin{proof}
In order to prove~\cref{T:Defect2} we need to check that in each individual case the bipartitions in question are Kleshchev bipartitions and that the subsequent submatrix of the graded decomposition matrix is as claimed. For the former we can use~\cref{P:FayersKlesh} and for the latter we can use~\cref{Fayers:DecompTable}. Here, we look at Case (i) in detail; for the other cases, we indicate where the relevant decomposition numbers can be found in~\cref{Fayers:DecompTable}. 

Note that our assumptions on $B$ impose the condition that $I,K \ne \emptyset$. Since $\tb$ is minimal, we then have that $N_2(\tb)>0$ and that $(\tb_0,\tb_1) \in \{(1,1),(1,2),(2,1),(2,2),(2,3)\}$. 
\begin{enumerate}
\item We have $\tb=(1,1,\ast,\dots,\ast)$ so that $0,1 \in I$. Let $x=P_2(\tb)$. Then $0,1 \prec x$ so by~\cref{P:FayersKlesh} we have $[0]^2,[1]^2,[0]^1,[1]^1 \in \Kl$.
We want to show that the entries in the graded decomposition matrix are 
\[
\begin{matrix*}[r] [0]^2 \\ [1]^2 \\ [0]^1 \\ [1]^1 \end{matrix*} \quad 
\begin{pmatrix} 1 &&& \\
v & 1 & & \\
v & 0 & 1 & \\
v^2 & v & v & 1 
\end{pmatrix}
.\]
Let $i=0$ so that $i^+=1$. Take $\bmu=[i]^2$, which corresponds to Case B of~\cref{Fayers:DecompTable}. Then $\bmu^{\ast}=[i^+]^1=[1]^1$ so that 
\begin{align*} 
d^{p}_{[i^+]^2[i]^2}(v) &= d^{p}_{[1]^2[0]^2}(v)  = v,\\
d^{p}_{[i]^1[i]^2}(v)&=d^{p}_{[0]^1[0]^2}(v)=v, \\
d^{p}_{[i^+]^1[i]^2}(v) &= d^{p}_{[1]^1[0]^2}(v) =v^2.
\end{align*}
Now take $\bmu=[i]^1$. This corresponds to the first instance of Case A of~\cref{Fayers:DecompTable}, where $k=x$. In particular,
\[
d^{p}_{[0]^2[0]^1}(v)=d^{p}_{[1]^2[0]^1}=0 \text{ and } d^{p}_{[1]^1[0]^1}(v)=v.
\]

Let $i=1$ and take $\bmu=[i]^2$.
If $\tb_2=1$, this corresponds to Case B of~\cref{Fayers:DecompTable}.
Otherwise $\tb_2=2$; set $k=2$. We see that this corresponds to either the first instance of Case C if $k \npreccurlyeq i^+$ or the second instance of Case C if $k \preceq i^+$. (Recall that if $i^+$ does not exist then we ignore conditions involving it so that the condition $k \preceq i^+$ is trivially true.) 
In all cases, we have
\[
d^{p}_{[0]^2[1]^2}(v)=d^{p}_{[0]^1[1]^2}(v)=0 \text{ and } d^{p}_{[1]^1[1]^2}(v)=v.
\]
Finally, take $\bmu=[i]^1$. This corresponds to one of the instances of Case A. In either instance we have 
\[
d^{p}_{[0]^2[1]^1}(v)=d^{p}_{[1]^2[1]^1}(v)=d^{p}_{[0]^1[1]^1}(v)=0.
\]
\end{enumerate}
In fact, in Case (i) above, it was unnecessary to look at $\bmu=[i]^1$. The previous calculations showed that $d_{[1]^1\bla} \ne 0$ for $\bla \in \{[0]^2,[1]^2,[0]^2$ so that $[1]^1 \dom [0]^2,[1]^2,[0]^2$ and by~\cref{L:DomOrder}, the corresponding column of the decomposition matrix is indeed as shown. 

For all cases, we record below where in~\cref{Fayers:DecompTable} we can find the Specht modules that contain $\rD{\bmu}$ as a composition factor. As above, we do not need to consider the fourth bipartition.

{\renewcommand{\arraystretch}{1.2}
\[
\begin{array}{c@{\quad}c@{\quad}c@{\quad}c@{\quad}c@{\quad}c}
\text{Case} && \bla_1 & \bla_2 & \bla_3 & \bla_4 \\
\hline

\multirow{2}{*}{(i)} && [0]^2 & [1]^2 & [0]^1 & [1]^1 \\
 && \text{B} & \text{B / C1 / C2} & \text{A1} & - \\
\hline

\multirow{2}{*}{(ii)} && [a_1]^2 & [a_2]^2 & [a_3,x] & [a_2,x] \\
 && \text{C1 / D1} & \text{D1} & \text{D'1 / G'1 / H} & - \\
\hline

\multirow{2}{*}{(iii)} && [0]^2 & [x]^2 & [1]^2 & [x,1] \\
 && \text{C1} & \text{D2} & \text{C'1 / F1 / F3 / G1 / G2} & - \\
\hline

\multirow{4}{*}{(iv)} &
\multirow{2}{*}{(a)} & [3,1] & [2,1] & [1]^2 & [2]^2 \\
 & & \text{A'1} & \text{A'2} & \text{C'2} & - \\
\cline{2-6}

 & \multirow{2}{*}{(b)} & [0]^2 & [1]^2 & [0]^1 & [1]^1 \\
 & & \text{C2} & \text{C'2 / F2 / F4} & \text{A2} & - \\
\hline

\multirow{2}{*}{(v)} && [a_1]^2 & [a_2]^2 & [0]^2 & [a_2,0] \\
 && \text{D1} & \text{D2} & \text{C'1 / F1 / F3 / G1 / G2} & - \\
\hline

\multirow{12}{*}{(vi)} & (a) & [0]^2 & [2]^2 & [1]^1 & [2]^1 \\
                       &     & \text{F3}     & \text{F4}     & \text{A2}    & -    \\
\cline{2-6}

 & \multirow{2}{*}{(b)} & [0]^2 & [1]^2 & [a,0] & [a,1] \\
 & & \text{G2} & \text{F3 / G2} & \text{G'2} & - \\
\cline{2-6}

 & \multirow{2}{*}{(c)} & [c,0] & [c,b] & [0]^2 & [b]^2 \\
 & & \text{E'2} & \text{A'2 / E'2} & \text{F3 / G2} & - \\
\cline{2-6}

 & \multirow{2}{*}{(d)} & [d,0] & [1,0] & [0]^2 & [1]^2 \\
 & & \text{A'1 / E'1} & \text{A'2} & \text{C'1 / F1} & - \\
\cline{2-6}

 & \multirow{2}{*}{(e)} & [0]^2 & [1]^2 & [b]^2 & [1,b] \\
 & & \text{F1 / G1} & \text{D2} & \text{F3 / G2} & - \\
\cline{2-6}

 & \multirow{2}{*}{(f)} & [0]^2 & [1]^2 & [a]^1 & [1]^1 \\
 & & \text{C'1} & \text{C2} & \text{A1} & - \\
\hline

\multirow{10}{*}{(vii)} &
\multirow{2}{*}{(a)} & [c,0] & [c,1] & [0]^2 & [1]^2 \\
 & & \text{E'2} & \text{A'2 / E'2} & \text{F4} & - \\
\cline{2-6}

 & \multirow{2}{*}{(b)} & [0]^2 & [1]^2 & [2]^2 & [1,2] \\
 & & \text{F2} & \text{D2} & \text{C'1 / G2} & - \\
\cline{2-6}

 & \multirow{2}{*}{(c)} & [d,0] & [1,0] & [0]^2 & [1]^2 \\
 & & \text{A'1 / E'1} & \text{A'2} & \text{C'2 / F2} & - \\
\cline{2-6}

 & \multirow{2}{*}{(d)} & [0]^2 & [1]^2 & [b]^2 & [1,b] \\
 & & \text{F2} & \text{D2} & \text{F3 / G2} & - \\
\cline{2-6}

 & \multirow{2}{*}{(e)} & [1]^2 & [2]^2 & [1]^1 & [2]^1 \\
 & & \text{C2} & \text{F4} & \text{A2} & - \\
\hline

\end{array}
\qedhere
\]
}
\end{proof}

This concludes the proof of the main theorem of this section.

\begin{thm} \label{T:Defect2}
Suppose $e\geq 4$, $s_1 \ne s_2$, and that $B$ is a block of $\hhh$ with $\defect(B)=2$, and that $B$ is not a core block.
Then $B$ is Schurian-infinite. 
\end{thm}

\section{Defect 2 blocks and $\tau$-tilting theory}\label{sec:def2}

\subsection{Review on $\tau$-tilting theory}

Here we review necessary materials from $\tau$-tilting theory.
Let $A$ be a finite-dimensional $\bbf$-algebra, $\mod A$ the category of finite-dimensional $A$-modules.
For $M\in \mod A$, we denote by $|M|$ the number of isomorphism classes of indecomposable direct summands of $M$, and by $\add(M)$ the full subcategory of $\mod A$ whose objects are direct summands of finite direct sums of copies of $M$.
We will apply theorems in this section to $R^\Lambda(\beta)$, which admits an anti-involution that fixes the KLR generators, so that the category of right $A$-modules is equivalent to the category of left $A$-modules.
Since almost all the references on $\tau$-tilting theory use right modules, we state theorems using right modules in this section. In the next section, we often compute the bound quiver presentation of the basic algebra of $R^\Lambda(\beta)$ using Fock space theory.
In that process, we choose explicit elements that represent arrows of the Gabriel quiver and we obtain the radical structure of indecomposable projective left modules over the basic algebra and we view them as right modules. Then, we may compute the support $\tau$-tilting quiver in the category of right modules to decide the $\tau$-tilting finiteness of $A$.

\begin{defn}[{\cite[Definition 0.1]{AIR}}]
Let $M\in \mod A$. We denote by $\tau$ the Auslander--Reiten translation. Then,
\begin{enumerate}
\item $M$ is called \emph{$\tau$-rigid} if $\Hom_A(M,\tau M)=0$.
\item $M$ is called \emph{$\tau$-tilting} if it is $\tau$-rigid and $\left | M \right |=\left | A \right |$.
\item $M$ is called \emph{support $\tau$-tilting} if $M$ is a $\tau$-tilting $\left ( A/A e A\right )$-module for an idempotent $e$ of $A$. In this case, $(M,P)$ with $P:=eA$ is called a \emph{support $\tau$-tilting pair}.
\end{enumerate}
\end{defn}

We denote by $\tautilt A$ and $\stautilt A$ the set of isomorphism classes of basic $\tau$-tilting $A$-modules and basic support $\tau$-tilting $A$-modules, respectively. 
It is shown in \cite[Theorem 0.2]{AIR} that any $\tau$-rigid $A$-module is a direct summand of some $\tau$-tilting $A$-module. If the number of isomorphism classes of $\tau$-tilting $A$-modules is finite, we say that $A$ is $\tau$-tilting finite. Then, $A$ is $\tau$-tilting finite if and only if the number of isomorphism classes of indecomposable $\tau$-rigid $A$-modules is finite, or equivalently, if and only if $\stautilt A$ is a finite set. 

Schurian modules are called bricks in the field of $\tau$-tilting theory. 
The following theorem established the fact that Schurian-finiteness is equivalent to $\tau$-tilting finiteness.

\begin{thm}[{\cite[Theorem 4.2]{DIJ}}]\label{Schurian-rigid}
There exists an injective map from the set of isomorphism classes of indecomposable $\tau$-rigid $A$-modules to the set of isomorphism classes of Schurian $A$-modules given by
$X\longmapsto X/\Rad_B(X)$, where $B:=\End_A(X)$.
If one of the sizes of the two sets is finite, then the map is bijective. 
\end{thm}

\begin{rem}
    The image of the map is the set of isomorphism classes of Schurian modules $M$ such that the smallest torsion class $\mathsf{T}(M)$ containing $M$ is functorially finite. 
\end{rem}

We may check $\tau$-tilting finiteness using the support $\tau$-tilting quiver, defined as follows.
Its vertex set is $\stautilt A$.
The directed edges of the quiver are given by left mutation: let $T=M\oplus N\in \stautilt A$ with an indecomposable direct summand $M$ such that $M$ is not a factor module of a direct sum of copies of $N$.
We take an exact sequence with a minimal left $\add(N)$-approximation $f$:
\[
M\overset{f}{\longrightarrow}N'\longrightarrow \mathsf{coker}\ f\longrightarrow 0.
\]
It is shown in \cite[Theorem 2.30]{AIR} that $\mu_M^-(T):=(\mathsf{coker}\ f)\oplus N$ is again a basic support $\tau$-tilting module. We call $\mu_M^-(T)$ the \emph{left mutation} of $T$ with respect to $M$\footnote{Dually, we have the \emph{right mutation} $\mu_M^+(T)$ of $T$ with respect to $M$.}.
Then we draw an arrow $T\rightarrow \mu_M^-(T)$. It is known that $N$ is not a sincere module if and only if $f$ is an epimorphism.

\begin{thm}[{\cite[Corollary 2.38]{AIR}}]
If the support $\tau$-tilting quiver of a finite-dimensional $\bbf$-algebra $A$ has a connected component with finitely many vertices, there is no vertex outside the component. 
In particular, $A$ is $\tau$-tilting finite. Equivalently, 
$A$ is Schurian-finite. 
\end{thm}

\begin{rem}
We may also use the quiver of basic two-term silting complexes, and the quiver is isomorphic to the support $\tau$-tilting quiver \cite[Corollary 3.9]{AIR}.
\end{rem}

Recalling that a two-term silting complex is determined by its $g$-vector, we may deduce the following result from 
\cite[Theorem 11]{ejr18}. This result is very useful for showing that an algebra is Schurian-finite.

\begin{thm}\label{ejr reduction}
  Let $I$ be an ideal of a finite-dimensional $\bbf$-algebra $A$ generated by central elements in $\Rad A$. Then $A$ is Schurian-finite if and only if $A/I$ is Schurian-finite. 
\end{thm}

\subsection{Core blocks of defect $2$ when $s_1 = s_2$}\label{subsec:core-def-2-equal}

We may suppose that $(s_1,s_2)=(0,0)$.
Then, by \cite[Lemma 2.1, Corollary 3.3]{arikirep}, we consider $R^{2\Lambda_0}(\beta)$ such that $2\Lambda_0-\beta$ is in the affine symmetric group orbit of $2\Lambda_0-\lambda_2^0$ where $\lambda^0_2=\alpha_{e-1}+2\alpha_0+\alpha_1$. 
If we are considering $e=\infty$, this should be replaced with $\lambda^0_2=\alpha_{-1}+2\alpha_0+\alpha_1$, but as noted before, it suffices to consider $e$ large enough to avoid this necessity.
Kakei computed the Gabriel quiver of $R^{2\Lambda_0}(\lambda_2^0)$, for $e\ge 4$, and it is given in \cite[page 840]{arikirep}. The quiver is
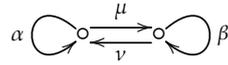
\[
\begin{xy}
(20,0) *{\circ}="A", (30,0) *{\circ}="B", 
\ar @<1mm> "A";"B"^{\mu}
\ar @(lu,ld) "A";"A"_{\alpha}
\ar @<1mm> "B";"A"^{\nu}
\ar @(ru,rd) "B";"B"^{\beta}
\end{xy}
\]
Our convention for the multiplication in the path algebra is that, a path $\xymatrix{\ar[r]|-x&\ar[r]|-y&}$ is written as $xy$.
\begin{itemize}
\item[(I)]
If $p \ne 2$, then the relations are
\begin{equation}\label{SF-alg-1}
\alpha\mu=\mu\beta=\beta\nu=\nu\alpha=0, \;\; \alpha^2+\mu\nu\mu\nu=0, \;\; \beta^2+\nu\mu\nu\mu=0.
\end{equation}
\item[(II)]    
If $p = 2$, then the relations are
\begin{equation}\label{SF-alg-2}
\alpha^2=\beta^2=\mu\nu\mu=\nu\mu\nu=0,\;\; \alpha\mu=\mu\beta, \;\; \beta\nu=\nu\alpha.
\end{equation}
\end{itemize}
The case (I) is exactly the case (b) from \cite[Theorem 2(2)]{ariki21}, where we mistakenly added the condition that $e$ was even. We allow $e\ge4$ to be odd. Recall that symmetric special biserial algebras are Brauer graph algebras. Hence, $R^{2\Lambda_0}(\lambda_2^0)$ is a Brauer graph algebra and its Brauer graph is as follows. 
\[
\begin{xy}
(0,0) *[o]+[Fo]{2}="A", (10,0) *[o]+[Fo]{2}="B", (20,0) *[o]+[Fo]{2}="C",
\ar@{-} "A";"B"
\ar@{-} "B";"C"
\end{xy}
\]
Then, any finite-dimensional algebra which is derived equivalent to $R^{2\Lambda_0}(\lambda_2^0)$ is Morita equivalent to $R^{2\Lambda_0}(\lambda_2^0)$ \cite[Proposition 5.3]{ariki21}. 
We may also deduce this fact from \cite[Theorem A]{OZ22}, which asserts that two Brauer graph algebras having two or more simple modules are derived equivalent if and only if
\begin{itemize}
\item[(a)]
the number of vertices, edges, faces all coincide.
\item[(b)]
the multiset of the multiplicities of vertices and the multiset of the perimeters of faces coincide.
\item[(c)]
either both are bipartite or both are not bipartite.
\end{itemize}
Thus, if $p\ne 2$, any defect two block $R^{2\Lambda_0}(\beta)$ is Schurian-finite.

Next, we consider case (II).  
We denote the algebra \eqref{SF-alg-2} by $\mathcal{SF}_2$.
We compute with right modules in the following. 
Then, the indecomposable projective right modules are as follows, where the dotted lines denote that the right multiplication by an arrow is nonzero.

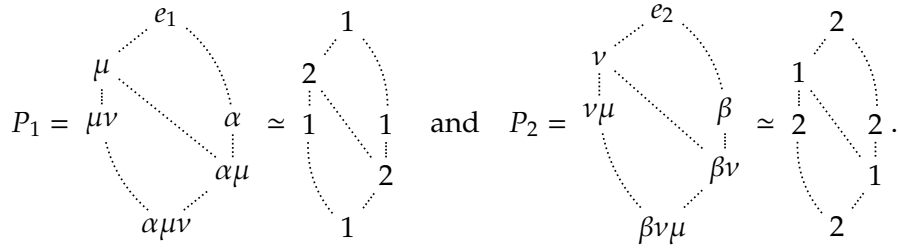
\begin{center}
$
P_1=\vcenter{\xymatrix@C=0.1cm@R=0.2cm{
&e_1\ar@/^0.2cm/@{.}[ddr]\ar@{.}[dl]&\\
\mu\ar@{.}[d]\ar@{.}[ddrr]&& \\
\mu\nu\ar@/_0.2cm/@{.}[ddr]&& \alpha\ar@{.}[d]\\
&&\alpha\mu\ar@{.}[dl]\\
&\alpha\mu\nu&
}}\simeq \vcenter{\xymatrix@C=0.1cm@R=0.2cm{
&1\ar@/^0.2cm/@{.}[ddr]\ar@{.}[dl]&\\
2\ar@{.}[d]\ar@{.}[ddrr]&& \\
1\ar@/_0.2cm/@{.}[ddr]&& 1\ar@{.}[d]\\
&&2\ar@{.}[dl]\\
&1&
}}\quad \text{and} \quad P_2=\vcenter{\xymatrix@C=0.1cm@R=0.2cm{
&e_2\ar@/^0.2cm/@{.}[ddr]\ar@{.}[dl]&\\
\nu\ar@{.}[d]\ar@{.}[ddrr]&& \\
\nu\mu\ar@/_0.2cm/@{.}[ddr]&& \beta\ar@{.}[d]\\
&&\beta\nu\ar@{.}[dl]\\
&\beta\nu\mu&
}}\simeq \vcenter{\xymatrix@C=0.1cm@R=0.2cm{
&2\ar@/^0.2cm/@{.}[ddr]\ar@{.}[dl]&\\
1\ar@{.}[d]\ar@{.}[ddrr]&& \\
2\ar@/_0.2cm/@{.}[ddr]&& 2\ar@{.}[d]\\
&&1\ar@{.}[dl]\\
&2&
}}.
$
\end{center}
We also know that $\mathcal{SF}_2$ is wild. By direct calculation of left mutations, one may find that the support $\tau$-tilting quiver $\stautilt \mathcal{SF}_2$ is displayed as follows.
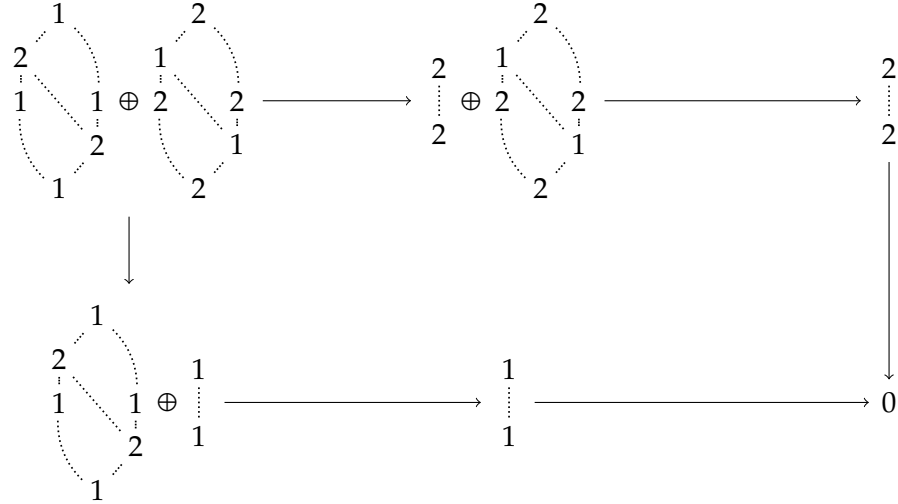
\begin{center}
\begin{tikzpicture}[shorten >=1pt, auto, node distance=0cm,
   node_style/.style={font=},
   edge_style/.style={draw=black}]
\node[node_style] (v0) at (0,0) {$\vcenter{\xymatrix@C=0.1cm@R=0.1cm{
&1\ar@/^0.2cm/@{.}[ddr]\ar@{.}[dl]&\\
2\ar@{.}[d]\ar@{.}[ddrr]&& \\
1\ar@/_0.2cm/@{.}[ddr]&& 1\ar@{.}[d]\\
&&2\ar@{.}[dl]\\
&1&
}}\oplus \vcenter{\xymatrix@C=0.1cm@R=0.1cm{
&2\ar@/^0.2cm/@{.}[ddr]\ar@{.}[dl]&\\
1\ar@{.}[d]\ar@{.}[ddrr]&& \\
2\ar@/_0.2cm/@{.}[ddr]&& 2\ar@{.}[d]\\
&&1\ar@{.}[dl]\\
&2&
}}$};
\node[node_style] (v1) at (5,0) {$\vcenter{\xymatrix@C=0.1cm@R=0.4cm{
2\ar@{.}[d]\\
2
}}\oplus \vcenter{\xymatrix@C=0.1cm@R=0.1cm{
&2\ar@/^0.2cm/@{.}[ddr]\ar@{.}[dl]&\\
1\ar@{.}[d]\ar@{.}[ddrr]&& \\
2\ar@/_0.2cm/@{.}[ddr]&& 2\ar@{.}[d]\\
&&1\ar@{.}[dl]\\
&2&
}}$};
\node[node_style] (v12) at (10,0) {$\vcenter{\xymatrix@C=0.1cm@R=0.4cm{
2\ar@{.}[d]\\
2
}}$};
\node[node_style] (v) at (10,-4) {$0$};
\node[node_style] (v2) at (0,-4) {$\vcenter{\xymatrix@C=0.1cm@R=0.1cm{
&1\ar@/^0.2cm/@{.}[ddr]\ar@{.}[dl]&\\
2\ar@{.}[d]\ar@{.}[ddrr]&& \\
1\ar@/_0.2cm/@{.}[ddr]&& 1\ar@{.}[d]\\
&&2\ar@{.}[dl]\\
&1&
}}\oplus\vcenter{\xymatrix@C=0.1cm@R=0.4cm{
1\ar@{.}[d]\\
1
}}$};
\node[node_style] (v21) at (5,-4) {$\vcenter{\xymatrix@C=0.1cm@R=0.4cm{
1\ar@{.}[d]\\
1
}}$};

\draw[->]  (v0) edge node{\ } (v1);
\draw[->]  (v1) edge node{\ } (v12);
\draw[->]  (v0) edge node{\ } (v2);
\draw[->]  (v2) edge node{\ } (v21);
\draw[->]  (v21) edge node{\ } (v);
\draw[->]  (v12) edge node{\ } (v);
\end{tikzpicture}
\end{center}
Since $R^{2\Lambda_0}(\lambda^0_2)$ is only a representative of the derived equivalent blocks, we show that blocks which are derived equivalent to $R^{2\Lambda_0}(\lambda^0_2)$ are Morita equivalent to $R^{2\Lambda_0}(\lambda^0_2)$.

\begin{prop}
Let $A=\mathcal{SF}_2$. We have $\End_{\Kb(\proj A)} \mu_1^-(A)\cong A$.
\end{prop}
\begin{proof}
By direct calculation, it is easy to find 
\begin{center}
$\mu_1^-(A)=
\left [\begin{smallmatrix}
\xymatrix@C=1.2cm{P_1\ar[r]^-{\nu}& P_2}\\
\oplus \\
\xymatrix@C=1.4cm{0\ar[r]& P_2}
\end{smallmatrix}  \right ]$.
\end{center}
In the diagram below, the top square means that if we set $f^{-1}$ and $f^0$ to be linear combination of $\{e_1, \alpha, \mu\nu, \alpha\mu\nu\}$ and $\{e_2, \beta, \nu\mu, \beta\nu\mu\}$, respectively, and force $\nu f^{-1}=f^0\nu$, then we obtain that $(f^{-1},f^0)$ is a linear combination of $(e_1,e_2), (\alpha,\beta), (\mu\nu,0), (0,\nu\mu), (\alpha\mu\nu,0), (0,\beta\nu\mu)$, and 
$\End_{\Kb(\proj A)}(P_1\stackrel{\nu}{\rightarrow} P_2)$ has basis $\{(e_1,e_2), (\alpha,\beta), (0,\nu\mu), (0,\beta\nu\mu)\}$. The meaning of the other three squares is similar. Besides, we denote by $\sim_h$ the homotopy equivalence in the homotopy category $\Kb(\proj A)$.

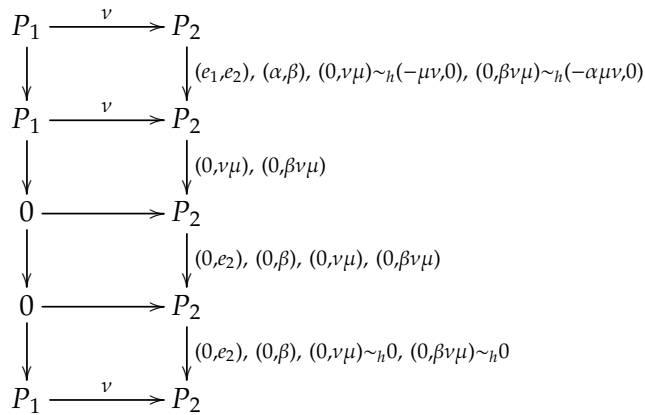
\begin{center}
$\xymatrix@C=1.5cm@R=0.7cm{
P_1\ar[r]^-{\nu}\ar[d]& P_2\ar[d]^{(e_1, e_2),\  (\alpha, \beta),\  (0, \nu\mu)\sim_h(-\mu\nu, 0), \  (0, \beta\nu\mu)\sim_h(-\alpha\mu\nu, 0)}\\
P_1\ar[r]^-{\nu}\ar[d]& P_2\ar[d]^{(0,\nu\mu),\ (0,\beta\nu\mu)}\\
0\ar[r]\ar[d]& P_2\ar[d]^{(0, e_2), \ (0, \beta),\ (0,\nu\mu),\ (0,\beta\nu\mu)}\\
0\ar[r]\ar[d]& P_2\ar[d]^{(0, e_2), \ (0, \beta),\ (0,\nu\mu)\sim_h 0,\ (0,\beta\nu\mu)\sim_h 0}\\
P_1\ar[r]^-{\nu}& P_2
}$
\end{center}
Set $x:=(0, e_2), y:=(0,\nu\mu), z:=(\alpha, \beta),t:=(0,\beta)$. We have 
\begin{center}
$\vcenter{\xymatrix@C=1cm{1 \ar@<0.5ex>[r]^{x}\ar@(dl,ul)^{z}&2 \ar@<0.5ex>[l]^{y}\ar@(ur,dr)^{t}}}$,
\end{center}
and 
\begin{itemize}
\item $z^2=t^2=0$, $zx=(0, \beta)=xt$, $ty=(0, \beta\nu\mu)=yz$, $xy=(0,\nu\mu)$, $yx=(0,\nu\mu)$;
\item $zxt=tyz=0$, $xyx=yxy=0$, $zxy=xty=xyz=(0,\beta\nu\mu)$, $tyx=yzx=yxz=(0,\nu\mu\alpha)$;
\item all paths of length at least 4 are zero.
\end{itemize}
It gives that $\End_{\Kb(\proj A)} \mu_1^-(A)\cong A$.
\end{proof}

\begin{prop}
Let $A=\mathcal{SF}_2$. We have $\End_{\Kb(\proj A)} \mu_2^-(A)\cong A $.
\end{prop}
\begin{proof}
We have 
\begin{center}
$\mu_2^-(A)=
\left [\begin{smallmatrix}
\xymatrix@C=1.4cm{0\ar[r]& P_1}\\
\oplus \\
\xymatrix@C=1.2cm{P_2\ar[r]^-{\mu}& P_1}
\end{smallmatrix}  \right ]$.
\end{center}
Then, the statement follows from the following diagram
\begin{center}
$\xymatrix@C=1.5cm@R=0.7cm{
P_2\ar[r]^-{\mu}\ar[d]& P_1\ar[d]^{(e_2, e_1),\  (\beta, \alpha),\ (0,\mu\nu)\sim_h(-\nu\mu,0), \ (0,\alpha\mu\nu)\sim_h(-\beta\nu\mu,0) }\\
P_2\ar[r]^-{\mu}\ar[d]& P_1\ar[d]^{(0,\mu\nu),\ (0,\alpha\mu\nu)}\\
0\ar[r]\ar[d]& P_1\ar[d]^{(0, e_1), \ (0, \alpha),\ (0,\mu\nu),\ (0,\alpha\mu\nu)}\\
0\ar[r]\ar[d]& P_1\ar[d]^{(0, e_1), \ (0, \alpha),\ (0,\mu\nu)\sim_h 0,\ (0,\alpha\mu\nu)\sim_h 0}\\
P_2\ar[r]^-{\mu}& P_1
}$
\end{center}\vspace{-26pt}\qedhere
\end{proof}

The next proposition shows that any defect two block $R^{2\La_0}(\beta)$ is Schurian-finite if $p=2$.

\begin{prop}[{\cite[Theorem 2.25]{ahswreptype}}]
Let $A_1, A_2, \dots, A_s$ be finite-dimensional symmetric algebras which are derived equivalent to each other and identify $\mathcal{T}=\Kb(\proj A_i)$ for all $1\le i\le s$. 
Suppose the following conditions hold. 
\begin{enumerate}
\item
The algebra $A_i$ is Schurian-finite, for $1\le i\le s$.
\item
For each indecomposable projective direct summand $X$ of the left regular module $A_i$, for $1\le i\le s$, we have
$\End_{\mathcal{T}}(\mu_X^-(A_i))\cong A_j$, for some $1\le j\le s$.
\end{enumerate}
Then, any finite-dimensional algebra $B$ which is derived equivalent to $A_1$, is Morita equivalent to $A_i$, for some $1\le i\le s$.
\end{prop}

\subsection{Non-core blocks of defect $2$ when $s_1 = s_2$}\label{subsec:def2s1=s2}

We consider $R^{2\La_0}(\delta)$ with $e\ge 4$, which is wild. In the proof of \cite[Lemma 6.2]{aswreptype}, the authors calculated the basis of $R^{k\La_0}(\delta)$, for $k\ge 3$, to find its quiver with relations.%
\footnote{The cyclotomic quiver Hecke algebra for the Cartan matrix $A^{(1)}_{e-1}$ depends on a parameter, and the parameter for the Hecke algebra of type $\ttb$ is $t=(-1)^e$ in \cite[Lemma 6.2]{aswreptype}.}

The method is valid for $k=2$. We review some necessary materials as follows.

Let $e_i=e(\nu_i)$ with $\nu_{i}=(0,1,2,\dots,i-1,e-1,e-2,\dots, i)$, for $1\leq i\leq e-1$.
Set $A=e'R^{2\Lambda_0}(\delta) e'$ with $e'=\sum_{i=1}^{e-1}e_i$. This is the basic algebra of $B=R^{2\Lambda_0}(\delta)$, that is, $A=\End_B(e'B)$. The graded dimensions are given by
\[
\begin{aligned}
\dim_v e_i A e_j&=0, \quad |i-j|>1,\\
\dim_v e_i A e_i &=1+2v^2+v^4, \quad 1\leq i\leq e-1,\\
\dim_v e_i A e_{i+1}&=\dim_v e_{i+1} A e_i=v+v^3, \quad 1\leq i\le e-2.
\end{aligned}
\]
In that proof in \cite{aswreptype}, a basis of $A$ is given as follows.
\begin{itemize}
\item $e_i A e_i$ has basis $\{e_i,\  x_1e_i,\ x_{e} e_i,\ x_1x_{e}e_i\}$, for $1\le i< e$.
\item $e_{i+1} A e_i$ has basis $\{\psi_{i+1}\psi_{i+2}\dots \psi_{e-1} e_i,\  x_1\psi_{i+1}\psi_{i+2}\dots \psi_{e-1} e_i\}$.
\item $e_{i} A e_{i+1}$ has basis $\{\psi_{e-1}\dots \psi_{i+2}\psi_{i+1} e_{i+1},\  x_1\psi_{e-1}\dots \psi_{i+2}\psi_{i+1} e_{i+1}\}$.
\end{itemize}

In the process of calculations in that proof, we used $c_i=(-1)^{i}$ and $d_i=(-1)^{\ell-i+1}t$. Since we consider cyclotomic Hecke algebras here, we have $c_i=d_i=(-1)^{i}$. 
Set $\alpha_i=x_1e_i$ with $1\le i<e$, $\mu_i=e_i\psi_{e-1}\dots \psi_{i+2}\psi_{i+1}e_{i+1}$ and $\eta_i=e_{i+1}\psi_{i+1}\psi_{i+2}\dots\psi_{e-1} e_i$, for $1\leq i< e-1$. The following relations were shown there. 
\begin{itemize}
\item $\alpha_i^2=x_1^2e_i=0$, for $1\le i\le e-1$.
\item $\eta_{i+1}\eta_i=\mu_i\mu_{i+1}=0$, for $1\leq i\leq e-3$.
\item $\eta_i\mu_i=\mu_{i+1}\eta_{i+1}=(x_e+(-1)^ix_1)e_{i+1}$,  for $1\leq i\leq e-3$.
\end{itemize}
We compute the relations between $\alpha_i$ and $\mu_i$, $\eta_i$ as follows.
\begin{align*}
\alpha_i\mu_i&=x_1\psi_{e-1}\dots \psi_{i+2}\psi_{i+1}e_{i+1} 
=\psi_{e-1}\dots \psi_{i+2}\psi_{i+1}x_1e_{i+1}=\mu_i\alpha_{i+1}, \\
\alpha_{i+1}\eta_i&=x_1\psi_{i+1}\psi_{i+2}\dots\psi_{e-1} e_i=\psi_{i+1}\psi_{i+2}\dots\psi_{e-1}x_1e_i=\eta_i\alpha_i.
\end{align*}
We observe now that the cases considered in that proof required $c_i\ne d_i$, and $e_iAe_i$ had a basis 
$\{e_i, \eta_{i-1}\mu_{i-1}, \mu_i\eta_i, (\eta_{i-1}\mu_{i-1})^2=- (\mu_i\eta_i)^2\}$, so that $\alpha_i$, for $2\le i<e$,  did not appear as arrows of the quiver, while we have 
$c_i=d_i$ and $\eta_{i-1}\mu_{i-1}=\mu_i\eta_i=x_ee_i+(-1)^{i-1}\alpha_i$. Since $2v^2$ appears in $\dim_v e_iAe_i$,  it implies that we need all $\alpha_i$ as arrows. 
We conclude that $A$ is isomorphic to $\mathcal{SF}_3:=\bbf Q/I$ with $n=e-1$, where the quiver is 
\begin{equation}\label{SF-alg-3}
Q:\xymatrix@C=1.2cm{1\ar@<0.5ex>[r]^{\mu_1}\ar@(ul,ur)^{\alpha_1}&2\ar@(ul,ur)^{\alpha_2}\ar@<0.5ex>[l]^{\eta_1}\ar@<0.5ex>[r]^{\mu_2}&3\ar@(ul,ur)^{\alpha_3}\ar@<0.5ex>[l]^{\eta_2}\ar@<0.5ex>[r]^{\mu_3}&\cdots\ar@<0.5ex>[l]^{\eta_4}\ar@<0.5ex>[r]^{\mu_{n-1}}&n \ar@(ul,ur)^{\alpha_{n}}\ar@<0.5ex>[l]^{\eta_{n-1}}}
\end{equation}
and the admissible ideal is 
\begin{center}
$I: \left \langle \begin{matrix}
\alpha_i^2, \mu_i\mu_{i+1}, \eta_{i+1}\eta_{i}, \eta_i\mu_i-\mu_{i+1}\eta_{i+1}\\
\alpha_i\mu_i-\mu_i\alpha_{i+1}, \alpha_{i+1}\eta_{i}-\eta_i\alpha_i
\end{matrix}\right \rangle$.
\end{center}
The space of central elements in $\mathcal{SF}_3$ contain $\sum_{i=1}^n\alpha_i$ and $\mu_i\eta_i$. 
\cref{ejr reduction} implies that it suffices to consider the quotient $\widetilde{\mathcal{SF}_3}:=\mathcal{SF}_3/\langle \alpha_i\;(1\le i\le n), \mu_i\eta_i\;(1\le i<n) \rangle$ in order to judge Schurian-finiteness.
Then, $\widetilde{\mathcal{SF}_3}$ is a radical square zero algebra, and it is a quotient of a Brauer line algebra modulo certain central elements.
Thus, $\widetilde{\mathcal{SF}_3}$ has $\binom{2n}{n}$ support $\tau$-tilting modules.
We may also show its Schurian-finiteness by noting that every single subquiver of the separated quiver is a disjoint union of Dynkin quivers. See \cite[Theorem B]{Ad2}.
We have proved that $R^{2\La_0}(\delta)$ is Schurian-finite.

\section{Blocks of quantum characteristic $e=3$}\label{sec:e3}

Throughout this section, we assume that $e=3$.
Fix a block $B$ of a Hecke algebra with bicharge $(\hhh,\ba)$ and let $\cB$ be the corresponding combinatorial block with $\hk=\hk(\cB)$. 
Following~\cref{T:Equiv1} and~\cref{T:SUEquivReps}
we may assume that $\cB$ lies in the set
\[
\{\hat{D} \in \Bc_{\hk} \mid \psi(\hat{D}) \text{ is Type $1$ $\hk$-minimal}\}.
\]
Take $C$ to be the core block of $B$ and $\cC\in\Cl$ the corresponding core block. Write $\nn=\nn(\cC)$ and $\pp=\pp(\cC)$.  Recall that $\defect(B)=\defect(C)+2\hk$. 

\begin{lem}\label{lem:e3coredefects}
We have $\defect(C)\in \{0,1\}$.
\end{lem}
\begin{proof}
We use \cref{L:s}. Then $\nn+\pp\leq 3$ implies $\defect(C)=\min\{\nn,\pp\}\in \{0,1\}$. 
\end{proof}

As usual, we let $\bs=(s_1,s_2)$, where $s_k$ is the equivalence class of $a_k$ modulo $e$ for $k=1,2$.

\begin{lem} \label{L:01}
We may assume that $\bs=(0,s)$ where $s=0$ or $s=1$.
\end{lem}

\begin{proof}
Since we have assumed that $\psi(\cB)$ is Type $1$ $\hk$-minimal, we automatically have that $\bs=(0,s)$ for some $s\in \{0,1,2\}$. 
Suppose that $\bs=(0,2)$.
Let 
\[
\cB'= \{(\bla,(\ba_2+1,\ba_1+1)) \mid (\bla,\ba) \in \cB\}.
\]
Then $\ba_2+1 \equiv 0 \pmod 3$ and $\ba_1+1 \equiv 1 \pmod 3$. By~\cref{T:Equiv1}, $\cB'$ is Schurian-infinite if and only if $\cB$ is Schurian-infinite.
\end{proof}

Following \cref{L:01}, we will henceforth also assume $\bs=(0,s)$ where $s=0$ or $s=1$. 
For many of the blocks we will study in \cref{sec:e=3defect2,sec:e=3defect3} 
we will label simple $B$-modules with numbers and denote them by $D_i$ (for $i=1,2,\dots$). Then, we also denote the projective cover of $D_i$ by $P_i$. We write $\Hom_B(P_i, P_j)$ simply as $\Hom(P_i, P_j)$. 

By the work of Fayers and Putignano~\cite{fay06,fayput24}, it is known that decomposition matrices are characteristic-free in defects 2 and 3, so that it suffices to compute the decomposition matrices in characteristic zero in order to apply \cref{prop:matrixtrick} in \cref{sec:e=3defect2,sec:e=3defect3}.

When $\defect(B)\leq 1$, $B$ has finite representation type, and is therefore  Schurian-finite.
We deal with the blocks of defect at least two in \cref{sec:e=3defect2,sec:e=3defect3,sec:e=3defect4+}.

\subsection{Cyclotomic quiver Hecke algebras}\label{Cyclotomic quiver Hecke algebras}

Since we will carry out explicit computations later, it is convenient to recall the explicit construction of $R^\La(n)$ here. We adopt the sign-modified version $Q_{i,i+1}(u,v)=u-v$, following the convention introduced in \cite{bkisom}. 
Recall that the affine Cartan matrix of type $\tta^{(1)}_2$ is give as 
\[
\left(\begin{array}{ccc}
2 & -1 &-1 \\
-1 & 2 &-1\\
-1& -1 &2
\end{array}\right).
\]
We take polynomials $Q_{i,j}(u,v)\in\bbf[u,v]$ with $i,j\in \{0,1,2\}$ such that 
\[
Q_{i,i}(u,v)=0,\quad Q_{i,j}(u,v) =Q_{j,i}(v,u),\quad 
Q_{0,1}(u,v)=Q_{1,2}(u,v)=Q_{2,0}(u,v)=u-v.
\]
Let $\La=\La_0+\La_s$ with $s\in\{0,1,2\}$.
The cyclotomic quiver Hecke algebra $R^{\La}(n)$ is the $\Z$-graded $\bbf$-algebra generated by
\[
\{ e(\nu)\mid \nu=(\nu_1, \nu_2, \dots, \nu_n)\in \{0,1,2\}^n \}, \quad
\{x_i \mid 1\leq i \leq n \}, \quad
\{\psi_j \mid 1\leq j\leq n-1\},
\]
subject to the following relations:
\begin{enumerate}
\item $e(\nu)e(\nu')=\delta_{\nu, \nu'}e(\nu),\ \textstyle\sum_{\nu\in I^n}e(\nu)=1,\ x_ix_j=x_jx_i,\  x_ie(\nu)=e(\nu)x_i$.

\item $\psi_i e(\nu)=e(s_i(\nu))\psi_i,\ \psi_{i}\psi_j=\psi_j\psi_i$ if $|i-j|>1$.

\item $\psi_i^2 e(\nu)=Q_{\nu_i,\nu_{i+1}}(x_i,x_{i+1})e(\nu)$.

\item
\[
(\psi_ix_j-x_{s_i(j)}\psi_i)e(\nu)=\left\{
\begin{array}{ll}
-e(\nu)  & \text{ if } j=i \text{ and } \nu_i=\nu_{i+1}, \\
e(\nu)   & \text{ if } j=i+1 \text{ and } \nu_i=\nu_{i+1}, \\
 0       & \text{ otherwise.}
\end{array}\right.
\]

\item
\[
(\psi_{i+1}\psi_i\psi_{i+1}-\psi_i\psi_{i+1}\psi_{i})e(\nu)
=\left\{
\begin{array}{ll}
\frac{Q_{\nu_i,\nu_{i+1}}(x_i,x_{i+1})-Q_{\nu_i,\nu_{i+1}}(x_{i+2}, x_{i+1})}
{x_i-x_{i+2}}e(\nu) & \text{ if } \nu_i=\nu_{i+2}, \\
0 & \text{ otherwise.}
\end{array}\right.
\]

\item $x_1^2e(\nu)=0$ if $\nu_1=0$ and $s=0$; $x_1 e(\nu)=0$ if $\nu_1=0$ or $s$ and $s\neq 0$; $e(\nu)=0$ otherwise.
\end{enumerate}

The $\Z$-grading on $R^\Lambda(n)$ is defined by
\[
\deg(e(\nu))=0,\quad
\deg(x_ie(\nu))=2,\quad
\deg(\psi_ie(\nu))=-2 \text{ if } \nu_i=\nu_{i+1}, 1 \text{ otherwise}.
\]
We mention that the graded dimension of $e(\nu)R^\La(n)e(\nu)$ can be calculated explicitly; see \cite[Theorem 4.20]{bk09}. We then review some techniques that will be used repeatedly in our calculations. Let $e(\nu)=e(\nu_1,\nu_2, \dots ,\nu_n)$.
\begin{enumerate}
\item If the sequence $\nu$ cannot be realised as the residue sequence of a standard tableau for a bipartition of $n$, then $e(\nu)=0$. 

\item If $\nu_i=\nu_{i+1}$, then $\psi_ix_{i+1}\psi_ie(\nu)=\psi_i(\psi_ix_i+1)e(\nu)=\psi_i^2x_ie(\nu)+\psi_ie(\nu)=\psi_ie(\nu)$.

\item If $\nu_i=\nu_{i+1}$, then $x_i^2e(\nu)=0$ implies $x_{i+1}e(\nu)=-x_ie(\nu)$, see \cite[Lemma 11.3]{arikirep}.
\end{enumerate}

\subsection{Blocks of defect $2$}\label{sec:e=3defect2}
Suppose $\defect(B)=2$. By~\cref{lem:e3coredefects}, $B$ cannot be a core block; and in fact we must have $\defect(C)=0$ and $\hk(B)=1$. 
Hence $\min\{\nn,\pp\}=0$. Note that $|\cC|=1$. 
It suffices to consider blocks up to $\SU$-equivalence, and so we will begin by listing the equivalence classes, following \cref{T:SUEquivReps}. Recall that we are assuming that $s=0$ or $s=1$.

\begin{lem} \label{L:ListCases1Def2}
Suppose that $\psi(\cB)=\TT$. Then $\TT$ is one of the following: 
\begin{align*}
((2,2,2),0,0), \\
((1,1,2),2,0), && ((1,2,1),2,0), && ((2,1,1),2,0), && ((5,3,2),2,0), && (5,2,3),2,0), &&((4,5,1),2,0). 
\end{align*}
The first line corresponds to the case that $s=0$ and the second line corresponds to the case that $s=1$.
\end{lem}

\begin{proof}
We assumed that $\cB$ lies in the set
\[
\{\hat{D} \in \Bc_{\hk} \mid \psi(\hat{D}) \text{ is Type $1$ $1$-minimal}\}.
\]
Hence we want to find $\TT=(\tb,\nn,\pp) \in \Tl$ such that $\min\{\nn,\pp\}=0$ and $\TT$ is Type $1$ $1$-minimal.
Following \cref{D:ScopesMin,D:TypeMin} we therefore require $\TT=(\tb,\nn,\pp) \in \Tl$ to satisfy
\begin{itemize}
\item $\tb=(\tb_0,\tb_1,\tb_2)$ where $\tb_1-\tb_0 \leq 1$, $\tb_2-\tb_1 \le 1$ and $\tb_0-\tb_2 \le 3$; 
\item $\pp=0$ and $\nn\in \{0,1,2\}$; 
\item $\min\{\tb_0,\tb_1,\tb_2\} \in \{1,2\}$;
\item $\tb_0+\tb_1+\tb_2 \equiv -\nn \pmod 6$. 
\end{itemize}
It is not difficult to check that the $\TT \in \Tl$ satisfying these conditions are as above.
\end{proof}

\begin{cor} \label{C:ListCases1}
Suppose that $\cC = \{\bnu\}$. 
\begin{itemize}
\item Suppose $s=0$. 
Then $\bnu =(\vn,\vn)$. 
\item Suppose $s=1$. Then $\bnu \in \{(\vn,\vn),\,(\vn,(1)),\,(\vn,(2)),\, ((1),\vn),\,((1^2),\vn),\,((2),(1^2))\}$.
\end{itemize}
\end{cor}

\begin{proof}
We draw the abacus configurations corresponding to each $\TT$ from~\cref{L:ListCases1Def2} and write down the corresponding bipartition. 

$s=0$: 
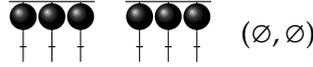
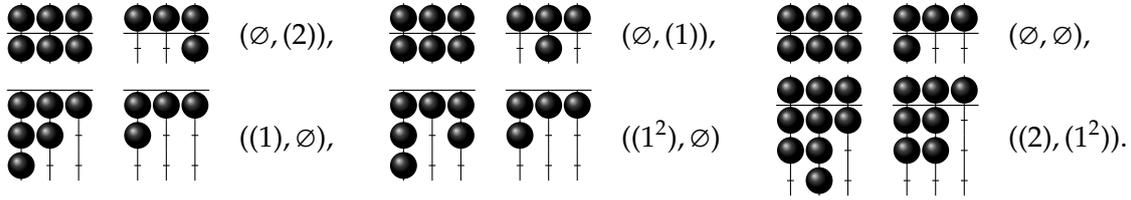
\[
\abacusline(3,1,bbb,nnn) \quad \abacusline(3,1,bbb,nnn) \quad (\vn,\vn)
\]

$s=1$: 
\begin{align*} 
\abacusline(3,0,bbb,bbb) \quad \abacusline(3,0,bbb,nnb) & \quad (\vn,(2)), & 
\abacusline(3,0,bbb,bbb) \quad \abacusline(3,0,bbb,nbn) & \quad (\vn,(1)),& 
\abacusline(3,0,bbb,bbb) \quad \abacusline(3,0,bbb,bnn) & \quad (\vn,\vn),\\
\abacusline(3,1,bbb,bbn,bnn) \quad \abacusline(3,1,bbb,bnn,nnn) & \quad ((1),\vn), &
\abacusline(3,1,bbb,bnb,bnn) \quad \abacusline(3,1,bbb,bnn,nnn) & \quad((1^2),\vn) & 
\abacusline(3,0,bbb,bbb,bbn,nbn) \quad \abacusline(3,0,bbb,bbn,bbn,nnn) & \quad ((2),(1^2)).
\end{align*}

\end{proof}

\begin{thm}\label{thm:e3def2a}
Let $e=3$, and $\La = 2\La_0$.
Then, up to $\Sc$-equivalence, the only defect two block is $R^{2\La_0}(\delta)$, which is Schurian-finite.
\end{thm}

\begin{proof}
The bound quiver presentation was computed in \cite[Proposition 11.8]{arikirep}. The parameter $t=(-1)^e=-1$ is denoted $\la$ there.
Then the algebra is isomorphic to the bound quiver algebra given by 
\[
\xymatrix@C=1cm@R=0.8cm{1\ar@<0.5ex>[r]^{\mu}\ar@(ld,lu)^{\alpha}&2\ar@<0.5ex>[l]^{\nu}\ar@(ru,rd)^{\beta}}
\]
with relations $\alpha\mu=\mu\beta, \beta\nu=\nu\alpha,\alpha^2=\beta^2=\mu\nu\mu=\nu\mu\nu=0$. 
This algebra is $\mathcal{SF}_2$, which already appeared in \eqref{SF-alg-2}. Thus, $R^{2\La_0}(\delta)$ is Schurian-finite.
\end{proof}

Next, we will consider $s=1$. The following lemma will be useful for this, and will also be used later in \cref{sec:e=3defect4+}.

\begin{lem}\label{lem:signshift}
Let $e=3$.
The algebras $R^{\La_0+\La_1}(c_0\alpha_0 + c_1 \alpha_1 + c_2 \alpha_2)$ and $R^{\La_0+\La_1}(c_1\alpha_0 + c_0 \alpha_1 + c_2 \alpha_2)$ are isomorphic.
Similarly, the algebras $R^{2\La_0}(c_0\alpha_0 + c_1 \alpha_1 + c_2 \alpha_2)$ and $R^{2\La_0}(c_0\alpha_0 + c_2 \alpha_1 + c_1 \alpha_2)$ are isomorphic.
\end{lem}

\begin{proof}
The Dynkin automorphism which swaps the nodes $0$ and $1$ gives the isomorphism as in the proof of \cref{T:Equiv1}.
One may consider the automorphism as the composition of the sign automorphism $R^{\La_0+\La_1}(c_0\alpha_0 + c_1 \alpha_1 + c_2 \alpha_2)\simeq R^{\La_0+\La_2}(c_0\alpha_0 + c_2 \alpha_1 + c_1 \alpha_2)$ given by
\[
e(\nu)\mapsto e(-\nu),\;\; x_i\mapsto -x_i, \;\; \psi_i\mapsto -\psi_i
\]
and the shift automorphism 
$R^{\La_0+\La_2}(c_0\alpha_0 + c_2 \alpha_1 + c_1 \alpha_2)
\simeq R^{\La_1+\La_0}(c_1\alpha_0 + c_0 \alpha_1 + c_2 \alpha_2)$.

The second statement is obtained by applying the sign automorphism to $R^{2\La_0}(c_0\alpha_0 + c_1 \alpha_1 + c_2 \alpha_2)$.
\end{proof}

Following \cref{C:ListCases1}, we can list the blocks up to $\SU$-equivalence as follows:

\begin{multicols}{3}
\begin{enumerate}
    \item[(i)] $R^{\La_0+\La_1}(\delta),$
    \item[(iv)] $R^{\La_0+\La_1}(\delta+\alpha_0),$
    \item[(ii)] $R^{\La_0+\La_1}(\delta+\alpha_1),$
    \item[(v)] $R^{\La_0+\La_1}(\delta+\alpha_0+\alpha_2),$
    \item[(iii)] $R^{\La_0+\La_1}(\delta+\alpha_1+\alpha_2),$
    \item[(vi)] $R^{\La_0+\La_1}(\delta+2\alpha_0+2\alpha_1).$
\end{enumerate}
\end{multicols}

Moreover, by \cref{lem:signshift}, the algebras appearing in (ii) and (iv) above are isomorphic, and likewise the algebras appearing in (iii) and (v) are too. So it suffices to study those algebras appearing in (i)--(iii) and (vi) above. This is done in \cref{prop:e3defect2first,prop:e3defect2second,prop:e3defect2third,prop:e3defect2fourth} below, and summarised in \cref{thm:e3def2b}.
We note that the decomposition matrices appearing in all of these proofs are those for column Specht modules, i.e.~$([\spe{\bla}: \D{\bmu}]_v)_{\bla,\bmu}$.

\begin{prop}\label{prop:e3defect2first}
Let $e=3$.
Then $B:=R^{\La_0+\La_1}(\delta)$ is Schurian-finite.
\end{prop}

\begin{proof}
The decomposition matrix of $B$ is as follows.
\[
\begin{matrix*}[r] 
((3),\varnothing)\\
((2,1),\varnothing)\\
((1^3), \varnothing)\\
((1^2), (1)) \\
((1), (2)) \\
(\varnothing, (3)) \\
(\varnothing, (2,1)) \\
(\varnothing, (1^3))
\end{matrix*} \quad 
\begin{pmatrix} 
1 &  &  &  \\
v & 1 &  &  \\
0 & v &  &  \\
v & v^2 & 1 &  \\
v^2 & 0 & v & 1 \\
0 & 0 & 0 & v \\
0 & 0 & v & v^2 \\
0 & 0 & v^2 & 0
\end{pmatrix}
\]
Set $A:=\End_{B}(P_1\oplus P_2\oplus P_3\oplus P_4)^{\rm op}$, where
\[
D_1:=\D{((3), \varnothing)}, \quad 
D_2:=\D{((2,1), \varnothing)}, \quad 
D_3:=\D{((1^2), (1))}, \quad 
D_4:=\D{((1), (2))}.
\]
We obtain the graded dimensions of $A$ as follows\footnote{This table agrees with the graded Cartan matrix of $A$.}. 
\[
\renewcommand\arraystretch{1.2}
\begin{tabular}{c|ccccccccccccc}
$\Hom$ & $P_1$ & $P_2$ & $P_3$ & $P_4$ \\ \hline
$P_1$ & $1 + 2v^2 + v^4$ & $v + v^3$ & $v + v^3$ & $v^2$ \\
$P_2$ & $v + v^3$ & $1 + v^2 + v^4$ & $v^2$ & 0 \\
$P_3$ & $v + v^3$ & $v^2$ & $1 + 2v^2 + v^4$ & $v + v^3$ \\
$P_4$ & $v^2$ & 0 & $v + v^3$ & $1 + v^2 + v^4$ \\
\end{tabular}
\]
There are six non-zero degree one elements in
\begin{gather*}
\Hom(P_2,P_1), \quad \Hom(P_1,P_3), \quad \Hom(P_3,P_4), \\
\Hom(P_1,P_2), \quad \Hom(P_3,P_1), \quad \Hom(P_4,P_3).
\end{gather*}
We show that those elements generate $A$. 
\begin{itemize}
\item We start by determining the radical structure of $P_2$ and $P_4$.
Note that $P_2$ has Specht filtration whose successive quotients are $\spe{((2,1),\varnothing)}$, $\spe{((1^3),\varnothing)}$, $\spe{((1^2),(1))}$ in this order. Since $\spe{((1^2),(1))}$ is a submodule of $P_2$, we may conclude that $\spe{((1^2),(1))}$ is uniserial, so that the successive quotients are 
\begin{center}
$\vcenter{\xymatrix@C=0.1cm@R=0.2cm{
D_2\ar@{-}[d]\\
D_1
}}$ 
$\qquad$$D_2$$\qquad$
$\vcenter{\xymatrix@C=0.1cm@R=0.2cm{
D_3\ar@{-}[d]\\
D_1\ar@{-}[d]\\
D_2
}}$
\end{center}
With the self-duality of $P_2$ in mind, there is a unique radical structure of $P_2$ displayed as
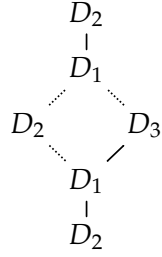
\[
\xymatrix@C=0.1cm@R=0.2cm{
&D_2\ar@{-}[d]&\\
&D_1\ar@{.}[dl]\ar@{.}[dr]&\\
D_2\ar@{.}[dr]&& D_3\ar@{-}[dl]\\
&D_1\ar@{-}[d]&\\
&D_2&
}
\]
where the solid lines indicate the Specht filtrations. We obtain
\begin{gather*}
\Ext^1(D_2,D_2)=0, \quad \Ext^1(D_2,D_3)=\Ext^1(D_3,D_2)=0, \quad \Ext^1(D_2,D_4)=\Ext^1(D_4,D_2)=0.
\end{gather*}
As the lower $D_1$ sits in the heart of $\spe{((1^2),(1))}$, it cannot appear in $\Rad P_2/\Rad^2 P_2$, which implies $\Ext^1(D_2,D_1)=\Ext^1(D_1,D_2)=\bbf$. 
Similarly, $\spe{((1),(2))}$ and $\spe{(\varnothing,(2,1))}$ are uniserial because each of them is a submodule of $P_1$ or $P_4$, respectively, and must therefore have a simple socle.
Then, successive quotients of the Specht filtration of $P_4$ are as follows,
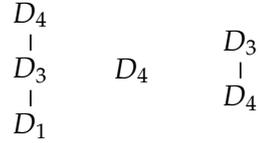
\begin{center}
$\vcenter{\xymatrix@C=0.1cm@R=0.2cm{
D_4\ar@{-}[d]\\
D_3\ar@{-}[d]\\
D_1
}}$$\qquad$$D_4$$\qquad$
$\vcenter{\xymatrix@C=0.1cm@R=0.2cm{
D_3\ar@{-}[d]\\
D_4
}}$
\end{center}
so that there is a unique radical structure of $P_4$ displayed as
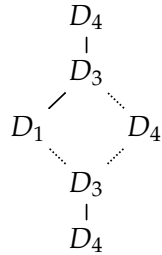
\begin{center}
$\vcenter{\xymatrix@C=0.1cm@R=0.2cm{
&D_4\ar@{-}[d]&\\
&D_3\ar@{-}[dl]\ar@{.}[dr]&\\
D_1\ar@{.}[dr]&& D_4\ar@{.}[dl]\\
&D_3\ar@{-}[d]&\\
&D_4&
}}$
\end{center}
and we obtain
\begin{gather*}
\Ext^1(D_4,D_1)=\Ext^1(D_1,D_4)=0, \quad \Ext^1(D_4,D_2)=\Ext^1(D_2,D_4)=0, \quad \Ext^1(D_4,D_4)=0.
\end{gather*}
As the lower $D_3$ cannot appear in $\Rad P_4/\Rad^2 P_4$, $\Ext^1(D_4,D_3)=\Ext^1(D_3,D_4)=\bbf$. 

\item We find that $\Rad P_1$ has Specht filtration with successive quotients $\spe{((2,1),\varnothing)}$, $\spe{((1^2),(1))}$, $\spe{((1),(2))}$ in this order.
We already know that $\spe{((2,1),\varnothing)}$ and $\spe{((1),(2))}$ are uniserial.
Since $\spe{((1^2),(1))}$ is a submodule of $P_2$, it is uniserial again.
Hence, those successive quotients are as follows.
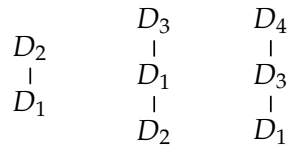
\begin{center}
$\vcenter{\xymatrix@C=0.1cm@R=0.2cm{
D_2\ar@{-}[d]\\
D_1
}}$ 
$\qquad$
$\vcenter{\xymatrix@C=0.1cm@R=0.2cm{
D_3\ar@{-}[d]\\
D_1\ar@{-}[d]\\
D_2
}}$ 
$\qquad$$\vcenter{\xymatrix@C=0.1cm@R=0.2cm{
D_4\ar@{-}[d]\\
D_3\ar@{-}[d]\\
D_1
}}$ 
\end{center}
Since $D_1$ appears in $\Rad \spe{((2,1),\varnothing)}$, $\Rad \spe{((1^2),(1))}$ and $\Rad^2 \spe{((1),(2))}$, $D_1$ cannot appear in \linebreak $\Rad P_1/\Rad^2 P_1$.
Moreover, $D_3$ in $\Rad \spe{((1),(2))}$ cannot appear in $\Rad P_1/\Rad^2 P_1$.
We then obtain
\[
\Ext^1(D_1,D_1)=0, \quad \Ext^1(D_1,D_3)=\Ext^1(D_3,D_1)=\bbf.
\]

\item We see that successive quotients of the Specht filtration of $P_3$ are as follows. 
\begin{center}
$\vcenter{\xymatrix@C=0.1cm@R=0.2cm{
D_3\ar@{-}[d]\\
D_1\ar@{-}[d]\\
D_2
}}$ 
$\qquad$
$\vcenter{\xymatrix@C=0.1cm@R=0.2cm{
D_4\ar@{-}[d]\\
D_3\ar@{-}[d]\\
D_1
}}$$\qquad$
$\vcenter{\xymatrix@C=0.1cm@R=0.2cm{
D_3\ar@{-}[d]\\
D_4
}}$$\qquad$
$D_3$
\end{center}
If $D_3$ appeared in $\Rad P_3/\Rad^2 P_3$, it would appear in $\Soc^2 P_3/\Soc P_3$, but it is impossible as $\Soc \spe{((1^2),(1))}$, $\Soc \spe{((1),(2))}$ and $\Soc \spe{(\varnothing, (2,1))}$ are not $D_3$.
We conclude that $\Ext^1(D_3,D_3)=0$. 
\end{itemize}
We have proved that the six degree one elements generate the algebra $A$, and we obtain the (Gabriel) quiver of $A$, which we display as follows. 
\[
\xymatrix{2\ar@<0.5ex>[r]^{\alpha_1}&1\ar@<0.5ex>[l]^{\beta_1}\ar@<0.5ex>[r]^{\alpha_2}&3\ar@<0.5ex>[l]^{\beta_2}\ar@<0.5ex>[r]^{\alpha_3}&4 \ar@<0.5ex>[l]^{\beta_3}}.
\]

By graded dimensions, we have $\alpha_1\alpha_2\alpha_3=0$ and $\beta_3\beta_2\beta_1=0$.
In order to find other relations, it is easier to work directly with the KLR generators and relations. We find  
\[
[P_1]=f_2f_1f_0v_\La, \quad [P_2]=f_1f_2f_0v_\La, \quad  [P_3]=f_2f_0f_1v_\La, \quad [P_4]=f_0f_2f_1v_\La
\]
in the deformed Fock space\footnote{The deformed Fock space here is the one whose crystal $B(\La)$ generated by the vacuum vector $v_\La$ is the set of conjugate-Kleshchev bipartitions.}, and we have
\[
P_1=Be(012), \quad P_2=Be(021), \quad  P_3=Be(102), \quad P_4=Be(120). 
\]
From this description, we also see that the swap of $f_0$ and $f_1$ 
induces an algebra automorphism of $A$, which swaps $P_1\leftrightarrow P_3$ and $P_2\leftrightarrow P_4$. Moreover, 
$A=eBe$ for $e=e(021)+e(012)+e(102)+e(120)$, 
and we may choose $\alpha_i$ and $\beta_i$, for $i=1,2,3$, as follows.
\begin{gather*}
\alpha_1=e(021)\psi_2e(012), \quad \alpha_2=e(012)\psi_1e(102), \quad \alpha_3=e(102)\psi_2e(120), \\
\beta_1=e(012)\psi_2e(021), \quad \beta_2=e(102)\psi_1e(012), \quad \beta_3=e(120)\psi_2e(102).
\end{gather*}
In the following, we show that $A$ is isomorphic to $\mathcal{SF}_4:=\bbf Q/I$ with 
\begin{equation}\label{SF-alg-4}
\begin{matrix}
Q:\ \xymatrix{2\ar@<0.5ex>[r]^{\alpha_1}&1\ar@<0.5ex>[l]^{\beta_1}\ar@<0.5ex>[r]^{\alpha_2}&3\ar@<0.5ex>[l]^{\beta_2}\ar@<0.5ex>[r]^{\alpha_3}&4 \ar@<0.5ex>[l]^{\beta_3}}
\qquad 
I:\ 
\begin{matrix}
\alpha_1\alpha_2\alpha_3=\beta_3\beta_2\beta_1=\alpha_2\beta_2\alpha_2=\beta_2\alpha_2\beta_2=0,\\
\beta_1\alpha_1\alpha_2=-\alpha_2\alpha_3\beta_3,\quad \alpha_3\beta_3\beta_2=-\beta_2\beta_1\alpha_1, \\
\alpha_1\beta_1\alpha_1=-\alpha_1\alpha_2\beta_2,\quad  
\beta_1\alpha_1\beta_1=-\alpha_2\beta_2\beta_1, \\
\alpha_3\beta_3\alpha_3=-\beta_2\alpha_2\alpha_3, \quad 
\beta_3\beta_2\alpha_2=-\beta_3\alpha_3\beta_3.
\end{matrix}
\end{matrix}
\end{equation}
\begin{itemize}
\item $\alpha_1\alpha_2\alpha_3=\beta_3\beta_2\beta_1=0$ follows from the 
graded dimensions of $A$.

\item Since $x_1e(0**)=x_1e(1**)=0$, we have
\[
\begin{aligned}
\alpha_2\beta_2\alpha_2&=\psi_1^3e(102)=(x_1-x_2)\psi_1e(102)
=-\psi_1x_1e(102)=0,\\
\beta_2\alpha_2\beta_2&=\psi_1^3e(012)=(x_2-x_1)\psi_1e(012)=\psi_1x_1e(012)=0.
\end{aligned}
\]

\item Observe that $\psi_1e(120)=0$, then $\psi_1^2e(120)=(x_1-x_2)e(120)$ implies $x_2e(120)=0$. Similarly, $x_2e(021)=0$. We obtain $\alpha_1\beta_1\alpha_1=-\alpha_1\alpha_2\beta_2$ using the following calculation:
\[
\begin{aligned} 
\alpha_1\beta_1\alpha_1&=\psi_2^3e(012)=(x_3-x_2)\psi_2e(012)=x_3\psi_2e(012),\\
\alpha_1\alpha_2\beta_2&=\psi_2\psi_1^2e(012)=\psi_2(x_1-x_2)e(012)=-x_3\psi_2e(012).
\end{aligned}
\]
We obtain $\beta_1\alpha_1\beta_1=-\alpha_2\beta_2\beta_1$ using the following calculation:
\[
\begin{aligned} 
\beta_1\alpha_1\beta_1&=\psi_2^3e(021)=\psi_2(x_3-x_2)e(021)=\psi_2x_3e(021),\\
\alpha_2\beta_2\beta_1&=\psi_1^2\psi_2e(021)=(x_1-x_2)\psi_2e(021)=-\psi_2x_3e(021).
\end{aligned}
\]

\item
Similar to the above, 
\begin{align*} 
\alpha_3\beta_3\alpha_3&=\psi_2^3e(120)=\psi_2(x_2-x_3)e(120)= -\psi_2x_3e(120), \\
\beta_2\alpha_2\alpha_3&=\psi_1^2\psi_2e(120)=(x_2-x_1)e(102)\psi_2 =\psi_2x_3e(120),
\end{align*}
imply $\alpha_3\beta_3\alpha_3=-\beta_2\alpha_2\alpha_3$. We also obtain $\beta_3\alpha_3\beta_3=-\beta_3\beta_2\alpha_2$ by the anti-involution of $A$ swapping $\alpha_i$ and $\beta_i$, for $i=1,2,3$.

\item One finds that $\beta_1\alpha_1\alpha_2=-x_3 \psi_1 e(102)$ and $\alpha_2\alpha_3\beta_3=x_3\psi_1e(102)$, which implies  $\beta_1\alpha_1\alpha_2=-\alpha_2\alpha_3\beta_3$ and $\alpha_3\beta_3\beta_2=-\beta_2\beta_1\alpha_1$.
\end{itemize}
It remains to look at the structure of indecomposable projective modules in order to know that we have obtained all the relations.
We use solid lines to indicate the Specht filtrations, and the dotted lines indicate how multiplication by an arrow connects two basis elements.

\begin{itemize}
\item The structure of $P_1$ is displayed as follows. 
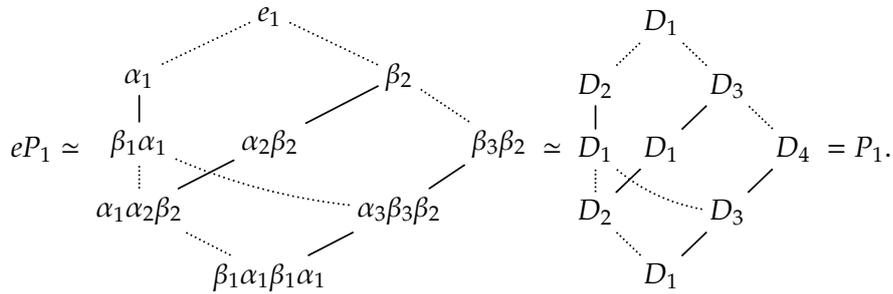
\begin{center}
$eP_1\simeq \vcenter{\xymatrix@C=0.2cm@R=0.3cm{
&e_1\ar@{.}[dl]\ar@{.}[dr]&&\\
\alpha_1\ar@{-}[d]&&\beta_2\ar@{-}[dl]\ar@{.}[dr]&\\
\beta_1\alpha_1\ar@{.}[d]\ar@/_0.2cm/@{.}[drr]&\alpha_2\beta_2\ar@{-}[dl]&& \beta_3\beta_2\\
\alpha_1\alpha_2\beta_2\ar@{.}[dr]&&
\alpha_3\beta_3\beta_2\ar@{-}[ur]\ar@{-}[dl]&\\
&\beta_1\alpha_1\beta_1\alpha_1&&
}}\simeq\vcenter{\xymatrix@C=0.2cm@R=0.3cm{
&D_1\ar@{.}[dl]\ar@{.}[dr]&&\\
D_2\ar@{-}[d]&&D_3\ar@{-}[dl]\ar@{.}[dr]&\\
D_1\ar@{.}[d]\ar@/_0.2cm/@{.}[drr]&D_1\ar@{-}[dl]&& D_4\\
D_2\ar@{.}[dr]&&D_3\ar@{-}[ur]\ar@{-}[dl]&\\
&D_1&&
}}=P_1.
$
\end{center}
We should understand the structure of $\Rad^2 P_1$ as follows.
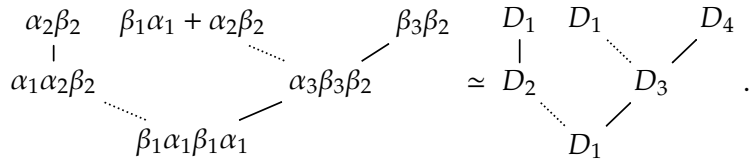
\begin{center}
$\vcenter{\xymatrix@C=0.1cm@R=0.2cm{\alpha_2\beta_2\ar@{-}[d]&\beta_1\alpha_1+\alpha_2\beta_2\ar@{.}[dr]&& \beta_3\beta_2\\
\alpha_1\alpha_2\beta_2\ar@{.}[dr]&&
\alpha_3\beta_3\beta_2\ar@{-}[ur]\ar@{-}[dl]&\\
&\beta_1\alpha_1\beta_1\alpha_1&&
}}\simeq\vcenter{\xymatrix@C=0.2cm@R=0.3cm{
D_1\ar@{-}[d]&D_1\ar@{.}[dr]&& D_4\\
D_2\ar@{.}[dr]&&D_3\ar@{-}[ur]\ar@{-}[dl]&\\
&D_1&&
}}.$
\end{center}
Note that $\beta_1\alpha_1\beta_1\alpha_1=-\beta_1\alpha_1\alpha_2\beta_2=\alpha_2\alpha_3\beta_3\beta_2$. 

\item The structure of $P_2$ is displayed as follows. 
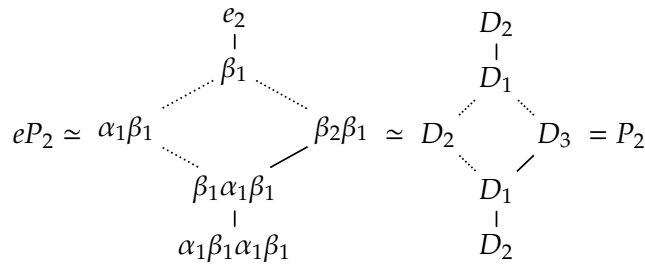
\begin{center}
$eP_2\simeq\vcenter{\xymatrix@C=0.1cm@R=0.2cm{
&e_2\ar@{-}[d]&\\
&\beta_1\ar@{.}[dl]\ar@{.}[dr]&\\
\alpha_1\beta_1\ar@{.}[dr]&& \beta_2\beta_1\ar@{-}[dl]\\
&\beta_1\alpha_1\beta_1\ar@{-}[d]&\\
&\alpha_1\beta_1\alpha_1\beta_1&
}}\simeq\vcenter{\xymatrix@C=0.1cm@R=0.2cm{
&D_2\ar@{-}[d]&\\
&D_1\ar@{.}[dl]\ar@{.}[dr]&\\
D_2\ar@{.}[dr]&& D_3\ar@{-}[dl]\\
&D_1\ar@{-}[d]&\\
&D_2&
}}=P_2$
\end{center}
Note that $\beta_2\beta_1\alpha_1\beta_1=-\alpha_3\beta_3\beta_2\beta_1=0$. 

\item Similar to $P_1$, the structure of $P_3$ is displayed as follows. 
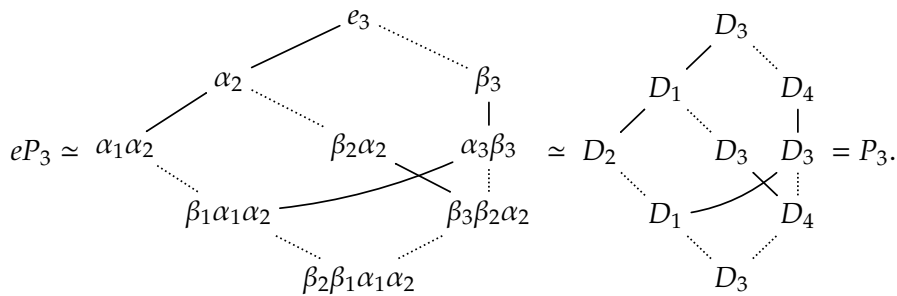
\begin{center}
$eP_3\simeq\vcenter{\xymatrix@C=0.2cm@R=0.3cm{
&&e_3\ar@{-}[dl]\ar@{.}[dr]&\\
&\alpha_2\ar@{-}[dl]\ar@{.}[dr]&&\beta_3\ar@{-}[d]\\
\alpha_1\alpha_2&& \beta_2\alpha_2\ar@{-}[dr]&\alpha_3\beta_3\ar@{.}[d]\ar@/^0.2cm/@{-}[dll]\\
&\beta_1\alpha_1\alpha_2\ar@{.}[ul]\ar@{.}[dr]&&\beta_3\beta_2\alpha_2\ar@{.}[dl]\\
&&\beta_2\beta_1\alpha_1\alpha_2&
}}\simeq\vcenter{\xymatrix@C=0.2cm@R=0.3cm{
&&D_3\ar@{-}[dl]\ar@{.}[dr]&\\
&D_1\ar@{-}[dl]\ar@{.}[dr]&&D_4\ar@{-}[d]\\
D_2&& D_3\ar@{-}[dr]&D_3\ar@{.}[d]\ar@/^0.2cm/@{-}[dll]\\
&D_1\ar@{.}[ul]\ar@{.}[dr]&&D_4\ar@{.}[dl]\\
&&D_3&
}}=P_3.$
\end{center}
We should understand the structure of $\Rad^2 P_3$ as follows.
\begin{center}
$\vcenter{\xymatrix@C=0.2cm@R=0.3cm{
\alpha_1\alpha_2&& \beta_2\alpha_2+\alpha_3\beta_3\ar@{-}[dl]&\beta_2\alpha_2\ar@{-}[d]\\
&\beta_1\alpha_1\alpha_2\ar@{.}[ul]\ar@{.}[dr]&&\beta_3\beta_2\alpha_2\ar@{.}[dl]\\
&&\beta_2\beta_1\alpha_1\alpha_2&
}}\simeq\vcenter{\xymatrix@C=0.2cm@R=0.3cm{
D_2&& D_3\ar@{-}[dl]&D_3\ar@{-}[d]\\
&D_1\ar@{.}[ul]\ar@{.}[dr]&&D_4\ar@{.}[dl]\\
&&D_3&
}}.$
\end{center}

\item Similar to $P_2$, the structure of $P_4$ is displayed as follows.
\begin{center}
$eP_4\simeq\vcenter{\xymatrix@C=0.1cm@R=0.2cm{
&e_4\ar@{-}[d]&\\
&\alpha_3\ar@{-}[dl]\ar@{.}[dr]&\\
\alpha_2\alpha_3\ar@{.}[dr]&& \beta_3\alpha_3\ar@{.}[dl]\\
&\beta_2\alpha_2\alpha_3\ar@{-}[d]&\\
&\beta_3\alpha_3\beta_3\alpha_3&
}}\simeq\vcenter{\xymatrix@C=0.1cm@R=0.2cm{
&D_4\ar@{-}[d]&\\
&D_3\ar@{-}[dl]\ar@{.}[dr]&\\
D_1\ar@{.}[dr]&& D_4\ar@{.}[dl]\\
&D_3\ar@{-}[d]&\\
&D_4&
}}=P_4$
\end{center}
\end{itemize}

We have proved that $A\cong \mathcal{SF}_4$ and are ready to study the $\tau$-tilting finiteness of $A$. Since $A\simeq A^{\rm op}$, we may view $P_i$ as right $A$-modules. We notice that $\beta_3\alpha_3-\beta_2\alpha_2$ and $\alpha_1\beta_1-\alpha_3\beta_3- \beta_2\alpha_2+\beta_1\alpha_1$ are central elements of $\mathcal{SF}_4$, and $\alpha_2\beta_2$ is a central element of the quotient algebra of $\mathcal{SF}_4$ modulo the ideal generated by the aforementioned two elements. We define
\[
\widetilde{\mathcal{SF}_4}:=\mathcal{SF}_4/\langle \alpha_1\beta_1, \alpha_2\beta_2, \alpha_3\beta_3, \beta_3\alpha_3, \beta_2\alpha_2, \beta_1\alpha_1\rangle.
\]
Then, $\widetilde{\mathcal{SF}_4}$ can be shown by the String Applet~\cite{StringApplet} to be a representation-finite string algebra, which is obviously Schurian-finite.
In fact, there are only 28 Schurian modules.
It follows from \cref{ejr reduction} that $\mathcal{SF}_4$ is also Schurian-finite.
\end{proof}

\begin{prop}\label{prop:e3defect2second}
Let $e=3$.
Then $B:=R^{\La_0+\La_1}(\delta+\alpha_1)$ is Schurian-finite.
\end{prop}

\begin{proof}
The decomposition matrix of $B$ is given by 
\[
\begin{matrix*}[r] 
((2,1^2), \varnothing) \\
((3), (1)) \\
((2,1), (1)) \\
((1^3), (1)) \\
((2), (2)) \\
(\varnothing, (4)) \\
(\varnothing, (2^2)) \\
(\varnothing, (1^4))
\end{matrix*} \quad 
\begin{pmatrix} 
1 &  &  &  \\
0 & 1 &  &  \\
v & v & 1 &  \\
v^2 & 0 & v &  \\
0 & v^2 & v & 1 \\
0 & 0 & 0 & v \\
0 & 0 & v & v^2 \\
0 & 0 & v^2 & 0 
\end{pmatrix}.
\]
We set 
\[
D_1:=\D{((2,1^2), \varnothing) }, \quad 
D_2:=\D{((3), (1))}, \quad 
D_3:=\D{((2,1), (1))}, \quad 
D_4:=\D{((2), (2))}, \quad 
\]
and $A:=(\End_{B}(P_1\oplus P_2\oplus P_3\oplus P_4))^{\rm op}$. 

$\spe{((1^3),(1))} \subseteq P_1$ implies that
$\spe{((1^3),(1))}$ is uniserial of length two, with socle $D_1$ and head $D_3$. 
Consider $\Rad P_1/\Soc P_1$. Since $\Rad P_1$ has Specht filtration whose successive quotients are $\spe{((2,1),(1))}$, $\spe{((1^3),(1))}$ in this order, the self-dual module $\Rad P_1/\Soc P_1$ has $D_3$ in the socle and in the head. 
Then, $D_3$ in the head must be the head of $\spe{((2,1),(1))}$ because $\spe{((2,1),(1))}$ has a simple head.
Therefore, the self-duality of $\Rad P_1/\Soc P_1$ implies the radical structure of $P_1$ depicted below.
In particular, $\spe{((2,1),(1))}$ is of length two and $\Soc \spe{((2,1),(1))}=D_1\oplus D_2$. 

$\spe{(\vn,(2^2))}\subseteq P_4$ implies that 
$\spe{(\vn,(2^2))}$ is uniserial of length two, with socle $D_4$ and head $D_3$.
Note that $\spe{((2),(2))}$ has simple head $D_4$ because 
$\spe{((2),(2))}$ is a factor module of $P_4$.
Then, $\spe{((2),(2))}\subseteq P_2$ implies that 
$\spe{((2),(2))}$ is uniserial of length three, with head $D_4$, heart $D_3$, and socle $D_2$. 
Now we consider $\Rad P_4/\Soc P_4$.
Since $P_4$ has Specht filtration whose successive quotients are $\spe{((2),(2))}$, $\spe{(\vn,(4))}$, $\spe{(\vn,(2^2))}$ in this order, 
$\Rad P_4/\Soc P_4$ has $D_3$ in the head and in the socle. 
Therefore, the radical structure of $P_4$ is as depicted below, since $\Rad P_4/\Soc P_4$ is self-dual. 

$\Rad P_2$ has Specht filtration whose successive quotients are $\spe{((2,1),(1))}$, $\spe{((2),(2))}$ in this order. 
Since we already know the radical structure of 
$\spe{((2,1),(1))}$ and $\spe{((2),(2))}$, we see that 
$\Rad P_2/\Soc P_2$ has $D_3$ in the head and in the socle, and 
the self-duality of $\Rad P_2/\Soc P_2$ implies that the radical structure of $P_2$ is as depicted below. 
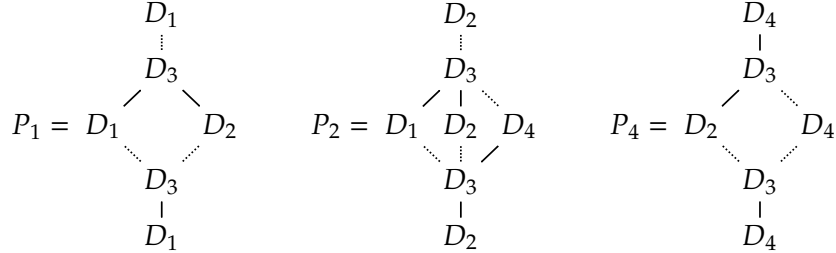
\begin{center}
$P_1=\vcenter{\xymatrix@C=0.1cm@R=0.2cm{
&D_1\ar@{.}[d]&\\
&D_3\ar@{-}[dl]\ar@{-}[dr]&\\
D_1\ar@{.}[dr]&& D_2\ar@{.}[dl]\\
&D_3\ar@{-}[d]&\\
&D_1&
}}$ \qquad
$P_2=\vcenter{\xymatrix@C=0.1cm@R=0.2cm{
&D_2\ar@{.}[d]&\\
&D_3\ar@{-}[dl]\ar@{-}[d]\ar@{.}[dr]&\\
D_1 \ar@{.}[dr] & D_2\ar@{.}[d] & D_4\ar@{-}[dl]\\
&D_3\ar@{-}[d]&\\
&D_2&
}}$ \qquad
$P_4=\vcenter{\xymatrix@C=0.1cm@R=0.2cm{
&D_4\ar@{-}[d]&\\
&D_3\ar@{-}[dl]\ar@{.}[dr]&\\
D_2\ar@{.}[dr]&& D_4\ar@{.}[dl]\\
&D_3\ar@{-}[d]&\\
&D_4&
}}$
\end{center}
Here, the solid black lines indicate the structure of Specht modules.
By a similar argument, we also find the radical layers of $P_3$ as follows.
\begin{gather*}
    \Rad P_3/\Rad^2 P_3=D_1\oplus D_2\oplus D_4,\quad
    \Rad^2 P_3/\Rad^3 P_3=D_3\oplus D_3\oplus D_3,\\
    \Rad^3 P_3/\Rad^4 P_3=D_1\oplus D_2\oplus D_4.
\end{gather*}
Then, it follows that $\Ext^1(D_3,D_i)=\Ext^1(D_i,D_3)=\bbf$, for $i=1,2,4$, and all the other extensions vanish.
Hence, we obtain the quiver of $A$ displayed as follows. 
\[
\xymatrix@C=1cm@R=0.8cm{1\ar@<-0.5ex>[r]_{\beta_1}&3\ar@<-0.5ex>[l]_{\alpha_1}\ar@<-0.5ex>[d]_{\alpha_2}\ar@<0.5ex>[r]
^{\alpha_3}&4\ar@<0.5ex>[l]^{\beta_3}\\
&2\ar@<-0.5ex>[u]_{\beta_2}&}
\]
We need to look at the structure of PIMs more closely. To do this, we observe that 
\[
[P_1]=f_1^{(2)}f_2f_0v_\Lambda, \quad [P_2]=f_2f_1^{(2)}f_0v_\Lambda, \quad [P_3]=f_1f_2f_0f_1v_\Lambda, \quad [P_4]=f_1f_0f_2f_1v_\Lambda
\]
in the deformed Fock space. Hence, if we set
\[
e_1=x_4\psi_3e(0211), \quad e_2=x_3\psi_2e(0112), \quad e_3=e(1021), \quad e_4=e(1201)
\]
and $e=e_1+e_2+e_3+e_4$, 
then $A=eBe$ and we obtain $P_1=Be_1\langle 1\rangle$,
$P_2=Be_2\langle 1\rangle$, $P_3=Be_3$, $P_4=Be_4$. 
Then, for $i=1,2$, we have
\[
\begin{aligned}
\dim_v e_3Be_i&=\dim_v \Hom(Be_3, Be_i)=\dim_v \Hom(P_3, P_i\langle -1\rangle)\\
&=\dim_v \Hom(P_3, P_i)\langle -1\rangle=1+v^2,\\
\dim_v e_iBe_3&=\dim_v \Hom(Be_i, Be_3)=\dim_v \Hom(P_i\langle -1\rangle, P_3)\\
&=\dim_v \Hom(P_i, P_3)\langle 1\rangle=v^2+v^4, \\
\dim_v e_3Be_4&=\dim_v e_4Be_3=v+v^3.
\end{aligned}
\]
Moreover, we may choose 
\[
\alpha_1=e(1021)\psi_1\psi_2\psi_3e(0211),\quad
\alpha_2=e(1021)\psi_3\psi_1\psi_2e(0112),\quad 
\alpha_3=e(1021)\psi_2e(1201),
\]
\[
\beta_1=e(0211)x_4\psi_3\psi_2\psi_1e(1021),\quad \beta_2=e(0112)x_3\psi_2\psi_1\psi_3e(1021),\quad 
\beta_3=e(1201)\psi_2 e(1021).
\]
We have the following relations.
\begin{itemize}
\item $\beta_1\alpha_3=\beta_3\alpha_1=0$, by
$\dim_v \Hom(P_1,P_4)=\dim_v \Hom(P_4,P_1)=0$.
\item We have $\alpha_3\beta_3=\psi_2^2e(1021)=(x_3-x_2)e(1021)$,
\[
\begin{aligned}
\alpha_1\beta_1&=\psi_1\psi_2\psi_3x_4\psi_3\psi_2\psi_1e(1021) \\
&=\psi_1\psi_2\psi_3\cdot x_4\psi_3e(0211)\cdot\psi_2\psi_1 \\
&=\psi_1\psi_2\psi_3\cdot (\psi_3x_3+1)e(0211)\cdot\psi_2\psi_1 \quad \text{by} \quad (\psi_3x_3-x_4\psi_3)e(0211)=-e(0211)\\
&=\psi_1\cdot \psi_2\psi_3 \psi_2 e(0121)\cdot\psi_1 \quad \text{by} \quad \psi_3^2e(0211)=0\\
&=\psi_1\cdot (\psi_3 \psi_2\psi_3-1)e(0121)\cdot\psi_1 \quad \text{by} \quad (\psi_3\psi_2\psi_3-\psi_2\psi_3\psi_2)e(0121)=e(0121)\\
&=\psi_1\psi_3 \psi_2\psi_3\psi_1e(1021) -\psi_1^2e(1021),
\end{aligned}
\]
\[
\begin{aligned}
\alpha_2\beta_2&=\psi_3\psi_1\psi_2x_3\psi_2\psi_1\psi_3e(1021) \\
&=\psi_3\psi_1\psi_2\cdot x_3\psi_2e(0112)\cdot \psi_1\psi_3 \\
&=\psi_3\psi_1\psi_2\cdot (\psi_2x_2+1)e(0112)\cdot \psi_1\psi_3 \quad \text{by} \quad (\psi_2x_2-x_3\psi_2)e(0112)=-e(0112) \\
&=\psi_3\psi_1\psi_2\cdot e(0112)\cdot \psi_1\psi_3 \quad \text{by} \quad \psi_2^2e(0112)=0\\
&=\psi_1\psi_3\psi_2 \psi_3\psi_1e(1021).\\
\end{aligned}
\]
It gives $\alpha_1\beta_1=\alpha_2\beta_2-\psi_1^2e(1021)$.
Since $3v^2$ appears in $\dim_v \Hom(P_3,P_3)$, 
$\alpha_1\beta_1, \alpha_2\beta_2, \alpha_3\beta_3$ 
are linearly independent.
In particular $\psi_1^2e(1021)\neq 0$.

\item Since 
$\psi_1^3e(1021)=\psi_1x_2e(1021)=x_1e(0121)\psi_1=0$, 
\[
\begin{aligned}
\beta_1\alpha_1\beta_1&=\beta_1\alpha_2\beta_2-x_4\psi_3\psi_2\psi_1^3e(1021)=\beta_1\alpha_2\beta_2,\\
\beta_2\alpha_1\beta_1&=\beta_2\alpha_2\beta_2-x_3\psi_2\psi_3\psi_1^3e(1021)=\beta_2\alpha_2\beta_2.
\end{aligned}
\]
Similarly, using $\psi_1^3e(0121)=-\psi_1x_2e(0121)=-x_1e(1021)\psi_1=0$, we have
\[
\begin{aligned}
\alpha_1\beta_1\alpha_1&=\alpha_2\beta_2\alpha_1-\psi_1^3e(0121) \psi_2\psi_3=\alpha_2\beta_2\alpha_1,\\
\alpha_1\beta_1\alpha_2&=\alpha_2\beta_2\alpha_2-\psi_1^3e(0121)\psi_3\psi_2=\alpha_2\beta_2\alpha_2.
\end{aligned}
\]

Now, we may give the shape of the projective right $A$-module $e_1A$ as follows. 

\begin{center}
$e_1A\simeq\vcenter{\xymatrix@C=0.1cm@R=0.2cm{
&e_1\ar@{.}[d]&\\
&\beta_1\ar@{-}[dl]\ar@{-}[dr]&\\
\beta_1\alpha_1\ar@{.}[dr]&& \beta_1\alpha_2\ar@{.}[dl]\\
& \beta_1\alpha_1\beta_1 \ar@{-}[d]&\\
& \beta_1\alpha_1\beta_1\alpha_1&
}}\simeq\vcenter{\xymatrix@C=0.1cm@R=0.2cm{
&D_1\ar@{.}[d]&\\
&D_3\ar@{-}[dl]\ar@{-}[dr]&\\
D_1\ar@{.}[dr]&& D_2\ar@{-}[dl]\\
&D_3\ar@{-}[d]&\\
&D_1&
}}$
\end{center}

\item Since $e(2101)=0$, we have
\begin{gather*}
x_2e(1201)=(x_2-x_1)e(1201)=-\psi_1^2e(1201)
=-\psi_1e(1201)\psi_1=0 \;\;\text{and}\\
\beta_3\alpha_3\beta_3=\psi_2^2e(1201)\psi_2=(x_2-x_3)e(1201)\psi_2=-x_3\psi_2e(1021)=-\psi_2x_2e(1021).
\end{gather*}
Recall that $\beta_3\alpha_1=0$ and $(\psi_3\psi_2\psi_3-\psi_2\psi_3\psi_2)e(0121)=e(0121)$. Thus, 
\[
\begin{aligned}
\beta_3\alpha_2\beta_2&=\psi_2\psi_1\psi_3\psi_2 \psi_3\psi_1e(1021)\\
&=\psi_2\psi_1\cdot \psi_3\psi_2\psi_3e(0121)\cdot \psi_1\\
&=\psi_2\psi_1\cdot (\psi_2\psi_3\psi_2+1)e(0121)\cdot \psi_1 \\
&=\psi_2\psi_1\psi_2\psi_3e(0211)\psi_2\psi_1 +\psi_2\psi_1^2 e(1021)\\
&=\beta_3\alpha_1\psi_2\psi_1 +\psi_2\psi_1^2 e(1021)\\
&=\psi_2x_2e(1021).
\end{aligned}
\]
We obtain $\beta_3\alpha_2\beta_2=-\beta_3\alpha_3\beta_3$. We have the following shape of $e_4A$. 

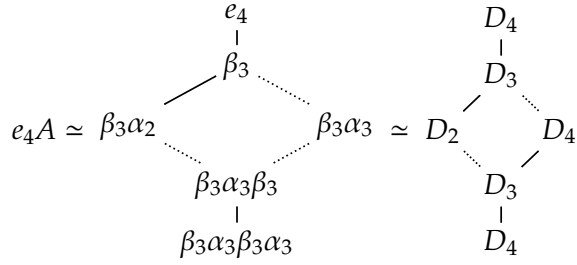
\begin{center}
$e_4A\simeq\vcenter{\xymatrix@C=0.1cm@R=0.2cm{
&e_4\ar@{-}[d]&\\
&\beta_3\ar@{-}[dl]\ar@{.}[dr]&\\
\beta_3\alpha_2\ar@{.}[dr]&& \beta_3\alpha_3\ar@{.}[dl]\\
& \beta_3\alpha_3\beta_3\ar@{-}[d]&\\
& \beta_3\alpha_3\beta_3\alpha_3&
}}\simeq\vcenter{\xymatrix@C=0.1cm@R=0.2cm{
&D_4\ar@{-}[d]&\\
&D_3\ar@{-}[dl]\ar@{.}[dr]&\\
D_2\ar@{.}[dr]&& D_4\ar@{-}[dl]\\
&D_3\ar@{-}[d]&\\
&D_4&
}}$
\end{center}

\item Since $x_2e(1201)=0$, we have 
\[
\alpha_3\beta_3\alpha_3=\psi_2(x_2-x_3)e(1201)=-\psi_2x_3e(1201)=-x_2\psi_2e(1201).
\]
Using $(\psi_3\psi_2\psi_3-\psi_2\psi_3\psi_2)e(0121)=e(0121)$ and $(\psi_2\psi_1\psi_2-\psi_1\psi_2\psi_1)e(1201)=0$, 
\[
\begin{aligned}
\alpha_2\beta_2\alpha_3&=\psi_1\psi_3\psi_2 \psi_3\psi_1\psi_2e(1201)\\
&=\psi_1\cdot \psi_3\psi_2 \psi_3e(0121)\cdot \psi_1\psi_2\\
&=\psi_1\cdot (\psi_2\psi_3\psi_2+1)e(0121)\cdot \psi_1\psi_2\\
&=\psi_1\psi_2\psi_3\psi_2 \psi_1\psi_2e(1201)+\psi_1^2e(1021)\psi_2\\
&=\psi_1\psi_2\psi_3\psi_1 \psi_2e(2101)\psi_1+x_2\psi_2e(1201)\\
&=x_2\psi_2e(1201)
\end{aligned}.
\]
We obtain $\alpha_2\beta_2\alpha_3=-\alpha_3\beta_3\alpha_3$.

\item By $\psi_1x_2e(1021)=x_1e(0121)\psi_1=0$ and $(\psi_3x_3-x_4\psi_3)e(1021)=0$, we have 
\[
\begin{aligned}
\beta_2\alpha_3\beta_3&=x_3\psi_2\psi_1\psi_3\psi_2^2e(1021)\\
&=x_3\psi_2\psi_1\psi_3(x_3-x_2)e(1021)\\
&=x_3\psi_2\psi_1\psi_3x_3e(1021) \\
&=x_3x_4\psi_2\psi_1\psi_3e(1021) 
\end{aligned}.
\]
Since $\psi_1^2e(0112)=-x_2e(0112)$, $\psi_2^2e(0112)=0$, $\psi_3^2e(0112)=(x_3-x_4)e(0112)$, and
\[
(x_4-x_3)\psi_2e(0112)=\psi_2(x_4-x_2)e(0112)-e(0112),
\]
\[
\begin{aligned}
\beta_2\alpha_2\beta_2&=x_3\psi_2\psi_1\psi_3\psi_1\psi_3e(0112)\psi_2\psi_3\psi_1\\
&=x_3\psi_2\psi_1^2\psi_3^2e(0112)\psi_2\psi_3\psi_1\\
&=x_3\psi_2x_2(x_4-x_3)e(0112)\psi_2\psi_1\psi_3 \\
&=x_3(x_3\psi_2-1)(x_4-x_3)\psi_2e(0112)\psi_1\psi_3 \\
&=-x_3^2e(0112)\psi_2\psi_1\psi_3-x_3(x_4-x_3)e(0112)\psi_2\psi_1\psi_3 \\
&=-x_3x_4\psi_2\psi_1\psi_3e(1021).
\end{aligned}
\]
We obtain $\beta_2\alpha_1\beta_1=\beta_2\alpha_2\beta_2=-\beta_2\alpha_3\beta_3$. 

\item Since $x_2\psi_1e(0121)=\psi_1x_1e(0121)=0$ and $(\psi_3x_4-x_3\psi_3)e(1012)=0$, we have 
\[
\begin{aligned}
\alpha_3\beta_3\alpha_2&=\psi_2^2e(1021)\psi_3\psi_1\psi_2e(0112)\\
&=(x_3-x_2)\psi_3\psi_1\psi_2e(0112)\\
&=x_3\psi_3\psi_1\psi_2e(0112)\\
&=\psi_3\psi_1\psi_2x_4e(0112).
\end{aligned}
\]
On the other hand, using $(\psi_2x_2-x_3\psi_2)e(0112)=-e(0112)$ and 
\[
(x_4-x_3)\psi_2e(0112)=\psi_2(x_4-x_2)e(0112)-e(0112), 
\]
\[
\begin{aligned}
\alpha_2\beta_2\alpha_2&=\psi_1\psi_3\psi_2 \psi_3\psi_1\psi_3\psi_1\psi_2e(0112)\\
&=\psi_3\psi_1\psi_2\psi_1^2\psi_3^2\psi_2e(0112)\\
&=\psi_3\psi_1\psi_2 x_2(x_4-x_3)\psi_2e(0112)\\
&=\psi_3\psi_1(x_3\psi_2-1)(x_4-x_3)\psi_2e(0112) \\
&=-\psi_3\psi_1x_3\psi_2e(0112)-\psi_3\psi_1(x_4-x_3)\psi_2e(0112) \\
&=-\psi_3\psi_1\psi_2x_4e(0112).
\end{aligned}
\]
We obtain $\alpha_1\beta_1\alpha_2=\alpha_2\beta_2\alpha_2=-\alpha_3\beta_3\alpha_2$. 
The following are the shapes of $e_2A$ and $e_3A$. 

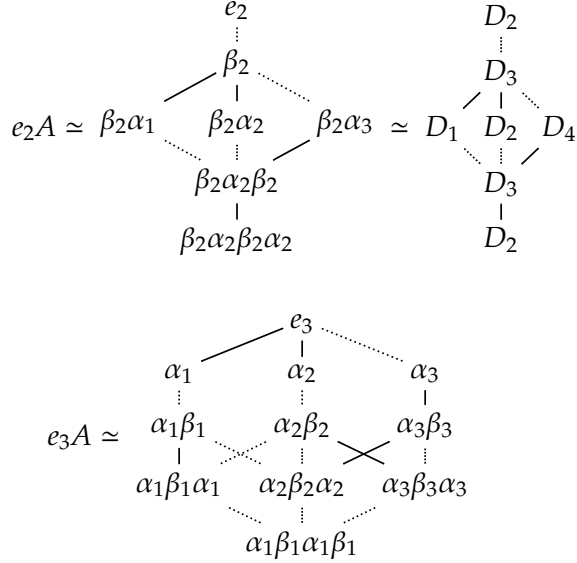
\begin{center}
$e_2A\simeq\vcenter{\xymatrix@C=0.1cm@R=0.2cm{
&e_2\ar@{.}[d]&\\
&\beta_2\ar@{-}[dl]\ar@{-}[d]\ar@{.}[dr]&\\
\beta_2\alpha_1\ar@{.}[dr]& \beta_2\alpha_2\ar@{.}[d] & \beta_2\alpha_3\ar@{-}[dl]\\
& \beta_2\alpha_2\beta_2\ar@{-}[d]&\\
& \beta_2\alpha_2\beta_2\alpha_2&
}}\simeq\vcenter{\xymatrix@C=0.1cm@R=0.2cm{
&D_2\ar@{.}[d]&\\
&D_3\ar@{-}[dl]\ar@{-}[d]\ar@{.}[dr]&\\
D_1\ar@{.}[dr]& D_2\ar@{.}[d] & D_4\ar@{-}[dl]\\
&D_3\ar@{-}[d]&\\
&D_2&
}}$
\end{center}
\end{itemize}

\begin{center}
$e_3A\simeq\vcenter{\xymatrix@C=0.1cm@R=0.2cm{
&e_3\ar@{-}[dl]\ar@{-}[d]\ar@{.}[dr]&\\
\alpha_1\ar@{.}[d]&\alpha_2\ar@{.}[d]& \alpha_3\ar@{-}[d]\\
\alpha_1\beta_1\ar@{-}[d]\ar@{.}[dr]& \alpha_2\beta_2 \ar@{.}[dl]\ar@{.}[d]\ar@{-}[dr]& \alpha_3\beta_3\ar@{-}[dl]\ar@{.}[d]\\
\alpha_1\beta_1\alpha_1\ar@{.}[dr] & \alpha_2\beta_2\alpha_2\ar@{.}[d]& \alpha_3\beta_3\alpha_3\ar@{.}[dl]\\
&\alpha_1\beta_1\alpha_1\beta_1 & 
}}$
\end{center}
To understand the structure of $\Rad^2 e_3A$, we choose a basis $\{\alpha_1\beta_1, \alpha_2\beta_2+\alpha_3\beta_3, \alpha_1\beta_1-\alpha_2\beta_2\}$ or 
$\{\alpha_2\beta_2+\alpha_3\beta_3, \alpha_1\beta_1-\alpha_2\beta_2, \alpha_3\beta_3 \}$ as before.

To summarise, $A$ is isomorphic to $\mathcal{SF}_5:=\bbf Q/I$ with 
\begin{equation}\label{SF-alg-5}
\begin{matrix}
Q:\quad\xymatrix@C=1cm@R=0.8cm{1\ar@<-0.5ex>[r]_{\beta_1}&3\ar@<-0.5ex>[l]_{\alpha_1}\ar@<-0.5ex>[d]_{\alpha_2}\ar@<0.5ex>[r]
^{\alpha_3}&4\ar@<0.5ex>[l]^{\beta_3}\\
&2\ar@<-0.5ex>[u]_{\beta_2}&}
\qquad 
I:\quad 
\begin{matrix}
\beta_1\alpha_3=\beta_3\alpha_1=0,\\
\beta_1\alpha_1\beta_1=\beta_1\alpha_2\beta_2, \alpha_1\beta_1\alpha_1=\alpha_2\beta_2\alpha_1, \\
\beta_3\alpha_3\beta_3=-\beta_3\alpha_2\beta_2, \alpha_3\beta_3\alpha_3=-\alpha_2\beta_2\alpha_3, \\
\beta_2\alpha_1\beta_1=\beta_2\alpha_2\beta_2=-\beta_2\alpha_3\beta_3, \\
\alpha_1\beta_1\alpha_2=\alpha_2\beta_2\alpha_2=-\alpha_3\beta_3\alpha_2.
\end{matrix}
\end{matrix}
\end{equation}
Now, we are ready to check the Schurian-finiteness. Since $\alpha_1\beta_1-\alpha_2\beta_2+\beta_3\alpha_3$, $\alpha_3\beta_3+\beta_3\alpha_3-\beta_2\alpha_2$, $\alpha_2\beta_2-\beta_3\alpha_3+\beta_2\alpha_2+\beta_1\alpha_1$ are central elements of $\mathcal{SF}_5$, we take
\[
\widetilde{\mathcal{SF}_5}:=\mathcal{SF}_5/\langle \alpha_1\beta_1, \beta_1\alpha_1, \alpha_2\beta_2, \beta_2\alpha_2, \alpha_3\beta_3, \beta_3\alpha_3 \rangle.
\]
Using Aoki's GAP-QPA code \footnote{This is beta version and no public version is available, yet.}, we find that $\widetilde{\mathcal{SF}_5}$ admits 32 Schurian modules and 124 support $\tau$-tilting modules. Thus, it is Schurian-finite.
\end{proof}

\begin{prop}\label{prop:e3defect2third}
Let $e=3$.
Then $B:=R^{\La_0+\La_1}(\delta + \alpha_1 + \alpha_2)$ is Schurian-infinite.
\end{prop}

\begin{proof}
In this case, the decomposition matrix is~(\ref{targetmatrixalt}): 
\[
\begin{matrix*}[r] 
(\varnothing,(2,1^3)) \\(\varnothing,(2^2,1)) \\ ((1^3),(2)) \\ ((2,1),(2)) \end{matrix*} \quad 
\begin{pmatrix} 1 &&& \\
v & 1 & & \\
v & 0 & 1 & \\
v^2 & v & v & 1 
\end{pmatrix} \qedhere
\]
\end{proof}

\begin{prop}\label{prop:e3defect2fourth}
Let $e=3$.
Then $B:=R^{\La_0+\La_1}(\delta+2\alpha_0+2\alpha_1)$ is Schurian-finite.
\end{prop}

\begin{proof}
The decomposition matrix of $B$ is given by 
\[
\begin{matrix*}[r]
((2^2,1^2), (1)) \\
((5), (1^2)) \\
((2^2,1), (1^2)) \\
((2,1^3), (1^2)) \\
((2), (4,1)) \\
((2), (3,2)) \\
((2), (1^5)) \\
((1), (4,2))
\end{matrix*} \quad
\begin{pmatrix}
1 &  &  &  \\
0 & 1 &  &  \\
v & v & 1 &  \\
v^2 & 0 & v &  \\
0 & v & 0 & 1 \\
0 & v^2 & v & v \\
0 & 0 & v^2 & 0 \\
0 & 0 & 0 & v^2
\end{pmatrix}.
\]
There are 4 simple modules of $B$ labelled by $((2^2,1^2), (1))$, $((5), (1^2))$, $((2^2, 1), (1^2))$, $((2), (4,1))$. Set 
\[
D_1:=\D{((2^2,1^2), (1))}, \quad 
D_2:=\D{((5), (1^2))}, \quad 
D_3:=\D{((2^2, 1), (1^2))}, \quad 
D_4:=\D{((2), (4,1))}.
\]
We find $[P_1]=f_0^{(2)}f_1f_2f_1f_0f_1v_\La$ in the deformed Fock space. Similarly, 
\[
[P_2]=f_1f_0f_2f_1f_0f_1f_0v_\Lambda, \quad
[P_3]=f_0f_1f_2f_0f_1f_0f_1v_\Lambda, \quad
[P_4]=f_1^{(2)}f_0f_2f_0f_1f_0v_\Lambda.
\]
We have the Dynkin automorphism which swaps the nodes $0$ and $1$, and the automorphism swaps $D_1$ and $D_4$, $D_2$ and $D_3$, respectively. 
We define $A:=\End_B(P_1\oplus P_2\oplus P_3\oplus P_4)^{\rm op}$ and 
\[
e_1:=x_7\psi_6e(1012100), \quad e_2:=e(0101201), \quad e_3:=e(1010210), \quad e_4:=x_7\psi_6e(0102011).
\]
We obtain the graded Cartan matrix of $A$ (i.e.~graded dimensions of $\Hom(P_i,P_j)$) from the above decomposition matrix as follows.
\[
\renewcommand\arraystretch{1.2}
\begin{tabular}{c|ccccccccccccc}
$\Hom$& $P_1$ & $P_2$ & $P_3$ & $P_4$ \\ \hline
$P_1$ & $1 + v^2 + v^4$ & $v^2$ & $v + v^3$ & 0 \\
$P_2$ & $v^2$ & $1 + 2v^2 + v^4$ & $v + v^3$ & $v + v^3$ \\
$P_3$ & $v + v^3$ & $v + v^3$ & $1 + 2v^2 + v^4$ & $v^2$ \\
$P_4$ & 0 & $v + v^3$ & $v^2$ & $1 + v^2 + v^4$ \\
\end{tabular}
\]
Observe that we have 6 degree one elements. We show that they generate $A$ by looking at the radical structure of $P_i$, for $i=1,2,3,4$. 
By the Specht filtration and the self-duality, $P_1$ and $P_4$ have the following shape. 
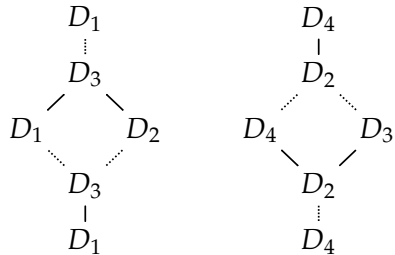
\begin{center}
$\vcenter{\xymatrix@C=0.1cm@R=0.2cm{
&D_1\ar@{.}[d]&\\
&D_3\ar@{-}[dl]\ar@{-}[dr]&\\
D_1\ar@{.}[dr]&& D_2\ar@{.}[dl]\\
&D_3\ar@{-}[d]&\\
&D_1&
}}$$\qquad$
$\vcenter{\xymatrix@C=0.1cm@R=0.2cm{
&D_4\ar@{-}[d]&\\
&D_2\ar@{.}[dl]\ar@{.}[dr]&\\
D_4\ar@{-}[dr]&& D_3\ar@{-}[dl]\\
&D_2\ar@{.}[d]&\\
&D_4&
}}$
\end{center}
Hence, we have $\Ext^1(D_1, D_1)=\Ext^1(D_4,D_4)=0$ and
\begin{gather*}
\Ext^1(D_1,D_2)=\Ext^1(D_2,D_1)=0,\quad  \Ext^1(D_3,D_4)=\Ext^1(D_4,D_3)=0,\\
\Ext^1(D_1,D_4)=\Ext^1(D_4,D_1)=0,\quad \Ext^1(D_1,D_3)=\Ext^1(D_3,D_1)=\bbf, \\
\Ext^1(D_2,D_4)=\Ext^1(D_4,D_2)=\bbf,\quad \Ext^1(D_2,D_3)=\Ext^1(D_3,D_2)=\bbf.
\end{gather*}
The last $\Ext^1(D_2,D_3)=\Ext^1(D_3,D_2)=\bbf$ needs some explanation.
Since $\spe{((2^2,1),(1^2))}$ is a factor module of $P_3$, $\spe{((2^2,1),(1^2))}$ has the simple head $D_3$.
Further, $D_3$ appears twice in $P_2$, and since $\spe{((2^2,1),(1^2))}$ is a factor module of $\Rad P_2$, $D_3$ appears in $\Rad P_2/\Rad^2 P_2$, so that the other $D_3$ appears in $\Soc^2 P_2/\Soc P_2$, which is a composition factor of $\spe{((2),(3,2))}$.
If the latter $D_3$ appeared in $\Rad P_2/\Rad^2 P_2$, then $\Rad \spe{((2^2,1),(1^2))}$ would appear in $\Soc P_2$, a contradiction.
Thus, $D_3$ appears in $\Rad P_2/\Rad^2 P_2$ only once.
On the other hand, since $\spe{((2),(4,1))}$ is a factor module of $P_4$, $\spe{((2),(4,1))}$ is uniserial of length two, whose head is $D_4$ and whose socle is $D_2$.
Hence, the other possible simple module that would occur in $\Rad P_2/\Rad^2 P_2$ is $D_4$ only. 
Then it follows that $\Ext^1(D_2,D_3)=\Ext^1(D_3,D_2)=\bbf$. 

We can also read off the shape of $\spe{((2^2,1),(1^2))}$ and $\spe{((2),(3,2))}$. 
Using $\Ext^1(D_2,D_4)=\Ext^1(D_2,D_3)=\bbf$ additionally, we see 
$\Rad P_2/\Rad^2 P_2=D_4\oplus D_3$ and $\Rad P_3/\Rad^2 P_3=D_1\oplus D_2$ by the Dynkin automorphism. 
We have proved that $A$ is generated by the 6 degree one elements, and the quiver of $A$ is as follows. 
\[
\xymatrix{1\ar@<0.5ex>[r]^{\alpha_1}&3\ar@<0.5ex>[l]^{\beta_1}\ar@<0.5ex>[r]^{\alpha_2}&2\ar@<0.5ex>[l]^{\beta_2}\ar@<0.5ex>[r]^{\alpha_3}&4 \ar@<0.5ex>[l]^{\beta_3}}
\]
Since
\[
\begin{aligned}
\dim_v e_1Be_3&=\dim_v \Hom(Be_1, Be_3)=\dim_v \Hom(P_1\langle -1\rangle, P_3)\\
&=\dim_v \Hom(P_1, P_3)\langle 1\rangle=v^2+v^4,\\
\dim_v e_3Be_1&=\dim_v \Hom(Be_3, Be_1)=\dim_v \Hom(P_3, P_1\langle -1\rangle)\\
&=\dim_v \Hom(P_3, P_1)\langle -1\rangle=1+v^2,\\
\dim_v e_2Be_3&=\dim_v e_3Be_2=v+v^3,\\
\dim_v e_2Be_4&=\dim_v \Hom(Be_2, Be_4)=\dim_v \Hom(P_2, P_4\langle -1\rangle)\\
&=\dim_v \Hom(P_2, P_4)\langle -1\rangle=1+v^2,\\
\dim_v e_4Be_2&=\dim_v \Hom(Be_4, Be_2)=\dim_v \Hom(P_4\langle -1\rangle, P_2)\\
&=\dim_v \Hom(P_4, P_2)\langle 1\rangle=v^2+v^4,
\end{aligned}
\]
we may choose the generators of $A$ as
\begin{align*}
\alpha_1:=& e(1012100)x_7\psi_6 \psi_5\psi_4 e(1010210), \quad
\;\beta_1:= e(1010210)\psi_4\psi_5 \psi_6e(1012100), \\
\alpha_2:=& e(1010210)\psi_1\psi_2\psi_3\psi_6e(0101201), \quad
\beta_2:= e(0101201)\psi_6\psi_3\psi_2\psi_1e(1010210), \\
\alpha_3:=& e(0101201) \psi_4\psi_5 \psi_6e(0102011), \quad
\quad\beta_3:= e(0102011)x_7\psi_6\psi_5\psi_4e(0101201).
\end{align*}
We check some calculations of KLR generators as follows. 
\begin{itemize}
\item $\alpha_1\alpha_2\alpha_3=\beta_3\beta_2\beta_1=0$ follows from the 
graded dimensions of $A$.

\item Since $e(1010021)=0$, we have $\psi_5e(1010201)=0$ and  $\psi_4\psi_5\psi_4e(1010201)=e(1010201)$. Then,
\[
\begin{aligned}
\beta_1\alpha_1&=e(1010210)\psi_4\psi_5 \psi_6\cdot x_7\psi_6 \psi_5\psi_4 e(1010210)\\
&=\psi_4\cdot \psi_5 \psi_6\psi_5e(1012010)\cdot \psi_4 \\
&=\psi_4 \cdot \psi_6\psi_5\psi_6e(1012010)\cdot \psi_4-\psi_4^2e(1010210)\\
&=\psi_6 \cdot \psi_4\psi_5\psi_4e(1010201)\cdot \psi_6-\psi_4^2e(1010210)\\
&=(\psi_6^2-\psi_4^2)e(1010210).
\end{aligned}
\]
Since $e(00*****)=0$, we have $\psi_2e(0101210)=0$ and $\psi_3\psi_2\psi_3e(0101210)=-e(0101210)$. Similarly, $e(11*****)=0$ gives $\psi_1\psi_2\psi_1e(1010210)=e(1010210)$ and $\psi_3\psi_2\psi_3e(1010210)=e(1010210)$. Then, we have
\[
\begin{aligned}
\alpha_2\beta_2&=e(1010210)\psi_1\psi_2\psi_3\psi_6\cdot \psi_6\psi_3\psi_2\psi_1e(1010210)\\
&=\psi_6^2\psi_1\psi_2\psi_3^2\psi_2\psi_1\cdot \psi_3\psi_2\psi_3e(1010210)\\
&=\psi_6^2\psi_1\psi_2\psi_3\cdot \psi_3\psi_2\psi_3 e(0101210)\cdot \psi_1\psi_2\psi_3\\
&=-\psi_6^2\psi_1\psi_2\psi_3\psi_1\psi_2\psi_3e(1010210)\\
&=-\psi_6^2\cdot \psi_1\psi_2\psi_1\cdot \psi_3\psi_2\psi_3e(1010210)\\
&=-\psi_6^2e(1010210).
\end{aligned}
\]
We obtain $\beta_1\alpha_1=-\alpha_2\beta_2-\psi_4^2e(1010210)$. 

\item Since $\psi_2e(0101201)=e(0011201)\psi_2=0$, we have $\psi_1\psi_2\psi_1e(0101201)=-e(0101201)$ and $\psi_3\psi_2\psi_3e(0101201)=-e(0101201)$. We also have $\psi_1\psi_2\psi_1e(1010201)=e(1010201)$. Then, 
\[
\begin{aligned}
\beta_2\alpha_2&=e(0101201)\psi_6\psi_3\psi_2\psi_1\cdot \psi_1\psi_2\psi_3\psi_6e(0101201)\\
&=-\psi_6^2\psi_3\psi_2\psi_1^2\psi_2\psi_3\cdot \psi_1\psi_2\psi_1 e(0101201)\\
&=-\psi_6^2\psi_3\psi_2\psi_1\cdot \psi_1\psi_2\psi_1 e(1010201)\cdot \psi_3\psi_2\psi_1 \\
&=-\psi_6^2\cdot \psi_3\psi_2\psi_3\cdot \psi_1 \psi_2\psi_1e(0101201) \\
&=-\psi_6^2e(0101201). \\
\end{aligned}
\]
Similarly, $\psi_5e(0101210)=e(0101120)\psi_5=0$ implies $\psi_4\psi_5\psi_4e(0101210)=-e(0101210)$, and
\[
\begin{aligned}
\alpha_3\beta_3&=e(0101201) \psi_4\psi_5 \psi_6\cdot x_7\psi_6\psi_5\psi_4e(0101201)\\
&=\psi_4\cdot \psi_5 \psi_6\psi_5e(0102101)\cdot \psi_4\\
&=\psi_4 \cdot \psi_6\psi_5\psi_6e(0102101)\cdot \psi_4+\psi_4^2e(0101201)\\
&=\psi_6 \cdot \psi_4\psi_5\psi_4e(0101210)\cdot \psi_6+\psi_4^2e(0101201)\\
&=(-\psi_6^2+\psi_4^2)e(0101201).\\
\end{aligned}
\]
We obtain $\alpha_3\beta_3=\beta_2\alpha_2+\psi_4^2e(0101201)$.

\item Since $\psi_5e(0101201)=e(0101021)\psi_5=0$, $\psi_5^2e(0101201)=(x_5-x_6)e(0101201)=0$. Similarly, $(x_5-x_6)e(0101210)=0$. It then gives 
\[
\begin{aligned}
\psi_6^3e(0101201)&=\psi_6(x_6-x_7)e(0101201)=(x_5-x_6)e(0101210)\psi_6=0,\\
\psi_6^3e(0101210)&=\psi_6(x_7-x_6)e(0101210)=(x_6-x_5)e(0101201)\psi_6=0.
\end{aligned}
\]
We deduce that
\[
\begin{aligned}
\alpha_2\beta_2\alpha_2&=-\psi_6^2\cdot \psi_1\psi_2\psi_3\psi_6e(0101201)=-\psi_1\psi_2\psi_3\cdot \psi_6^3e(0101201)=0,\\
\beta_2\alpha_2\beta_2&=-\psi_6^2e(0101201)\cdot\psi_6\psi_3\psi_2\psi_1=-\psi_6^3e(0101210)\cdot\psi_3\psi_2\psi_1=0.
\end{aligned}
\]
\item Since $\psi_1^2e(0110201)=(x_1-x_2)e(0110201)=-x_2e(0110201)$, we obtain
\[
x_2^2e(0110201)=\psi_1^4e(0110201)=\psi_1(x_2-x_1)e(1010201)\psi_1=x_1e(0110201)\psi_1^2=0.
\]
Using \cite[Lemma 11.3]{arikirep}, we have $x_3e(0110201)=-x_2e(0110201)$. We also note that 
\[
x_2e(0101201)=-x_2\psi_1\psi_2\psi_1e(0101201)=-\psi_1x_1e(1001201)\psi_2\psi_1=0,
\]
and hence, $(x_3-x_2)e(0101201)=\psi_2^2e(0101201)=0$ gives $x_3e(0101201)=0$. Moreover, 
\begin{align*}
x_4e(0101201)&=-x_4\psi_3\psi_2\psi_3e(0101201)=-\psi_3x_3e(0110201)\psi_2\psi_3 \\
&=\psi_3x_2e(0110201)\psi_2\psi_3=
x_2\psi_3\psi_2\psi_3e(0101201) \\
&=x_2e(0101201)=0.
\end{align*}
We deduce that
\[
\begin{aligned}
\beta_1\alpha_1\alpha_2&=-\alpha_2\beta_2\alpha_2-\psi_4^2\cdot \psi_1\psi_2\psi_3\psi_6e(0101201)\\
&=(x_4-x_5)e(1010210)\cdot \psi_1\psi_2\psi_3\psi_6\\
&=\psi_1\psi_2\psi_3 \psi_6(x_3-x_5)e(0101201)\\
&=-\psi_1\psi_2\psi_3 \psi_6x_5e(0101201),\\
\alpha_2\alpha_3\beta_3&=\alpha_2\beta_2\alpha_2+\psi_1\psi_2\psi_3\psi_6\cdot \psi_4^2e(0101201)\\
&=\psi_1\psi_2\psi_3\psi_6(x_4-x_5)e(0101201)\\
&=-\psi_1\psi_2\psi_3\psi_6x_5e(0101201),
\end{aligned}
\]
and thus, $\beta_1\alpha_1\alpha_2=\alpha_2\alpha_3\beta_3$. Similarly, $\beta_2\beta_1\alpha_1=\alpha_3\beta_3\beta_2$ due to the following calculation, omitting details.
\[
\begin{aligned}
\alpha_3\beta_3\beta_2&=\beta_2\alpha_2\beta_2+\psi_4^2 \psi_6\psi_3\psi_2\psi_1e(1010210)\\
\beta_2\beta_1\alpha_1&=-\beta_2\alpha_2\beta_2-\psi_6\psi_3\psi_2\psi_1\psi_4^2e(1010210)
\end{aligned}
\]

\item We note that $\psi_4e(0102011)=0$ gives $\psi_3\psi_4\psi_3e(0102011)=e(0102011)$. Since $e(0210011)=0$, $(x_2-x_3)e(0120011)=\psi_2^2e(0120011)=0$ and then
\[
x_4e(0102011)=x_4\psi_3\psi_4\psi_3e(0102011)=\psi_3x_3e(0120011)\psi_4\psi_3=x_2e(0102011).
\]
We have known in the above that $x_ie(0101201)=0$ for $i=1,2,3,4$. It turns out that
\[
\begin{aligned}
\psi_5\psi_4^3e(0101201)&=\psi_5\psi_4(x_4-x_5)e(0101201)=-\psi_5x_4e(0102101)\psi_4\\
&=-x_4e(0102011)\psi_5\psi_4=-\psi_5\psi_4x_2e(0101201)=0,\\
\psi_4^3\psi_5e(0102011)&=\psi_4x_5e(0102101)\psi_5-\psi_4x_4e(0102101)\psi_5\\
&=x_4e(0101201)\psi_4\psi_5-x_2e(0101201)\psi_4\psi_5=0.
\end{aligned}
\]
We leave the calculations of $\psi_5\psi_4^3e(1010210)=\psi_4^3\psi_5e(1012100)=0$ to the reader.
We then conclude the following relations.
\[
\begin{aligned}
\beta_3\alpha_3\beta_3&=\beta_3\beta_2\alpha_2+x_7\psi_6\psi_5\psi_4^3e(0101201)=\beta_3\beta_2\alpha_2\\
\alpha_3\beta_3\alpha_3&=\beta_2\alpha_2\alpha_3+\psi_4^3\psi_5 \psi_6e(0102011)=\beta_2\alpha_2\alpha_3\\
\beta_1\alpha_1\beta_1&=-\alpha_2\beta_2\beta_1-\psi_4^3\psi_5 \psi_6e(1012100)=-\alpha_2\beta_2\beta_1\\
\alpha_1\beta_1\alpha_1&=-\alpha_1\alpha_2\beta_2-x_7\psi_6 \psi_5\psi_4^3e(1010210)=-\alpha_1\alpha_2\beta_2\\
\end{aligned}
\]
\end{itemize}
If we replace $\beta_3$ by $\beta_3':=-\beta_3$ in the above relations, then we have obtained all relations in the quiver presentation of $\mathcal{SF}_4$. By looking at the indecomposable projective modules of $\mathcal{SF}_4$, we conclude that $A$ is isomorphic to $\mathcal{SF}_4$. Hence, $A$ is Schurian-finite.
\end{proof}

\begin{Remark}\label{rem:newMorita}
    A consequence of the above proof is that $R^{\La_0+\La_1}(\delta)$ and $R^{\La_0+\La_1}(\delta+2\alpha_0+2\alpha_1)$ are Morita equivalent.
    This is the first example we know of a Morita equivalence between such algebras that does not arise from a sequence of Scopes equivalences and twists by the sign automorphism.
\end{Remark}

Combining the results of \cref{thm:e3def2a,lem:signshift,prop:e3defect2first,prop:e3defect2second,prop:e3defect2third,prop:e3defect2fourth} yields the following theorem.

\begin{thm}\label{thm:e3def2b}
Let $e=3$. If $\La=2\La_0$ and $B$ is a defect 2 block, then $B$ is Schurian-finite.
For $\La = \La_0 + \La_1$, we record in the table below representatives of all $\SU$-classes $R^{\La_0+\La_1}(\beta)$ of defect 2 blocks, and their Schurian-finiteness is as indicated in the table. Note that $R^{\La_0+\La_1}(\delta+\alpha_0)\cong R^{\La_0+\La_1}(\delta + \alpha_1)$ and $R^{\La_0+\La_1}(\delta+\alpha_0+\alpha_2) \cong R^{\La_0+\La_1}(\delta+\alpha_1 + \alpha_2)$.
\[
\renewcommand{\arraystretch}{1.25}
\begin{array}{l l}
\hline
\beta & \text{Schurian-finiteness} \\
\hline
\delta & \text{Schurian-finite (\cref{prop:e3defect2first})} \\
\delta + \alpha_1 & \text{Schurian-finite (\cref{prop:e3defect2second})} \\
\delta + \alpha_0 & \text{Schurian-finite} \\
\delta + \alpha_1 + \alpha_2 & \text{Schurian-infinite (\cref{prop:e3defect2third})} \\
\delta + \alpha_0 + \alpha_2 & \text{Schurian-infinite} \\
\delta + 2\alpha_0 + 2\alpha_1 & \text{Schurian-finite (\cref{prop:e3defect2fourth})} \\
\hline
\end{array}
\]
\end{thm}

\subsection{Blocks of defect $3$}\label{sec:e=3defect3}
We now suppose that $\defect(B)=3$ so that $\defect(C)=\min\{\nn,\pp\}=1$ and $\hk=1$.
We note that the decomposition matrices appearing in the proofs for all of these blocks are those for column Specht modules, i.e.~$([\spe{\bla}: \D{\bmu}]_v)_{\bla,\bmu}$.

At the end of~\cref{S:ScopesSec}, we showed that up to $\SU$-equivalance, there are $20$ blocks of defect $3$ when $e=3$ by 
calculating all elements of the set \[\{\hat{D} \in \Bc_{\hk} \mid \psi(\hat{D}) \text{ is Type $1$ $1$-minimal}\}\]
corresponding to this case. 

As in \cref{sec:e=3defect2}, we can assume that $\bs=(0,0)$ or $(0,1)$, i.e.~that $\La = 2\La_0$ or $\La_0 + \La_1$, respectively.
In \cref{L:ListCases1,L:ListCases2}, we list the Scopes equivalence classes that we must resolve.

\begin{lem} \label{L:ListCases1}
Suppose that $s=0$, i.e. $(\nn,\pp)=(1,1)$. Up to $\Sc$-equivalence, there are $10$ blocks of defect $3$. For each representative $\TT$, we draw the abacus configuration $\phi^{\ast}(\TT)$ giving the minimal element of the corresponding core block.
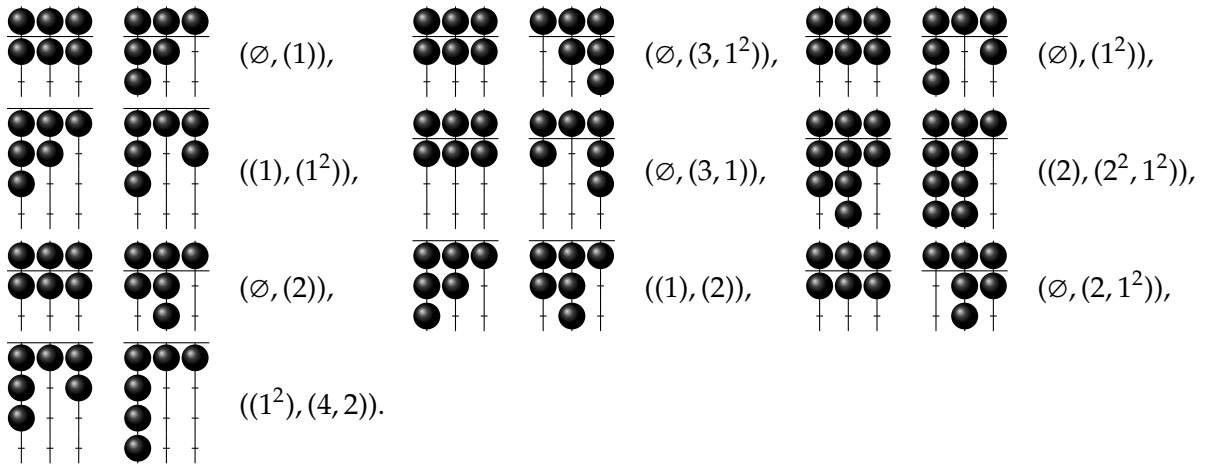
\begin{align*} 
\abacusline(3,0,bbb,bbb,nnn) \quad \abacusline(3,0,bbb,bbn,bnn) & \quad (\vn,(1)), &
\abacusline(3,0,bbb,bbb,nnn) \quad \abacusline(3,0,bbb,nbb,nnb) & \quad (\vn,(3,1^2)), & 
\abacusline(3,0,bbb,bbb,nnn) \quad \abacusline(3,0,bbb,bnb,bnn) & \quad(\vn),(1^2)), \\
\abacusline(3,1,bbb,bbn,bnn,nnn) \quad \abacusline(3,1,bbb,bnb,bnn,nnn) & \quad ((1),(1^2)), &
\abacusline(3,0,bbb,bbb,nnn,nnn) \quad \abacusline(3,0,bbb,bnb,nnb,nnn) & \quad (\vn,(3,1)), & 
\abacusline(3,0,bbb,bbb,bbn,nbn) \quad \abacusline(3,0,bbb,bbn,bbn,bbn) & \quad ((2),(2^2,1^2)), \\
\abacusline(3,0,bbb,bbb,nnn) \quad \abacusline(3,0,bbb,bbn,nbn) & \quad (\vn,(2)), &
\abacusline(3,1,bbb,bbn,bnn) \quad \abacusline(3,1,bbb,bbn,nbn) & \quad((1),(2)), & 
\abacusline(3,0,bbb,bbb,nnn) \quad \abacusline(3,0,bbb,nbb,nbn) & \quad (\vn,(2,1^2)),\\
\abacusline(3,1,bbb,bnb,bnn,nnn) \quad \abacusline(3,1,bbb,bnn,bnn,bnn) & \quad((1^2),(4,2)).
\end{align*}
\end{lem}

\begin{lem} \label{L:ListCases2}
Suppose that $s=1$, i.e. $(\nn,\pp)=(1,2)$. Up to $\SU$-equivalence, there are $5$ blocks of defect $3$. For each representative $\TT$, we draw the abacus configuration $\phi^{\ast}(\TT)$ giving the minimal element of the corresponding core block. We have reordered these $\TT$ so they appear in the order in which we deal with them in \cref{thm:e3def3s1}.

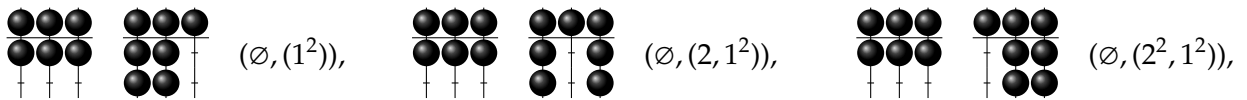
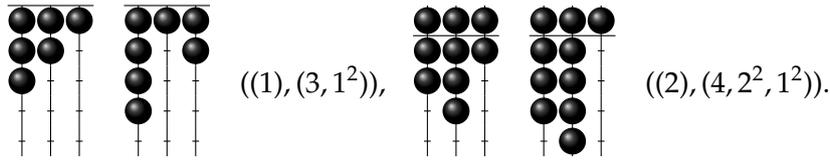
\begin{align*} 
\abacusline(3,0,bbb,bbb,nnn) \quad \abacusline(3,0,bbb,bbn,bbn) & \quad (\vn,(1^2)), & 
\abacusline(3,0,(bbb,bbb,nnn) \quad \abacusline(3,0,bbb,bnb,bnb) & \quad (\vn,(2,1^2)), &
\abacusline(3,0,bbb,bbb,nnn) \quad \abacusline(3,0,bbb,nbb,nbb) & \quad (\vn,(2^2,1^2)),  
\\
\abacusline(3,1,bbb,bbn,bnn,nnn,nnn) \quad \abacusline(3,1,bbb,bnb,bnn,bnn,nnn) & \quad ((1),(3,1^2)), &
\abacusline(3,0,bbb,bbb,bbn,nbn,nnn) \quad \abacusline(3,0,bbb,bbn,bbn,bbn,nbn) & \quad((2),(4,2^2,1^2)). & 
\end{align*}
\end{lem}

We first resolve  the $s=1$ case, as it is significantly simpler to solve.

\begin{thm}\label{thm:e3def3s1}
Let $e=3$ and $\La = \La_0+\La_1$. Then every defect 3 block $R^{\La}(\beta)$ is Schurian-infinite.
\end{thm}

\begin{proof}
We prove that each case from \cref{L:ListCases2} is Schurian-infinite.
By \cite{fayput24}, the decomposition matrices are known to be characteristic-free, so it suffices to compute the matrices in characteristic 0 using the LLT algorithm -- for each block, we give SI-subsets that give the submatrix~(\ref{targetmatrixalt}).

\begin{enumerate}
    \item 
    The algebra $R^{\La_0+\La_1}(\delta+ \alpha_0 + \alpha_1)$ has an SI-subset given by $((5),\vn)$, $((2^2,1),\vn)$, $((4),(1))$, and $((2^2),(1))$.

    \item 
    The algebra $R^{\La_0+\La_1}(2 \delta + \alpha_2)$ has an SI-subset given by $((6,1),\vn)$, $((3,2^2),\vn)$, $((4,1),(2))$, and $((3,2),(2))$.

    \item
    The algebra $R^{\La_0+\La_1}(2 \delta + 2 \alpha_1 + \alpha_2)$ has an SI-subset given by $((6,1^2),(1))$, $((3^2,2), (1))$, $((5,1^2), (2))$, and $((3^2,1), (2))$.

    \item
    The algebra $R^{\La_0+\La_1}(2 \delta + 2 \alpha_0 + \alpha_2)$ is isomorphic to the algebra in the previous case, by \cref{lem:signshift}. Alternatively, we may see directly that it has an SI-subset given by $((7,2),\vn)$, $((4,2^2,1),\vn)$, $((4,2), (3))$, and $((4,2), (2,1))$.

    \item
    The algebra $R^{\La_0+\La_1}(3\delta + 3 \alpha_0 + 3 \alpha_1)$ has an SI-subset given by $((8,3,1^2), (1^2)$, $((5,3^2,2), (1^2))$, $((5,3,1^2), (4,1))$, and $((5,3,1^2), (3,2))$.\qedhere
\end{enumerate}
\end{proof}

Next, we solve the $s=0$ case. Setting $\La=2\La_0$, one may check that there are 10 $\Sc$-classes of blocks of defect 3.
Following \cref{L:ListCases1}, we list them as follows:

\begin{multicols}{3}
\begin{enumerate}
    \item[(i)] $R^{2\La_0}(\delta+\alpha_0)$,
    \item[(iv)] $R^{2\La_0}(\delta + 2\alpha_0 + \alpha_2)$,
    \item[(vii)] $R^{2\La_0}(\delta+\alpha_0+\alpha_1)$,
    \item[(x)] $R^{2\La_0}(2\delta+3\alpha_0 + 2\alpha_2)$.
    \item[(ii)] $R^{2\La_0}(2\delta+\alpha_1+\alpha_2)$,
    \item[(v)] $R^{2\La_0}(2\delta+\alpha_2)$,
    \item[(viii)] $R^{2\La_0}(\delta + 2\alpha_0 + \alpha_1)$,
    \item[\vspace{\fill}]
    \item[(iii)] $R^{2\La_0}(\delta+\alpha_0+\alpha_2)$,
    \item[(vi)] $R^{2\La_0}(2\delta+3\alpha_0 + 2\alpha_1)$,
    \item[(ix)] $R^{2\La_0}(2\delta+\alpha_1)$,
    \item[\vspace{\fill}]
\end{enumerate}
\end{multicols}

Moreover, by \cref{lem:signshift}, the algebras appearing in (iii) and (vii) above are isomorphic, as are those appearing in (iv) and (viii), in (v) and (ix), and in (vi) and (x).
So it suffices to study those algebras appearing in (i)--(vi) above.
The ones we can currently prove are treated in \cref{prop:e3def3s0easy} below, and summarised in \cref{thm:e3def3s0}, where the Scopes classes and current status are presented succinctly.


\begin{prop}\label{prop:e3def3s0easy}
    The blocks $R^{2\La_0}(\delta+\alpha_0)$ and $R^{2\La_0}(2\delta+ \alpha_1+\alpha_2)$ are Schurian-infinite.
\end{prop}

\begin{proof}
    The submatrix of the decomposition matrix of $R^{2\La_0}(\delta+\alpha_0)$ given by all four simple modules is~(\ref{targetmatrixalt}).
    By the work of Fayers and Putignano~\cite[Corollary~1.2]{fayput24}, it is known that decomposition matrices are characteristic-free in defect 3, so the block is Schurian-infinite.

    Similarly, the submatrix of the decomposition matrix of $R^{2\La_0}(2\delta+\alpha_1+\alpha_2)$ given by all four simple modules is also~(\ref{targetmatrixalt}), so the block is Schurian-infinite as above.    
\end{proof}

The next theorem summarises our results for blocks of defect 3, listing the Scopes classes of blocks and their Schurian-finiteness, as in \cref{thm:e3def2b}.

\begin{thm}\label{thm:e3def3s0}
Let $e=3$ and $\La = 2 \La_0$.
The defect 3 blocks fall into 10 $\Sc$-classes in \cref{L:ListCases1}, or 6 up to isomorphism.
We record in the table below the 10 Scopes classes, and record those that we show are Schurian-infinite.
\[
\renewcommand{\arraystretch}{1.25}
\begin{array}{l l c l l}
\hline
 & \beta & & \beta & \text{SI} \\
\hline
 & \delta + \alpha_0 & & & \text{Schurian-infinite (\cref{prop:e3def3s0easy})} \\
 & 2\delta + \alpha_1 + \alpha_2 & & & \text{Schurian-infinite (\cref{prop:e3def3s0easy})} \\
 & \delta + \alpha_0 + \alpha_1 & \cong & \delta + \alpha_0 + \alpha_2 & \text{Unconfirmed} \\
 & \delta + 2\alpha_0 + \alpha_1 & \cong & \delta + 2\alpha_0 + \alpha_2 & \text{Unconfirmed} \\
 & 2\delta + \alpha_1 & \cong & 2\delta + \alpha_2 & \text{Unconfirmed} \\
 & 2\delta + 3\alpha_0 + 2\alpha_1 & \cong & 2\delta + 3\alpha_0 + 2\alpha_2 & \text{Unconfirmed} \\
\hline
\end{array}
\]
\end{thm}

\subsection{Blocks of defect at least $4$}\label{sec:e=3defect4+}
In this section we assume that $B$ is a block with $\defect(B) \geq 4$. By \cref{lem:e3coredefects}, we have $\defect(C) \in \{0,1\}$ so that
\[
\hk(\cB)=\left\lfloor \dfrac{\defect(B)}{2} \right\rfloor \ge 2.
\]
Recall the argument used to prove Schurian-infiniteness in \cref{thm:2ormorehooks}. If $\bmu=(\mu^{(1)},\mu^{(2)})$ is the minimal bipartition in $\cC$ and $\tilde{B}$ is the level one block with core $\mu^{(2)}$ and defect $\hk(\cB)$, then if there exists an SI-subset of $\tilde{B}$ we have that $B$ is Schurian-infinite.
Recall that we are working over an algebraically closed field $\bbf$ of characteristic $p\geq 0$. If $p \neq 2$, the argument in the proof of \cref{thm:2ormorehooks} goes through without change and we have the following result.

\begin{thm}\label{thm:2ormorehookscasesNot2}
Suppose that $e=3$ and $p\ne2$.
Suppose that $\defect(B)\geq 4$.
Then $B$ is Schurian-infinite. 
\end{thm}

When $p=2$, we must argue separately.
If we try to argue as in the proof of \cref{thm:2ormorehooks}, the block $\tilde{B}$ that arises does not always have an SI-subset, and we often had to prove that $\tilde{B}$ was Schurian-infinite by other means in \cite{als23}.
Nonetheless, we will prove the following result.

\begin{thm}\label{thm:2ormorehooksbadcases}
Suppose that $e=3$ and $p=2$.
Suppose that $B$ is a block with $\defect(B) \geq 4$.
Then $B$ is Schurian-infinite. 
\end{thm}

The argument in the proof of \cref{thm:2ormorehooks} goes through without change so long as the level one block $\tilde{B}$ that arises has an SI-subset.
In particular, if the defect is large enough, the result is easy to obtain, and we have the following proposition.

\begin{prop}\label{prop:e3p2defectlarge}
    Suppose that $e=3$ and $p=2$. Suppose that $\defect(B) \geq 8$.
    Then $B$ is Schurian-infinite.
\end{prop}

\begin{proof}
    In this case, we have bipartitions in our block which have four removable $3$-hooks. We may thus argue as in the proof of \cref{thm:2ormorehooks}, noting that the level one blocks $\tilde{B}$ that arise in this case have defect at least four.
    Following the arguments of \cite[\S6.2]{als23} in this case, we may always find an SI-subset, so that these blocks are Schurian-infinite in any characteristic.
\end{proof}

By the above proposition, it remains to resolve blocks of defect $4, 5, 6,$ or $7$, that is, whenever $\hk(\cB) \in \{2,3\}$. 
Our first task in each case is to determine which Scopes classes of blocks do not yield such a block $\tilde{B}$ that will provide an easy SI-subset, as in the proof of \cref{thm:2ormorehooks}. In the two lemmas below, we use the following notation.

Suppose that $\mu$ is a partition which has no removable $e$-rim hooks. If $\tilde{B}$ is a block of the type $\tta$ Hecke algebra with $\defect(\tilde{B})=d$, we write $\tilde{B} \sim_{\text{Sc}} \mu$ if $\tilde{B}$ is Scopes-equivalent to the block of the type $\tta$ Hecke algebra with core $\mu$ and defect $d$, and write $\tilde{B} \not\sim_{\text{Sc}} \mu$ otherwise. 

As noted at the start of \cref{sec:e3}, it suffices to assume that $s=0$ or $1$. Recall that \cref{L:s} classified which bipartitions in a core block are Kleshchev bipartitions. In particular, we have the following result.

\begin{cor}
If $\defect(C)=0$ then $\cC$ contains a unique bipartition, which is a Kleshchev bipartition. If $\defect(C)=1$ and $s=0$ then $\cC$ contains two bipartitions, one of which is a Kleshchev bipartition. If $\defect(C)=1$ and  $s=1$ then $\cC$ contains three bipartitions, two of which are Kleshchev bipartitions.
\end{cor}

\begin{lem} \label{RouqBad46}
Suppose that $\defect(B)=4$ or $\defect(B)=6$ and let $d=\hk(\cB)$. Suppose that $\cC=\{(\mu^{(1)},\mu^{(2)})\}$.  
For $k=1,2$, let $\tilde{B}^k$ be the block of the type $\tta$ Hecke algebra of defect $d$ with core $\mu^{(k)}$. 
\begin{itemize}
\item Suppose $\defect(B)=4$. 
If $\tilde{B}^k \not\sim_{\text{Sc}} (3,1^2)$ for $k=1$ or $k=2$ then $B$ is Schurian-infinite.
\item Suppose $\defect(B)=6$. 
If $\tilde{B}^k \not\sim_{\text{Sc}} (6,4,2^2,1^2)$ for $k=1$ or $k=2$ then $B$ is Schurian-infinite.
\end{itemize}
\end{lem}

\begin{proof}
Both cases are proved identically, so we will suppose that $\defect(B)=4$.
If $\tilde{B}^2 \not\sim_{\text{Sc}} (3,1^2)$, then the result follows by the same argument as the proof of \cref{thm:2ormorehooks} once again.
Suppose $\tilde{B}^2 \sim_{\text{Sc}} (3,1^2)$ but $\tilde{B}^1 \not\sim_{\text{Sc}} (3,1^2)$.
Let $\cC'=\{((\mu^{(2)},\mu^{(1)}),\ba^{\text{Sw}})\}$,
and let $B'$ be the defect 4 block with combinatorial core block $\cC'$.
Then $(\mu^{(2)},\mu^{(1)})$ is a Kleshchev bipartition, and we may therefore argue as above to see that $B'$ is Schurian-infinite.
The result now follows from \cref{L:Swapsies}.
\end{proof}

\begin{lem} \label{RouqBad57}
Suppose that $\defect(B)=5$ or $\defect(B)=7$ and let $d=\hk(\cB)$.
\begin{itemize}
\item Suppose $s=0$ and that the unique Kleshchev bipartition in $\cC$ is $(\mu^{(1)},\mu^{(2)})$. Let $\tilde{B}$ be the block of the type $\tta$ Hecke algebra of defect 2 with core $\mu^{(2)}$. 
\begin{itemize}
\item Suppose $\defect(B)=5$. If $\tilde{B} \not\sim_{\text{Sc}} (3,1^2)$ then $B$ is Schurian-infinite.
\item Suppose $\defect(B)=7$. If $\tilde{B} \not\sim_{\text{Sc}} (6,4,2^2,1^2)$ then $B$ is Schurian-infinite.
\end{itemize}
\item Suppose that $s=1$. 
Then $\cC$ contains two Kleschev bipartitions, $\bmu_1=(\mu_1^{(1)},\mu_1^{(2)})$ and $\bmu_2=(\mu_2^{(1)},\mu^{(2)}_2)$ say, where we assume $\bmu_1 \domsby \bmu_2$. For $m,k=1,2$, let $\tilde{B}^{(k)}_m$ be the block of the type $\tta$ Hecke algebra of defect $d$ with core $\mu_m^{(k)}$.
\begin{itemize}
\item Suppose that $\defect(B)=5$. If $\tilde{B}^{(k)}_{m} \not\sim_{\text{Sc}} (3,1^2)$ for some $(m,k) \in \{(1,2),(2,2),(2,1)\}$ then $B$ is Schurian-infinite. 
\item Suppose that $\defect(B)=7$. If $\tilde{B}^{(k)}_{m} \not\sim_{\text{Sc}} (6,4,2^2,1^2)$ for some $(m,k) \in \{(1,2),(2,2),(2,1)\}$ then $B$ is Schurian-infinite. 
\end{itemize}
\end{itemize}
\end{lem}

\begin{proof}
When $s=0$, the result follows immediately as in the proof of \cref{RouqBad46}.
So suppose $s=1$, and we will assume that $\defect(B)=5$ -- the case $\defect(B)=7$ is almost identical.
If $\tilde{B}^{(2)}_{1} \not\sim_{\text{Sc}} (3,1^2)$ or $\tilde{B}^{(2)}_{2} \not\sim_{\text{Sc}} (3,1^2)$, the result again follows as in the proof of \cref{RouqBad46}.
So now assume that $\tilde{B}^{(2)}_{1} \sim_{\text{Sc}} (3,1^2)$ and $\tilde{B}^{(2)}_{2} \sim_{\text{Sc}} (3,1^2)$, but that $\tilde{B}^{(1)}_{2} \not\sim_{\text{Sc}} (3,1^2)$.
Then $(\mu_2^{(2)},\mu^{(1)}_2)$ is a Kleshchev bipartition, and therefore we may argue as in the proof of \cref{RouqBad46}.
\end{proof}

\begin{rem}
    The proof of the $s=1$ statement in \cref{RouqBad57} fails in the case that $\tilde{B}^{(2)}_{m} \sim_{\text{Sc}} (3,1^2)$ for $m=1$ and $2$, $\tilde{B}^{(1)}_{2} \sim_{\text{Sc}} (3,1^2)$, and $\tilde{B}^{(1)}_{1} \not\sim_{\text{Sc}} (3,1^2)$.
    In this case, $(\mu_2^{(2)},\mu^{(1)}_2)$ is not a Kleshchev bipartition, so that we cannot apply \cref{T:StillKlesh} as we need.
    A similar statement holds for the $s=0$ situation.
    We note that $(\mu_2^{(2)'},\mu^{(1)'}_2)$ is also not a Kleshchev bipartition, so that twisting by the sign automorphism will also not help us to deduce the result.
\end{rem}

We first suppose that $\defect(B)=4$ and describe in detail how to find the representatives of the Scopes equivalence classes where Schurian-infiniteness cannot be confirmed by \cref{RouqBad46}.
Thus $\defect(C)=\min\{\nn,\pp\}=0$ and $\hk(B)=2$.
Suppose $s=0$.
Since $\defect(B)=4$, we have $\nn=\pp=0$. 
Following \cref{T:SUEquivReps}, there is a one-to-one correspondence between the Scopes classes and the set of Type 1 $2$-minimal elements $(\TT,0,0) \in \Tl$.
Take $\cC=\{\bnu\}$. 

\begin{lem} \label{L:List00}
There are 5 Type 1 $2$-minimal elements $(\TT,0,0) \in \Tl$, given by the following possibilities for $\TT$:  
\begin{align*}
&(2,2,2), &&(2,4,6),&&(4,6,2),&&(6,2,4),&&(6,4,2).
\end{align*}
Hence if $\defect(B)=4$ and $s=0$ then there are 5 Scopes classes. Below we give the unique bipartition $\bnu$ in the corresponding core blocks. 

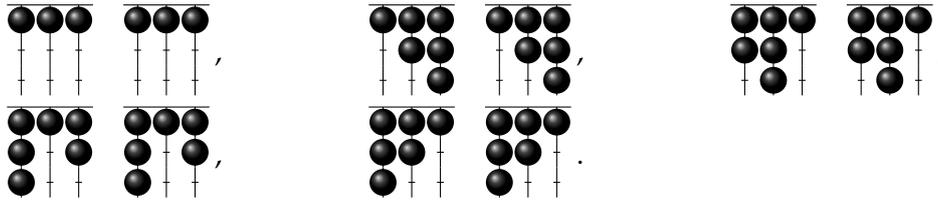
\begin{align*}
&\abacusline(3,1,bbb,nnn,nnn) \quad \abacusline(3,1,bbb,nnn,nnn) \,,
&&\abacusline(3,1,bbb,nbb,nnb) \quad \abacusline(3,1,bbb,nbb,nnb) \,,
&&\abacusline(3,1,bbb,bbn,nbn) \quad \abacusline(3,1,bbb,bbn,nbn) \,,\\
&\abacusline(3,1,bbb,bnb,bnn) \quad \abacusline(3,1,bbb,bnb,bnn) \,,
&&\abacusline(3,1,bbb,bbn,bnn) \quad \abacusline(3,1,bbb,bbn,bnn) \,.
\end{align*}
\end{lem}

Now suppose $s=1$. Since $\defect(B)=4$, we have $\min\{\nn,\pp\}=0$ and $\hk(B)=2$. 
Following \cref{T:SUEquivReps}, there is a one-to-one correspondence between the Scopes classes and the set of Type 1 $2$-minimal elements $(\TT,0,2) \in \Tl$. 

\begin{lem} \label{L:List02}
There are 21 Type 1 $2$-minimal elements $(\TT,0,2) \in \Tl$, given by the following possibilities for $\TT$:  
\begin{align*}
&(2, 3, 5), &&(2, 5, 3), &&(3, 2, 5), &&(1, 1, 2), &&(1, 2, 1), 
&&(3, 5, 2), &&(1, 3, 6), \\
&(1, 4, 5), &&(2, 1, 1), &&(5, 2, 3), 
&&(5, 3, 2), &&(3, 6, 1), &&(4, 5, 1), &&(5, 1, 4), \\
&(5, 4, 1), &&
(7, 7, 2), &&(6, 1, 3), &&(6, 3, 1), &&(6, 9, 1), &&(9, 2, 5), 
&&(7, 1, 2).
\end{align*}

Hence if $\defect(B)=4$ and $s=1$ then there are 21 Scopes classes. 
Below we give the unique bipartition $\bnu$ in the corresponding core blocks.

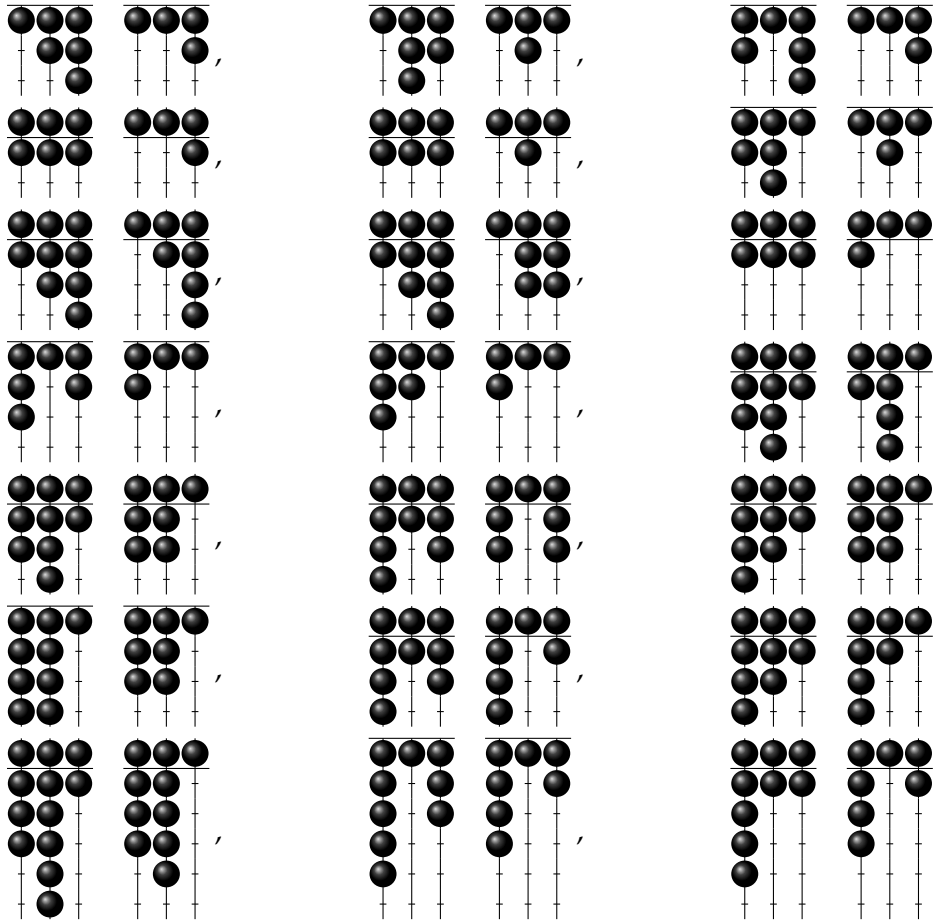
\begin{align*}
& \abacusline(3,1,bbb,nbb,nnb) \quad \abacusline(3,1,bbb,nnb,nnn) \,,
&& \abacusline(3,1,bbb,nbb,nbn) \quad \abacusline(3,1,bbb,nbn,nnn) \,,
&& \abacusline(3,1,(bbb,bnb,nnb) \quad \abacusline(3,1,bbb,nnb,nnn) \,, \\
& \abacusline(3,0,bbb,bbb,nnn) \quad \abacusline(3,0,bbb,nnb,nnn) \,, 
&& \abacusline(3,0,bbb,bbb,nnn) \quad \abacusline(3,0,bbb,nbn,nnn) \,,
&&\abacusline(3,1,bbb,bbn,nbn) \quad \abacusline(3,1,bbb,nbn,nnn) \,, \\ 
&\abacusline(3,0,bbb,bbb,nbb,nnb) \quad \abacusline(3,0,bbb,nbb,nnb,nnb) \,,
&&\abacusline(3,0,bbb,bbb,nbb,nnb) \quad \abacusline(3,0,bbb,nbb,nbb,nnn) \,, 
&&\abacusline(3,0,bbb,bbb,nnn,nnn) \quad \abacusline(3,0,bbb,bnn,nnn,nnn) \,, \\ 
& \abacusline(3,1,bbb,bnb,bnn,nnn) \quad \abacusline(3,1,bbb,bnn,nnn,nnn) \,,
&& \abacusline(3,1,bbb,bbn,bnn,nnn) \quad \abacusline(3,1,bbb,bnn,nnn,nnn) \,,
&& \abacusline(3,0,bbb,bbb,bbn,nbn) \quad \abacusline(3,0,bbb,bbn,nbn,nbn) \,,\\
& \abacusline(3,0,bbb,bbb,bbn,nbn) \quad \abacusline(3,0,bbb,bbn,bbn,nnn) \,, 
&& \abacusline(3,0,bbb,bbb,bnb,bnn) \quad \abacusline(3,0,bbb,bnb,bnb,nnn) \,, 
&& \abacusline(3,0,bbb,bbb,bbn,bnn) \quad \abacusline(3,0,bbb,bbn,bbn,nnn) \,, \\
& \abacusline(3,1,bbb,bbn,bbn,bbn) \quad \abacusline(3,1,bbb,bbn,bbn,nnn) \,,
&& \abacusline(3,0,bbb,bbb,bnb,bnn) \quad \abacusline(3,0,bbb,bnb,bnn,bnn) \,,
&&\abacusline(3,0,bbb,bbb,bbn,bnn) \quad \abacusline(3,0,bbb,bbn,bnn,bnn) \,, \\
&\abacusline(3,0,bbb,bbb,bbn,bbn,nbn,nbn) \quad \abacusline(3,0,bbb,bbn,bbn,bbn,nbn,nnn) \,, 
&&\abacusline(3,1,bbb,bnb,bnb,bnn,bnn,nnn) \quad \abacusline(3,1,bbb,bnb,bnn,bnn,nnn,nnn) \,, 
&&\abacusline(3,0,bbb,bbb,bnn,bnn,bnn,nnn) \quad \abacusline(3,0,bbb,bnb,bnn,bnn,nnn,nnn)\,. 
\end{align*}

\end{lem}

We record in the following lemma which exceptional blocks there are -- the others may be resolved as above.

\begin{lem}\label{lem:e3p2maincases}
Let $e=3$, $p=2$, and suppose that $B$ is a block with core block $C$ such that $\defect(C) = 0$ and $\defect(B) = 4$.
Suppose we are \emph{not} in one of the following cases.
\begin{enumerate}
\item $s=0$ and $B$ is Scopes equivalent to the block whose core block $\cC$ contains the unique bipartition $((3,1^2),(3,1^2))$.


\item $s=1$ and $B$ is Scopes equivalent to one of the three blocks whose core blocks contain the unique bipartitions $((3,1^2),(5,3,1^2))$, $((4,2^2,1^2),(3,1^2))$, or $((5,3,1^2),(4,2^2,1^2))$, respectively.
\end{enumerate}
Then $B$ is Schurian-infinite.
\end{lem}

\begin{proof}
The result follows from \cref{RouqBad46,L:List00} if $s=0$ and from \cref{RouqBad46,L:List02} if $s=1$. 
\end{proof}

Next, we move to blocks of defects 5, 6, and 7.
As with \cref{lem:e3p2maincases}, we have the following result which narrows down which blocks need particular attention.
The proof is entirely analogous to that of \cref{lem:e3p2maincases}.

\begin{lem} \label{lem:e3p2def567maincases5}
    Let $e=3$, $p=2$, and suppose that $B$ is a block of defect $5$ and that $B$ is \emph{not} Scopes equivalent to one of the following blocks.    
        \begin{itemize}
        \item $s=0$ and $B$ is the block whose core block $\cC$ is one of the eleven blocks containing the following bipartitions:

\begin{enumerate}
\item $\{((2,1^2),(3,1^2)),\,((3,1^2),(2,1^2))\}$,
\item $\{((3,1),(3,1^2)),\,((3,1^2),(3,1))\}$,
\item $\{(\varnothing,(3,1^2)),\,((3,1^2),\varnothing)\}$, 
\item $\{((3,1^2),(6,4,2^2,1^2)),\,((6,4,2^2,1^2),(3,1^2))\}$, $(\ast)$ 
\item $\{((2),(5,3,1^2)),\,((5,3,1^2),((2))\}$, 
\item $\{((2^2,1^2),(5,3,1^2)),\,((5,3,1^2),(2^2,1^2))\}$,
\item $\{((1^2),(4,2^2,1^2)),\,((4,2^2,1^2),(1^2))\}$,
\item $\{((1),(4,2,1^2)),\,((4,2,1^2),(1))\}$,
\item $\{((5,3,1^2),(5,3^2,2^2,1^2))\,,((5,3^2,2^2,1^2),(5,3,1^2))\}$, $(\ast)$
\item $\{((4,2),(4,2^2,1^2)),\,((4,2,1^2),(4,2))\}$, 
\item $\{((4,2^2,1^2),(7,5,3,1^2)),\,((7,5,3,1^2),(4,2^2,1^1))$. $(\ast)$
\end{enumerate}
 
        \item $s=1$ and $B$ is the block whose core block $\cC$ contains the bipartitions
        \[
            \{((5,3,1^2),(7,5,3^2,2^2,1^2)), \, ((5,3^2,2^2,1^2),\,(7,5,3,1^2)),  ((8,6,4,2^2,1^2),(4, 2^2, 1^2))\}.
        \]
        \end{itemize}
            Then $B$ is Schurian-infinite.
\end{lem}

\begin{lem} \label{lem:e3p2def567maincases6}
    Let $e=3$, $p=2$, and suppose that $B$ is a block of defect $6$ and that $B$ is \emph{not} Scopes equivalent to one of the following blocks.    
        \begin{itemize}
        \item $s=0$ and $B$ is the block whose core block $\cC$ contains the unique bipartition $((6,4,2^2,1^2),(6,4,2^2,1^2))$.
        
        \item $s=1$ and $B$ is the block whose core block $\cC$ is one of the five blocks whose core blocks contain a unique bipartition equal to  $((8,6,4,2^2,1^2), (7,5,3^2,2^2,1^2))$, $((7,5,3^2,2^2,1^2), (6,4,2^2,1^2))$, or $((6,4,2^2,1^2), (8,6,4,2^2,1^2))$, respectively.
        \end{itemize}
            Then $B$ is Schurian-infinite.
\end{lem}

\begin{lem} \label{lem:e3p2def567maincases7}
    Let $e=3$, $p=2$, and suppose that $B$ is a block of defect $7$ and that $B$ is \emph{not} Scopes equivalent to one of the following blocks. 
        \begin{itemize}
        \item $s=0$ and $B$ is the block whose core block $\cC$ is one of the ten blocks containing the following bipartitions:

\begin{enumerate}
\item $\{((3^2, 2^2, 1^2),(6, 4, 2^2, 1^2)),\, ((6, 4, 2^2, 1^2),(3^2, 2^2, 1^2))\}$, 
\item $\{((6, 4, 2),(6, 4, 2^2, 1^2)),\, ((6, 4, 2^2, 1^2),(6, 4, 2))\}$,
\item $\{((3, 1^2),(6, 4, 2^2, 1^2)),\,((6, 4, 2^2, 1^2),(3,1^2))\}$, 
\item $\{((6, 4, 2^2, 1^2),(9, 7, 5, 3^2, 2^2, 1^2)),\,((9, 7, 5, 3^2, 2^2, 1^2),(6, 4, 2^2, 1^2)\}$, $(\ast)$  
\item $\{((5, 3, 1^2),(8, 6, 4, 2^2, 1^2)),\, ((8, 6, 4, 2^2, 1^2),(5, 3, 1^2))\}$, 
\item $\{((5, 3^2, 2^2, 1^2),(8, 6, 4, 2^2, 1^2)),\,((8, 6, 4, 2^2, 1^2),(5, 3^2, 2^2, 1^2))\}$, 
\item $\{((8, 6, 4, 2^2, 1^2),(8, 6, 4^2, 3^2, 2^2, 1^2)),\,((8, 6, 4^2, 3^2, 2^2, 1^2),(8, 6, 4, 2^2, 1^2))\}$, $(\ast)$ 
\item $\{((4, 2^2, 1^2),(7, 5, 3^2, 2^2, 1^2)),\,((7, 5, 3^2, 2^2, 1^2),(4, 2^2, 1^2))\}$, 
\item $\{((7, 5, 3, 1^2), (7, 5, 3^2, 2^2, 1^2)),\,((7, 5, 3^2, 2^2, 1^2),(7, 5, 3, 1^2))\}$, 
\item $\{((7, 5, 3^2, 2^2, 1^2),(10, 8, 6, 4, 2^2, 1^2)),\,((10, 8, 6, 4, 2^2, 1^2),(7, 5, 3^2, 2^2, 1^2))\}$. $(\ast)$.
\end{enumerate}

        \item $s=1$ and $B$ is the block whose core block $\cC$ is the block containing the following three bipartitions:
        \begin{align*}
            \{&((8,6,4,2^2,1^2), (10,8,6,4^2,3^2,2^2,1^2)),\,
              ((8,6,4^2,3^2,2^2,1^2), (10,8,6,4,2^2,1^2)), \\
              & ((11,9,7,5,3^2,2^2,1^2), (7,5,3^2,2^2,1^2))\}.
        \end{align*}
        \end{itemize}
    Then $B$ is Schurian-infinite.
\end{lem}

We may further reduce the number of cases that need to be resolved. Suppose $B$ and $B'$ are blocks of the same defect with combinatorial core blocks $\cC$ and $\cC'$ respectively and that 
\[\cC'=\{(\bla',\ba) \mid (\bla,\ba) \in \cC\},\] so that $B'$ contains bipartitions conjugate to those in $B$ (although we note that $B$ and $B'$ have the same bicharge). Using \cref{lem:signshift}, we have that $B$ and $B'$ are isomorphic.
In order to prove \cref{thm:2ormorehooksbadcases}, we will find an SI-subset for each exceptional block $B$ described in \cref{lem:e3p2maincases,lem:e3p2def567maincases5,lem:e3p2def567maincases6,lem:e3p2def567maincases7} or show that $B$ is conjugate to a block where an SI-subset has been described; the theorem then follows from~\cref{prop:matrixtrick}. 

To check that a given subset is an SI-subset, we must confirm that the bipartitions are Kleshchev bipartitions, that the corresponding submatrix of the graded decomposition matrix in characteristic $0$ is equal to one of the matrices in~\cref{prop:matrixtrick}, and that the ungraded decomposition numbers agree in characteristic $0$ and characteristic $2$.
The first property may be checked using \cite[Theorem~9.5]{akt08} and the submatrix in characteristic $0$ can be computed using the higher level LLT-type algorithm of Fayers~\cite{Fay10}. It remains to show that the corresponding submatrix in characteristic $2$ does not change.

Suppose we want to compute $d^{2}_{\bla\bmu}(v)$. 
Note that if $|\la^{(1)}|=|\mu^{(1)}|$ then \cref{T:BowmanSpeyer1} shows that this computation reduces to the computation of two decomposition numbers for Hecke algebras of type $\tta$; in the cases we require, these type $\tta$ decomposition numbers will be of defect $1$ or $2$.
Decomposition numbers of blocks of type $\tta$ Hecke algebras of defect 1 are well known, and decomposition numbers blocks of type $\tta$ Hecke algebras of defect $2$ in characteristic $2$ can be found in~\cite{faywt2}. Furthermore, if $\bla$ is formed from $\bmu$ by moving one node then we know from ~\cref{L:OneNodeShift} that the graded decomposition number is independent of the characteristic of the field.  

The remaining tool that we need to compute the necessary decomposition numbers is the Jantzen--Schaper theorem \cite{jm00}.
We refer to \cite[\S2.9]{fayput24} for the setup and an explanation of how the Jantzen--Schaper coefficients $\nu_\fkp(j_{\bla\bnu})$ are computed.
Key for us is that $J_{\bla\bmu}$ is defined to be the integer
\[
J_{\bla\bmu} = \sum_{\bnu \domsby \bla} \nu_\fkp(j_{\bla\bnu}) [\rspe\bnu : \rD{\bmu}],
\]
and the following theorem.

\begin{thmc}{jm00}{Theorem~4.6}\label{thm:JS}
    Suppose $\bla$ and $\bmu$ are bipartitions, and that $\bmu$ is a Kleshchev bipartition.
    Then the decomposition number $[\rspe\bla : \rD{\bmu}]$ is at most $J_{\bla\bmu}$, and is non-zero if and only if $J_{\bla\bmu}$ is non-zero.
\end{thmc}

The decomposition matrices we give are for row Specht modules, i.e.~the entries are $[\rspe\bla : \rD{\bmu}]$.

\begin{rem}
    For many of the blocks which we find SI-subsets for in the remainder of the paper, we were also able to resolve them using `optimal sequences' and \cite[Proposition~4.2 and Corollary~4.3]{aj10}.
    One may compare dimensions of weight spaces for certain carefully-chosen residue sequences in the relevant four Specht modules (for different sets of bipartitions than those we have used above).
    However, the dimensions are incredibly large, and we used a computer to solve the cases in this way.
    Though we suspect this method works for all the necessary blocks, the proofs we provide below have the advantage that they may be directly checked by hand.
\end{rem}




\begin{prop}\label{prop:e3p2def4edgecases}
Let $e=3$ and $p=2$.
The defect 4 blocks $B$ excluded from \cref{lem:e3p2maincases} have SI-subsets, and are therefore Schurian-infinite.
\end{prop}

\begin{proof}
We will explicitly construct SI-subsets in each case.
\begin{enumerate}

\item
Suppose $s=0$, and let $B$ be the block with combinatorial core block $\cC = \{((3,1^2),(3,1^2))\}$.
Define $\bla^1, \bla^2, \bla^3,$ and $\bla^4$ to be the following four bipartitions in $B$.

\begin{align*}
\bla^1 &= (\vn,(6,4,2^2,1^2)), & \bla^2&= ((3,1),(3^2,2^2,1^2)),\\
\bla^3 &= ((3,1^2),(3,2^3,1^2)), & \bla^4&= ((3,1^2),(3^2,2,1^3)).
\end{align*}

We may compute that these give the submatrix~(\ref{targetmatrix}) in characteristic $0$ and verify that it is the same matrix in characteristic $2$ by computing the ungraded decomposition numbers using \cref{thm:JS}.
Then it's easy to check that
\[
J_{\bla^2\bla^1}= \nu_\fkp(j_{\bla^2\bla^1}) d_{\bla^1 \bla^1} = 1,
\]
so that $d_{\bla^2\bla^1}(v) = v$ even in characteristic 2.
Similarly 
\begin{align*}
J_{\bla^3\bla^1} &= \nu_\fkp(j_{\bla^3\bla^1})d_{\bla^1 \bla^1} + \nu_\fkp(j_{\bla^3\bla^2})d_{\bla^2\bla^1} = -1 +1  = 0,\\
J_{\bla^4 \bla^1} & = \nu_\fkp(j_{\bla^4\bla^1}) d_{\bla^1\bla^1}+ \nu_\fkp(j_{\bla^4\bla^2})d_{\bla^2 \bla^1} + \nu_\fkp(j_{\bla^4\bla^3}) d_{\bla^3\bla^1} = 0 +1+0 = 1.
\end{align*}
It follows that $d_{\bla^3\bla^1}(v) = 0$ and $d_{\bla^4\bla^1}(v) = v$ in characteristic 2.

For $d_{\bla^3\bla^2}(v)$, $d_{\bla^4\bla^2}(v)$, and $d_{\bla^4\bla^3}(v)$, we may apply \cref{L:OneNodeShift}, from which we see that $d_{\bla^3\bla^2}(v) = d_{\bla^4\bla^3}(v) = v$ and $d_{\bla^4\bla^2}(v) = v^2$ are all characteristic-free.
We have thus found an SI-subset for $B$.

\item 
Suppose $s=1$.
We note that the two blocks that have core blocks with bipartitions
$((4,2^2,1^2),(3,1^2))$ and $((3,1^2),(5,3,1^2))$ are isomorphic.
Thus we only need to consider two different blocks here, not three.

First, let $B$ be the block with combinatorial core block $\cC = \{((4,2^2,1^2),(3,1^2))\}$.
Define $\bla^1, \bla^2, \bla^3,$ and $\bla^4$ to be the following four bipartitions in $B$.
\begin{align*}
\bla^1 & = ((1^5),(6,4,2^2,1^2)), & \bla^2 & = ((4,2),(3^2,2^2,1^5)), \\
\bla^3 & = ((4,2),(5,4,2^2,1^2)), & \bla^4 & = ((4,2,1^3),(3^2,2^2,1^2)). 
\end{align*}
As in the previous case, we check that the matrix is~(\ref{targetmatrixaltsquare}) as follows.
Set $\bnu= ((4, 1),(6,4,2^2,1^2))$, $\bsig=((3,2),(6,4,2^2,1^2))$ and $\btau=((4,2),(6,3,2^2,1^2))$.
We have $d_{\bla^2\bla^1}(v) = 0$ by dominance order considerations  and
\[
J_{\bla^3\bla^1} = \nu_\fkp(j_{\bla^3\bla^1}) d_{\bla^1\bla^1} + \nu_\fkp(j_{\bla^3\bsig}) d_{\bsig\bla^1} + \nu_\fkp(j_{\bla^3\bnu}) d_{\bnu\bla^1} = 0 + 1 + 0,
\]
since $d_{\bsig\bla^1}=1$ and $d_{\bnu\bla^1} = 0$ by \cref{T:BowmanSpeyer1}.
It follows that $d_{\bla^3\bla^1}(v) = v$ is characteristic-free.
Next, 
\begin{align*}
J_{\bla^4\bla^1} &= \nu_\fkp(j_{\bla^4\bla^1}) d_{\bla^1\bla^1} + \nu_\fkp(j_{\bla^4\bsig}) d_{\bsig\bla^1} + 
\nu_\fkp(j_{\bla^4\bnu}) d_{\bnu\bla^1} +
\nu_\fkp(j_{\bla^4\bla^3}) d_{\bla^3\bla^1} + \nu_\fkp(j_{\bla^4\btau}) d_{\btau\bla^1}\\
&= 1+0+0+0+0 = 1
\end{align*}
so that $d_{\bla^4\bla^1}(v)=v$ is characteristic-free, as is $d_{\bla^3\bla^2}(v) = v$ by \cref{T:BowmanSpeyer1}. 
Finally 
\begin{align*}
J_{\bla^4\bla^2} &= \nu_\fkp(j_{\bla^4\bla^2}) d_{\bla^2\bla^2} + \nu_\fkp(j_{\bla^4\bla^3}) d_{\bla^3\bla^2} + \nu_\fkp(j_{\bla^4\btau}) d_{\btau\bla^2} = 1+0+0=1,\\
J_{\bla^4\bla^3} &= \nu_\fkp(j_{\bla^4\bla^3}) d_{\bla^3\bla^3} + \nu_\fkp(j_{\bla^4\btau}) d_{\btau\bla^3} = 0+0.
\end{align*}
It follows that $d_{\bla^4\bla^2}(v) =v$ and $d_{\bla^4\bla^3}(v) =0$ are characteristic-free, and we have an SI-subset for $B$.

\medskip

Finally, let $B$ be the block whose combinatorial core block is $\cC = \{((5,3,1^2),(4,2^2,1^2))\}$.
Define $\bla^1, \bla^2, \bla^3,$ and $\bla^4$ to be the following four bipartitions in $B$.
\begin{align*}
\bla^1 &= ((2^2,1^5),(7,5,3,1^2)), & \bla^2&= ((5,3,1^2),(4^2,3,1^5)),\\
\bla^3 &= ((5,3,1^2),(6,5,3,1^2)), & \bla^4&= ((5,3,1^5),(4^2,3,1^2)).
\end{align*}
As in the previous case, we check that the matrix is~(\ref{targetmatrixaltsquare}) as follows.
Set
\begin{align*}
\balph &= ((4,3,1^2),(7,5,3,1^2)), & \bbeta &= ((5,2,1^2),(7,5,3,1^2)), & \bgam &= ((5,3,1^2),(7,2^2,1^5)),\\
\bdelta &= ((5,3,1^2),(4^2,3^2,2)), & \beps &= ((5,3,1^2),(7,4,3,1^2)), & \bzeta &= ((5,3,1^5),(4,2^2,1^5)).
\end{align*}
We have $d_{\bla^2\bla^1}(v) = 0$ by dominance order considerations and
\[
J_{\bla^3\bla^1} = \nu_\fkp(j_{\bla^3\bla^1}) d_{\bla^1\bla^1} + \nu_\fkp(j_{\bla^3\balph}) d_{\balph\bla^1} + \nu_\fkp(j_{\bla^3\bbeta}) d_{\bbeta\bla^1} = 0 + 1 + d_{\bbeta\bla^1},
\]
using that $d_{\balph\bla^1} = 1$ by \cref{T:BowmanSpeyer1}.
Then
\[
J_{\bbeta\bla^1} = \nu_\fkp(j_{\bbeta\bla^1}) d_{\bla^1\bla^1} + \nu_\fkp(j_{\bbeta\balph}) d_{\balph\bla^1} = -1 + 1 =0,
\]
so that $d_{\bbeta\bla^1} = 0$, and therefore $J_{\bla^3\bla^1} = 1$ and $d_{\bla^3\bla^1}(v)=v$.

Next, 
\begin{align*}
J_{\bla^4\bla^1} &= \nu_\fkp(j_{\bla^4\bla^1}) d_{\bla^1\bla^1}
+ \nu_\fkp(j_{\bla^4\balph}) d_{\balph\bla^1}
+ \nu_\fkp(j_{\bla^4\bbeta}) d_{\bbeta\bla^1}
+ \nu_\fkp(j_{\bla^4\bla^3}) d_{\bla^3\bla^1}
+ \nu_\fkp(j_{\bla^4\beps}) d_{\beps\bla^1}\\
&= 1 + 0+0+0+0+0 = 1,
\end{align*}
so that $d_{\bla^4\bla^1}(v) = v$.
Now $d_{\bla^3\bla^2}(v) = v$ by \cref{T:BowmanSpeyer1} and
\begin{align*}
J_{\bla^4\bla^2} &= \nu_\fkp(j_{\bla^4\bla^2}) d_{\bla^2\bla^2}
+ \nu_\fkp(j_{\bla^4\bgam}) d_{\bgam\bla^2}
+ \nu_\fkp(j_{\bla^4\bdelta}) d_{\bdelta\bla^2}
+ \nu_\fkp(j_{\bla^4\bla^3}) d_{\bla^3\bla^2}
+ \nu_\fkp(j_{\bla^4\beps}) d_{\beps\bla^2}
+ \nu_\fkp(j_{\bla^4\bzeta}) d_{\bzeta\bla^2}\\
&= 1 + 0+0+0+0+ d_{\bzeta\bla^2} = 1 + d_{\bzeta\bla^2}.
\end{align*}
Then
\[
J_{\bzeta\bla^2} = \nu_\fkp(j_{\bzeta\bla^2}) d_{\bla^2\bla^2} + \nu_\fkp(j_{\bzeta\bgam}) d_{\bgam\bla^2} = 0+0,
\]
so that $d_{\bzeta\bla^2} = 0$, and thus $J_{\bla^4\bla^2} = 1$ and $d_{\bla^4\bla^2}(v) = v$.
Finally,
\[
J_{\bla^4\bla^3} = \nu_\fkp(j_{\bla^4\bla^3}) d_{\bla^3\bla^3}
+ \nu_\fkp(j_{\bla^4\beps}) d_{\beps\bla^3} = 0+0,
\]
so $d_{\bla^4\bla^3}(v) = 0$.
We have found an SI-subset for $B$, which completes the proof.\qedhere
\end{enumerate}

\end{proof}


\begin{prop}\label{prop:e3p2def5edgecases}
Let $e=3$ and $p=2$.
The defect 5 blocks $B$ excluded from \cref{lem:e3p2def567maincases5} have SI-subsets, and are therefore Schurian-infinite.
\end{prop}

\begin{proof}
    Suppose $s=0$.
    Several of the blocks consist of bipartitions that are conjugate to each other, and these blocks are therefore isomorphic, with the isomorphism given by the sign automorphism.
    Namely, this applies to: blocks (i) and (ii); blocks (v) and (vii); blocks (vi) and (x); blocks (ix) and (xi).
    So it suffices to prove the result for blocks (i), (iii), (iv), (v), (vi), (viii) and (ix) from \cref{lem:e3p2def567maincases5}.
    \begin{enumerate}
    \item
    First, suppose that $B$ is the block whose combinatorial core block is
    \[
    \cC = \{((2,1^2),(3,1^2)), ((3,1^2),(2,1^2))\}.
    \]
    Then we define $\bla^1, \bla^2, \bla^3,$ and $\bla^4$ to be the following four bipartitions in $B$.
    \begin{align*}
    \bla^1 & = (\vn,(5,4,2^2,1^2)), & \bla^2 & = ((2,1),(3^2,2^2,1^2), \\
    \bla^3 & = ((2,1^2),(3,2^3,1^2)), & \bla^4 & = ((2,1^2),(3^2,2,1^3)). 
    \end{align*}
    We may see that these give the submatrix~(\ref{targetmatrix}) in characteristic $0$ using the higher-level LLT algorithm. We may then compute the Jantzen coefficents $J_{\bla^2\bla^1}=J_{\bla^4\bla^1}=1$ and $J_{\bla^3\bla^1}=0$ to give the first column of the matrix in characteristic $2$. The entries $d_{\bla^3\bla^2}(v)$ and $d_{\bla^4 \bla^2}(v)$ are independent of characteristic by \cref{L:OneNodeShift} and computation of $d_{\bla^4\bla^3}(v)$ reduces to a Type $\tta$ defect 1 calculation by \cref{T:BowmanSpeyer1}.
    Thus, the decomposition matrix is~(\ref{targetmatrix}) in any characteristic, and therefore $B$ is Schurian-infinite.

    \item[(iii)]
    Next, suppose that $B$ is the block whose combinatorial core block is
    \[
    \cC = \{(\varnothing,(3,1^2)), ((3,1^2),\varnothing)\}.
    \]
    Then we define $\bla^1, \bla^2, \bla^3,$ and $\bla^4$ to be the following four bipartitions in $B$.
    \begin{align*}
    \bla^1 &= ((1^3),(3,1^5)), & \bla^2&= ((2,1),(3,1^5)),\\
    \bla^3 &= ((1^3),(3^2,2)), & \bla^4&= ((2,1),(3^2,2)).
    \end{align*}
    Since $\bla^2$ and $\bla^3$ are incomparable in the dominance order, $d_{\bla^3\bla^2}(v) = 0$.
    By \cref{T:BowmanSpeyer1}, we can check that
    \[
    d_{\bla^2\bla^1}(v) = d_{(\bla^2)^{(2)}(\bla^1)^{(2)}}(v), \quad%
    d_{\bla^3\bla^1}(v) = d_{(\bla^3)^{(1)}(\bla^1)^{(1)}}(v), \quad%
    d_{\bla^4\bla^1}(v) = d_{(\bla^4)^{(1)}(\bla^1)^{(1)}}(v) \times d_{(\bla^4)^{(2)}(\bla^1)^{(2)}}(v).
    \]
    Similarly, we see that
    \[
    d_{\bla^4\bla^2}(v) = d_{(\bla^4)^{(1)}(\bla^2)^{(1)}}(v), \quad%
    d_{\bla^4\bla^3}(v) = d_{(\bla^4)^{(2)}(\bla^3)^{(2)}}(v).
    \]
    Now all of the necessary decomposition numbers for level 1 partitions come from blocks of defect 1, so an easy check reveals that the decomposition matrix is~(\ref{targetmatrixalt}) in any characteristic, and therefore $B$ is Schurian-infinite.

    \item[(iv)]
    Next, suppose that $B$ is the block whose combinatorial core block is
    \[
    \cC = \{((6,4,2^2,1^2), \allowbreak (3,1^2)), ((3,1^2),(6,4,2^2,1^2))\}.
    \]
    Then we define $\bla^1, \bla^2, \bla^3,$ and $\bla^4$ to be the following four bipartitions in $B$.
    \begin{align*}
    \bla^1 &= ((3,1^5),(6,4,2^2,1^5)), & \bla^2&= ((3^2,2), (6,4,2^2,1^5)),\\
    \bla^3 &= ((3,1^5), (6,4^2,3,1^2)), & \bla^4&= ((3^2,2), (6,4^2,3,1^2)).
    \end{align*}
    The proof is now identical to case (iii) above, yielding the decomposition matrix~(\ref{targetmatrixalt}) in any characteristic, and therefore $B$ is Schurian-infinite.

    \item[(v)]
     Next, suppose that $B$ is the block whose combinatorial core block is
     \[
     \cC = \{((2),(5,3,1^2)), ((5,3,1^2),(2))\}.
     \]
     Then we define $\bla^1, \bla^2, \bla^3,$ and $\bla^4$ to be the following four bipartitions in $B$.
    \begin{align*}
    \bla^1&= ((2,1^3),(5,3,1^5)),& \bla^2&= ((2^2,1),(5,3,1^5)),\\
    \bla^3 &= ((2,1^3),(5,3^2,2)), & \bla^4 &= ((2^2,1),(5,3^2,2)).
    \end{align*}
    The proof is now identical to cases (iii) and (iv) above, yielding the decomposition matrix~(\ref{targetmatrixalt}) in any characteristic, and therefore $B$ is Schurian-infinite.
    
    \item[(vi)]
    Next, suppose that $B$ is the block whose combinatorial core block is
    \[
    \cC = \{((2^2,1^2),(5,3,1^2)), ((5,3,1^2),(2^2,1^2))\}.
    \]
    Then we define $\bla^1, \bla^2, \bla^3,$ and $\bla^4$ to be the following four bipartitions in $B$.
    \begin{align*}
    \bla^1 & = ((2),(5^2,4,2^2,1^2)), & \bla^2 & = ((2^2,1),(5,3^2,2^2,1^2), \\
    \bla^3 & = ((2^2,1^2),(5,3,2^3,1^2)), & \bla^4 & = ((2^2,1^2),(5,3^2,2,1^3)). 
    \end{align*}
    We may prove that the decomposition matrix is~(\ref{targetmatrix}) in any characteristic as in case (i), and therefore $B$ is Schurian-infinite.

    \item[(viii)]
    Next, suppose that $B$ is the block whose combinatorial core block is
    \[
    \cC = \{((1),(4,2,1^2)), ((4,2,1^2),(1))\}.
    \]
    Then we define $\bla^1, \bla^2, \bla^3,$ and $\bla^4$ to be the following four bipartitions in $B$.
    \begin{align*}
    \bla^1 &= ((1^4),(4,2,1^5)), & \bla^2&= ((2^2),(4,2,1^5)),\\
    \bla^3 &= ((1^4),(4,3^2,1)), & \bla^4&= ((2^2),(4,3^2,1)).
    \end{align*}
    The proof is now identical to cases (iii), (iv) and (v) above, yielding the decomposition matrix~(\ref{targetmatrixalt}) in any characteristic, and therefore $B$ is Schurian-infinite.

    \item[(ix)]
    Next, suppose that $B$ is the block whose combinatorial core block is
    \[
    \cC = \{((5,3^2,2^2,1^2),(5,3,1^2)), ((5,3,1^2),(5,3^2,2^2,1^2))\}.
    \]
    Then we define $\bla^1, \bla^2, \bla^3,$ and $\bla^4$ to be the following four bipartitions in $B$.
    \begin{align*}
    \bla^1 &= ((2^2,1^5),(8,6,4,2^2,1^2)), & \bla^2&= ((5,3,1^2),(5^2,4,2^2,1^5)),\\
    \bla^3 &= ((5,3,1^2),(7,6,4,2^2,1^2)), & \bla^4&= ((5,3,1^5),(5^2,4,2^2,1^2)).
    \end{align*}
    In characteristic 2, the computation of the decomposition number $d_{\bla^3\bla^2}(v)$ reduces to a Type $\tta$ computation by \cref{T:BowmanSpeyer1}. All other decomposition numbers are obtained by computing Jantzen--Schaper coefficients as in the proof of \cref{prop:e3p2def4edgecases}. We obtain the decomposition matrix~(\ref{targetmatrixaltsquare}) in both characteristic 0 and 2, therefore $B$ is Schurian-infinite.
\end{enumerate}

    \medskip

    Finally, let $s=1$ and suppose that $B$ is the block whose combinatorial core block is
    \[
    \cC = \{((8,6,4,2^2,1^2),(4, 2^2, 1^2)), ((5,3^2,2^2,1^2),(7,5,3,1^2)), ((5,3,1^2),(7,5,3^2,2^2,1^2))\}.
    \]
    Then we define $\bla^1, \bla^2, \bla^3,$ and $\bla^4$ to be the following four bipartitions in $B$.
    \begin{align*}
    \bla^1 &= ((5,3,1^5),(7,5,3^2,2^2,1^5)), & \bla^2&= ((5,3^2,2),(7,5,3^2,2^2,1^5)),\\
    \bla^3 &= ((5,3,1^5), (7,5^2,4,2^2,1^2)), & \bla^4&= ((5,3^2,2), (7,5^2,4,2^2,1^2)).
    \end{align*}
    As in cases (iii), (iv), (v), and (viii) above, we may use \cref{T:BowmanSpeyer1} to reduce the computation to level 1 computations in defect 1, yielding the submatrix~(\ref{targetmatrixalt}) in any characteristic.
    It follows that $B$ is Schurian-infinite.\qedhere

\end{proof}

\begin{prop}\label{prop:e3p2def6edgecases}
Let $e=3$ and $p=2$.
The defect 6 blocks $B$ excluded from \cref{lem:e3p2def567maincases6} have SI-subsets, and are therefore Schurian-infinite.
\end{prop}

\begin{proof}
The proof proceeds as in \cref{prop:e3p2def4edgecases,prop:e3p2def5edgecases}.

Let $s=0$, and suppose that $B$ is the block with combinatorial core block
\[
\cC = \{((6,4,2^2,1^2),(6,4,2^2,1^2))\}.
\]
Then we define $\bla^1, \bla^2, \bla^3,$ and $\bla^4$ to be the following four bipartitions in $B$.
\begin{align*}
\bla^1 &= ((3,1^5),(9,7,5,3^2,2^2,1^2)), & 
\bla^2&= ((6,4,2),(6,4^2,3^2,2^2,1^5)),\\
\bla^3 &= ((6,4,2),(6^2,5,3^2,2^2,1^2)), & 
\bla^4&= ((6,4,2,1^3),(6,4^2,3^2,2^2,1^2)).
\end{align*}
These yield the submatrix~(\ref{targetmatrixaltsquare}) in any characteristic -- we leave the explicit computations as an exercise for the reader.
It follows that $B$ is Schurian-infinite.

\medskip

When $s=1$, the second and third blocks consist of bipartitions that are conjugate to each other, and are therefore isomorphic.
So it suffices to prove the result for the first two blocks.

First, suppose that $B$ is the block with combinatorial core block
\[
\cC = \{((8,6,4,2^2,1^2), (7,5,3^2,2^2,1^2))\}.
\]
Then we define $\bla^1, \bla^2, \bla^3,$ and $\bla^4$ to be the following four bipartitions in $B$.
\begin{align*}
\bla^1 &= ((5,3^2,2^2,1^5), (10,8,6,4,2^2,1^5)), & \bla^2&= ((5^2,4,2^2,1^2), (10,8,6,4,2^2,1^5)),\\
\bla^3 &= ((5,3^2,2^2,1^5), (10,8,6,4^2,3,1^2)), & \bla^4&= ((5^2,4,2^2,1^2), (10,8,6,4^2,3,1^2)).
\end{align*}
By cutting to level 1 defect 1 blocks using \cref{T:BowmanSpeyer1}, these yield the submatrix~(\ref{targetmatrixalt}) in any characteristic, and $B$ is therefore Schurian-infinite.

Finally, suppose that $B$ is the block with combinatorial core block
\[
\cC = \{((7,5,3^2,2^2,1^2),(6,4,2^2,1^2))\}.
\]
Then we define $\bla^1, \bla^2, \bla^3,$ and $\bla^4$ to be the following four bipartitions in $B$.

\begin{align*}
\bla^1 &= ((4,2^2,1^5), (9,7,5,3^2,2^2,1^5)), & \bla^2&= ((4^2,3,1^2), (9,7,5,3^2,2^2,1^5)),\\
\bla^3 &= ((4,2^2,1^5), (9,7,5^2,4,2^2,1^2)), & \bla^4&= ((4^2,3,1^2), (9,7,5^2,4,2^2,1^2)).
\end{align*}
By cutting to level 1 defect 1 blocks using \cref{T:BowmanSpeyer1} again, these yield the submatrix~(\ref{targetmatrixalt}) in any characteristic, and $B$ is therefore Schurian-infinite.\qedhere

\end{proof}

\begin{prop}\label{prop:e3p2def7edgecases}
Let $e=3$ and $p=2$.
The defect 7 blocks $B$ excluded from \cref{lem:e3p2def567maincases7} have SI-subsets, and are therefore Schurian-infinite.
\end{prop}

\begin{proof}
    Suppose $s=0$.
    Several of the blocks consist of bipartitions that are conjugate to each other, and these blocks are therefore isomorphic, with the isomorphism given by the sign automorphism.
    Namely, this applies to: blocks (i) and (ii); blocks (v) and (viii); blocks (vi) and (ix); blocks (vii) and (x).
    So it suffices to prove the result for blocks (i) and (iii)--(vii) from \cref{lem:e3p2def567maincases7}.
    Now, all blocks may be proved as in \cref{prop:e3p2def5edgecases} cases (iii), (iv), (v), (viii), and (ix), so that we will write the combinatorial core blocks $\cC$ in each case and then give the necessary bipartitions.
    In each case, we get the submatrix~(\ref{targetmatrixalt}) in any characteristic.
    
    \begin{enumerate}
    \item
    $\cC=\{((3^2, 2^2, 1^2),(6, 4, 2^2, 1^2)),\, ((6, 4, 2^2, 1^2),(3^2, 2^2, 1^2))\}$.
    \begin{align*}
    \bla^1 &= ((3,1^5), (6,4^2,3^2,2^2,1^5)), & \bla^2&= ((3^2,2), (6,4^2,3^2,2^2,1^5)),\\
    \bla^3 &= ((3,1^5), (6^2,5,3^2,2^2,1^2)), & \bla^4&= ((3^2,2), (6^2,5,3^2,2^2,1^2)).
    \end{align*}
    
    \item[(iii)]
    $\cC=\{((3, 1^2),(6, 4, 2^2, 1^2)),\,((6, 4, 2^2, 1^2),(3,1^2))\}$. 
    \begin{align*}
    \bla^1 &= ((3^2,2,1^3), (6,4,2^2,1^5)), & \bla^2&= ((5,4,2), (6,4,2^2,1^5)),\\
    \bla^3 &= ((3^2,2,1^3), (6,4^2,3,1^2)), & \bla^4&= ((5,4,2), (6,4^2,3,1^2)).
    \end{align*}
    
    \item[(iv)]
    $\cC=\{((6, 4, 2^2, 1^2),(9, 7, 5, 3^2, 2^2, 1^2)),\,((9, 7, 5, 3^2, 2^2, 1^2),(6, 4, 2^2, 1^2)\}$.
    \begin{align*}
    \bla^1&= ((6,4,2^2,1^5), (9,7,5^2,4,2^2,1^5)),
  & \bla^2&= ((6,4^2,3,1^2), (9,7,5^2,4,2^2,1^5)),\\
    \bla^3 &= ((6,4,2^2,1^5), (9,7,5^2,4^2,3,1^2)), &   
    \bla^4 &= ((6,4^2,3,1^2), (9,7,5^2,4^2,3,1^2)).
    \end{align*}
    
    \item[(v)]
    $\cC=\{((5, 3, 1^2),(8, 6, 4, 2^2, 1^2)),\, ((8, 6, 4, 2^2, 1^2),(5, 3, 1^2))\}$.
    \begin{align*}
    \bla^1 &= ((5,3^2,2,1^3), (8,6,4,2^2,1^5)), & \bla^2&= ((5^2,4,2), (8,6,4,2^2,1^5)),\\
    \bla^3 &= ((5,3^2,2,1^3), (8,6,4^2,3,1^2)), & \bla^4&= ((5^2,4,2), (8,6,4^2,3,1^2)).
    \end{align*}
    
    \item[(vi)]
    $\cC=\{((5, 3^2, 2^2, 1^2),(8, 6, 4, 2^2, 1^2)),\,((8, 6, 4, 2^2, 1^2),(5, 3^2, 2^2, 1^2))\}$. 
    \begin{align*}
    \bla^1 &= ((5,3,1^5), (8,6,4^2,3^2,2^2,1^5)), & \bla^2&= ((5,3^2,2), (8,6,4^2,3^2,2^2,1^5)),\\
    \bla^3 &= ((5,3,1^5), (8,6^2,5,3^2,2^2,1^2)), & \bla^4&= ((5,3^2,2), (8,6^2,5,3^2,2^2,1^2)).
    \end{align*}
    
    \item[(vii)]
    $\cC=\{((8, 6, 4, 2^2, 1^2),(8, 6, 4^2, 3^2, 2^2, 1^2)),\,((8, 6, 4^2, 3^2, 2^2, 1^2),(8, 6, 4, 2^2, 1^2))\}$. 
    \begin{align*}
    \bla^1 &= ((5,3^2,2^2,1^5), (11,9,7,5,3^2,2^2,1^5)), & 
    \bla^2&= ((5^2,4,2^2,1^2), (11,9,7,5,3^2,2^2,1^5)),\\
    \bla^3 &= ((5,3^2,2^2,1^5), (11,9,7,5^2,4,2^2,1^2)), & 
    \bla^4&= ((5^2,4,2^2,1^2), (11,9,7,5^2,4,2^2,1^2)).
    \end{align*}
    \end{enumerate}

    \medskip
 
    Finally, let $s=1$ and suppose that $B$ is the block with combinatorial core block
    \begin{align*}
    \cC = \{&((11,9,7,5,3^2,2^2,1^2), (7,5,3^2,2^2,1^2)),
    ((8,6,4^2,3^2,2^2,1^2),(10,8,6,4,2^2,1^2)),\\
    &((8,6,4,2^2,1^2), (10,8,6,4^2,3^2,2^2,1^2))\}.
    \end{align*}
    Then we define $\bla^1, \bla^2, \bla^3,$ and $\bla^4$ to be the following four bipartitions in $B$.
    \begin{align*}
    \bla^1 &= ((8,6,4,2^2,1^5), (10,8,6,4^2,3^2,2^5,1^2)), & \bla^2&= ((8,6,4^2,3,1^2), (10,8,6,4^2,3^2,2^5,1^2)),\\
    \bla^3 &= ((8,6,4,2^2,1^5), (10,8,6^2,5,3^2,2^2,1^5)), & \bla^4&= ((8,6,4^2,3,1^2), (10,8,6^2,5,3^2,2^2,1^5)).
    \end{align*}
    By the same argument as the first block in this proof, these yield the submatrix~(\ref{targetmatrixalt}) in any characteristic, and $B$ is therefore Schurian-infinite.
\qedhere

\end{proof}

Thus, combining \cref{prop:e3p2defectlarge,lem:e3p2maincases,lem:e3p2def567maincases5,lem:e3p2def567maincases6,lem:e3p2def567maincases7,prop:e3p2def4edgecases,prop:e3p2def5edgecases,prop:e3p2def6edgecases,prop:e3p2def7edgecases}, we have proved \cref{thm:2ormorehooksbadcases}.
Combined with \cref{thm:2ormorehookscasesNot2}, we have that when $e=3$, all blocks $B$ satisfying $\defect(B) \geq 4$ are Schurian-infinite.

\section{Cyclotomic $q$-Schur algebras}\label{sec:Schuralg}

Finally, we will discuss the cyclotomic $q$-Schur algebras $\mathcal{S}_n$ corresponding to the type $\ttb$ Hecke algebras studied in this paper.
Recall we briefly discussed these in \cref{subsec:Schur functor}, along with the graded lifts $S^\La(\beta)$ for their blocks.
We end with the following theorem.

\begin{thm}\label{thm:Schuralg}
Let $e\geq 3$ and let $B$ be a block of the cyclotomic $q$-Schur algebra $\mathcal{S}_n$.
If $e=3$, then $B$ is Schurian-infinite if and only if $\defect(B)\geq 2$.

Suppose $e\geq 4$, and that the corresponding block of $\hhh$ is not a defect $3$ block with core block $C$ of defect $1$, and is not a defect $4$ block with core block $C$ of defect $2$ with $s=0$.
Then $B$ is Schurian-infinite in any characteristic.
\end{thm}

\begin{proof}
If $\defect(B) \leq 1$, then $B$ has finite representation type, and therefore is Schurian-finite.

Next, we note that each block of $\hhh$ is an idempotent truncation of a block of $\mathcal{S}_n$ -- this is precisely how the Schur functor $F$ is defined.
It therefore follows from \cref{reduction} that whenever the block $F(B)$ of $\hhh$ is Schurian-infinite, so is the corresponding block $B$ of $\mathcal{S}_n$.
By \cref{thm:maine3,thm:main}, the result follows for most blocks.

It remains to examine the blocks $F(B)$ of $\hhh$ that we could not show to be Schurian-infinite.
These include Schurian-finite blocks in defect $2$ when $e=3$, and blocks for which we have not yet determined that they are Schurian-infinite -- several blocks in defect $3$ when $e=3$, as well as the defect 3 and 4 families excluded in the statement of the theorem.
Recall from \cref{thm:e3def2b} that the Schurian-finite blocks in defect $2$ when $e=3$ are, up to Morita equivalence, the blocks $R^{2\La_0}(\delta)$, $R^{\La_0+\La_1}(\delta)$, $R^{\La_0+\La_1}(\delta+\alpha_1)$, and $R^{\La_0+\La_1}(\delta+2\alpha_0 + 2\alpha_1)$.
For each, we may construct SI-subsets as follows.
We note that whenever we apply \cref{T:BowmanSpeyer1} for decomposition numbers of $\mathcal{S}_n$, we may drop the assumption that $\bmu \in \Kl$, as all bipartitions index simple modules for $\mathcal{S}_n$.
The result~\cite[Main Theorem]{bs16} applies directly to $\mathcal{S}_n$.
\begin{enumerate}
    \item Let $B$ be the block such that $F(B) \cong R^{2\La_0}(\delta)$. Then we choose bipartitions
    \begin{align*}
    \bla^1 &= ((2,1), \vn), & \bla^2&= ((1^3), \vn),\\
    \bla^3 &= (\vn, (2,1)), & \bla^4&= (\vn, (1^3)).
    \end{align*}
    Then, using the Janzten--Schaper formula in conjunction with the higher-level LLT algorithm (to obtain the grading) as in \cref{prop:e3p2def4edgecases}, it is easy to show that this yields all but one entry of the submatrix (\ref{targetmatrixalt}).
    The exception is that $J_{\bla^4 \bla^1} = 2$, not 1.
    However, the entry $[\Delta(\bla^4):L(\bla^1)] = d_{\bla^4\bla^1} = q^2$ is easily obtained by the fact that $\spe{\bla^4}$ is one-dimensional.

    \item Let $B$ be the block such that $F(B) \cong R^{\La_0+\La_1}(\delta)$. Then we choose bipartitions
    \begin{align*}
    \bla^1 &= ((3), \vn), & \bla^2&= ((2,1), \vn),\\
    \bla^3 &= ((1^3),\vn), & \bla^4&= ((1^2), (1)).
    \end{align*}
    In this case, we may reduce to level 1 computations for the first 3 rows of the submatrix by \cref{T:BowmanSpeyer1}, and may deduce $[\Delta(\bla^4):L(\bla^2)]$ and $[\Delta(\bla^4):L(\bla^3)]$ by \cref{L:OneNodeShift}.
    The final entry may be deduced using the Jantzen--Schaper formula, yielding the submatrix (\ref{targetmatrix}).

    \item Let $B$ be the block such that $F(B) \cong R^{\La_0+\La_1}(\delta+\alpha_1)$. Then we choose bipartitions
    \begin{align*}
    \bla^1 &= ((2,1), (1)), & \bla^2&= ((1^3), (1)),\\
    \bla^3 &= (\vn,(2^2)), & \bla^4&= (\vn, (1^4)).
    \end{align*}
    Then the result follows exactly as in case (i) to yield the submatrix (\ref{targetmatrixalt}).

    \item Let $B$ be the block such that $F(B) \cong R^{\La_0+\La_1}(\delta+2\alpha_0 + 2\alpha_1)$. Then we choose bipartitions
    \begin{align*}
    \bla^1 &= ((2^2,1), (1^2)), & \bla^2&= ((2,1^3), (1^2)),\\
    \bla^3 &= ((2),(3,2)), & \bla^4&= ((2), (1^5)).
    \end{align*}
    In this case, we may reduce to level 1 computations for $[\Delta(\bla^2):L(\bla^1)]$ and $[\Delta(\bla^4):L(\bla^3)]$ by \cref{T:BowmanSpeyer1}, while the remaining decomposition numbers may be computed by the Jantzen--Schaper formula, yielding the submatrix (\ref{targetmatrixalt}).
\end{enumerate}

\medskip

Next, recall from \cref{thm:e3def3s0} that the blocks $F(B)$ in defect $3$ when $e=3$ that we have not yet determined to be Schurian infinite are, up to Morita equivalence, the blocks $R^{2\La_0}(\delta+\alpha_0+\alpha_1)$, $R^{2\La_0}(\delta+2\alpha_0+\alpha_1)$, $R^{2\La_0}(2\delta+\alpha_1)$, and $R^{2\La_0}(2\delta+3\alpha_0 + 2\alpha_1)$.
For each, we may construct SI-subsets for $B$ as follows.
\begin{enumerate}
    \item[(v)]
    Let $B$ be the block such that $F(B) \cong R^{2\La_0}(\delta+\alpha_0+\alpha_1)$. Then we choose bipartitions
    \begin{align*}
    \bla^1 &= ((5), \vn), & \bla^2&= ((2^2,1), \vn),\\
    \bla^3 &= ((2,1^3),\vn), & \bla^4&= ((2,1^2), (1)).
    \end{align*}
    This yields the submatrix (\ref{targetmatrix}) exactly as in case (ii) above.

    \item[(vi)]
    Let $B$ be the block such that $F(B) \cong R^{2\La_0}(\delta+2\alpha_0+\alpha_1)$. Then we choose bipartitions
    \begin{align*}
    \bla^1 &= ((1), (5)), & \bla^2&= ((1), (2^2,1)),\\
    \bla^3 &= ((1), (2,1^3)), & \bla^4&= (\vn, (2^2,1^2)).
    \end{align*}
    This yields the submatrix (\ref{targetmatrix}) exactly as in case (ii) above.

    \item[(vii)]
    Let $B$ be the block such that $F(B) \cong R^{2\La_0}(2\delta+\alpha_1)$. Then we choose bipartitions
    \begin{align*}
    \bla^1 &= ((3^2,1), \vn), & \bla^2&= ((3,1^2), (2)),\\
    \bla^3 &= ((2,1^2), (3)), & \bla^4&= ((2,1^2), (2,1)).
    \end{align*}
    We may deduce $[\Delta(\bla^2):L(\bla^1)]$, $[\Delta(\bla^3):L(\bla^1)]$, and $[\Delta(\bla^4):L(\bla^1)]$ using the Jantzen--Schaper formula.
    The entries $[\Delta(\bla^3):L(\bla^2)]$ and $[\Delta(\bla^4):L(\bla^2)]$ follow by \cref{L:OneNodeShift}, while $[\Delta(\bla^4):L(\bla^3)]$ follows by cutting to level one using \cref{T:BowmanSpeyer1}.
    This yields the submatrix (\ref{targetmatrix}).

    \item[(viii)]
    Let $B$ be the block such that $F(B) \cong R^{2\La_0}(2\delta+3\alpha_0 + 2\alpha_1)$. Then we choose bipartitions
    \begin{align*}
    \bla^1 &= ((2^2,1^2), (5)), & \bla^2&= ((2^2,1^2), (2^2,1)),\\
    \bla^3 &= ((2^2,1^2), (2,1^3)), & \bla^4&= ((2^2,1), (2^2,1^2)).
    \end{align*}
    This yields the submatrix (\ref{targetmatrix}) exactly as in case (ii) above.\qedhere
\end{enumerate}
\end{proof}

\begin{rmk}
    If we are able to prove that all defect $3$ blocks of $\hhh$ are Schurian-infinite when $e=3$, then the result for those blocks follows for free without needing to find the SI subsets above.
    Similarly, if we are able to complete the proof of \cref{conj:e>3}, then \cref{conj:mainSchuralg} will follow.
\end{rmk}

\bibliographystyle{lspaper}  
\addcontentsline{toc}{section}{\refname}
\bibliography{master}

\providecommand{\bysame}{\leavevmode\hbox to3em{\hrulefill}\thinspace}
\providecommand{\MR}{\relax\ifhmode\unskip\space\fi MR }
\providecommand{\MRhref}[2]{%
  \href{http://www.ams.org/mathscinet-getitem?mr=#1}{#2}
}
\providecommand{\href}[2]{#2}
\begin{thebibliography}{BKW11}

\bibitem[Ada16]{Ad2}
T.~Adachi, \emph{\href{http://dx.doi.org/10.1090/proc/13162}{Characterizing
  {$\tau$}-tilting finite algebras with radical square zero}}, Proc.\ Amer.\
  Math.\ Soc. \textbf{144} (2016), no.~11, 4673--4685.

\bibitem[AHSW]{ahswreptype}
S.~Ariki, B.~Hudak, L.~Song, and Q.~Wang, \emph{Representation type of higher
  level cyclotomic quiver {Hecke} algebras in affine type {C}}, Ann.\
  Represent.\ Theory, to appear.

\bibitem[AIR14]{AIR}
T.~Adachi, O.~Iyama, and I.~Reiten,
  \emph{\href{http://dx.doi.org/10.1112/S0010437X13007422}{{$\tau$}-tilting
  theory}}, Compos.\ Math. \textbf{150} (2014), no.~3, 415--452.

\bibitem[AJ10]{aj10}
S.~Ariki and N.~Jacon,
  \emph{\href{http://dx.doi.org/10.1007/978-0-8176-4697-4\_2}{{Dipper}--{James}--{Murphy's}
  conjecture for {Hecke} algebras of type {$B_n$}}}, Representation theory of
  algebraic groups and quantum groups, Progr.\ Math., vol. 284,
  Birkh\"{a}user/Springer, New York, 2010, pp.~17--32.

\bibitem[AK94]{ak94}
S.~Ariki and K.~Koike, \emph{\href{http://dx.doi.org/10.1006/aima.1994.1057}{A
  {Hecke} algebra of {$({\mathbb Z}/r{\mathbb Z})\wr{\mathfrak S}_n$} and
  construction of its irreducible representations}}, Adv.\ Math. \textbf{106}
  (1994), no.~2, 216--243.

\bibitem[AKT08]{akt08}
S.~Ariki, V.~Kreiman, and S.~Tsuchioka,
  \emph{\href{http://dx.doi.org/10.1016/j.aim.2007.11.018}{On the tensor
  product of two basic representations of {$U_v(\widehat{\mathfrak{sl}}_e)$}}},
  Adv.\ Math. \textbf{218} (2008), no.~1, 28--86.

\bibitem[ALS23]{als23}
S.~Ariki, S.~Lyle, and L.~Speyer,
  \emph{\href{http://dx.doi.org/10.1112/jlms.12808}{Schurian-finiteness of
  blocks of type {$A$} {Hecke} algebras}}, J.\ Lond.\ Math.\ Soc.\ (2)
  \textbf{108} (2023), no.~6, 2333--2376.

\bibitem[Ari17]{arikirep}
S.~Ariki,
  \emph{\href{http://dx.doi.org/10.1016/j.aim.2017.07.018}{Representation type
  for block algebras of {Hecke} algebras of classical type}}, Adv.\ Math.
  \textbf{317} (2017), 823--845.

\bibitem[Ari21]{ariki21}
\bysame, \emph{\href{http://dx.doi.org/10.1017/S1446788719000326}{Tame block
  algebras of {Hecke} algebras of classical type}}, J.\ Aust.\ Math.\ Soc.
  \textbf{111} (2021), no.~2, 179--201.

\bibitem[ASW23]{aswreptype}
S.~Ariki, L.~Song, and Q.~Wang,
  \emph{\href{http://dx.doi.org/10.1016/j.aim.2023.109329}{Representation type
  of cyclotomic quiver {Hecke} algebras of type {$A_{\ell}^{(1)}$}}}, Adv.\
  Math. \textbf{434} (2023), Paper No. 109329, 68.

\bibitem[AW21]{aokiwang21}
T.~Aoki and Q.~Wang, \emph{On $\tau$-tilting finiteness of blocks of {Schur}
  algebras}, \href{https://arxiv.org/abs/2110.02000}{arXiv:2110.02000}, 2021,
  preprint.

\bibitem[BFS13]{bfs13}
O.~Barshevsky, M.~Fayers, and M.~E. Schaps,
  \emph{\href{http://dx.doi.org/10.1007/s11856-013-0031-x}{A non-recursive
  criterion for weights of a highest-weight module for an affine {Lie}
  algebra}}, Israel J.\ Math. \textbf{197} (2013), no.~1, 237--261.

\bibitem[BK09a]{bkisom}
J.~Brundan and A.~S. Kleshchev,
  \emph{\href{http://dx.doi.org/10.1007/s00222-009-0204-8}{Blocks of cyclotomic
  {Hecke} algebras and {Khovanov}--{Lauda} algebras}}, Invent.\ Math.
  \textbf{178} (2009), no.~3, 451--484.

\bibitem[BK09b]{bk09}
\bysame, \emph{\href{http://dx.doi.org/10.1016/j.aim.2009.06.018}{Graded
  decomposition numbers for cyclotomic {Hecke} algebras}}, Adv.\ Math.
  \textbf{222} (2009), no.~6, 1883--1942.

\bibitem[BKW11]{bkw11}
J.~Brundan, A.~S. Kleshchev, and W.~Wang,
  \emph{\href{http://dx.doi.org/10.1515/CRELLE.2011.033}{Graded {Specht}
  modules}}, J.\ Reine Angew.\ Math. \textbf{655} (2011), 61--87.

\bibitem[BS18]{bs15}
C.~Bowman and L.~Speyer,
  \emph{\href{http://dx.doi.org/10.1090/tran/7054}{Kleshchev's decomposition
  numbers for diagrammatic {Cherednik} algebras}}, Trans.\ Amer.\ Math.\ Soc.
  \textbf{370} (2018), no.~5, 3551--3590.

\bibitem[BS19]{bs16}
\bysame, \emph{\href{http://dx.doi.org/10.1007/s00209-018-2222-y}{An analogue
  of row removal for diagrammatic {Cherednik} algebras}}, Math.\ Z.
  \textbf{293} (2019), no.~3-4, 935--955.

\bibitem[Del24]{d'ascopes}
A.~Dell'Arciprete,
  \emph{\href{http://dx.doi.org/10.1016/j.jpaa.2024.107639}{Equivalence of
  {$v$}-decomposition matrices for blocks of {Ariki}--{Koike} algebras}}, J.\
  Pure Appl.\ Algebra \textbf{228} (2024), no.~7, Paper No.~107639, 24.

\bibitem[DIJ19]{DIJ}
L.~Demonet, O.~Iyama, and G.~Jasso,
  \emph{\href{http://dx.doi.org/10.1093/imrn/rnx135}{{$\tau$}-tilting finite
  algebras, bricks, and {$g$}-vectors}}, Int.\ Math.\ Res.\ Not.\ IMRN (2019),
  no.~3, 852--892.

\bibitem[DJ92]{dj92}
R.~Dipper and G.~D. James,
  \emph{\href{http://dx.doi.org/10.1016/0021-8693(92)90078-Z}{Representations
  of {Hecke} algebras of type {$B_n$}}}, J.\ Algebra \textbf{146} (1992),
  no.~2, 454--481.

\bibitem[DM02]{dm02}
R.~Dipper and A.~Mathas,
  \emph{\href{http://dx.doi.org/10.1007/s002090100371}{Morita equivalences of
  {Ariki}--{Koike} algebras}}, Math.\ Z. \textbf{240} (2002), no.~3, 579--610.

\bibitem[EJR18]{ejr18}
F.~Eisele, G.~Janssens, and T.~Raedschelders,
  \emph{\href{http://dx.doi.org/10.1007/s00209-018-2067-4}{A reduction theorem
  for {$\tau$}-rigid modules}}, Math.\ Z. \textbf{290} (2018), no.~3-4,
  1377--1413.

\bibitem[EN02]{enreptype}
K.~Erdmann and D.~K. Nakano,
  \emph{\href{http://dx.doi.org/10.1090/S0002-9947-01-02848-3}{Representation
  type of {Hecke} algebras of type {$A$}}}, Trans.\ Amer.\ Math.\ Soc.
  \textbf{354} (2002), no.~1, 275--285.

\bibitem[Fay05]{faywt2}
M.~Fayers, \emph{\href{http://dx.doi.org/10.1017/S0305004105008637}{Weight two
  blocks of {Iwahori}--{Hecke} algebras in characteristic two}}, Math.\ Proc.\
  Cambridge Philos.\ Soc. \textbf{139} (2005), no.~3, 385--397.

\bibitem[Fay06a]{fay06}
\bysame, \emph{\href{http://dx.doi.org/10.1016/j.jalgebra.2006.05.006}{Weight
  two blocks of {Iwahori}--{Hecke} algebras of type {$B$}}}, J.\ Algebra
  \textbf{303} (2006), no.~1, 154--201.

\bibitem[Fay06b]{fay06wts}
\bysame, \emph{\href{http://dx.doi.org/10.1016/j.aim.2005.07.017}{Weights of
  multipartitions and representations of {Ariki}--{Koike} algebras}}, Adv.\
  Math. \textbf{206} (2006), no.~1, 112--144.

\bibitem[Fay07]{fay07core}
\bysame, \emph{\href{http://dx.doi.org/10.1007/s10801-006-0048-x}{Core blocks
  of {Ariki}--{Koike} algebras}}, J.\ Algebraic Combin. \textbf{26} (2007),
  no.~1, 47--81.

\bibitem[Fay10]{Fay10}
\bysame, \emph{\href{http://dx.doi.org/10.1016/j.jpaa.2010.02.021}{An
  {LLT}-type algorithm for computing higher-level canonical bases}}, J.\ Pure
  Appl.\ Algebra \textbf{214} (2010), no.~12, 2186--2198.

\bibitem[FP25]{fayput24}
M.~Fayers and L.~Putignano,
  \emph{\href{http://dx.doi.org/10.5802/alco.446}{Decomposition numbers for
  weight $3$ blocks of {Iwahori}--{Hecke} algebras of type {$B$}}}, Algebr.\
  Comb. \textbf{8} (2025), no.~5, 1251--1284.

\bibitem[FS16]{fs16}
M.~Fayers and L.~Speyer,
  \emph{\href{http://dx.doi.org/10.1007/s10801-016-0674-x}{Generalised column
  removal for graded homomorphisms between {Specht} modules}}, J.\ Algebraic
  Combin. \textbf{44} (2016), no.~2, 393--432.

\bibitem[Geu]{StringApplet}
J.~Geuenich, \emph{String applet: Web applet for special biserial algebras},
  \url{https://www.math.uni-bielefeld.de/~jgeuenich/string-applet/}, Bielefeld
  University.

\bibitem[GL96]{GL98}
J.~J. Graham and G.~I. Lehrer,
  \emph{\href{http://dx.doi.org/10.1007/BF01232365}{Cellular algebras}},
  Invent.\ Math. \textbf{123} (1996), no.~1, 1--34.

\bibitem[HM10]{hm10}
J.~Hu and A.~Mathas,
  \emph{\href{http://dx.doi.org/10.1016/j.aim.2010.03.002}{Graded cellular
  bases for the cyclotomic {Khovanov}--{Lauda}--{Rouquier} algebras of
  type~{$A$}}}, Adv.\ Math. \textbf{225} (2010), no.~2, 598--642.

\bibitem[HM16]{hmseminorm}
\bysame, \emph{\href{http://dx.doi.org/10.1007/s00208-015-1242-8}{Seminormal
  forms and cyclotomic quiver {Hecke} algebras of type~{$A$}}}, Math.\ Ann.
  \textbf{364} (2016), no.~3, 1189--1254.

\bibitem[JM00]{jm00}
G.~D. James and A.~Mathas,
  \emph{\href{http://dx.doi.org/10.1090/S0002-9947-00-02492-2}{The {Jantzen}
  sum formula for cyclotomic {$q$}-{Schur} algebras}}, Trans.\ Amer.\ Math.\
  Soc. \textbf{352} (2000), no.~11, 5381--5404.

\bibitem[Jos99]{jost_typeBHeckeAlg}
T.~Jost, \emph{\href{http://dx.doi.org/10.1006/jabr.1998.7802}{Morita
  equivalence for blocks of {Hecke} algebras of type {$B$}}}, J.\ Algebra
  \textbf{217} (1999), no.~1, 95--113.

\bibitem[KK12]{kk12}
S.-J. Kang and M.~Kashiwara,
  \emph{\href{http://dx.doi.org/10.1007/s00222-012-0388-1}{Categorification of
  highest weight modules via {Khovanov}--{Lauda}--{Rouquier} algebras}},
  Invent.\ Math. \textbf{190} (2012), no.~3, 699--742.

\bibitem[KL09]{kl09}
M.~Khovanov and A.~D. Lauda,
  \emph{\href{http://dx.doi.org/10.1090/S1088-4165-09-00346-X}{A diagrammatic
  approach to categorification of quantum groups {I}}}, Represent.\ Theory
  \textbf{13} (2009), 309--347.

\bibitem[KMR12]{kmr}
A.~S. Kleshchev, A.~Mathas, and A.~Ram,
  \emph{\href{http://dx.doi.org/10.1112/plms/pds019}{Universal graded {Specht}
  modules for cyclotomic {Hecke} algebras}}, Proc.\ Lond.\ Math.\ Soc.\ (3)
  \textbf{105} (2012), no.~6, 1245--1289.

\bibitem[LM07]{lm07}
S.~Lyle and A.~Mathas,
  \emph{\href{http://dx.doi.org/10.1016/j.aim.2007.06.008}{Blocks of cyclotomic
  {Hecke} algebras}}, Adv.\ Math. \textbf{216} (2007), no.~2, 854--878.

\bibitem[Lyl24]{Lyle24coreblocks}
S.~Lyle, \emph{\href{http://dx.doi.org/10.1080/00927872.2023.2278669}{Core
  blocks for {Hecke} algebras of type {$B$} and sign sequences}}, Comm.\
  Algebra \textbf{52} (2024), no.~5, 1965--1981.

\bibitem[Mak14]{Mak14}
R.~Maksimau,
  \emph{\href{http://dx.doi.org/10.1016/j.jalgebra.2014.02.029}{Quiver {Schur}
  algebras and {Koszul} duality}}, J.\ Algebra \textbf{406} (2014), 91--133.

\bibitem[Mat98]{m98}
A.~Mathas, \emph{\href{http://dx.doi.org/10.1090/pspum/063/1603195}{Simple
  modules of {Ariki}--{Koike} algebras}}, Group representations: cohomology,
  group actions and topology ({Seattle}, {WA}, 1996), Proc.\ Sympos.\ Pure
  Math., vol.~63, Amer.\ Math.\ Soc., Providence, RI, 1998, pp.~383--396.

\bibitem[Mat99]{Mathas}
\bysame, \emph{\href{http://dx.doi.org/10.1090/ulect/015}{Iwahori--{Hecke}
  algebras and {Schur} algebras of the symmetric group}}, University Lecture
  Series, vol.~15, American Mathematical Society, Providence, RI, 1999.

\bibitem[Mat04]{m04surv}
\bysame, \emph{\href{http://dx.doi.org/10.2969/aspm/04010261}{The
  representation theory of the {Ariki}--{Koike} and cyclotomic {$q$}-{Schur}
  algebras}}, Representation theory of algebraic groups and quantum groups,
  Adv.\ Stud.\ Pure Math., vol.~40, Math.\ Soc.\ Japan, Tokyo, 2004,
  pp.~261--320.

\bibitem[Mat15]{m14surv}
\bysame, \emph{\href{http://dx.doi.org/10.1142/9789814651813\_0005}{Cyclotomic
  quiver {Hecke} algebras of type~{$A$}}}, Modular representation theory of
  finite and {$p$}-adic groups, Lect.\ Notes Ser.\ Inst.\ Math.\ Sci.\ Natl.\
  Univ.\ Singap., vol.~30, World Sci.\ Publ., Hackensack, NJ, 2015,
  pp.~165--266.

\bibitem[MSS23]{mss22}
R.~Muth, L.~Speyer, and L.~Sutton,
  \emph{\href{http://dx.doi.org/10.1007/s10468-021-10093-3}{Decomposable
  {Specht} modules indexed by bihooks {II}}}, Algebr.\ Represent.\ Theory
  \textbf{26} (2023), no.~1, 241--280.

\bibitem[OZ22]{OZ22}
S.~Opper and A.~Zvonareva,
  \emph{\href{http://dx.doi.org/10.1016/j.aim.2022.108341}{Derived equivalence
  classification of {Brauer} graph algebras}}, Adv.\ Math. \textbf{402} (2022),
  Paper No. 108341, 59.

\bibitem[Rou08a]{rouq}
R.~Rouquier, \emph{$2$-{Kac}--{Moody} algebras},
  \href{http://arxiv.org/abs/0812.5023}{arXiv:0812.5023}, 2008, preprint.

\bibitem[Rou08b]{rouq08}
\bysame,
  \emph{\href{http://dx.doi.org/10.17323/1609-4514-2008-8-1-119-158}{{$q$}-{Schur}
  algebras and complex reflection groups}}, Mosc.\ Math.\ J. \textbf{8} (2008),
  no.~1, 119--158.

\bibitem[Spe25]{lsstrictlywild}
L.~Speyer, \emph{\href{http://dx.doi.org/10.1112/blms.70115}{Wild blocks of
  type {$A$} {Hecke} algebras are strictly wild}}, Bull.\ London Math.\ Soc.
  \textbf{57} (2025), no.~9, 2658--2679.

\bibitem[SS20]{ss20}
L.~Speyer and L.~Sutton,
  \emph{\href{http://dx.doi.org/10.2140/pjm.2020.304.655}{Decomposable {Specht}
  modules indexed by bihooks}}, Pacific J.\ Math. \textbf{304} (2020), no.~2,
  655--711.

\bibitem[Wan22]{qiwang22}
Q.~Wang, \emph{\href{http://dx.doi.org/10.1016/j.jpaa.2021.106818}{On
  {$\tau$}-tilting finiteness of the {Schur} algebra}}, J.\ Pure Appl.\ Algebra
  \textbf{226} (2022), no.~1, Paper No.~106818, 29.

\bibitem[Web24]{websterScopes}
B.~Webster, \emph{\href{http://dx.doi.org/10.4171/jca/88}{{RoCK} blocks for
  affine categorical representations}}, J.\ Comb.\ Algebra (2024), to appear.

\end{thebibliography}

\end{document}